\documentclass[letterpaper,11pt,oneside,reqno]{amsart}
\usepackage{amsmath,amssymb,amsthm,amsfonts}
\usepackage{hyperref}
\usepackage{graphicx,color,colortbl}
\usepackage{upgreek}
\usepackage[mathscr]{euscript}
\usepackage{hhline}
\DeclareMathAlphabet{\mathpzc}{OT1}{pzc}{m}{it}
\usepackage{mathbbol,mathtools}
\usepackage[labelfont={rm}]{subfig}
\usepackage{bm}
\allowdisplaybreaks
\numberwithin{equation}{section}

\usepackage{tikz}
\usetikzlibrary{
	shapes,
	arrows,
	positioning,
	decorations.markings,
	decorations.pathmorphing,
	circuits.logic.US,
	circuits.logic.IEC,
	fit,
	calc,
	plotmarks,
	matrix
}

\tikzset{
>=stealth',
help lines/.style={dashed, thick},
axis/.style={<->},
important line/.style={thick},
connection/.style={thick, dotted},
punkt/.style={
rectangle,
rounded corners,
draw=black, thick,
text width=4.5em,
minimum height=2em,
text centered,
},
pil/.style={
->,
thick,
gray,
shorten <=2pt,
shorten >=2pt,}
}

\usepackage{array}
\usepackage{adjustbox}
\usepackage{cleveref}
\usepackage{enumitem}

\usepackage[DIV=12]{typearea}

\newtheorem{proposition}{Proposition}[section]
\newtheorem{lemma}[proposition]{Lemma}
\newtheorem{corollary}[proposition]{Corollary}
\newtheorem{theorem}[proposition]{Theorem}
\newtheorem*{theorem*}{Theorem}
\newtheorem{conjecture}[proposition]{Conjecture}
\theoremstyle{definition}
\newtheorem{definition}[proposition]{Definition}

\newtheorem{remark}[proposition]{Remark}
\newtheorem*{remark*}{Remark}
\newtheorem{example}[proposition]{Example}

\newcommand{\rightcornertext}{\pmb{\lrcorner}}

\newcommand{\leftwalltext}{\bm{\vdash}}

\newcommand{\rightwalltext}{\bm{\dashv}}

\newcommand{\bulktext}{\bm{+}}

\newcommand{\rightcorner}{%
\tikz[]{\draw[thick] (0,0) -- (.1,0) -- ( .1,.1 );}%
}   

\newcommand{\rightwall}{%
\tikz[]{\draw[thick] (0,0.05) -- (.1,0.05); \draw[thick] (.1, 0) -- (.1,.1);}%
}

\newcommand{\bulk}{%
	\tikz[baseline=-.6,scale=1.2]{\draw[thick] (0,0.05) -- (.1,0.05); \draw[thick] (0.05, 0) --++ (0,.1);}%
}

\newcommand{\leftwall}{%
\tikz[]{\draw[thick] (0,0.05) -- (.1,0.05); \draw[thick] (0, 0) -- (0,.1);}%
}
\newcommand{\Ufwd}{%
\mathsf{U}^{\mathrm{fwd}}%
}

\newcommand{\Ubwd}{%
\mathsf{U}^{\mathrm{bwd}}%
}

\newcommand{\floor}[1]{\left\lfloor #1 \right\rfloor}

\newcommand{\doubleunderline}[1]{\underline{\underline{#1}}}

\newcommand{\smallO}{{\mathsf{o}}}
\newcommand{\bigO}{{\mathsf{O}}}

\begin{document}
\title{Spin $q$-Whittaker polynomials and deformed quantum Toda}

\author[M. Mucciconi]{Matteo Mucciconi}\address{M. Mucciconi, 
Department of Physics,
Tokyo Institute of Technology, Tokyo, 152-8551 Japan}\email{matteomucciconi@gmail.com}

\author[L. Petrov]{Leonid Petrov}\address{L. Petrov, Department of Mathematics, University of Virginia, Charlottesville,
VA, 22904 USA,
	and
	Institute for Information Transmission
	Problems, Moscow, 117279 Russia}\email{lenia.petrov@gmail.com}

\date{}

\begin{abstract}
	Spin $q$-Whittaker symmetric polynomials 
	labeled by partitions $\lambda$ 
	were recently
	introduced by Borodin and Wheeler
	\cite{BorodinWheelerSpinq}
	in the context of integrable $\mathfrak{sl}_2$
	vertex models.
	They are 
	a one-parameter deformation of
	the $t=0$ Macdonald polynomials.
	We present a new more convenient modification of 
	spin $q$-Whittaker polynomials
	and find two Macdonald type $q$-difference operators
	acting diagonally in these polynomials
	with
	eigenvalues,
	respectively,
	$q^{-\lambda_1}$
	and $q^{\lambda_N}$ (where $\lambda$ is the polynomial's label).
	We study probability measures 
	on interlacing arrays
	based on spin $q$-Whittaker polynomials, 
	and
	match their observables
	with known stochastic particle systems
	such as the $q$-Hahn TASEP.
	
	In a scaling limit as $q\nearrow 1$, 
	spin $q$-Whittaker polynomials
	turn into a new one-parameter deformation
	of the $\mathfrak{gl}_n$ Whittaker functions.
	The rescaled Pieri type rule 
	gives rise to a one-parameter deformation 
	of the
	quantum Toda Hamiltonian. The deformed Hamiltonian acts diagonally on
	our new spin Whittaker functions.
	On the stochastic side, as $q\nearrow 1$
	we discover a
	multilevel extension of the beta polymer
	model of Barraquand and Corwin \cite{CorwinBarraquand2015Beta},
	and relate it to spin Whittaker functions.
\end{abstract}

\maketitle

\setcounter{tocdepth}{1}
\tableofcontents
\setcounter{tocdepth}{3}

\section{Introduction}
\label{sec:introduction}

\subsection{Overview}

This paper deals with new classes of symmetric
functions inspired by 
the $U_q(\widehat{\mathfrak{sl}_2})$ Yang--Baxter equation 
and applications to integrable stochastic
interacting particle systems and random polymer models.

\medskip

Symmetric functions have been very useful in 
studying integrable
stochastic systems
in the past two decades,
starting from the works on asymptotic fluctuations
in longest increasing subsequences of random permutations
\cite{baik1999distribution} and
the TASEP (totally asymmetric simple exclusion process)
\cite{johansson2000shape},
and continuing through the frameworks of Schur processes
\cite{okounkov2001infinite}, \cite{okounkov2003correlation} and
Macdonald processes \cite{BorodinCorwin2011Macdonald}.
Here and below by a \emph{process} associated with a family of symmetric
functions (like Schur or Macdonald)
we mean a 
probability measure
on sequences of partitions 
with probability weights expressed through these 
functions in a certain way (cf. 
\Cref{def:F_G_process} in the text).
See the scheme of symmetric functions in \Cref{fig:scheme}.

More recently, quantum integrability 
(in the form of the Yang--Baxter equation /
Bethe ansatz \cite{baxter2007exactly}) has brought new structures 
allowing to extend the range of exactly solvable stochastic systems
to 
the ASEP (partially asymmetric simple exclusion process)
\cite{TW_ASEP1}, \cite{TW_ASEP2}
and stochastic vertex models 
\cite{BCG6V}, \cite{CorwinPetrov2015},
\cite{CorwinTsai2015KPZ},
\cite{lin2020kpz},
and discover new asymptotic phenomena
around the Kardar-Parisi-Zhang universality class
\cite{Corwin2016Notices}.
In the process of
exploring quantum integrability from these perspectives,
two new families of symmetric functions were discovered:
\begin{itemize}
	\item Spin Hall--Littlewood symmetric
		functions \cite{Borodin2014vertex}. They are a one-parameter
		generalization of the classical Hall--Littlewood polynomials
		\cite[Ch. III]{Macdonald1995}, 
		and are Bethe Ansatz 
		eigenfunctions of a number of integrable stochastic systems,
		including ASEP (under a certain choice of parameters).
		These functions retain many properties
		of Hall--Littlewood polynomials including
		Cauchy type summation identities, Pieri type rules, torus scalar product orthogonality,
		and the presence of difference operators acting on them diagonally
		\cite{BCPS2014_arXiv_v4}, \cite{BorodinPetrov2016inhom},
		\cite[Section 8]{BufetovMucciconiPetrov2018}.
	\item Spin $q$-Whittaker polynomials
		\cite{BorodinWheelerSpinq}.
		They form a one-parameter 
		generalization of the 
		$q$-deformed $\mathfrak{gl}_n$ Whittaker functions
		\cite{GLO2010}, and 
		also possess Cauchy type summation identities, Pieri type rules,
		and certain first-order difference operators acting on them diagonally
		\cite[Section 8]{BufetovMucciconiPetrov2018}.
		Notably, torus orthogonality relation for 
		spin $q$-Whittaker polynomials  is not known.
\end{itemize}

Marginals of 
spin Hall--Littlewood
and
spin $q$-Whittaker processes
are matched in distribution to various
$U_q(\widehat{\mathfrak{sl}_2})$ 
stochastic vertex models, including
the stochastic six vertex model
\cite{GwaSpohn1992},
\cite{BCG6V} (with its ``dynamic'' extension \cite{BufetovPetrovYB2017});
the stochastic higher spin six vertex model
\cite{CorwinPetrov2015},
\cite{BorodinPetrov2016inhom};
and the more recent
$_4\phi_3$ stochastic vertex model
\cite{BufetovMucciconiPetrov2018}
(which is close to the
$q$-Hahn PushTASEP of \cite{CMP_qHahn_Push}).

\begin{figure}[t]
	\centering
	\includegraphics[width=\textwidth]{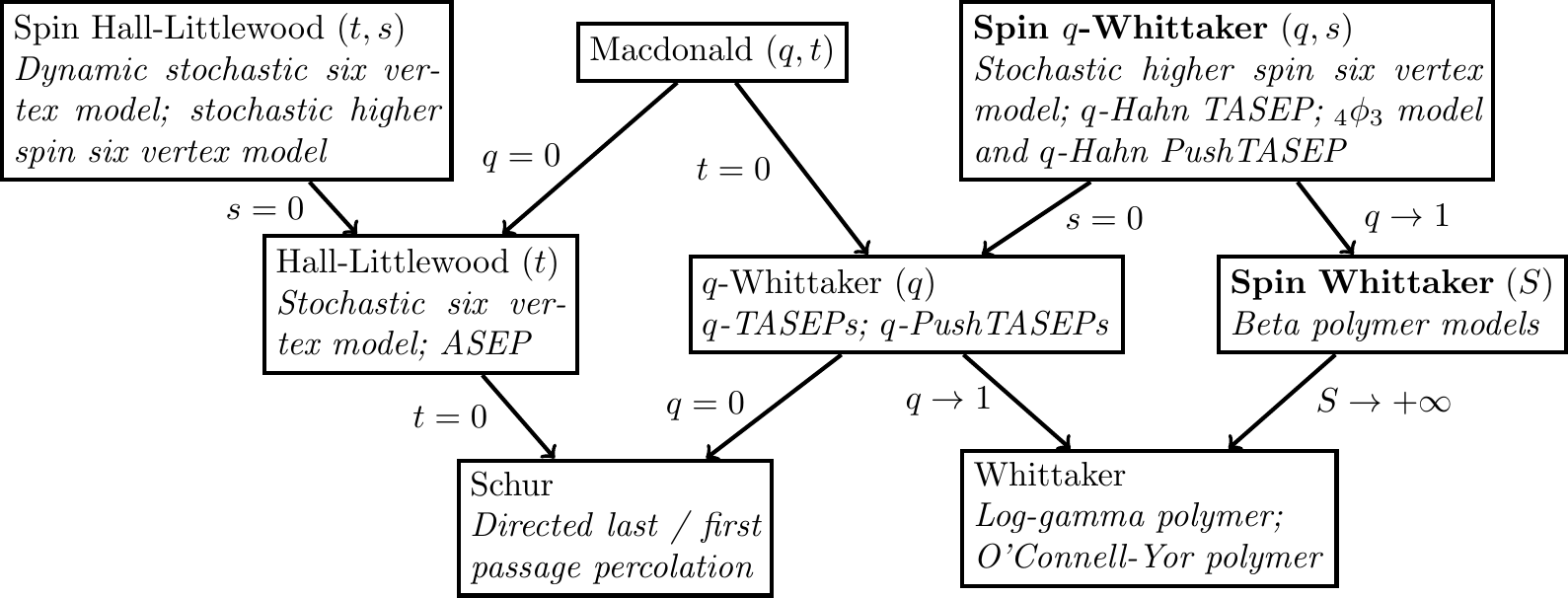}
	\caption{A scheme of various families of symmetric
	functions together with stochastic systems based on them.
	Arrows indicate degenerations or scaling limits.
	The two families which are our main focus are indicated in bold.}
	\label{fig:scheme}
\end{figure}

\medskip

The (undeformed) $q$-Whittaker polynomials
admit a nontrivial scaling limit as $q\to 1$.
In this limit
\cite{GerasimovLebedevOblezin2011}, \cite[Theorem 4.1.7]{BorodinCorwin2011Macdonald}, 
the $q$-Whittaker polynomials become 
the 
$\mathfrak{gl}_N$ Whittaker functions.
The latter play an important role in representation theory and 
integrable systems
\cite{Kostant1977Whitt}, \cite{givental1996stationary},
\cite{Etingof1999}. 
In particular, the Whittaker functions 
$\psi_{\underline \lambda}(\underline{u}_N)$, 
$\underline{\lambda}=(\lambda_1,\ldots,\lambda_N )\in \mathbb{C}^N$, 
$\underline{u}_N=(u_1,\ldots,u_N )\in \mathbb{R}^N$,
are eigenfunctions of the 
quantum $\mathfrak{gl}_N$ Toda lattice Hamiltonian
\begin{equation}
	\label{eq:intro_Whit_eigenrelations}
	\mathscr{H}_2^{\mathrm{Toda}}=
	-\frac{1}{2}\sum_{i=1}^{N}\frac{\partial^2}{\partial u_i^2}
	+\sum_{i=1}^{N-1}e^{u_{i+1}-u_i},\qquad 
	\mathscr{H}_2^{\mathrm{Toda}}
	\psi_{\underline \lambda}(\underline{u}_N)=
	\left( -\frac12\sum_{i=1}^{N}\lambda_i^2 \right)\psi_{\underline\lambda}(\underline{u}_N).
\end{equation}

Probability measures based on Whittaker functions 
describe distributions of
integrable models of directed random polymers:
the semi-discrete Brownian polymer 
\cite{Oconnell2009_Toda},
and fully 
discrete polymers in random environments with independent log-gamma distributed
weights
\cite{COSZ2011}, \cite{OSZ2012},
\cite{OConnellOrtmann2014}, \cite{CorwinSeppalainenShen2014}.

\medskip

The contribution of this paper is two-fold. 
First, we present a new version of spin $q$-Whittaker polynomials
which generalizes the ones from \cite{BorodinWheelerSpinq}
and strengthen their properties.
Second, in a $q\to1$ limit,
we discover a nontrivial one-parameter deformation of the $\mathfrak{gl}_N$
Whittaker functions.
The new spin Whittaker functions are eigenfunctions of a deformed quantum Toda Hamiltonian, 
and are also related to random polymers with beta distributed weights.
Let us briefly describe our main results.

\subsection{A new version of spin $q$-Whittaker polynomials}

First, we introduce \emph{modified versions of the spin $q$-Whittaker
symmetric polynomials} $\mathbb{F}_\lambda(x_1,\ldots,x_n )$,
where $\lambda=(\lambda_1\ge \ldots\ge \lambda_n\ge0 )$, 
$\lambda_i\in \mathbb{Z}$, are (nonnegative) signatures.
Our polynomials are more general than the 
Borodin--Wheeler's version $\mathbb{F}_\lambda^{BW}$
\cite{BorodinWheelerSpinq}. More precisely, we have
\begin{equation*}
	\mathbb{F}_\lambda(x_1,x_2,\ldots,x_n )\big\vert_{x_1=0}=
	\mathbb{F}^{BW}_\lambda(x_2,\ldots,x_n ).
\end{equation*}
Under the degeneration $s=0$, 
both families $\mathbb{F}_\lambda$ and $\mathbb{F}_\lambda^{BW}$
coincide and turn into the usual $q$-Whittaker
polynomials.

The new spin $q$-Whittaker polynomials $\mathbb{F}_\lambda$
share all the properties known for the $\mathbb{F}_\lambda^{BW}$'s, including symmetry,
Cauchy summation identities, and Pieri type rules.
Moreover, we strengthen other 
known properties of the spin $q$-Whittaker polynomials:
\begin{itemize}
	\item (\Cref{sub:sqW_eigenrelations})
		We present $q$-difference operators $\mathfrak{D}_1$, $\overline{\mathfrak{D}}_1$
		which act on our new spin $q$-Whittaker polynomials diagonally
		as
		$\mathfrak{D}_1 \mathbb{F}_\lambda=q^{\lambda_N}\mathbb{F}_\lambda$
		and
		$\overline{\mathfrak{D}}_1 \mathbb{F}_\lambda=q^{-\lambda_1}\mathbb{F}_\lambda$.
		The operator $\overline{\mathfrak{D}}_1$ 
		reduces, as $x_1\to0$,
		to the known eigenoperator $\mathfrak{E}$ \cite{BufetovMucciconiPetrov2018}
		acting on $\mathbb{F}_\lambda^{BW}$ with the same eigenvalue $q^{-\lambda_1}$.
		The operator $\overline{\mathfrak{D}}_1$, and the fact that the other eigenvalue
		$q^{\lambda_N}$ can be extracted from spin $q$-Whittaker polynomials, are new.
	\item (\Cref{sub:qW_conjugation}) 
		We observe that the operators 
		$\mathfrak{D}_1$, $\overline{\mathfrak{D}}_1$
		can be represented as conjugations of the 
		first $q$-Whittaker $q$-difference operators (these are $t=0$ degenerations
		of the Macdonald operators from \cite{Macdonald1995}).
		From higher $q$-Whittaker operators we thus get higher order $q$-difference operators
		commuting with either 
		$\mathfrak{D}_1$ or $\overline{\mathfrak{D}}_1$
		(the conjugations leading to 
		$\mathfrak{D}_1$ and $\overline{\mathfrak{D}}_1$ are different,
		even though
		these operators commute).
		The higher order operators
		coming from the $q$-Whittaker operators
		are not diagonal in the 
		spin $q$-Whittaker polynomials.
	\item (\Cref{sec:marginals_sqW_processes}) 
		For spin $q$-Whittaker processes
		on interlacing arrays of signatures,
		we construct sampling algorithms 
		(``Yang--Baxter fields'')
		based on bijectivizations
		of the Yang--Baxter equation, using ideas 
		of the previous works \cite{BufetovPetrovYB2017},
		\cite{BufetovMucciconiPetrov2018}.
		We consider several Yang--Baxter fields, 
		and by the very construction
		each of them possesses a 
		marginally Markovian projection to a one-dimensional
		system: stochastic higher spin six vertex model; 
		$_4\phi_3$ stochastic vertex model / $q$-Hahn PushTASEP;
		or the $q$-Hahn TASEP from 
		\cite{Povolotsky2013},
		\cite{Corwin2014qmunu}. 
		The former two connections already appeared in \cite{BufetovMucciconiPetrov2018} (for 
		processes based on $\mathbb{F}_\lambda^{BW}$),
		while the matching to the $q$-Hahn TASEP is new.
	\item (\Cref{sec:rsk_from_YB})
		Setting $q=s=0$, 
		we see how the Yang--Baxter equation
		reduces to the classical Robinson--Schensted--Knuth
		row insertion algorithm \cite{Knuth1970}, \cite{fulton1997young}, \cite{Stanley1999}.
	\item (\Cref{sub:sqw_sqw_last_continuous})
		In a simplified ``Plancherel'' (or ``Poisson-type'') 
		continuous time limit 
		we construct a Markov dynamics on interlacing arrays 
		under which the last rows marginally evolve as a continuous time version of
		the $q$-Hahn TASEP (appeared in \cite{barraquand2015q}).
		Our new two-dimensional continuous time dynamics 
		is a one-parameter deformation of the 
		$q$-Whittaker 2d-growth model introduced in
		\cite[Definition 3.3.3]{BorodinCorwin2011Macdonald}.
		The latter growth model has continuous time $q$-TASEP as the 
		last row marginal dynamics.
\end{itemize}

Our modification of the spin $q$-Whittaker polynomials originates
from computer experiments 
informed by the existing definition from
\cite{BorodinWheelerSpinq}
combined with the desire to have $q$-difference eigenoperators
(a particular case of one of the eigenoperators
appeared earlier in \cite{BufetovMucciconiPetrov2018}).
The new spin $q$-Whittaker polynomials
can be formulated as partition functions of up-right path 
ensembles (cf. \Cref{fig:paths_examples}, left), 
where paths must stay above the diagonal, and 
the vertex weights at the diagonal are special. 
These special \emph{corner vertex weights} turn out to
satisfy a version of the Yang--Baxter equation
(given in \Cref{prop:YBE_corner_W}
in Appendix). Combined with the Yang--Baxter equation
for the spin $q$-Whittaker \emph{bulk vertex weights} 
written down in \cite{BorodinWheelerSpinq} (which is a fusion of the most basic
Yang--Baxter equation for the six vertex model),
this brings most of the desired properties of the new polynomials,
including their symmetry and Cauchy summation identities. 
We also note that for $s=0$, corner and bulk vertex weights coincide, 
so the effect of the new corner weights is present only at the $s\ne 0$ level.

\medskip
It would be very interesting to connect our 
corner vertex weights and the corresponding
Yang--Baxter equation with known integrable vertex 
model constructions.

\subsection{Spin Whittaker functions, random polymers, and deformed quantum Toda}

Our second series of results deals with the 
$q\to1$ scaling limit of the spin $q$-Whittaker polynomials.
Stochastic systems which we have associated with
the spin $q$-Whittaker polynomials
already known to possess such limits:
\begin{itemize}
	\item 
		The $q$-Hahn TASEP becomes the strict-weak 
		directed polymer model in an environment built from 
		independent random variables with beta distribution \cite{CorwinBarraquand2015Beta}. 
		We recall it in \Cref{def:strict_weak_beta}. 
	\item 
		The $q$-Hahn PushTASEP scales \cite{CMP_qHahn_Push} to another beta polymer 
		type model --- a rather complicated system 
		determined by a random recursion with 
		negative beta binomial random weights.
		We recall (a slight generalization of)
		this model in \Cref{def:polymer_like}. 
\end{itemize}

Introduce the scaling
\begin{equation*}
	q\to1, \qquad 
	x_i=q^{X_i},\qquad 
	s=-q^S,\qquad 
	q^{-\lambda_i}=L_i,
\end{equation*}
where $S>0$, $|X_i|<S$, and 
$1\le L_N\le \ldots\le L_1 $ are fixed real numbers.
We show (\Cref{thm:sqW_to_sW})
that under this scaling, the spin $q$-Whittaker polynomial
$\mathbb{F}_\lambda(x_1,\ldots,x_N )$
converges to a new object --- the \emph{spin Whittaker function}
$\mathfrak{f}_{X_1,\ldots,X_N }(L_1,\ldots,L_N)$ (which also depends on $S$).

\medskip

The functions 
$\mathfrak{f}_{X_1,\ldots,X_N }(L_1,\ldots,L_N)$
may be defined via a recursive Givental-type integral representation.
Let $\underline{L}'_{N-1}=(L_{N-1}',\ldots,L_1' )$
and $\underline{L}_N=(L_N,\ldots,L_1 )$
be interlacing sequences:
\begin{equation*}
	1\le L_N\le L_{N-1}'\le L_{N-1}\le \ldots\le L_1'\le L_1 
\end{equation*}
(notation: $\underline{L}_{N-1}'\prec \underline{L}_N$).
Define
\begin{multline*}
	\mathfrak{f}_X(\underline{L}_{N-1}';\underline{L}_N)\coloneqq
	\frac{1}{(\mathrm{B}(S+X,S-X))^{N-1}}
	\left( \frac{L_{N}\cdots L_1}{L_{N-1}'\cdots L_1' } \right)^{-X}
	\\\times\prod_{j=1}^{N-1}
	\biggl( 1-\frac{L_j'}{L_j}			\biggr)^{S-X-1}
	\biggl( 1-\frac{L_{j+1}}{L_j'}	\biggr)^{S+X-1}
	\biggl( 1-\frac{L_{j+1}}{L_j}		\biggr)^{1-2S},
\end{multline*}
where $\mathrm{B}(S+X,S-X)$ is the Beta function. 
Set
$\mathfrak{f}_{X_1}(L_1)\coloneqq L_1^{-X_1}$, and, inductively, 
\begin{equation}
	\label{eq:intro_sW_recursive}
	\mathfrak{f}_{X_1,\ldots,X_N }(\underline{L}_N)
	\coloneqq
	\int_{\underline{L}'_{N-1}\colon
	\underline{L}'_{N-1}\prec \underline{L}_N}
	\mathfrak{f}_{X_1,\dots,X_{N-1}}(\underline{L}_{N-1}')
	\,
	\mathfrak{f}_{X_N}(\underline{L}_{N-1}' ; \underline{L}_N) 
	\,
	\frac{d \underline{L}_{N-1}}{\underline{L}_{N-1}'}.
\end{equation} 

\begin{example}
	In the simplest nontrivial case $N=2$ we have
	\begin{equation*}
			\mathfrak{f}_{X,Y}(u,z)
			=
			(z/u)^S 
			u^{-X-Y}
			{}_2F_1 
			\left(\begin{minipage}{2.15cm}
			\center{$S+X \, , \, S+Y$}
			\\
			\center{$2S$}
	\end{minipage} \Big\vert\, 1-\frac{z}{u}\right),
	\qquad 1\le u\le z,
	\end{equation*}
	where $_2F_1$ is the Gauss hypergeometric function \eqref{eq:2F1}.
\end{example}

\begin{remark}
	Observe that in contrast with the usual
	Whittaker functions, the spin deformations depend on 
	\emph{ordered} tuples $\underline{L}_N$.
	This also corresponds to the fact that 
	the integration in \eqref{eq:intro_sW_recursive}
	is over sequences $\underline{L}_{N-1}'$
	interlacing with $\underline{L}_N$.
\end{remark}

The
spin Whittaker functions 
$f_{X_1,\ldots,X_N }(\underline{L}_N)$ 
defined by the recursion 
\eqref{eq:intro_sW_recursive}
are \emph{symmetric} in 
the $X_i$'s. This fact is far from being obvious
from this recursive representation,
and follows from the symmetry of the spin $q$-Whittaker polynomials
(which ultimately is a consequence of the Yang--Baxter equation).

\medskip

We show (\Cref{thm:deformed_Toda_eigen}) 
that the functions $\mathfrak{f}_{X_1,\ldots,X_N }(\underline{L}_N)$
are eigenfunctions of a deformation of the $\mathfrak{gl}_N$ quantum
Toda Hamiltonian
\begin{equation}
	\label{eq:intro_deformed_Toda}
	\mathscr{H}_2\coloneqq
	-\frac{1}{2} \sum_{i=1}^N \partial^2_{u_i} 
	+
	\sum_{1\le i<j \le N} 
	S^{-2(j-i)} e^{u_j - u_i} 
	( 
		S - \partial_{u_i} 
	) 
	( 
		S+ \partial_{u_j} 
	).
\end{equation}
Introduce a change of variables
$L_j=S^{N+1-2j}e^{u_j}$. Then in these variables we have
\begin{equation*}
	\mathscr{H}_2
	\mathfrak{f}_{X_1,\ldots,X_N }
	=
	\left( -\frac12\sum_{j=1}^{N}X_j^2 \right)
	\mathfrak{f}_{X_1,\ldots,X_N }.
\end{equation*}

\medskip

The functions $\mathfrak{f}_{X_1,\ldots,X_N }(\underline{L}_N)$
satisfy a version of the Cauchy type identity
with integration over $1\le L_N\le \ldots\le L_1$:
\begin{equation}
	\label{eq:intro_sW_Cauchy}
	\int
	\mathfrak{f}_{X_1,\dots,X_N}(\underline{L}_N)\,
	\mathfrak{g}_{Y_1,\dots,Y_M}(\underline{L}_N)\,
	\frac{d \underline{L}_N}{\underline{L}_N}
	=
	\prod_{j=1}^M \frac{\Gamma(X_1+Y_j)}{\Gamma(S+X_1)} 
	\Bigg(\prod_{i=2}^N \frac{\Gamma(X_i+Y_j) \Gamma(2S)}{\Gamma(S+X_i) \Gamma(S+Y_j)}\Bigg)
	.
\end{equation}
Here $\mathfrak{g}_{Y_1,\dots,Y_M}(\underline{L}_N)$
are certain dual spin Whittaker functions, 
see \Cref{sub:dual_sw}.
For the usual Whittaker functions, 
first Cauchy type identity with $M=N$
is due to Bump and Stade 
\cite{Bump1989},
\cite{Stade2002},
\cite{gerasimov2008baxter},
and was later generalized in
\cite[(1.2)]{COSZ2011}, \cite[Section 4.2.1]{BorodinCorwin2011Macdonald}.

\medskip

We also define \emph{spin Whittaker processes}.
These are probability measures on 
interlacing sequences of reals
$\underline{L}_1\prec \underline{L}_2\prec \ldots \prec \underline{L}_N $,
$\underline{L}_k=(L_{k,k}\le L_{k,k-1}\le \ldots\le L_{k,1} )$,
with probability weights expressed through 
the spin Whittaker functions.
Cauchy type identity \eqref{eq:intro_sW_Cauchy}
provides an explicit normalizing constant for the spin Whittaker
process.
We match the distribution of the marginal process
$L_{k,k}^{-1}$ to the strict-weak beta polymer
model of \cite{CorwinBarraquand2015Beta} (\Cref{thm:strict_weak_beta_polymer}),
and the distribution of the other marginal process
$L_{k,t}$ to 
the other beta polymer like model which appeared
in \cite{CMP_qHahn_Push} (\Cref{thm:sW_process_first_row}).

\medskip

As $S\to+\infty$ and under the scaling
$L_j=S^{N+1-2j}e^{u_j}$, $X_j=-\mathrm{i}\lambda_j$,
the spin Whittaker functions $\mathfrak{f}_{X_1,\ldots,X_N }(\underline{L}_N)$
formally reduce to the 
usual Whittaker functions $\psi_{\underline{\lambda}}(\underline{u}_N)$.
A similar reduction brings spin Whittaker processes
to Whittaker processes from
\cite{Oconnell2009_Toda}, 
\cite{COSZ2011},
\cite{BorodinCorwin2011Macdonald}.
We do not fully justify these limit transitions, as this requires a 
much finer analysis and justification of the interchange of the 
$S\to+\infty$ limit with Givental-type representations, which is outside the 
scope of this paper.
However, we note that at the level of marginals, 
the strict-weak beta polymer becomes the strict-weak log-gamma polymer
\cite[Remark 1.5]{CorwinBarraquand2015Beta}.
We also show (\Cref{prop:log_gamma_reduction}) that the other beta polymer type model
turns into the log-gamma polymer from \cite{Seppalainen2012}.

\medskip

Finally, we note that at the level of quantum Toda Hamiltonians
the limit 
$\lim\limits_{S\to +\infty} \mathscr{H}_2=\mathscr{H}_2^{\mathrm{Toda}}$ 
is quite straightforward. 
Indeed, the only terms surviving this limit have $j=i+1$, and 
then $S^{-2}(S-\partial_{u_i})(S+\partial_{u_{i+1}})\to 1$
leads to \eqref{eq:intro_Whit_eigenrelations}

\begin{remark}
	In representation theory, the term ``Whittaker functions'' 
	refers to special matrix elements of certain
	(infinite-dimensional) representations.
	The name ``spin Whittaker'' functions for our 
	new objects 
	$\mathfrak{f}_{X_1,\ldots,X_N }(\underline{L}_N)$,
	$\mathfrak{g}_{Y_1,\dots,Y_M}(\underline{L}_N)$
	is naturally suggested by their place in the hierarchy
	of symmetric functions in \Cref{fig:scheme}.
	At this point \emph{we do not claim}
	the existence of a
	matrix element interpretation of these
	functions (that would make them ``Whittaker'' in a representation-theoretic sense).
	Finding such an interpretation
	is an interesting challenge which falls out of the scope of the present paper.
\end{remark}

\subsection{Outline}

The paper has two main parts. 
The \emph{discrete}
part, \Cref{sec:symm_functions,sec:difference_operators,sec:dynamics_on_arrays,sec:marginals_sqW_processes,sec:rsk_from_YB},
discusses spin Hall--Littlewood functions and 
our new variant of $q$-Whittaker functions, Cauchy type summation identities via Yang--Baxter equations,
difference operators diagonalized by these functions, and related
integrable stochastic models.
The \emph{continuous} part,
\Cref{sec:sW_functions,sec:sW_processes,sec:Toda},
deals with continuously labeled spin Whittaker functions and their properties.
These include
Givental-type integral representations for spin Whittaker functions,
Cauchy type identities, 
connections to random polymer models with beta weights,
and a new deformation of the quantum Toda Hamiltonian.

In \Cref{sec:concluding_remarks} we discuss a number of 
further directions, and formulate conjectures about torus scalar product 
orthogonality of spin $q$-Whittaker and spin Whittaker functions.

\Cref{app:special} collects notation relevant to special functions
used in the paper. In \Cref{app:YBE} we list Yang--Baxter equations
used throughout the discrete part.
\Cref{app:proof_prop,app:triangular_sums} contain certain technical proofs
used in the main text (in \Cref{sec:sW_functions,sec:Toda}, respectively).

\subsection{Notation}

We use the $q$-Pochhammer symbol notation
\begin{equation}
	\label{eq:q_Pochhammer}
	(a;q)_k \coloneqq (1-a)(1-aq)\ldots(1-aq^{k-1}),
	\qquad 
	(a;q)_0\coloneqq 1.
\end{equation}
Occasionally we will need multiple $q$-Pochhammer
symbols
$(a_{1}, \ldots, a_{m}; q)_{k} \coloneqq (a_1;q)_k\ldots (a_m;q)_k $.
Certain special functions
such as $q$-hypergeometric and hypergeometric functions,
as well as useful probability distributions
based on them are described in \Cref{app:special}.

Throughout the paper,
$\mathbf{1}_{A}$ denotes the indicator of an event $A$.

\subsection{Acknowledgments}
MM is grateful to Alexander Garbali for useful discussions.
LP was partially supported by the NSF grant DMS-1664617.

\section{\texorpdfstring{Spin $q$-Whittaker and spin Hall--Littlewood functions}{Spin q-Whittaker and spin Hall--Littlewood functions}}
\label{sec:symm_functions}

Here we introduce symmetric functions we use
throughout the paper which are variants of 
the spin Hall--Littlewood and spin $q$-Whittaker 
functions of
\cite{Borodin2014vertex}, \cite{BorodinWheelerSpinq}.

\subsection{Signatures}
\label{sub:signatures}

Our symmetric functions are indexed by nonnegative signatures
(i.e., partitions with a specified number of parts $N$).
We will drop the word ``nonnegative'',
and refer to them simply as ``signatures''.
Signatures
form a set
\begin{equation*}
	\mathrm{Sign}_N \coloneqq \{ \lambda = (\lambda_1 \ge \cdots \ge \lambda_N \ge 0) \colon \lambda_i \in \mathbb{Z}_{\ge 0} \},
	\qquad N\in \mathbb{Z}_{\ge0}.
\end{equation*}
By agreement, $\mathrm{Sign}_0=\{\varnothing\}$.
The number of positive parts of a signature $\lambda$ is denoted by 
\begin{equation*}
    \ell(\lambda) = \# \{i\colon \lambda_i >0 \}.
\end{equation*}
When $\lambda$ is a partition (and not a signature), the
quantity $\ell(\lambda)$ takes the name of \emph{length}. We \emph{will
not} use such terminology as it creates confusion with the number $N$
of coordinates of the signature $\lambda$.

For notational convenience, we will also label certain symmetric
functions with the transpose of a signature. 
To define the transposition in the context of signatures,
introduce the set of \emph{boxed signatures}
\begin{equation*}
	\mathrm{Sign}_M^{\le N} \coloneqq \{ \lambda = (\lambda_1 \ge \cdots \ge \lambda_M \ge 0)\colon 0\le \lambda_i\le N  \}\subset \mathrm{Sign}_M.
\end{equation*}
Clearly, these signatures can be represented as belonging to 
the box
$\mathrm{Box}(N,M)$, where
\begin{equation*}
	\mathrm{Box}(N,M)=\{1,\dots,N\} \times \{1,\dots,M\}.
\end{equation*}
Let
$\lambda \in \mathrm{Sign}_M^{\le N}$.
By the \emph{transposed signature} 
$\lambda'$ we mean
\begin{equation*}
	\lambda_i'\coloneqq\#\{j\colon  \lambda_j\ge i \},\qquad i=1,\ldots,N.
\end{equation*}
Clearly, 
$\lambda' \in \mathrm{Sign}_N^{\le M}$. 
See \Cref{fig:boxed_transposition} for an illustration.

\begin{figure}[htpb]
	\centering
	\includegraphics[width=.3\textwidth]{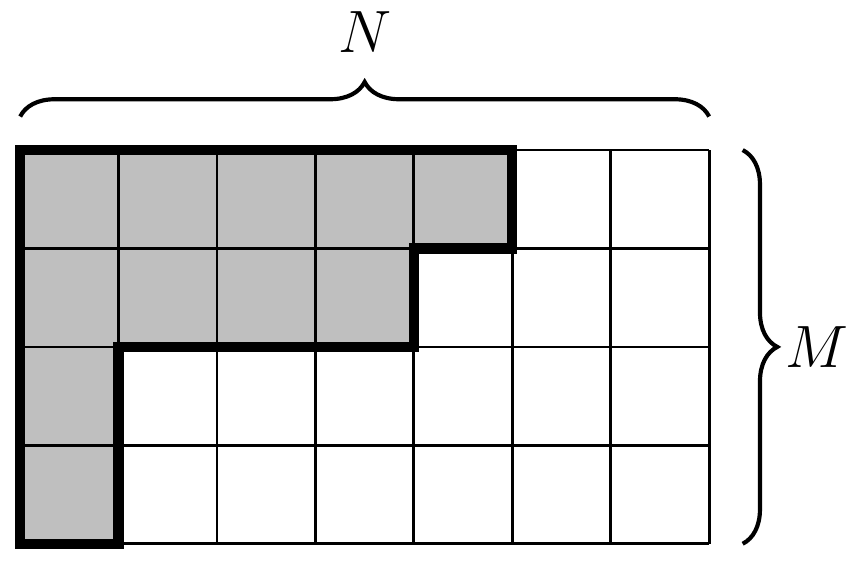}
	\caption{An example of a signature $\lambda=(5,4,1,1)\in \mathrm{Sign}_4^{\le 7}$.
	Its transposed signature is $\lambda'=(4,2,2,2,1,0,0)\in \mathrm{Sign}_7^{\le 4}$.}
	\label{fig:boxed_transposition}
\end{figure}

We will also use multiplicative notation for signatures:
\begin{equation*}
    \lambda = 1^{m_1(\lambda)}2^{m_2(\lambda)}\ldots , \qquad \text{where} \qquad m_i(\lambda)=\#\{ j \colon \lambda_j = i \}.
\end{equation*}

Given two signatures $\mu \in \mathrm{Sign}_{k}$ and $\lambda \in \mathrm{Sign}_{k+1}$ we say that they \emph{interlace} if
\begin{equation}\label{eq:interlacing}
	\lambda_1 \ge \mu_1 \ge \lambda_2 \ge \mu_2 \ge \cdots \ge \mu_{k} \ge \lambda_{k+1}.
\end{equation}
We will use notation
$\mu \prec \lambda$ for interlacing.
Interlacing
also extends to the case when $\lambda$ and $\mu$
have the same number of elements 
by dropping the last inequality in
\eqref{eq:interlacing}. We will use the same 
notation $\mu\prec \lambda$ in this case.
When $\lambda$ and $\mu$ are such that
$\mu' \prec \lambda'$, we say that they are \emph{transposed interlacing}, and
use the notation $\mu \prec' \lambda$.

\subsection{Directed path vertex models}
\label{sub:directed_paths}

Symmetric functions introduced in this section are constructed
through a
vertex model formalism.
That is, we define symmetric functions as partition 
functions (=~sum of weights of allowed configurations)
of ensembles of paths flowing
through a planar lattice, where
the \emph{global weight} of each path configuration
is the product of Boltzmann weights of local configurations around
each vertex. We need 
two separate classes of ensembles:
up-right and down-right.

\begin{definition}[Up-right paths]
	\label{def:up_right_paths}
	We consider up-right directed paths living in the
	half-quadrant
	$\{ (i,j) \in \mathbb{Z}_{\ge 0} \times \mathbb{Z}_{\ge 1}\colon j \ge i \}$.
	We divide its vertices into three categories: 
	\begin{itemize}
			\item left boundary vertices
				$\leftwalltext$
				at $(0,j)$, for $j\ge 1$;
			\item bulk vertices
				$\bulktext$
				at $(i,j)$, for $1 \le i < j$;
			\item right corner vertices
				$\rightcornertext$
			at $(i,i)$, for $i \ge 1$.
	\end{itemize}
	Paths we consider emanate from left boundary vertices and proceed in
	the up-right direction in the bulk of the lattice.
	Multiple paths are allowed to go along one
	horizontal or vertical
	edge.
	When a path 
	meets the diagonal it gets reflected in the upward direction.
	The reason why we distinguish the nature of the vertices is
	that we will use different weighting systems for each of them.
	See \Cref{fig:paths_examples}, left, for an illustration.
\end{definition}

For a configuration of up-right paths, define for each $k\in \mathbb{Z}_{\ge1}$
the signature $\lambda^k \in \mathrm{Sign}_k$ by
\begin{equation*}
	\lambda_i^k - \lambda_{i+1}^k = \#\{\text{paths occupying the edge } (i,k) \to (i,k+1)\},
\end{equation*}
where $i=1, \dots, k-1$. Let also $\lambda^k_k$ will be the number of paths reflected 
at the right corner $(k,k)$. 
In this way the up-right path ensemble is bijectively encoded by a sequence 
\begin{equation*}
	\lambda^1 \subseteq \lambda^2 \subseteq \cdots,
	\qquad \lambda^k\in\mathrm{Sign}_k.
\end{equation*}
Here the relation $\lambda^k \subseteq \lambda^{k+1}$
means inclusion of the respective Young diagrams.
For example, for the up-right path ensemble in \Cref{fig:paths_examples}, left,
we have $\lambda^2=(1,0)$ and $\lambda^3=(3,3,2)$.

\begin{definition}[Down-right paths]
	\label{def:down_right_paths}
	Down-right paths live inside the finite rectangle
	$\{ 0, \dots , N\} \times \{ 1, \dots , M\}$, where $N,M$ are
	fixed positive integers. We divide its vertices into three
	categories:
	\begin{itemize}
			\item left boundary vertices  
				$\leftwalltext$
				at $(0,j)$, for $1 \le j \le M$;
			\item bulk vertices 
				$\bulktext$
				at $(i,j)$, for $1 \le i < N$, $ 1 \le j \le M$;
			\item right boundary vertices
				$\rightwalltext$
				at $(N,j)$, for $ 1 \le j \le M$.
	\end{itemize}
	Down-right directed paths we consider originate at left boundary
	vertices and terminate at the lower base of the
	rectangle. Once paths hit the right boundary they are automatically
	sent all the way down.
	See \Cref{fig:paths_examples}, right, for an illustration.
\end{definition}

\begin{figure}[htpb]
	\centering
	\includegraphics{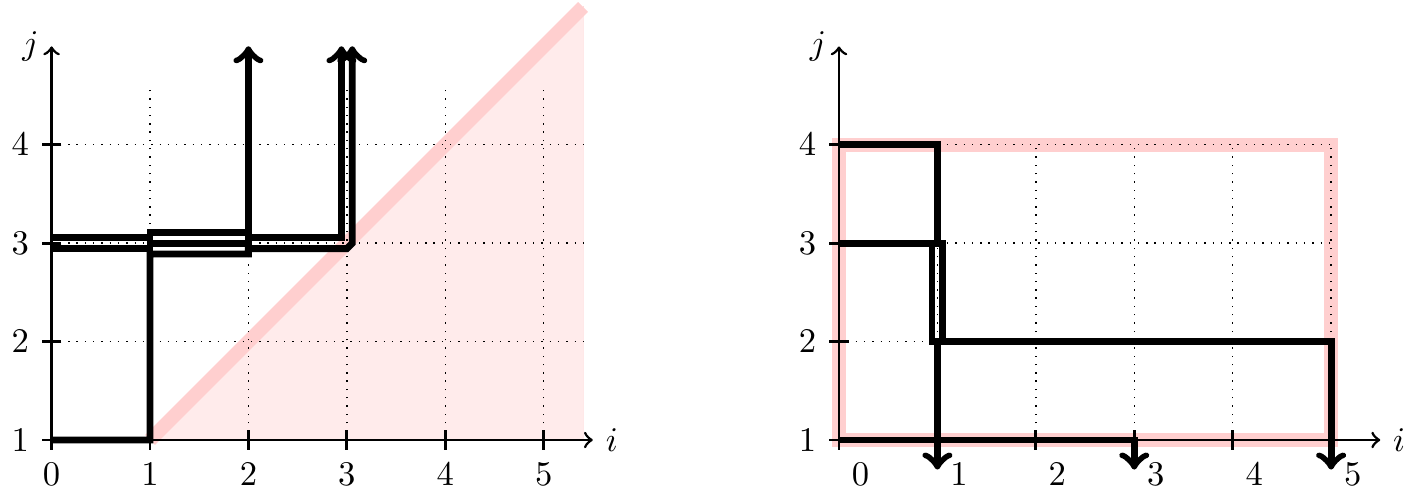}
	\caption{Left: Example of an up-right path ensemble. 
	All paths must be above the main diagonal.
	Right: Example of a down-right path ensemble with 
	$N=5$ and $M=4$. All paths must
	be inside the rectangle.}
	\label{fig:paths_examples}
\end{figure}

To each configuration of down-right paths we can also associate a sequence of growing signatures 
\begin{equation*}
	\lambda^1 \subseteq \lambda^2 \subseteq \cdots \subseteq \lambda^M,
	\qquad \lambda^k\in \mathrm{Sign}_N,
\end{equation*}
where 
\begin{equation*}
	\lambda_i^k - \lambda_{i+1}^k = \#\{\text{paths occupying the edge } (i,M-k+1) \to (i,M-k)\},
\end{equation*}
with $\lambda^k_{N+1}=0$, by agreement. 
For example, for the path ensemble in \Cref{fig:paths_examples}, right, 
we have
$\lambda^1=(1,0,0,0,0)$,
$\lambda^2=(2,0,0,0,0)$,
$\lambda^3=(2,1,1,1,1)$,
and $\lambda^4=(3,2,2,1,1)$.

\subsection{\texorpdfstring{Spin $q$-Whittaker polynomials}{Spin q-Whittaker polynomials}}
\label{sub:sqw_poly_def}

The spin $q$-Whittaker polynomials are partition functions of up-right path ensembles.
Assign the following weights to the left boundary, bulk, and right corner vertices 
(here and below we use the $q$-Pochhammer notation
\eqref{eq:q_Pochhammer}):
\begin{align}
  W^{\leftwall}_{x,s} (j) &\coloneqq
	x^j\,\frac{(-s/x;q)_j}{(q;q)_j} ;
	\label{eq:W_left_boundary}
	\\
	W^{\bulk}_{x,s}(i_1,j_1;i_2, j_2) 
	&\coloneqq 
	\mathbf{1}_{i_1 + j_1 = i_2 + j_2} \, \mathbf{1}_{i_1 \geq j_2}\,  x^{j_2}\, 
	\frac{(- s/x;q)_{j_2} (- s x ; q )_{i_1 - j_2} (q;q)_{i_2} }{(q;q)_{j_2} (q;q)_{i_1 - j_2} (s^2;q)_{i_2} };
	\label{eq:Whit_W}
	\\
	W^{\rightcorner}_{x,s}(j) &\coloneqq \,  \frac{(q;q)_j}{(-s/x;q)_j}.
	\label{eq:W_corner}
\end{align}
Here in \eqref{eq:W_left_boundary} and \eqref{eq:W_corner},
$j\in \mathbb{Z}_{\ge0}$ denotes
the number of paths going through the vertex,
and in \eqref{eq:Whit_W} the numbers 
$i_1,j_1,i_2,j_2\in \mathbb{Z}_{\ge0}$
denote, respectively, the numbers of entering vertical, entering horizontal,
exiting vertical, and exiting horizontal paths to/from the vertex.

The weights \eqref{eq:W_left_boundary}--\eqref{eq:W_corner}
depend on the main \emph{quantization parameter} $q$, on a
\emph{spectral parameter} $x$, a \emph{spin parameter} $s$.
While $q,s$ are assumed fixed, the spectral
parameter will depend on the vertical lattice coordinate.
 
\begin{remark}
	One can readily check that the condition $i_1\ge j_2$ 
	in \eqref{eq:Whit_W}
	implies that up-right
	path configurations with nonzero global weight are those associated
	with sequences of interlacing signatures $\lambda^1\prec
	\lambda^2 \prec \ldots$. In particular, the configuration in
	\Cref{fig:paths_examples}, left, has global weight zero.
\end{remark}

\begin{remark}
	When $s=0$, the bulk 
	and the corner weights
	\eqref{eq:Whit_W}-\eqref{eq:W_corner} coincide. More precisely, we have
	$W^{\rightcorner}_{x,s}(j)=W^{\bulk}_{x,s}(0,j;j,0)=(q;q)_j$.
\end{remark}

\begin{definition}[Spin $q$-Whittaker polynomials]
For given interlacing signatures $\mu \prec \lambda$ with $\mu \in
\mathrm{Sign}_{k}$ and $\lambda \in \mathrm{Sign}_{k+1}$, the
\emph{skew spin $q$-Whittaker polynomial} in one variable
is the weight of the
unique path configuration between $\mu$ and $\lambda$ at the 
$k$-th slice.
It is given by 
\begin{equation} \label{eq:inc_sqW}
	\mathbb{F}_{\lambda/\mu} (x) 
	\coloneqq 
	x^{|\lambda|-|\mu|} 
	\prod_{i=1}^{k} 
	\frac{(-s/x;q)_{\lambda_i - \mu_i} (-s x;q)_{\mu_i - \lambda_{i+1}} (q;q)_{\lambda_i - \lambda_{i+1}  }  }
	{(q;q)_{\lambda_i - \mu_i} (q;q)_{\mu_i - \lambda_{i+i} } (s^2;q)_{\lambda_i - \lambda_{i+i}}  }.
\end{equation}
This is clearly a polynomial in $x$, even though the right corner weight
\eqref{eq:W_corner} is not polynomial.
We will often abbreviate the name ``spin $q$-Whittaker'' as \emph{sqW}.
\end{definition}

For $\mu\in \mathrm{Sign}_k$ and $\nu\in \mathrm{Sign}_{k+n}$, we also define $n$-variable 
polynomials in a standard way via \emph{branching}:
\begin{equation}
	\label{eq:sqW_many_variables_branching}
	\mathbb{F}_{\nu/\mu} (x_1,\dots,x_n) = \sum_{ \varkappa  }  
	\mathbb{F}_{\varkappa / \mu}(x_1, \dots, x_{n-1})\,
	\mathbb{F}_{\nu/\varkappa}(x_n).
\end{equation}
The polynomials $\mathbb{F}_{\nu/\mu}(x_1,\ldots,x_n)$ 
are partition functions of 
up-right path ensembles as in \Cref{fig:paths_examples}, left,
in a domain with the bottom and the top boundary
conditions determined by $\mu\in \mathrm{Sign}_k$ and $\nu\in \mathrm{Sign}_{k+n}$, 
respectively.

We will use the shorthand notation $\mathbb{F}_\lambda(x_1,\ldots,x_n )\equiv
\mathbb{F}_{\lambda/\varnothing}(x_1,\ldots,x_n )$,
where $\lambda\in \mathrm{Sign}_n$.
\begin{remark}
	\label{rmk:sqW_fixed_number_of_variables}
	It is important to 
	notice that the number of variables in a sqW polynomial
	$\mathbb{F}_{\nu/\mu}$
	is determined by the signatures $\nu,\mu$. If $\nu\in \mathrm{Sign}_{n+k}$ and 
	$\mu\in \mathrm{Sign}_k$, then we can only evaluate $\mathbb{F}_{\nu/\mu}$
	at $n$ variables.
\end{remark}

\subsection{\texorpdfstring{Comparison with Borodin--Wheeler's spin $q$-Whittaker polynomials}{Comparison with Borodin--Wheeler's spin q-Whittaker polynomials}}
\label{sub:compare_BW_sqW}

It is important to note that our version of the spin $q$-Whittaker polynomials is
\emph{different} from the original definition
of Borodin and Wheeler \cite{BorodinWheelerSpinq}.
Namely, the one-variable skew polynomials in 
\cite{BorodinWheelerSpinq}
have the form 
\begin{equation} \label{eq:sqW_BW}
	\mathbb{F}^{BW}_{\lambda/\mu} (x) 
	=
	x^{|\lambda|-|\mu|} 
	\prod_{i\ge1}
	\frac{(-s/x;q)_{\lambda_i - \mu_i} (-s x;q)_{\mu_i - \lambda_{i+1}} (q;q)_{\lambda_i - \lambda_{i+1}  }  }
	{(q;q)_{\lambda_i - \mu_i} (q;q)_{\mu_i - \lambda_{i+i} } (s^2;q)_{\lambda_i - \lambda_{i+i}}  },
\end{equation}
where 
$\mu\in \mathrm{Sign}_k$, 
$\lambda\in \mathrm{Sign}_{k+1}$, 
and the product over $i$ extends to $i=k+1$ with the
agreement that $\lambda_{k+2}=\mu_{k+1}=0$.
That is, 
our one-variable functions differ from 
\eqref{eq:sqW_BW} as
\begin{equation} \label{eq:comparison_BW_sqW}
	\mathbb{F}^{BW}_{\lambda/\mu} (x) 
	=
	\frac{(-s/x;q)_{\lambda_{k+1}}}{(s^2;q)_{\lambda_{k+1}}}
	\,\mathbb{F}_{\lambda/\mu} (x).
\end{equation}

The $n$-variable polynomials $\mathbb{F}^{BW}_{\nu/\mu}(x_1,\ldots,x_n) $
are defined from $\mathbb{F}^{BW}_{\lambda/\mu}(x)$
by branching as in \eqref{eq:sqW_many_variables_branching}.
They admit a lattice path construction similarly to $\mathbb{F}_{\nu/\mu}$,
but with the right corner weights $W_{x,s}^{\rightcorner}$
replaced by the bulk weights $W_{x,s}^{\bulk}$. 

The Borodin--Wheeler's spin $q$-Whittaker polynomials 
arise from our $\mathbb{F}_\lambda$ as a particular case:
\begin{proposition}
	\label{prop:reduction_of_ours_to_BW}
	For all $\lambda\in \mathrm{Sign}_{n}$ we have 
	\begin{equation}
		\label{eq:reduction_to_BW}
		\mathbb{F}_{\lambda}(0,x_2,\ldots,x_n )
		=
		\mathbb{F}^{BW}_{\lambda}(x_2,\ldots,x_{n} ).
	\end{equation}
\end{proposition}
\begin{proof}
	The up-right paths that start at the left boundary
	at height $1$ must immediately turn up
	at the right corner at $(1,1)$.
	If there are $j$ such paths, their
	contribution to the global weight is
	$W^{\leftwall}_{x_1,s}(j)W^{\rightcorner}_{x_1,s}(j)=x_1^j$.
	For $x_1=0$, this forces no paths to start at height $1$.
	Next, due to the presence of $\mathbf{1}_{i_1\ge j_2}$ in the 
	bulk weight $W^{\bulk}_{x,s}$, we see that paths started at 
	height $2,\ldots,n $
	cannot reach the diagonal with the special right corner
	weights. Therefore, 
	the partition function for 
	$\mathbb{F}_{\lambda}(x_1,x_2,\ldots,x_n )$
	with
	$x_1=0$
	involves only 
	left boundary and bulk weights, and is thus the same as 
	the partition function for 
	$\mathbb{F}^{BW}_{\lambda}(x_2,\ldots,x_{n} )$.
\end{proof}

Note that $\mathbb{F}^{BW}_\lambda(x_2,\ldots,x_n )$
is well-defined for any $\lambda$,
and vanishes if $\ell(\lambda)$, the number of nonzero parts in $\lambda$,
exceeds $n-1$. 
If $\ell(\lambda)\le n-1$, then we can treat $\lambda$
as an element of $\mathrm{Sign}_n$ with $\lambda_n=0$, 
and then \eqref{eq:reduction_to_BW} holds. 
Moreover, one readily sees that both sides of
\eqref{eq:reduction_to_BW} vanish if $\lambda_n>0$.
Therefore, any polynomial 
$\mathbb{F}^{BW}_\lambda$ can be obtained from 
our polynomial $\mathbb{F}_\lambda$ by specializing 
one of the variables to zero. (By symmetry, see \Cref{sub:sqW_properties} below,
we can specialize to zero any variable, and not necessarily the first one.)

\subsection{\texorpdfstring{Properties of the spin $q$-Whittaker polynomials}{Properties of the spin q-Whittaker polynomials}}
\label{sub:sqW_properties}

The fact that the Borodin--Wheeler's sqW polynomials
are symmetric in their variables follows 
from the Yang--Baxter equation
which we reproduce in 
\Cref{app:YBE} as \Cref{prop:YBE_bulk_W}.
By looking at \eqref{eq:comparison_BW_sqW},
it is not immediately clear why our version of the sqW polynomials
should also be symmetric.
We prove this next.

\begin{proposition} \label{prop:sqW_are_symmetric}
	For any $\mu\in \mathrm{Sign}_k$, $\nu\in \mathrm{Sign}_{n+k}$
	the polynomial
	$\mathbb{F}_{\nu/\mu}(x_1,\ldots,x_n )$ is 
	symmetric
	with respect to permutations of its variables $x_i$.
\end{proposition}
\begin{proof}
	We use the Yang--Baxter equations
	of \Cref{prop:YBE_bulk_W,prop:YBE_corner_W}
	and employ the standard ``cross dragging'' / commuting transfer matrices
	argument, cf. \cite[Theorem 3.6]{Borodin2014vertex}.
	Using branching, it suffices to consider the
	two-variable case.
	The two-variable polynomial
	$\mathbb{F}_{\lambda/\mu}(x,y)$
	is a partition function of up-right paths 
	on two consecutive levels, with 
	parameters $x$, $y$ at the bottom and at the top, respectively,
	and boundary conditions determined by $\lambda,\mu$.

	First we use the new relation \eqref{eq:YBE_RWW_wall} that, as shown in 
	\Cref{fig:YBE_W_W}\,(b), implies that 
	swapping the spectral parameters $x\leftrightarrow y$
	at the right corners 
	makes a cross appear at their left. 
	Then we sequentially move the
	cross to the left while swapping the spectral parameters using
	the bulk Yang--Baxter equation
	\eqref{eq:YBE_RWW}, as shown by \Cref{fig:YBE_W_W}\,(a).
	We proceed till the left boundary of the domain.
	
	At the left boundary, we can 
	swap the last two spectral parameters
	by noticing that
	\begin{equation} 
		\label{eq:boundary_W_infinite_lines}
		W^{\leftwall}_{x,s}(j) = \frac{(s^2;q)_\infty}{(-s x;q)_{\infty}}\,
		W^{\bulk}_{x,s} (\infty,l;\infty,j), \qquad \text{for any } l\in\mathbb{Z}_{\ge0}.
	\end{equation}
	This means that the left boundary weights 
	$W^{\leftwall}$ also satisfy the Yang--Baxter equation \eqref{eq:YBE_RWW},
	and so we can take the cross out of the lattice.
	This completes the proof.
\end{proof}

Our sqW polynomials also satisfy an index shifting property which
is the same as for the classical homogeneous Macdonald polynomials
$P_\lambda(\cdot;q,t)$ \cite[VI(4.17)]{Macdonald1995}:
\begin{proposition}
	\label{prop:sqW_shifting}
For any signature $\lambda \in \mathrm{Sign}_N$ with $\lambda_N>0$, we have
\begin{equation*}
	\mathbb{F}_\lambda (x_1,\dots, x_N) = x_1 \cdots x_N \,
	\mathbb{F}_{\lambda-1^N} (x_1,\dots, x_N),
	\qquad 
	\lambda - 1^N = (\lambda_1-1,  \ldots , \lambda_N -1).
\end{equation*}
\end{proposition}
\begin{proof}
	First,
	note that \eqref{eq:inc_sqW}
	implies that
	the one-variable skew polynomials
	satisfy the shifting property as
	\begin{equation} 
		\label{eq:sqW_shifting_one_variable}
		\mathbb{F}_{\nu/\mu}(x) = x \, \mathbb{F}_{(\nu-1^{k+1})/(\mu-1^k)}(x),  
	\end{equation}
	for any $\nu\in \mathrm{Sign}_{k+1}$ and $\mu\in \mathrm{Sign}_{k}$ with $\nu_{k+1}>0$ 
	(this also implies $\mu_k>0$, since $\mu_k \ge \nu_{k+1}$). 
	Next, we use the expansion 
	\begin{equation} 
		\label{eq:branching_GT_full_expansion_for_sqW}
		\mathbb{F}_{\lambda} (x_1,\dots, x_N) 
		= 
		\sum_{\lambda^1 \prec \cdots \prec \lambda^{N-1}\prec\lambda} 
		\mathbb{F}_{\lambda^1}(x_1) \,\mathbb{F}_{\lambda^2/\lambda^1}(x_2) 
		\ldots 
		\mathbb{F}_{\lambda^{N-1}/\lambda^{N-2}}(x_{N-1})\,
		\mathbb{F}_{\lambda/\lambda^{N-1}}(x_N)
	\end{equation}
	coming from iterating the branching rule,
	and apply the shifting property
	\eqref{eq:sqW_shifting_one_variable}
	to each of the terms to get the desired result.
\end{proof}

\begin{remark}
	The polynomials $\mathbb{F}_\lambda^{BW}$ \emph{do
	not} satisfy the index shifting property of
	\Cref{prop:sqW_shifting}, which can be seen from
	\eqref{eq:comparison_BW_sqW}.
    
	On the other hand, 
	the polynomials $\mathbb{F}_\lambda^{BW}$ satisfy the stability property 
	$$
	\mathbb{F}_\lambda^{BW}(x_1,\dots,x_{N-1},-s)=\mathbb{F}_\lambda^{BW}(x_1,\dots,x_{N-1}),
	$$
	whereas the polynomials $\mathbb{F}_{\lambda}$ \emph{do not}. More precisely, we have
    \begin{equation*}
        \mathbb{F}_\lambda(x_1,\dots,x_{N-1},-s)= (-s)^{\lambda_N} \mathbb{F}_{\widetilde{\lambda}}(x_1,\dots,x_{N-1}),
    \end{equation*}
    where $\widetilde{\lambda}=(\lambda_1\ge \cdots \ge \lambda_{N-1})$ and this is easily proven since the 
		branching coefficient 
		\eqref{eq:inc_sqW}
		evaluates as $\mathbb{F}_{\lambda/\mu}(-s)=(-s)^{\lambda_N} \prod_{i=1}^{N-1} \mathbf{1}_{\lambda_i=\mu_i}$.
\end{remark}
In the following proposition we use the coefficient
\begin{equation*}
    \mathsf{c}_{\lambda} = \prod_{i=1}^{N-1} \frac{(s^2;q)_{\lambda_i - \lambda_{i+1} } }{ (q;q)_{\lambda_i - \lambda_{i+1} } },
		\qquad \lambda\in \mathrm{Sign}_N.
\end{equation*}

\begin{proposition}\label{prop:quasi_Cauchy_identity}
    Let $|s x_i|<1$ for $i=1,\dots,N$. Then we have
    \begin{equation} \label{eq:quasi_Cauchy_identity}
        \sum_{\substack{\lambda \in \mathrm{Sign}_N \\ \lambda_N=0}} \mathsf{c}_{\lambda} (-s)^{|\lambda|} \mathbb{F}_\lambda (x_1,\dots,x_N)= 
			\frac{((-s)^N x_1\ldots x_N ;q)_\infty (s^2;q)_\infty^{N-1}}
			{(-sx_1;q)_{\infty}\dots(-sx_N ;q)_\infty}.
    \end{equation}
\end{proposition}
\begin{proof}
    We will use the identity     
		\begin{equation} \label{eq:pfaff_saalschutz}
		\sum_{k=0}^n a^k \frac{(b;q)_k}{(q;q)_k} \frac{(a;q)_{n-k}}{(q;q)_{n-k}} = \frac{(ab;q)_n}{(q;q)_n},
    \end{equation}
    that follows from the $q$-Chu--Vandermonde identity 
		(e.g., see \cite[(II.6)]{GasperRahman}).
		Expand the left-hand side of \eqref{eq:quasi_Cauchy_identity} as:
    \begin{equation*}
        \begin{split}
            \sum_{\substack{\lambda=\lambda^N\in \mathrm{Sign}_N \\ \lambda^N_N=0}} \sum_{\lambda^{N-1} \in \mathrm{Sign}_{N-1}}
            &
            (-s x_N)^{\lambda^N_1-\lambda^{N-1}_1} \frac{(-s/x_N;q)_{\lambda^N_1-\lambda^{N-1}_1} }
						{ (q;q)_{\lambda^N_1-\lambda^{N-1}_1} } 
            \\
            &
            \times
            \prod_{k = 2}^{N-1} \left( (-s x_N)^{ \lambda^N_{k}-\lambda^{N-1}_k } \frac{ (-s x_N;q )_{\lambda^{N-1}_{k-1} -\lambda^N_k }  }{ (q;q )_{\lambda^{N-1}_{k-1} -\lambda^N_k } } 
            \frac{ (-s/x_N;q )_{\lambda^{N}_{k} -\lambda^{N-1}_k }  }{ (q;q )_{ \lambda^{N}_{k} -\lambda^{N-1}_k } } \right)
            \\
            &
           \times
           \frac{ (-s x_N;q)_{\lambda^{N-1}_{N-1} } }{ (q;q)_{\lambda^{N-1}_{N-1} } }
           \times
           (-s)^{|\lambda^{N-1}|}
           \mathbb{F}_{\lambda^{N-1}}(x_1, \dots, x_{N-1}).
        \end{split}
    \end{equation*}
    Summing over $\lambda^N_1$ by means of the 
		$q$-binomial theorem
		gives us the factor $(s^2;q)_\infty / (-sx_N;q)_\infty$. We then sum sequentially over indices $\lambda^N_2, \dots, \lambda^N_{N-1}$ and using \eqref{eq:pfaff_saalschutz} we are left with
    \begin{equation*}
        \frac{(s^2;q)_\infty}{(-s x_N;q)_\infty}
				\sum_{\lambda^{N-1} \in \mathrm{Sign}_{N-1}} \frac{ (-s x_N;q)_{\lambda^{N-1}_{N-1} } }{ (q;q)_{\lambda^{N-1}_{N-1} } }\,
        \mathsf{c}_{\lambda^{N-1}}
				(-s)^{|\lambda^{N-1}|} \mathbb{F}_{\lambda^{N-1}}(x_1, \dots, x_{N-1}),
    \end{equation*}
		where $\mathsf{c}_{\lambda^{N-1}}$
		is the result of applying \eqref{eq:pfaff_saalschutz}.
    Repeating the same procedure we can reduce the previous expression to
    \begin{multline*}
				\frac{(s^2;q)_\infty^2}{(-s x_N;q)_{\infty}(-sx_{N-1};q)_\infty}
        \\\times
        \sum_{\lambda^{N-2} \in \mathrm{Sign}_{N-2}} 
				\frac{ ((-s)^2 x_{N-1}x_N;q)_{\lambda^{N-2}_{N-2} } 
				}{ (q;q)_{\lambda^{N-2}_{N-2} } }
        \mathsf{c}_{\lambda^{N-2}}
				(-s)^{|\lambda^{N-2}|} \mathbb{F}_{\lambda^{N-2}}(x_1, \dots, x_{N-2}).
    \end{multline*}
		Here the reason for the appearance of the product
		$(-sx_N)(-sx_{N-1})$ in the $q$-Pochhammer
		symbol is again \eqref{eq:pfaff_saalschutz},
		where we also used that $\lambda^{N-1}_{N-1}$
		is not necessarily zero (in contrast with the 
		first summation over $\lambda^N_N$).
		Continuing 
		inductively, we exhaust all the summations down to the bottom
		one over $\lambda^1_1$, from which we recover the factor
		$((-sx_1)\cdots(-sx_N);q)_\infty/(-sx_1;q)_\infty$. This completes
		the proof.
\end{proof}

\subsection{Dual spin Hall--Littlewood rational functions}
\label{sub:sHL_func_def}

Along with the sqW polynomials $\mathbb{F}_\lambda$ we 
will define two families of 
\emph{dual} functions, with which the $\mathbb{F}_\lambda$'s
satisfy Cauchy-type summation identities.
The first are the dual spin Hall--Littlewood rational functions.
For them we use down-right path ensembles
as in \Cref{fig:paths_examples}, right,
and define the weights by
\begin{align}
		\label{eq:w_sHL_leftwall}
    w_{v}^{*,\leftwall}(j) &\coloneqq v^j;
    \\
		\label{eq:w_sHL_bulk}
		w_{v,s}^{*,\bulk}& \coloneqq
		\textnormal{see \Cref{fig:table_w_star}};
		\\
		W^{*,\rightwall}(i_1,j_1;i_2) &\coloneqq \mathbf{1}_{i_1 = j_1 + i_2 }.
		\label{eq:W_rightwall}
\end{align}
These weights depend on the main parameters $s,q$, and 
on the spectral parameter $v$.
It is easy to see that with this choice of vertex weights the only
allowed configurations of down-right paths in the rectangular grid
$\{0,\dots,N \} \times \{0,\dots, M\}$ are those associated with sequences of
transposed interlacing signatures $\lambda^1 \prec' \cdots \prec'
\lambda^M$.

\begin{figure}[htbp]
    \centering
    \includegraphics{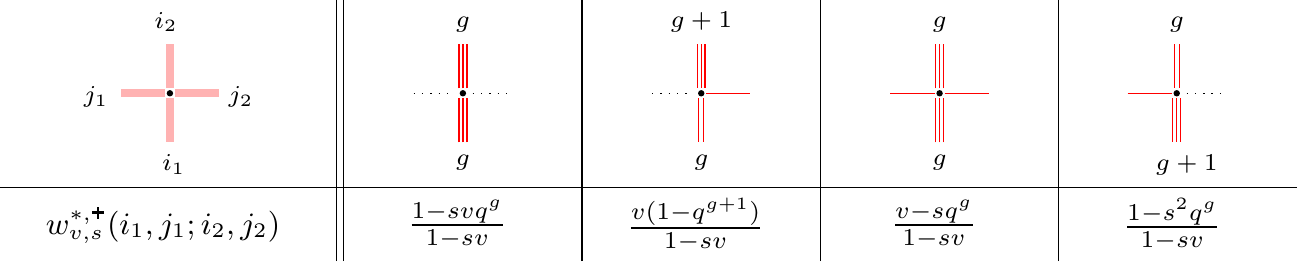}
    \caption{Bulk vertex weights used in the construction of the dual spin
			Hall--Littlewood functions. 
			Vertex configurations not listed are assigned weight zero.
			Note that the weights vanish unless $i_1+j_2=j_1+i_2$.}
    \label{fig:table_w_star}
\end{figure}

\begin{definition}
	\label{def:sHL_func}
	For given $\mu,\lambda \in \mathrm{Sign}_N^{\le M}$
	with $\mu\prec'\lambda$,
	the
	\emph{skew dual spin Hall--Littlewood function} in one
	variable $\mathsf{F}_{\lambda'/ \mu'}^{*}(v)$ is the
	weight of the unique down-right path configuration between $\mu$ and
	$\lambda$ (at the top and at the bottom, respectively)
	at a single 
	row of vertices, where we take weights
	\eqref{eq:w_sHL_leftwall}--\eqref{eq:W_rightwall} with the spectral parameter
	$v$. More explicitly, we have
	\begin{equation*}
			\mathsf{F}_{\lambda'/ \mu'}^{*}(v) 
			\coloneqq 
			\sum_{j_0,\dots,j_{N-1} \in \{0,1\} } 
			v^{j_0} 
			W^{*, \rightwall}(l_N',j_{N-1}; m_N')
			\prod_{r=1}^{N-1} w^{*,\bulk}_{v,s}(l'_r, j_{r-1}; m'_r,j_r) 
			,
	\end{equation*}
	where $\lambda' = 1^{l_1'} \cdots N^{l_N'}$ and $\mu' = 1^{m_1'} \cdots N^{m_N'}$
	belong to $\mathrm{Sign}_{M}^{\le N}$.

	Multi-variable extensions
	$\mathsf{F}^*_{\nu/\varkappa}(v_1,\ldots,v_k )$,
	where $\nu,\varkappa\in \mathrm{Sign}_M^{\le N}$ and 
	$k$ is arbitrary, are defined using the 
	branching rule in the same way as in \eqref{eq:sqW_many_variables_branching}.
	The single-index (non-skew) functions are defined by
	$\mathsf{F}^*_{\nu}(v_1,\ldots,v_k )
	=
	\mathsf{F}^*_{\nu/0^M}(v_1,\ldots,v_k )$, where
	$\nu\in \mathrm{Sign}_{M}^{\le N}$,
	and $0^M$ is the signature from $\mathrm{Sign}_{M}^{\le N}$
	with all parts equal to zero.
\end{definition}

\begin{remark}
	\label{rmk:sHL_any_number_of_variables}
	The sHL functions $\mathsf{F}^*_{\nu/\varkappa}(v_1,\ldots,v_k )$,
	are defined for any number of variables $k$, regardless of the 
	signatures $\nu$ and $\varkappa$.
	This should be contrasted with the sqW polynomials, 
	cf.~\Cref{rmk:sqW_fixed_number_of_variables}.
\end{remark}

The functions 
$\mathsf{F}^*_{\nu/\varkappa}$ are \emph{stable} in the sense that
\begin{equation*}
	\mathsf{F}^*_{\nu/\varkappa}(v_1,\ldots,v_k,0 )=
	\mathsf{F}^*_{\nu/\varkappa}(v_1,\ldots,v_k).
\end{equation*}
Indeed, this readily follows from the 
vertex weights 
\eqref{eq:w_sHL_leftwall}--\eqref{eq:W_rightwall}.

The version of the spin Hall--Littlewood functions
of \Cref{def:sHL_func}
is essentially a particular case of the inhomogeneous 
spin Hall--Littlewood functions from
\cite{BorodinPetrov2016inhom},
where the $N$-th spin parameter $s_N$ is set to zero.
This allows to derive a lot of their properties 
by specializing the corresponding results of
\cite{BorodinPetrov2016inhom}.
As the functions $\mathsf{F}_{\nu/\varkappa}^*$ from 
\Cref{def:sHL_func}
are central to our discussion and we do not use other versions
in the present paper, 
we simply refer to the $\mathsf{F}^*_{\nu/\varkappa}$'s 
as (dual) spin Hall--Littlewood functions.
For convenience, we will omit the dependence on $N$ in their notation.
We will often abbreviate the name ``spin Hall--Littlewood'' as \emph{sHL}.

The sHL functions $\mathsf{F}^*_\lambda$
admit an explicit symmetrization formula:
\begin{proposition}
	\label{prop:dual_sHL_symmetric_sum}
	Let $\lambda \in \mathrm{Sign}_M^{\le N}$, then
	for all $k\ge M$ we have
	\begin{equation} \label{eq:dual_sHL_symmetric_sum}
		\mathsf{F}_\lambda^{*}(v_1, \dots, v_k) 
		=
		\mathscr{C}(\lambda)
		\sum_{\sigma \in \mathfrak{S}_k} 
		\sigma \Biggl\{ 
			\prod_{1\le i < j\le k} \frac{v_i - q v_j}{v_i - v_j} 
			\prod_{i=1}^{\ell(\lambda)} v_i \left( \frac{v_i - s}{1 - s v_i} \right)^{\lambda_i-1} 
			\left(\frac{1}{1 - s v_i}\right)^{\mathbf{1}_{\lambda_i<N}}
		\Biggr\},
\end{equation}
where the symmetric group $\mathfrak{S}_k$
acts on the variables $v_i$ but not on elements of the signature $\lambda_i$,
and the constant prefactor has the form
\begin{equation}
	\label{eq:c_tilde_constant_dual_sHL}
	\mathscr{C}(\lambda) = \frac{(1-q)^k}{(q;q)_{k - \ell(\lambda)}} \prod_{i=1}^N \frac{(s^2;q)_{m_i(\lambda)}}{(q;q)_{m_i(\lambda)}}.
\end{equation}
\end{proposition}
\begin{proof}
	This formula follows from 
	\cite[Theorem 4.14, part 1]{BorodinPetrov2016inhom} via several 
	specializations. 
	The latter result is a symmetrization
	formula for a more general vertex model partition function
	$\mathsf{F}_{\lambda}^{\textnormal{non-stab, non-dual}}$
	which involves inhomogeneity
	parameters $\xi_j$ and $s_j$ depending on the horizontal
	lattice coordinate $j\in \mathbb{Z}_{\ge0}$. 
	Let us describe the necessary specializations.
	In the first step, we set all the parameters $\xi_j$ to $1$.

	For the second step, we take a \emph{stable limit} described in, e.g.,
	\cite[Section 3.3]{BufetovMucciconiPetrov2018} 
	(the second of those limits). Namely, put $s_0=0$, then by 
	\cite[(3.7)]{BufetovMucciconiPetrov2018}
	we have
	\begin{equation}
		\label{eq:dual_sHL_symmetric_sum_proof}
		\mathsf{F}^{\text{stab, non-dual}}_{\lambda}
		(v_1,\ldots,v_k )=
		\frac{1}{(q;q)_{k-\ell(\lambda)}}
		\mathsf{F}^{\text{non-stab, non-dual}}
		_{\lambda\cup 0^{k-\ell(\lambda)}}(v_1,\ldots,v_k )
		\big\vert_{s_0=0}.
	\end{equation}
	Here $\lambda\cup 0^{k-\ell(\lambda)}\in \mathrm{Sign}_k$
	is obtained by appending the partition $\lambda=(\lambda_1,\ldots,\lambda_{\ell(\lambda)})$
	by $k-\ell(\lambda)$ zeroes.
	Then \eqref{eq:dual_sHL_symmetric_sum_proof} 
	is given by the symmetrization formula
	\cite[(4.23)]{BorodinPetrov2016inhom}
	\begin{equation*}
		\frac{(1-q)^k}{(q;q)_{k-\ell(\lambda)}}
		\sum_{\sigma\in \mathfrak{S}_k}
		\sigma
		\Biggl\{ 
			\prod_{1\le i < j\le k} \frac{v_i - q v_j}{v_i - v_j} 
			\prod_{j=1}^{\ell(\lambda)}
			\frac{v_j}{1-s_{\lambda_j}v_j}
			\prod_{i=1}^{\lambda_j-1}\frac{v_i-s_j}{1-s_jv_i}
		\Biggr\},
	\end{equation*}
	where we used the fact that $s_0=0$ to 
	pass from the product over $1\le j\le k$ to 
	$1\le j\le \ell(\lambda)$.
	
	For the third step, we use the fact that the 
	weights $w^{*,\bulk}_{v,s}$ we use differ from the 
	$w_{u,s}$'s for 
	$\mathsf{F}_{\lambda}^{\textnormal{non-stab, non-dual}}$
	by a conjugation factor
	$(s^2;q)_i / (q;q)_i$ \cite[(2.2)]{BorodinPetrov2016inhom},
	which brings the product over $i$ in 
	$\mathscr{C}(\lambda)$
	\eqref{eq:c_tilde_constant_dual_sHL}
	involved in our function $\mathsf{F}^*_{\lambda}$.

	For the fourth step, 
	we add the right boundary at $N$ to our vertex model
	by setting $s_N=0$ (recall that $\lambda_1\le N$). This turns the 
	factor $\frac{1}{1-s_{\lambda_j}v_j}$ into 
	$\bigl( \frac{1}{1-s_{\lambda_j}v_j} \bigr)^{\mathbf{1}_{\lambda_j<N}}$.

	Finally, 
	we set $s_1=\ldots=s_{N-1}=s$ to recover the homogeneous
	parameter $s$, and arrive at the desired symmetrization formula.
\end{proof}

The Yang--Baxter equations of 
\Cref{prop:sHL_sqW_YBE_bulk,prop:sHL_sqW_YBE_border}
translate into Cauchy identities 
for the functions $\mathbb{F}$ and $\mathsf{F}^{\mathrm{*}}$.

\begin{proposition}
	\label{prop:skew_Cauchy_ID_sqW_sHL}
	Fix $M\ge1$.
	For $N>0$, let $\mu \in \mathrm{Sign}_{N}^{\le M}$ and $\lambda \in 
	\mathrm{Sign}_{N+1}^{\le M}$. 
	Then, we have
	\begin{equation} \label{eq:skew_Cauchy_ID_sqW_sHL}
		\sum_{\nu\in \mathrm{Sign}_{N+1}^{\le M}} 
		\mathbb{F}_{\nu/\mu}(x) \, \mathsf{F}^{*}_{\nu'/\lambda'}(v) 
		=
		 \frac{1+v x}{1-sv} 
		 \sum_{\varkappa \in \mathrm{Sign}_{N}^{\le M}} 
		\mathbb{F}_{\lambda/\varkappa}(x) \, \mathsf{F}^{*}_{\mu'/\varkappa'}(v).
	\end{equation}
For $N=0$, we have
\begin{equation} \label{eq:skew_Cauchy_ID_sqW_sHL_first_row}
	\sum_{\nu \in \mathrm{Sign}_{1}^{\le M} } 
	\mathbb{F}_{\nu}(x) \, \mathsf{F}^{*}_{\nu'/\lambda'}(v) 
	=
	 (1+v x) 
    \,
    \mathbb{F}_{\lambda}(x).
\end{equation}
\end{proposition}
Note that all the sums in this proposition
are over finite sets of signatures, so there are no 
convergence issues.
\begin{proof}[Proof of \Cref{prop:skew_Cauchy_ID_sqW_sHL}]
	The proof of \eqref{eq:skew_Cauchy_ID_sqW_sHL} is similar to that of 
	\Cref{prop:sqW_are_symmetric} as it also uses a ``cross dragging'' argument. 
	The summation in the left-hand side of
	\eqref{eq:skew_Cauchy_ID_sqW_sHL} is the partition function of path
	configurations across two rows of vertices glued together:
	\begin{itemize}
		\item 
			the lower
			row has weights $W_{x,s}^{\leftwall},W_{x,s}^{\bulk},W^{\rightcorner}$
			and
			boundary condition $\mu$ at the bottom and $\nu$ at the top;
		\item the upper one has weights 
			$w_{v,s}^{*,\,
			\leftwall},w^{*,\bulk}_{v,s},W^{*,\rightwall}$,
			and
			boundary condition $\nu$ at the bottom and $\lambda$ at the top.
	\end{itemize}
	Recall that the encoding of arrow configurations by 
	signatures is described in detail in \Cref{sub:directed_paths}.

	The Yang--Baxter equation
	\eqref{eq:YBE_RWw_star_wall} implies that the action of swapping
	weights at the rightmost pair of columns, makes a cross weight
	appear at their left, as shown in \Cref{fig:YBE_W_Wstar}\,(b).
	We then push the cross to the left
	one vertical step at a time, 
	each time swapping the vertex weights
	and using the 
	Yang--Baxter equation
	\eqref{eq:YBE_RWw_star} as in \Cref{fig:YBE_W_Wstar}\,(a). 
	This
	procedure sequentially turns the left-hand side of
	\eqref{eq:skew_Cauchy_ID_sqW_sHL} into the right-hand side. 

	At the final step,
	we 
	push the cross out of the lattice at the leftmost site.
	Using 
	\eqref{eq:boundary_W_infinite_lines} and 
	\begin{equation*}
		w^{*,\leftwall}_{v}(j)= (1-sv)\, w^{*}_{v,s}(\infty,l;\infty,j), \qquad \text{for }l=0,1,
	\end{equation*}
	we obtain the combined
	contribution of the cross vertex weights 
	$\mathcal{R}_{x,v,s}$ (defined in \Cref{fig:table_R_cal} in the Appendix) 
	corresponding to the two cross configurations
	\begin{tikzpicture}[scale=0.6,baseline=0.5mm, thick]
			\draw[densely dotted] (0,0) -- (0.5,.5);
			\draw[densely dotted] (0,0.5) -- (.5,0);
	\end{tikzpicture} 
	and
	\begin{tikzpicture}[scale=0.6,baseline=0.5mm, thick]
			\draw[densely dotted] (0,0) -- (0.25,.25);
			\draw[densely dotted] (0,0.5) -- (.25,0.25);
			\draw[line width = .4mm,blue!50] (0.25,0.25) -- (0.5,0.5);
			\draw[line width = .4mm,red!50] (0.25,0.25) -- (.5,0);
	\end{tikzpicture}
	.
	Their sum gives the factor $(1+vx)/(1-sv)$ in the right-hand side of
	\eqref{eq:skew_Cauchy_ID_sqW_sHL}, as desired. 
	
	The second 
	identity \eqref{eq:skew_Cauchy_ID_sqW_sHL_first_row} can be verified
	by simply using definition of functions.
\end{proof}

Combining the skew Cauchy identities of Proposition
\ref{prop:skew_Cauchy_ID_sqW_sHL}, we come to the following corollary
for several variables:

\begin{corollary}
	\label{cor:sHL_sqW_Cauchy_2}
	For any positive integers $N,M,m$ we have
	\begin{equation} \label{eq:Cauchy_Id_sqW_sHL}
		\sum_{\lambda\in \mathrm{Sign}_M^{\le N}} 
		\mathbb{F}_\lambda(x_1,\dots, x_N)\,
		\mathsf{F}_{\lambda'}^{*}(v_1,\dots, v_m) 
		=
		\prod_{j=1}^m \left(\frac{1}{1-s v_j}\right)^{N-1} 
		\prod_{i=1}^N
		\prod_{j=1}^m (1 + v_j x_i).
	\end{equation}
\end{corollary}
\begin{proof}
	We use the branching expansion of functions
	$\mathbb{F}_{\lambda},\mathsf{F}^{*}_{\lambda'}$
	and then apply the single-variable skew
	Cauchy identities
	\eqref{eq:skew_Cauchy_ID_sqW_sHL} and
	\eqref{eq:skew_Cauchy_ID_sqW_sHL_first_row}.
\end{proof}

\begin{proposition}
	\label{prop:sHL_orthogonality}
	Let $0<q<1$ and $-1<s<0$.
	For any 
	$\lambda,\mu \in \mathrm{Sign}^{\le N}_M$
	and $k\ge M$,
	we have 
	\begin{equation} \label{eq:orthog_FF}
		\frac{1}{k!}
		\oint_{\gamma} 
		\frac{dz_1}{ 2 \pi \mathrm{i} z_1}
		\cdots 
		\oint_{\gamma} 
		\frac{dz_k}{ 2 \pi \mathrm{i} z_k}
		\prod_{1\le i \neq j\le k} \frac{z_i - z_j}{z_i - q z_j}\,
		\mathsf{F}^{*}_\lambda(z_1, \dots,z_k)
		\mathsf{F}^{*}_\mu(1/z_1, \dots,1/z_k)
		= \mathscr{C}(\lambda)
		\mathbf{1}_{\lambda = \mu},  
	\end{equation}
	where $\gamma$ is a positively oriented contour encircling $0$, $q^{j}
	s$ for all $j\geq 0$, and the contour $q\gamma$, but not the point $s^{-1}$.
\end{proposition}
\begin{proof}
	This follows from \cite[Corollary 7.5]{BorodinPetrov2016inhom} after
	specializing the inhomogeneous spin Hall--Littlewood functions 
	$\mathsf{F}_{\lambda}^{\textnormal{non-stab, non-dual}}$
	as described
	in the proof of \Cref{prop:dual_sHL_symmetric_sum}.
\end{proof}

\subsection{\texorpdfstring{Dual spin $q$-Whittaker polynomials}{Dual spin q-Whittaker polynomials}}
\label{sub:sqw_dual_def}

Let us also define the dual versions of the sqW weights.
These dual weights 
correspond to down-right lattice paths,
and are given by (we use the notation \eqref{eq:W_left_boundary}--\eqref{eq:Whit_W}):
\begin{align}
	W^{*,\leftwall}_{y,s}(j)&\coloneqq 
	W^{\leftwall}_{y,s}(j);
	\label{eq:Wstar_leftwall}
	\\
	W^{*,\bulk}_{y,s}(i_1,j_1;i_2,j_2) 
	&\coloneqq 
	\frac{(s^2;q)_{i_1}}{(q;q)_{i_1}} 
	\frac{(q;q)_{i_2}}{(s^2;q)_{i_2}}
	\,  
	W^{\bulk}_{y,s}(i_2,j_1;i_1,j_2).
	\label{eq:Wstar_bulk}
\end{align}
We will also use the right boundary weights
$W^{*,\rightwall}(i_1,j_1;i_2)$ as in \eqref{eq:W_rightwall}. 

This choice of
vertex weights
implies that nonzero global weights are assigned
to 
configurations of down-right paths in the grid $\{
0,\dots, N\} \times \{ 0, \dots ,k\}$ which are encoded by sequences of interlacing
signatures $\lambda^1 \prec \cdots \prec \lambda^k$.
(Compare this with the transposed interlacing property for the sHL functions.)

\begin{definition}
	For given interlacing signatures 
	$\lambda,\mu \in \mathrm{Sign}_N$, 
	the \emph{skew dual spin $q$-Whittaker 
	polynomial} in one variable $\mathbb{F}^{*}_{\lambda / \mu}(y)$ 
	is the weight of the unique down-right path configuration between $\mu$ and $\lambda$ 
	at a single row of vertices, with the
	weights \eqref{eq:Wstar_bulk}, \eqref{eq:Wstar_leftwall}
	and \eqref{eq:W_rightwall}. 
	Recall that the encoding of arrow configurations
	by signatures is described in \Cref{sub:directed_paths}.

	An explicit expression for the skew dual sqW polynomial is 
  \begin{equation}
		\label{eq:dual_sqw_single_variable_expression}
        \mathbb{F}^{*}_{\lambda / \mu}(y) 
        \coloneqq
		y^{|\lambda|-|\mu|} 
		\,\frac{(-s/y;q)_{\lambda_N - \mu_N}}{(q;q)_{\lambda_N - \mu_N}}
		\prod_{i=1}^{N-1} \frac{(-s/y;q)_{\lambda_i - \mu_i} (-s y;q)_{\mu_i - \lambda_{i+1}} 
		(q;q)_{\mu_i - \mu_{i+1}  }  }
		{(q;q)_{\lambda_i - \mu_i} (q;q)_{\mu_i - \lambda_{i+i} } (s^2;q)_{\mu_i - \mu_{i+i}}  }.
	\end{equation}
	Observe that $\mathbb{F}^{*}_{\lambda / \mu}(y)$ is a polynomial
	in $y$.

	Multi-variable extensions 
	$\mathbb{F}^{*}_{\lambda / \mu}(y_1,\ldots,y_k )$,
	where $\lambda,\mu\in \mathrm{Sign}_N$ are arbitrary,
	are defined via branching in the same say as in \eqref{eq:sqW_many_variables_branching}.
	The non-skew functions are 
	$\mathbb{F}_{\nu}^*\equiv \mathbb{F}^*_{\nu/0^N}$, where $\nu\in \mathrm{Sign}_N$,
	and $0^N\in \mathrm{Sign}_N$ is the signature with all parts equal to zero.
\end{definition}

\begin{remark}
	\label{rmk:dual_sqW_any_number_of_variables}
	Like the dual sHL functions (cf.~\Cref{rmk:sHL_any_number_of_variables})
	and unlike the usual sqW polynomials
	(cf.~\Cref{rmk:sqW_fixed_number_of_variables}), 
	the dual sqW polynomials 
	$\mathbb{F}^{*}_{\lambda / \mu}(y_1,\ldots,y_k )$ make 
	sense for any number of variables $k$, regardless of the signatures
	$\lambda,\mu$.
\end{remark}

\begin{proposition}
    The 
		polynomials $\mathbb{F}^{*}_{\lambda / \mu}(y_1,\dots, y_k)$ are symmetric.
\end{proposition}
\begin{proof}
	This follows from the Yang--Baxter equation \eqref{eq:YBE_RWW}
	and the sum-to-one property of the R-matrix $R$ given by 
	\eqref{eq:R_matrix_WW}. 
	It suffices to consider swapping two variables.
	We 
	apply the usual ``cross-dragging'' argument to exchange spectral
	parameters of two consecutive rows of vertices. 
	Similarly to 
	the proof of 
	\Cref{prop:sqW_are_symmetric},
	identity \eqref{eq:YBE_RWW}
	suffices to swap spectral parameters from the
	leftmost column up until the rightmost one. 
	Since the right boundary
	weights $W^{*,\rightwall}$ differ from the bulk weights $W^{*,\bulk}$, we
	have to prove that we can drag the cross 
	one more step to the right.
	We have using the definition of 
	$W^{*,\rightwall}$ that the partition function near the right wall 
	with the cross vertex is equal to
	\begin{equation*} 
			\sum_{k_1,k_2,k_3} R_{x,y} 
			(i_1,i_2;k_1,k_2) 
			W^{*,\rightwall}
			(k_3,k_2;j_3)
			W^{*,\rightwall}
			(i_3,k_1;k_3) 
			=
			\sum_{k}
			R_{x,y}(i_1,i_2;i_1+i_2-j_3-k,k-j_3).
	\end{equation*}
	(We used the arrow preservation property $i_1+i_2+j_3=i_3$.)
	The right-hand side is equal to one.
	Indeed, this sum-to-one property readily follows from the 
	$q$-Chu--Vandermonde identity.
	On the other hand, without the cross vertex, the partition function 
	near the right wall is equal to
	$\sum_{k}
	W^{*,\rightwall}(i_3,i_2;k) 
	W^{*,\rightwall}(k,i_1;j_3)$.
	This is also equal to $1$, because only the summand with $k=i_1+j_3$ is nonzero.
	This completes the proof.
\end{proof}

We finish this subsection by describing
Cauchy identities for our two sqW 
families $\mathbb{F},\mathbb{F}^{*}$.

\begin{proposition} 
	\label{prop:skew_Cauchy_ID_sqW_sqW}
	For $N > 0$, let $\mu \in \mathrm{Sign}_N$ and $\lambda \in \mathrm{Sign}_{N+1}$. 
	Then, for $|xy|<1$, we have
	\begin{equation} \label{eq:skew_Cauchy_ID_sqW_sqW}
		\sum_{\nu \in \mathrm{Sign}_{N+1}} \mathbb{F}_{\nu / \mu}(x)\, \mathbb{F}^{*}_{\nu / \lambda}(y) 
		=
	\frac{(-sx;q)_\infty (-sy;q)_\infty}{(s^2;q)_\infty (xy;q)_\infty} \sum_{\varkappa \in \mathrm{Sign}_N} \mathbb{F}_{\lambda / \varkappa}(x)\, \mathbb{F}^{*}_{\mu / \varkappa}(y). 
	\end{equation}
	For $N=0$, we have
	\begin{equation} \label{eq:skew_Cauchy_ID_sqW_sqW_first_row}
		\sum_{\nu \in \mathrm{Sign}_{N+1}} \mathbb{F}_{\nu}(x)\, \mathbb{F}^{*}_{\nu / \lambda}(y) 
		= 
		\frac{(-sx;q)_\infty}{( xy;q)_\infty}  \,\mathbb{F}_{\lambda}(x). 
	\end{equation}
\end{proposition}
\begin{proof}
	For $N>0$ this is proven using the same method
	explained in \Cref{prop:skew_Cauchy_ID_sqW_sHL}
	with the help of identity
	\eqref{eq:sum_R_matrix} when extracting the cross vertex weight from the
	rightmost column. For $N=0$ the statement is simply the $q$-binomial theorem.
\end{proof}

\begin{corollary}
	Let $|x_i y_j|<1$ for all $i=1,\dots,N$, $j=1,\dots,k$. Then, we have
	\begin{equation} \label{eq:Cauchy_Id_sqW_sqW}
		\sum_{\lambda \in \mathrm{Sign}_N} 
		\mathbb{F}_{\lambda}(x_1, \dots, x_N)\, \mathbb{F}^{*}_{ \lambda}(y_1, \dots, y_k) 
		= 
		\prod_{j=1}^k \left( \frac{(-sy_j;q)_\infty}{(s^2;q)_\infty} \right)^{N-1}
		\prod_{i=1}^N \prod_{j=1}^k \frac{(-sx_i ;q)_\infty}{(x_i y_j;q)_\infty}.
	\end{equation}
\end{corollary}

\subsection{Pieri rules}
\label{sub:Pieri}

Pieri type rules for the Borodin--Wheeler spin $q$-Whittaker polynomials
$\mathbb{F}_{\lambda}^{BW}$ are given in 
\cite{BorodinWheelerSpinq}.
These are analogs of the classical 
Pieri type rules for Macdonald polynomials.
The Pieri type rules follow from skew Cauchy identities, 
and here we present these rules for our version of the spin $q$-Whittaker polynomials.

\begin{proposition}\label{prop:Pieri1}
	Let $|x_iy|<1$ for all $i=1,\ldots,N $. Then we have
\begin{equation*}
	\sum_{\lambda \in \mathrm{Sign}_N } \mathbb{F}_\lambda (x_1, \dots x_N)\,
	\mathbb{F}^{*}_{\lambda / \mu} (y) 
	=
	\Biggl( \sum_{i \ge 0} y^i \frac{(-s/y;q)_i}{(q;q)_i}\, \mathbb{F}_{(i)}(x_1, \dots, x_N) \Biggr) 
	\mathbb{F}_{\mu}(x_1, \dots , x_N).  
\end{equation*}
\end{proposition}

\begin{proof}
	By the skew Cauchy identities of \Cref{prop:skew_Cauchy_ID_sqW_sqW}, we can write
\begin{equation}
	\sum_{\lambda \in \mathrm{Sign}_N } \mathbb{F}_\lambda (x_1, \dots x_N)\, \mathbb{F}^{*}_{\lambda / \mu} (y) =
    \left( \frac{(-sy;q)_\infty}{(s^2;q)_\infty} \right)^{N-1} \prod_{i=1}^N \frac{(-sx_i ;q)_\infty}{(x_i y;q)_\infty}
    \,
    \mathbb{F}_{\mu}(x_1, \dots , x_N),
\end{equation}
and the claim
follows by expanding $\bigl( \frac{(-sy;q)_\infty}{(s^2;q)_\infty} \bigr)^{N-1} \prod_{i=1}^N \frac{(-sx_i ;q)_\infty}{(x_i y;q)_\infty}$
using \eqref{eq:Cauchy_Id_sqW_sqW}.
\end{proof}

\begin{proposition}\label{prop:Pieri2}
We have
\begin{equation*} 
	\sum_{\lambda} \mathbb{F}_\lambda (x_1, \dots x_N)\, \mathsf{F}^{*}_{\lambda' / \mu'} (v)
    =
		\Biggl( \sum_{i=0}^N \mathsf{F}^{*}_{(i)} (v)  \,\mathbb{F}_{1^i} (x_1, \dots x_N) \Biggr)\, \mathbb{F}_{\mu}(x_1, \dots , x_N)
\end{equation*}
\end{proposition}

\begin{proof}
	By the skew Cauchy identities of \Cref{prop:skew_Cauchy_ID_sqW_sHL}, we have
\begin{equation} \label{eq:second_Pieri_rule}
	\sum_{\lambda \in \mathrm{Sign}_N } \mathbb{F}_\lambda (x_1, \dots x_N)\, \mathsf{F}^{*}_{\lambda / \mu} (v) =
    \left( \frac{1}{1-sv} \right)^{N-1} \prod_{i=i}^N (1+x_i v)
    \,
    \mathbb{F}_{\mu}(x_1, \dots , x_N),
\end{equation}
and the claim
follows by expanding $\bigl( \frac{1}{1-sv} \bigr)^{N-1} \prod_{i=i}^N (1+x_i v)$ using \eqref{eq:Cauchy_Id_sqW_sHL}.
\end{proof}

Pieri type rules of \Cref{prop:Pieri1,prop:Pieri2}
are eigenrelations 
on the spin $q$-Whittaker polynomials
in the \emph{label} variable.
Indeed, define 
operators $\mathfrak{H}^{\mathrm{sqW}}, \mathfrak{H}^{\mathrm{sHL}}$ as
\begin{equation} \label{eq:pieri_op}
	(\mathfrak{H}^{\mathrm{sqW}} f)(\mu) = \sum_\lambda f(\lambda)\, \mathbb{F}^{*}_{\lambda / \mu} (y),
	\qquad
	\qquad
	(\mathfrak{H}^{\mathrm{sHL}} f)(\mu) = \sum_\lambda f(\lambda) \,\mathsf{F}^{*}_{\lambda' / \mu'} (v).
\end{equation}
Then these operators act diagonally 
on spin $q$-Whittaker functions
$f(\lambda)=\mathbb{F}_\lambda(x_1,\ldots,x_N )$,
with respective eigenvalues
\begin{equation*}
	\sum_{i \ge 0} y^i \frac{(-s/y;q)_i}{(q;q)_i}\, \mathbb{F}_{(i)}(x_1, \dots, x_N) 
	\qquad \textnormal{and} \qquad 
	\sum_{i=0}^N \mathsf{F}^{*}_{(i)} (v)  \,\mathbb{F}_{1^i} (x_1, \dots x_N).
\end{equation*}

\section{Difference operators}
\label{sec:difference_operators}

Here we show that the sHL and sqW 
functions satisfy
certain eigenrelations under operators
acting the \emph{spectral parameters}
(as opposed to labels as in \Cref{sub:Pieri}).
These operators are 
$s$-deformations of the
($q=0$ or $t=0$) Macdonald difference operators.
Half of these eigenrelations
essentially appears in \cite{BufetovMucciconiPetrov2018}, but
here we obtain eigenrelations in a form which is more symmetric with
respect to $q,t$. 

We will denote the ``quantization'' parameter
by $q$ throughout this section, 
except for \Cref{sub:eigen_sHL}
where it will be denoted by $t$ instead of $q$.
This is done for consistency with classical literature (e.g., \cite{Macdonald1995}), where
Hall--Littlewood functions (obtained from sHL functions by setting $s=0$)
are the $q=0$ degenerations of the Macdonald polynomials $P_\lambda(\cdot;q,t)$.

Throughout the entire section we make use of the shift operator 
\begin{equation*}
    T_{q,z_i}f(z_1, \dots, z_M) = f(z_1,\dots,z_{i-1}, q z_i, z_{i+1},\dots, z_M).
\end{equation*}

\subsection{Eigenrelations for the spin Hall--Littlewood functions}
\label{sub:eigen_sHL}

We begin by essentially repeating the definition of a
family of eigenoperators for the spin Hall--Littlewood polynomials
from \cite{BufetovMucciconiPetrov2018}.

\begin{definition}
    For $r \in \{ 1, \dots, M \}$, the $r$-th Hall--Littlewood operator is given by
    \begin{equation*}
        \overline{\mathfrak{D}}_r^* \coloneqq 
				\sum_{\substack{ I \in \{ 1,\dots, M \}\\ |I|=r }} 
				\prod_{\substack{i \in I \\ k \in \{ 1, \dots, M\} \setminus I }} \frac{v_k - t v_i}{v_k - v_i} 
				\prod_{i\in I}T_{0,v_i}.
    \end{equation*}
		This is the $q=0$ specialization of the $r$-th Macdonald difference
		operator \cite[Ch. VI.3]{Macdonald1995}.
		The operators $\overline{\mathfrak{D}}_r^*$
		act diagonally on the Hall--Littlewood symmetric polynomials.
\end{definition}

It was discovered in \cite{BufetovMucciconiPetrov2018} that the
(stable) spin Hall--Littlewood functions (first introduced in \cite{deGierWheeler2016}), 
much like the classical Hall--Littlewood
polynomials \cite[Ch. III]{Macdonald1995}, 
are eigenfunctions of the difference operators $\overline{\mathfrak{D}}^*_r$.
The same result holds for our dual sHL functions 
$\mathsf{F}^*_\nu$, and it is given in the next theorem.

\begin{theorem}
	For any $\lambda \in \mathrm{Sign}_M$, we have
	\begin{equation} \label{eq:sHL_eigen_HL}
		\overline{\mathfrak{D}}_r^* \mathsf{F}^*_{\lambda} = e_r(1,t,\dots, t^{M-\ell(\lambda) -1})\, \mathsf{F}^*_{\lambda}.
	\end{equation}
	Here
	$\displaystyle e_r(x_1,\dots, z_n) = \sum_{1\le i_1 < \cdots i_r \le n} z_{i_1} \cdots z_{i_r}$
	is the $r$-th elementary symmetric polynomial.
\end{theorem}
\begin{proof}
	The proof is analogous to that of
	\cite[Theorem 8.2]{BufetovMucciconiPetrov2018}: we get
	\eqref{eq:sHL_eigen_HL} 
	by
	directly
	evaluating the action of $\overline{\mathfrak{D}}^*_r$ on the
	symmetrization formula \eqref{eq:dual_sHL_symmetric_sum}
	(with $q$ replaced by $t$). We do not
	repeat the argument here.
\end{proof}

The operator we introduce next depends on 
the number of variables $M$ and 
on an additional positive integer $N$.
Moreover, this operator acts only on a certain subspace of rational functions.
Namely, let $\mathbf{V}(M)$ be the space of symmetric rational 
functions in $M$ variables
$v_1,\dots,v_M$ 
of degree $\le 1$ in each variable. 
That is, its elements are
functions $f(v_1,\dots,v_M)=a(v_1,\dots,v_M)/b(v_1,\dots,v_M)$,
where $a$ and $b$ are polynomials such that $\deg_{v_i}(a)-\deg_{v_i}(b) \le 1$ for all $i=1,\ldots,M $.
One readily sees that $\mathbf{V}(M)$ is a linear space.
The
dual sHL functions 
$\mathsf{F}^*_\nu(v_1,\ldots,v_M )$
belong
to $\mathbf{V}(M)$, see
\eqref{eq:dual_sHL_symmetric_sum}. 

\begin{definition}
    For positive integers $M,N$ define the \emph{dual $s$-deformed Macdonald operator} by
		\begin{equation}\label{eq:dual_operator_HL}
			\mathfrak{D}_{1,N}^* \coloneqq \sum_{j=1}^M \prod_{\substack{l=1 \\ l \neq j}}^M \frac{v_j - t v_l}{v_j - v_l}\,
			\mathfrak{C}_{j,N},
    \end{equation}
    where
    \begin{equation*}
			\mathfrak{C}_{j,N} \coloneqq v_j \left( \frac{v_j - s}{1 - s v_j} \right)^{N-1} 
			(-s)^{N-1} \lim_{\varepsilon \to 0} \varepsilon \, T_{\varepsilon^{-1},v_j}.
    \end{equation*}
		The limit 
		$\lim_{\varepsilon \to 0} \varepsilon \, T_{\varepsilon^{-1},v_j}$ is well-defined
		on $\mathbf{V}(M)$, so
		$\mathfrak{D}_{1,N}^*$
		acts in the space $\mathbf{V}(M)$.
\end{definition}

\begin{theorem}
	\label{thm:new_boxed_sHL_eigenoperator}
		For any boxed signature $\lambda \subseteq \mathrm{Box}(N,M)$
		(recall that this is $\mathrm{Sign}_M^{\le N}$),
		we have
    \begin{equation}\label{eq:eigenrelation_sHL}
        \mathfrak{D}_{1,N}^* \mathsf{F}^{*}_{\lambda}
        =
        e_1(1,t,\dots,t^{\lambda_N'-1})
        \,
				\mathsf{F}^{*}_{\lambda},
    \end{equation}
		where $\lambda'$ is the transposed signature.
		In particular, $\lambda_N'=\#\left\{ i\colon \lambda_i=N \right\}$.
\end{theorem}
\begin{proof}
    We make use of the symmetrization formula \eqref{eq:dual_sHL_symmetric_sum}
		(recall that we have replaced the parameter $q$ by $t$ throughout this subsection).
		We use the notation
    \begin{equation*}
        A= \prod_{1\le l < r \le M} \frac{v_l - t v_r}{v_l - v_r} 
        \qquad 
        \text{and} 
        \qquad
        B=
        \prod_{r=1}^{\ell (\lambda)} v_r 
        \left( 
        \frac{v_r- s}{1 - s v_r}
        \right)^{\lambda_r-1} 
        \left( 
        \frac{1}{1 - s v_r}
        \right)^{\mathbf{1}_{\lambda_r < N}}, 
    \end{equation*}
    so that $\mathsf{F}^{*}_\lambda= \mathscr{C}(\lambda) \sum_{ \sigma \in \mathfrak{S}_M } \sigma \{ A B \}$. 
		The operator $\mathfrak{D}_{1,N}^*$ acts as
    \begin{equation*}
        \mathfrak{D}_{1,N}^* \mathsf{F}^{*}_{\lambda}
        =
        \mathscr{C}(\lambda) 
				\sum_{i=1}^M 
				\prod_{\substack{j=1 \\ j\neq i}}^M 
				\frac{v_i - t v_j}{v_i - v_j} 
				\sum_{ \sigma \in \mathfrak{S}_M }
				\mathfrak{C}_{i,N} ( \sigma \{AB\} ).
    \end{equation*}
    The action of $\mathfrak{C}_{i,N}$ on the product $\sigma\{AB\}$ can be split as
    \begin{equation} \label{eq:CAB}
			\mathfrak{C}_{i,N} (\sigma\{AB \})
			= 
			\lim_{\varepsilon \to 0} \sigma\{A\}\Big\vert_{v_i=1/\varepsilon} \, \times \mathfrak{C}_{i,N} (\sigma\{ B \}).
    \end{equation}
    Assume now that $\lambda'_N=L$, that is $\lambda_1=\cdots=\lambda_L=N$ and $\lambda_{L+1}<N$, for some $L \in \{0,\dots,M\}$. 
    We focus on the second factor of \eqref{eq:CAB}. A simple computation shows that
    \begin{equation}
			\label{eq:CiN_action_sigma_B}
			\mathfrak{C}_{i,N} (\sigma\{ B \} )
        =
        \begin{cases}
            \sigma\{ B \} \qquad & \text{if } i \in \sigma ( \{1, \dots, L\} ),\\
            0 \qquad & \text{else},
        \end{cases}
    \end{equation}
    that in particular, implies that $\mathfrak{C}_{i,N} \sigma\{ B \}=0$ when $L=0$, confirming \eqref{eq:eigenrelation_sHL} in this specific case. 
    
    For $L>0$ and a permutation $\sigma$ such that $i \in \sigma ( \{1,\dots,L\} )$, call $\bar{k}$ the element such that $\sigma(\bar{k})=i$. We rewrite $A$ into a product of factors $A=A_1 A_2 A_3$, obtained dividing the triangular product as
    \begin{equation*}
        A_1 = \prod_{1\le l < r < \bar{k} } \frac{v_l - t v_r}{v_l - v_r}, 
        \hspace{20pt}
        A_2 = \prod_{1\le l <  \bar{k} } \frac{v_l - t v_{\bar{k}}}{v_l - v_{\bar{k}}},
        \hspace{20pt}
        A_3 = \prod_{\substack{1\le l < M \\ \max(l,\bar{k}) < r \le M } } \frac{v_l - t v_r}{v_l - v_r}.
    \end{equation*}
    We can evaluate the first factor in the right-hand site of \eqref{eq:CAB} as
    \begin{equation*}
        \prod_{\substack{j=1 \\ j\neq i}}^M \frac{v_i - t v_j}{v_i - v_j} \lim_{\varepsilon\to 0} \sigma\{A\}\Big\vert_{v_i=1/\varepsilon}
        =
        t^{\bar{k}-1}
        \sigma
        \{
					A_1 \widetilde{A}_2 A_3
        \},
    \end{equation*}
    where
    \begin{equation*}
        \widetilde{A}_2 \coloneqq \prod_{1\le l <  \bar{k} } \frac{v_{\bar{k}} -t v_l }{v_{\bar{k}} -v_l }.
    \end{equation*}
		The action of $\mathfrak{D}_{1,N}^*$ on the sHL function can be therefore expressed (ignoring $\mathscr{C}(\lambda)$) as
    \begin{equation}\label{eq:dual_operator_HL_expansion}
        \sum_{\bar{k}=1}^L t^{\bar{k}-1}
        \sum_{i=1}^M
        \sum_{\substack{\sigma \in \mathfrak{S}_M \\ \sigma(\bar{k})=i}} \sigma\{ A_1 \widetilde{A}_2 A_3 B\}
        =
        \sum_{ \bar{k} =1 }^L t^{ \bar{k}-1 }
        \sum_{\sigma \in \mathfrak{S}_M } \sigma\{ A_1 \widetilde{A}_2 A_3 B \}.
    \end{equation}
		To prove relation \eqref{eq:eigenrelation_sHL} we
		show that each term $\sigma\{ A_1 \widetilde{A}_2 A_3 B \}$ is
		equal to one of the terms $\tau\{ A_1 A_2 A_3 B \}$ in the expansion of the original sHL function.
		For each permutation $\tau \in \mathfrak{S}_M$ and each $\bar{k}$ define the permutation $\sigma$ as
    \begin{equation*}
        \sigma(j)= 
        \begin{cases}
        \tau(j+1) \qquad &\text{if }j=1,\dots,\bar{k}-1,
        \\
        \tau(1) \qquad &\text{if }j=\bar{k},
        \\
        \tau(j) \qquad &\text{if }j=\bar{k}+1, \dots M.
        \end{cases}
    \end{equation*}
		With this choice we can easily check that $\sigma\{ A_3 B\}=\tau\{
		A_3 B \}$ and more crucially that $\sigma\{A_1 \widetilde{A}_2 \}=
		\tau\{A_1 A_2 \}$ since the cyclic shift in the first $\bar{k}$ terms of
		$\sigma$ makes up for the exchange of $\widetilde{A}_2$ and $A_2$.
		This in particular shows that the symmetric sum in the right-hand
		side of \eqref{eq:dual_operator_HL_expansion} is independent of
		$\bar{k}$ and it is equal, up to a factor
		$\mathscr{C}(\lambda)$ that we omitted, to
		$\mathsf{F}^{*}_{\lambda}(v_1,\dots,v_M)$. 
		The sum 
		$\sum_{\bar k=1}^L t^{\bar k-1}$
		is the desired eigenvalue $e_1(1,t,\dots,t^{\lambda_N'-1})$.
		This completes the
		proof.
\end{proof}

\begin{remark}[Limit to the Hall--Littlewood case]
	In the limit $s\to 0$, the new operator $\mathfrak{D}^*_{1,N}$
	\eqref{eq:dual_operator_HL} acting on the dual sHL functions
	should be replaced by 
	\begin{equation}\label{eq:HL_limit_of_new_sHL_operator}
		D_{1,N}^*=\sum_{j=1}^M \prod_{\substack{l=1 \\ l \neq j}}^M \frac{v_j - t v_l}{v_j - v_l} 
			\,v_j^N \lim_{\varepsilon \to 0} \varepsilon^N \, T_{\varepsilon^{-1},v_j},
	\end{equation}
	by mimicking the action \eqref{eq:CiN_action_sigma_B}.
	Similarly to \Cref{thm:new_boxed_sHL_eigenoperator},
	one can show that $D_{1,N}^*$
	acts diagonally on the 
	Hall--Littlewood polynomials $P_\lambda(\cdot;0,t)$.

	The same operator \eqref{eq:HL_limit_of_new_sHL_operator} 
	can be also obtained as a $q\to 0$ limit of a certain operator
	diagonal in the Macdonald polynomials $P_\lambda(\cdot;q,t)$. 
	Take the first Macdonald $q^{-1}$-difference operator 
	\begin{equation}
		\label{eq:M_1}
		M_1=\sum_{j=1}^M\prod_{\substack{i=1\\i\ne j}}^{M}
		\frac{tx_i-x_j}{x_i-x_j}\,T_{q^{-1},x_j}.
	\end{equation}
	It acts on the Macdonald polynomials 
	$P_\lambda(x_1,\ldots,x_M ;q,t)$ with eigenvalues
	$\sum_{i=1}^{M}q^{-\lambda_i}t^{i-1}$
	(this follows from, e.g., \cite[Section 2.2.3]{BorodinCorwin2011Macdonald}).
	Denote by $\mathbf{P}_N$ the subspace of polynomials in $x_1,\ldots,x_M $
	which have degree $\le N$ in each of the variables $x_i$.
	It is spanned by the Macdonald polynomials
	$P_\lambda(x_1,\ldots,x_M;q,t )$ with $\lambda_1\le N$, i.e., $\lambda\subseteq \mathrm{Box}(N,M)$.
	On $\mathbf{P}_N$ consider the operator
	$q^N M_1$. Its limit as $q\to0$ is well-defined.
	By looking at eigenvalues on Hall--Littlewood polynomials 
	$P_\lambda(x_1,\ldots,x_M;0,t )$ with $\lambda_1\le N$,
	one readily sees that this limit coincides with $D_{1,N}^{*}$.
\end{remark}

\subsection{\texorpdfstring{Eigenrelations for the spin $q$-Whittaker polynomials}{Eigenrelations for the spin q-Whittaker polynomials}}
\label{sub:sqW_eigenrelations}

The duality between sHL functions and sqW polynomials
(\Cref{cor:sHL_sqW_Cauchy_2} and \Cref{prop:sHL_orthogonality})
allows to pass from the eigenoperators for the sHL
functions to the 
ones for the sqW polynomials.

\begin{definition} 
	[Spin $q$-Whittaker difference operators]
	\label{def:sqW_operators}
	Fix a positive integer $N$, and define 
	the \emph{$s$-deformed $q$-Whittaker operators}
	\begin{equation} \label{eq:D_1}
		\mathfrak{D}_1 \coloneqq \sum_{i=1}^N \prod_{\substack{j=1 \\ j \neq i}}^N \frac{(1+s x_i)}{1-x_i/x_j}\, T_{q, x_i},
	\end{equation}
	and
	\begin{equation}
		\label{eq:D_1_bar}
		\overline{\mathfrak{D}}_1 \coloneqq \sum_{i=1}^N \prod_{\substack{j=1 \\ j \neq i}}^N \frac{(1+s / x_i)}{1-x_j/x_i}\, T_{q^{-1}, x_i}.
	\end{equation}
\end{definition}

Let us make two remarks after this definition.
\begin{remark}
	The operators $\mathfrak{D}_1$ and $\overline{\mathfrak{D}}_1$ 
	reduce for $s=0$ to the $t=0$ specializations
	of the two Macdonald $q$-difference operators.
	The first one is 
	the standard first order Macdonald operator
	$\sum_{i=1}^N\prod_{j\ne i}\frac{tx_i-x_j}{x_i-x_j}\,T_{q,x_i}$ 
	(denoted by $D_N^1$ in \cite[Ch. VI]{Macdonald1995}),
	and the second one is 
	$\sum_{i=1}^N\prod_{j\ne i}\frac{x_i-tx_j}{x_i-x_j}\,T_{q^{-1},x_i}$
	(denoted by $\tilde D_N^1$ in \cite[Section 2.2.3]{BorodinCorwin2011Macdonald}).
\end{remark}

\begin{remark}
	The operator $\mathfrak{D}_1$ is new.
	The other operator
	$\overline{\mathfrak{D}}_1$ is only a slightly more general version of
	the operator $\mathfrak{E}$ from \cite[Section 8]{BufetovMucciconiPetrov2018}.
	The latter is diagonal in the Borodin--Wheeler's 
	sqW polynomials $\mathbb{F}_\lambda^{BW}$.
	To recover $\mathfrak{E}$ from \eqref{eq:D_1_bar} one has to 
	take the limit $x_1 \to 0$, which agrees with 
	\Cref{prop:reduction_of_ours_to_BW} connecting the 
	$\mathbb{F}_\lambda^{BW}$'s with our sqW polynomials
	$\mathbb{F}_\lambda$.
\end{remark}

We establish two 
eigenrelations for the sqW polynomials in the next two theorems.

\begin{theorem} \label{thm:eigenrelation_sqW}
    For all signatures $\lambda \in \mathrm{Sign}_N$ we have
    \begin{equation} \label{eq:eigenrelation_sqW}
        \mathfrak{D}_1 \mathbb{F}_{\lambda} (x_1, \dots, x_N) = q^{\lambda_N} \mathbb{F}_{\lambda} (x_1, \dots, x_N).
    \end{equation}
\end{theorem}
\begin{proof}
    We will prove the identity
    \begin{equation} \label{eq:equality_operators_sHL_sqW}
        \left( 1 - (1-q) \mathfrak{D}_{1,N}^*  \right  ) \Pi(x;v) = \mathfrak{D}_1 \Pi(x;v),
    \end{equation}
    where
    \begin{equation}
			\label{eq:Pi_RHS}
        \Pi(x;v) = \prod_{j=1}^M \left(\frac{1}{1-s v_j}\right)^{N-1} \times\prod_{i=1}^N\prod_{j=1}^M (1 + v_j x_i).
    \end{equation}
		Indeed, 
		modulo \eqref{eq:equality_operators_sHL_sqW},
		the Cauchy Identity \eqref{eq:Cauchy_Id_sqW_sHL} and the eigenrelations \eqref{eq:eigenrelation_sHL}
		imply
    \begin{equation*}
        \sum_{\lambda \subseteq \mathrm{Box}(N,M)} 
				q^{\lambda_N} \mathbb{F}_\lambda(x_1,\dots,x_N) \mathsf{F}_{\lambda'}^{*}(v_1,\dots,v_M)
        = 
				\sum_{\lambda \subseteq \mathrm{Box}(N,M)} 
				\mathfrak{D}_1 \mathbb{F}_\lambda(x_1,\dots,x_N) \mathsf{F}_{\lambda'}^{*}(v_1,\dots,v_M),
    \end{equation*}
    and hence \eqref{eq:eigenrelation_sqW} follows by orthogonality of the sHL functions (\Cref{prop:sHL_orthogonality}).
    
		It thus suffices to establish \eqref{eq:equality_operators_sHL_sqW}.
		Define
    \begin{equation*}
        h(z)\coloneqq\prod_{j=1}^M (1+ v_j z).
    \end{equation*}
    We have
    \begin{equation*}
        \frac{\mathfrak{D}_1 \Pi(x;v) }{\Pi(x;v)} = \frac{\mathfrak{D}_1 h(x_1) \cdots h(x_N) }{ h(x_1) \cdots h(x_N) } = - \oint_{x_1,\dots,x_N} \prod_{i=1}^N \frac{x_i (1+s z)}{x_i - z} \frac{h(qz)}{h(z)} \frac{dz}{z(1+sz)},
    \end{equation*}
    where in the second equality we used the residue expansion of
		the complex integral and the contour encircles only the poles $x_1,\dots,x_N$. 
		By subtracting 1 from both sides, we can enlarge the complex
		contour to also include the pole at $z=0$ (note that $h(z)$ is
		nonsingular at $z=0$).
    After a change of variable $z=-1/w$, we get
    \begin{equation} \label{eq:1-D_integral}
        \frac{(-1+\mathfrak{D}_1) \Pi(x;v) }{\Pi(x;v)} 
        = 
        -\oint_{v_1,\dots,v_M} \prod_{k=1}^M\frac{w - q v_k}{w-v_k} 
				\,(w - s)^{N-1}
				\prod_{j=1}^N \frac{x_j}{1+x_j w}\, dw.
    \end{equation}
		In the right-hand side of \eqref{eq:1-D_integral}, after the
		change of variable, we switched the contour to a positively
		oriented curve around $v_1,\dots,v_M$, which yielded the
		negative sign in front. Using 
    \begin{equation*}
        \lim_{\varepsilon \to 0} \varepsilon \left(
        \frac{1}{(1 - s/\varepsilon)^{N-1}}
        \prod_{j=1}^N (1+x_j/\varepsilon)
        \right)
        =
        \frac{x_1 \cdots x_N}{(-s)^{N-1}}
    \end{equation*}
		and expanding the right-hand side of \eqref{eq:1-D_integral} as a
		sum of residues, we can rewrite it as
		\begin{equation*}
			-\frac{(1-q)\mathfrak{D}^*_{1,N}\Pi(x;v)}{\Pi(x;v)}.
		\end{equation*} 
		This proves
		\eqref{eq:equality_operators_sHL_sqW}, and hence the desired eigenrelation
		\eqref{eq:eigenrelation_sqW}.
\end{proof}

\begin{theorem} \label{thm:eigenrelation_sqW_conj}
    For all signatures $\lambda \in \mathrm{Sign}_N$ we have
    \begin{equation} \label{eq:eigenrelation_sqW_lambda_1}
        \overline{\mathfrak{D}}_1 \mathbb{F}_{\lambda} = q^{-\lambda_1} \mathbb{F}_{\lambda}.
    \end{equation}
\end{theorem}
\begin{proof}
	The proof of this eigenrelation is identical to that given in
	\cite{BufetovMucciconiPetrov2018} and similar to that of Theorem
	\ref{thm:eigenrelation_sqW}.
	It uses the fact that
	\begin{equation*}
		q^{-M} \left( 1-(1-q)\overline{\mathfrak{D}}_1^* 
		\right) \Pi(x;v) = \overline{\mathfrak{D}}_1 \Pi(x;v),
	\end{equation*}
	where $\Pi(x;v)$ is given by \eqref{eq:Pi_RHS}.
	We will not repeat the detailed argument here.
\end{proof}

\subsection{Commutation and conjugation}
\label{sub:qW_conjugation}

The $q$-difference operators 
$\mathfrak{D}_1$ \eqref{eq:D_1}
and $\overline{\mathfrak{D}}_1$ \eqref{eq:D_1_bar}
commute.
For this statement we cannot appeal to the
eigenrelations of \Cref{thm:eigenrelation_sqW,thm:eigenrelation_sqW_conj}
since we did not prove that the sqW polynomials 
form a basis for the ring of symmetric
polynomials in $N$ variables.
Nevertheless, the commutation can be checked independently:
\begin{proposition}
	\label{prop:sqW_operators_commute}
	We have $\mathfrak{D}_1\overline{\mathfrak{D}}_1 F
	=
	\overline{\mathfrak{D}}_1\mathfrak{D}_1 F$ for all 
	symmetric polynomials $F$ in $N$ variables.
\end{proposition}
\begin{proof}
	By polarization, it suffices to check 
	the action on product form functions $F(x_1,\ldots,x_N )=f(x_1) \ldots f(x_N) $,
	where $f(x)$ is an arbitrary polynomial.

	The action of each operator can we written as a contour integral:
	\begin{equation*}
		\begin{split}
			\mathfrak{D}_1 F&=
			-\frac{1}{2\pi \mathrm{i}}
			\oint \prod_{i=1}^{N}\left( f(x_i)\,\frac{x_i(1+sz)}{x_i-z} \right)\frac{f(qz)}{f(z)}
			\frac{dz}{z(1+sz)},\\
			\overline{\mathfrak{D}}_1 F&=\frac{1}{2\pi \mathrm{i}}
			\oint \prod_{i=1}^{N}\left( f(x_i)\,\frac{w+s}{w-x_i} \right)\frac{f(q^{-1}w)}{f(w)}
			\frac{dw}{w+s},\\
		\end{split}
	\end{equation*}
	where both integrals are over a contour containing $x_1,\ldots,x_N $ 
	and no other poles of the integrand. 
	Throughout the proof we assume that all contours exist, which might impose 
	some restrictions on the $x_i$'s.
	After checking the commutation under the restrictions, 
	we can lift these restrictions by an analytic continuation.

	We have
	\begin{equation*}
		\mathfrak{D}_1\overline{\mathfrak{D}}_1 F
		=-
		\frac{1}{(2\pi \mathrm{i})^2}
		\oint_{\gamma_z^1}\oint_{\gamma_w^1}
		\prod_{i=1}^{N}\left( f(x_i)\,\frac{w+s}{w-x_i}\frac{x_i(1+sz)}{x_i-z} \right)
		\frac{w-z}{w-qz}
		\frac{f(qz)f(q^{-1}w)}{f(z)f(w)}
		\frac{dw}{w+s}\,
		\frac{dz}{z(1+sz)},
	\end{equation*}
	where $\gamma_z^1$ contains both $\gamma_w^1$ and $q^{-1}\gamma_w^1$,
	while $\gamma_w^1$ is around $x_1,\ldots,x_N $
	and no other poles.
	In the other order, we have
	\begin{equation*}
		\overline{\mathfrak{D}}_1 \mathfrak{D}_1 F
		=-
		\frac{1}{(2\pi \mathrm{i})^2}
		\oint_{\gamma_w^2}\oint_{\gamma_z^2}
		\prod_{i=1}^{N}\left( f(x_i)\,\frac{x_i(1+sz)}{x_i-z}\frac{w+s}{w-x_i} \right)
		\frac{q^{-1}(w-z)}{q^{-1}w-z}
		\frac{f(q^{-1}w)f(qz)}{f(w)f(z)}
		\frac{dz}{z(1+sz)}
		\frac{dw}{w+s}
		,
	\end{equation*}
	but now $\gamma_w^2$ contains both $\gamma_z^2$
	and $q \gamma_z^2$, while $\gamma_z^2$ is 
	around $x_1,\ldots,x_N $ and no other poles. 
	Note that the integrands in both formulas coincide.

	In the first 
	expression, deform the integration contour $\gamma_z^1$
	to coincide with $\gamma_w^1$, which picks up the residue
	at $z=q^{-1}w$.
	In the second expression, deform the contour $\gamma_w^2$ to coincide with $\gamma_z^2$, 
	which picks up the residue at $w=qz$.
	The resulting double contour integrals are over the same contours and 
	are thus equal.
	It remains to check the equality of the single integrals of the 
	residues. We have
	\begin{equation*}
		\begin{split}
			\mathop{\mathrm{Res}}_{z=q^{-1}w}
			(\textnormal{integrand in 
			$
			\mathfrak{D}_1\overline{\mathfrak{D}}_1
			$})
			&=
			-
			(-1)^{N}(1-q)(s+w)^{N-1}(q+s w)^{N-1}
			\prod_{i=1}^{N}\frac{x_i f(x_i)}{(w-x_i)(w-q x_i)},
			\\
			\mathop{\mathrm{Res}}_{w=qz}
			(\textnormal{integrand in 
			$
			\overline{\mathfrak{D}_1}
			\mathfrak{D}_1
			$})
			&=
			(-1)^N(1-q)(1+s z)^{N-1}
			(s+q z)^{N-1}
			\prod_{i=1}^{N}\frac{x_if(x_i)}{(z-x_i)(q z-x_i)}
		\end{split}
	\end{equation*}
	We must show that the 
	integral of the first expression
	over $\gamma_w^1$ is the same 
	as
	the 
	integral of the second expression
	over $\gamma_z^2$.
	Noting that both expressions have zero residue 
	at infinity due to quadratic decay, we can 
	compute the first integral 
	as a sum of minus residues at $w=qx_i$.
	Then one readily sees that
	each minus residue at $w=qx_i$ is the 
	same as the residue of the second 
	expression at $z=x_i$. This shows the 
	desired commutation.
\end{proof}

\begin{remark}\label{rmk:inhom_D_1}
	Similarly to \cite{BorodinPetrov2016inhom},
	\cite{OrrPetrov2016}, one can define the following inhomogeneous generalizations of the 
	operators $\mathfrak{D}_1$ and $\overline{\mathfrak{D}}_1$, respectively:
	\begin{equation*} 
		\sum_{i=1}^N \prod_{r=1}^{N-1}(1+s_rx_i)\prod_{\substack{j=1 \\ j \neq i}}^N \frac{1}{1-x_i/x_j}
		\, T_{q, x_i},
		\qquad
		\sum_{i=1}^N
		\prod_{r=1}^{N-1}(1+s_r/x_i)
		\prod_{\substack{j=1 \\ j \neq i}}^N \frac{1}{1-x_j/x_i}\, T_{q^{-1}, x_i}.
	\end{equation*}
	A straightforward modification of the proof of \Cref{prop:sqW_operators_commute}
	shows that these operators commute, too.
\end{remark}

The discussion in the rest of this subsection 
aims in part to demonstrate why the result of
\Cref{prop:sqW_operators_commute}
is a rather unexpected one.

Both operators
$\mathfrak{D}_1$ \eqref{eq:D_1}
and $\overline{\mathfrak{D}}_1$ \eqref{eq:D_1_bar}
are related via conjugation to $q$-Whittaker difference operators.
The latter
are $t=0$ degenerations of the Macdonald $q$-difference operators 
from \cite{Macdonald1995}.
Denote for $r=1,\ldots,N $:
\begin{equation}
	\label{eq:qW_difference_op}
	W_N^r
	\coloneqq
	\sum_{|I|=r}
	\prod_{i\in I,\, j\notin I}\frac{1}{1-x_i/x_j}
	\prod_{i\in I}T_{q,x_i}
	,\qquad 
	\tilde W_N^r
	\coloneqq
	\sum_{|I|=r}\prod_{i\in I,\,j\notin I}\frac{1}{1-x_j/x_i}
	\prod_{i\in I}T_{q^{-1},x_i},
\end{equation}
where the sums are over subsets of $\left\{ 1,\ldots,N  \right\}$ of 
cardinality $r$.
These operators are diagonal in the usual $q$-Whittaker polynomials
(which are $t=0$ versions of the Macdonald polynomials).
In particular, $W_N^1$ and $\tilde W_N^1$ have eigenvalues
$q^{\lambda_N}$ and $q^{-\lambda_1}$, respectively,
on $q$-Whittaker polynomials.
All the operators $W_N^r, \tilde W_N^r$, $r=1,\ldots,N $,
commute. We refer to Sections 2.2.2 and 3.1.3 
in
\cite{BorodinCorwin2011Macdonald} for details.
Let 
\begin{equation*}
	\mathbb{U}_N \coloneqq
	\prod_{i=1}^{N}
	\frac{1}{(-sx_i;q)_\infty^{N-1}},
	\qquad 
	\mathbb{V}_N \coloneqq
	\prod_{i=1}^{N}
	\frac{1}{(-s/x_i;q)_\infty^{N-1}}
	.
\end{equation*}
A straightforward computation shows:
\begin{proposition}
	The spin $q$-Whittaker operators 
	\eqref{eq:D_1}, 
	\eqref{eq:D_1_bar}
	are conjugates of the 
	first $q$-Whittaker operators \eqref{eq:qW_difference_op}:
	\begin{equation*}
		\mathfrak{D}_1=
		\mathbb{U}_N^{-1}
		W_N^1 \mathbb{U}_N,
		\qquad 
		\overline{\mathfrak{D}}_1=
		\mathbb{V}_N^{-1}
		\tilde W_N^1 \mathbb{V}_N,
	\end{equation*}
	where $\mathbb{U}_N$, etc., mean multiplication operators.
\end{proposition}

Because the $q$-Whittaker operators \eqref{eq:qW_difference_op} commute, 
we get many operators commuting
with either $\mathfrak{D}_1$ or $\overline{\mathfrak{D}}_1$.
That is, for $r=1,\ldots,N $ we have:
\begin{equation}
	\label{eq:higher_comm_rel}
	\bigl[
		\mathfrak{D}_1,
		\mathbb{U}_N^{-1}
		W_N^r \mathbb{U}_N
	\bigr]=0,
	\qquad 
	\bigl[
		\overline{\mathfrak{D}}_1,
		\mathbb{V}_N^{-1}
		\tilde W_N^r \mathbb{V}_N
	\bigr]=0.
\end{equation}
For example, 
\begin{equation}
	\label{eq:pseudo_D_2}
	\mathbb{U}_N^{-1}
	W_N^2 \mathbb{U}_N=
	\sum_{i,k\colon i\ne k}
	(1+sx_i)^{N-1}(1+sx_k)^{N-1}
	\prod_{j\colon j\ne i,k}
	\frac{1}{(1-x_i/x_j)(1-x_k/x_j)}\,
	T_{q,x_i}T_{q,x_k}.
\end{equation}
However, one can directly check that the operator 
$\mathbb{U}_N^{-1} W_N^2 \mathbb{U}_N$
\emph{does not commute} with $\overline{\mathfrak{D}}_1$.
This suggests that the operators
$\mathbb{U}_N^{-1} W_N^r \mathbb{U}_N$
or 
$\mathbb{V}_N^{-1} \tilde W_N^r \mathbb{V}_N$, 
$r\ge 2$, \emph{should not be diagonal}
in the spin $q$-Whittaker polynomials
$\mathbb{F}_\lambda$.
The following example shows that this is indeed the case:
\begin{example}
	Take $N=2$, 
	then $(1-s^2)\mathbb{F}_{(1,0)}(x_1,x_2)=s+x_1+x_2+sx_1x_2$.
	Applying \eqref{eq:pseudo_D_2} to this function, we obtain
	$(1+sx_1)(1+sx_2)(s+qx_1+qx_2+sq^2x_1x_2)$,
	which is not proportional to
	$\mathbb{F}_{(1,0)}(x_1,x_2)$ unless $s=0.$
	Note that for $s=0$
	both $\mathbb{U}_N$ and $\mathbb{V}_N$
	are the same (and are equal to the identity), and 
	$\mathfrak{D}_1,\overline{\mathfrak{D}}_1$ 
	are the usual $q$-Whittaker difference operators.
\end{example}

We also observe that by 
\eqref{eq:higher_comm_rel}, 
polynomials of the form 
$\mathbb{U}_N^{-1} W_N^r \mathbb{U}_N\mathbb{F}_\lambda$,
$r=2,\ldots,N $, are eigenfunctions
of the operator $\mathfrak{D}_1$ with eigenvalues
$q^{\lambda_N}$.
Similarly, 
$\mathbb{V}_N^{-1} \tilde W_N^r \mathbb{V}_N \mathbb{F}_\lambda$ are 
eigenfunctions of $\overline{\mathfrak{D}}_1$ 
with eigenvalues $q^{-\lambda_1}$.
However, one
can check that 
$\overline{\mathfrak{D}}_1$
does not act diagonally on, say, the polynomial
$\mathbb{U}_N^{-1} W_N^2 \mathbb{U}_N\mathbb{F}_{(1,0)}$.

\medskip

It remains unclear 
how to construct
higher order $q$-difference 
operators which would be diagonal in the sqW polynomials
(and whether such operators exist at all).

\section{Integrable stochastic dynamics on interlacing arrays}
\label{sec:dynamics_on_arrays}

In this section we 
implement the general scheme of passing from
symmetric functions satisfying Cauchy type summation identities
to probability measures. 
This approach closely follows the ideas of
Schur / Macdonald processes
\cite{okounkov2003correlation}, \cite{BorodinCorwin2011Macdonald}.
We use the framework of \emph{skew Cauchy structures}
which is explained in detail in \cite[Section 2]{BufetovMucciconiPetrov2018}.

\subsection{Skew Cauchy structures and random fields} \label{sub:skew_Cauchy_structures}

We say that two families of functions $\mathfrak{F},\mathfrak{G}$ 
form a \emph{skew Cauchy structure} if they satisfy the following properties: 
\begin{enumerate}
	\item $\mathfrak{F}_{\lambda/\mu},\mathfrak{G}_{\lambda/\mu}$ are
		symmetric rational functions in their respective variables, 
		parametrized by pairs of signatures $\lambda/\mu$ (with appropriate numbers of 
		parts).
		In particular, $\mathfrak{F}_{\lambda/\mu},\mathfrak{G}_{\lambda/\mu}$
		are nonzero only if $\mu \subseteq \lambda$.
    \item Branching rules: for all $\mu,\lambda$ we have
    \begin{equation*}
        \mathfrak{F}_{\nu / \lambda}(u_1,\dots,u_n) = \sum_{\mu} \mathfrak{F}_{\mu / \lambda}(u_1,\dots,u_{n-1})
        \mathfrak{F}_{\nu / \mu}(u_n)
    \end{equation*}
		for any $n$ and any set of variables $u_1,\dots,u_n$, and analogously for $\mathfrak{G}$.
    \item There exists a 
			function $\Pi$ and a set $\mathsf{Adm}\subseteq \mathbb{C}^2$ such that 
    the skew Cauchy identity
    \begin{equation} \label{eq:sCI}
        \Pi(u;v)
	    \sum_{\varkappa} \,\mathfrak{F}_{\mu / \varkappa}(u)
	    \mathfrak{G}_{\lambda / \varkappa} (v)
	    =
        \sum_{\nu} \mathfrak{F}_{\nu / \lambda}(u) \,\mathfrak{G}_{\nu / \mu} (v)
    \end{equation}
		holds numerically for all $(u,v)\in\mathsf{Adm}$.
		Note that $u,v$ stand for single variables, as in 
		\Cref{prop:skew_Cauchy_ID_sqW_sHL,prop:skew_Cauchy_ID_sqW_sqW}.
	\item 
		There exist two sets $\mathsf{P},\Dot{\mathsf{P}}\subseteq
		\mathbb{C}$, with $\mathsf{P} \times \Dot{\mathsf{P}} \subseteq
		\mathsf{Adm}$, such that for any choice of $u \in \mathsf{P}$ and
		$v \in \Dot{\mathsf{P}}$ the functions
		$\mathfrak{F}_{\lambda/\mu}(u),\mathfrak{G}_{\lambda/\mu}(v)$ are
		non negative for all $\lambda,\mu$.
		In this case we say that $u,v$ are 
		\emph{positive specializations}.
		(Nonnegativity of single-variable functions together
		with branching implies nonnegativity of multi-variable versions of the 
		functions.)
\end{enumerate}

Consider now two sequences of signatures
$\vec\lambda=(\lambda^1, \dots, \lambda^n)$ and $\vec\mu=(\mu^1, \dots, \mu^{n-1})$
with
\begin{equation*}
	\lambda^1\supseteq \mu^1 
	\subseteq \lambda^2 \supseteq \mu^2 \subseteq \cdots \mu^{n-1}\subseteq \lambda^n,
\end{equation*}
and sequences of positive specializations $u_1,\dots, u_n$ and $v_1,\dots, v_n$ respectively of $\mathfrak{F}$ and $\mathfrak{G}$. The $\mathfrak{F}/\mathfrak{G}$ \emph{process} is the probability measure
\begin{equation}
\label{eq:F_G_process}
	\mathrm{Prob}(\vec\lambda, \vec\mu) = \frac{1}{Z}\,			
		\mathfrak{F}_{\lambda^1}(u_1) 
    \Biggl(
		\prod_{i=1}^{n-1}
    \mathfrak{G}_{\lambda^i/\mu^i}(v_i) \mathfrak{F}_{\lambda^{i+1}/\mu^i}(u_{i+1})\Biggr)
    \mathfrak{G}_{\lambda^n}(v_n),
\end{equation}
where the normalization constant is $Z=\prod_{i,j}\Pi(u_i; v_j)$.

\medskip

For applications to stochastic dynamics, 
it is of interest to consider \emph{random fields} 
$\{\lambda^{(i,j)}\}$ of signatures indexed by $\mathbb{Z}_{\ge0}^2$,
whose marginal distributions along down-right paths are given by suitable
$\mathfrak{F}/\mathfrak{G}$ processes.
A \emph{down-right path} is
\begin{equation*}
    \varpi=\{\varpi_k=(i_k,j_k)\colon 0 \le k \le L \},
    \hspace{10pt}
    \text{where}
    \hspace{10pt} 
    i_0 = j_L=0
    \hspace{10pt}
    \text{and}
    \hspace{10pt}
    \varpi_{k+1} - \varpi_k \in \{ \mathbf{e}_1, -\mathbf{e}_2 \}.
\end{equation*}
Here $L$ is arbitrary and depends on $\varpi$,
and $\mathbf{e}_1,\mathbf{e}_2$ 
are the standard basis vectors $(1,0),(0,1)$.

\begin{definition}
	\label{def:F_G_field}
	Consider positive specializations $u_1,u_2,\dots$ and $v_1,v_2,\dots$
	respectively of functions $\mathfrak{F}$ and $\mathfrak{G}$. An
	$\mathfrak{F}/\mathfrak{G}$ \emph{field} is a probability measure on
	the set $\{\lambda^{(i,j)}\colon i,j\in\mathbb{Z}_{\ge 0}\}$ that associates the probability
	\begin{equation}
		\label{eq:FG_process_in_field}
			\frac{1}{Z_\varpi}
			\prod_{k\colon \varpi_{k+1}=\varpi_k+\mathbf{e}_1  }
			\mathfrak{F}_{ \lambda^{\varpi_{k+1}} / \lambda^{\varpi_k}
			}(u_{i_{k+1}}) \prod_{ k\colon \varpi_{k+1}=\varpi_k-\mathbf{e}_2 }
			\mathfrak{G}_{ \lambda^{\varpi_k} / \lambda^{\varpi_{k+1}}
			}(v_{j_{k}}).
	\end{equation}
	to the event of finding signatures
	$\lambda^{\varpi_1},\dots,\lambda^{\varpi_{L-1}}$ along an down-right
	path $\varpi$.
	Here the normalization constant is
	$Z_\varpi = \prod_{(i,j)\text{ below }\varpi} \Pi(u_i;v_j)$, and at the boundary we fix
	$\lambda^{(0,j)}=\lambda^{(i,0)}=\varnothing$ with probability one.
\end{definition}

\begin{remark}
	\label{rmk:F_G_field_not_unique}
	While the $\mathfrak{F}/\mathfrak{G}$ process
	is defined uniquely by \eqref{eq:F_G_process}, 
	an $\mathfrak{F}/\mathfrak{G}$ is not determined uniquely by \Cref{def:F_G_field}.
	Below in this section we outline two different constructions of 
	a field in our particular cases. See also the discussion
	in \cite[Section 2.6]{BufetovMucciconiPetrov2018} for more 
	details and additional references.
\end{remark}

To visualize an $\mathfrak{F}/ \mathfrak{G}$ field, decorate edges
$(i-1,j) \to (i,j)$ of the first quadrant with specializations $u_i$,
and edges $(i,j-1)\to(i,j)$ with $v_j$. Then for each down-right
path $\varpi$, the probability of finding the sequence
$\lambda^{\varpi_k}$ is computed by 
climbing down $\varpi$ and picking up skew
functions $\mathfrak{F}(u_{i_k})$ along horizontal edges, and 
$\mathfrak{G}(v_{j_k})$ along vertical edges.
See \Cref{fig:field} for an illustration.

\begin{figure}[htpb]
	\centering
	\includegraphics{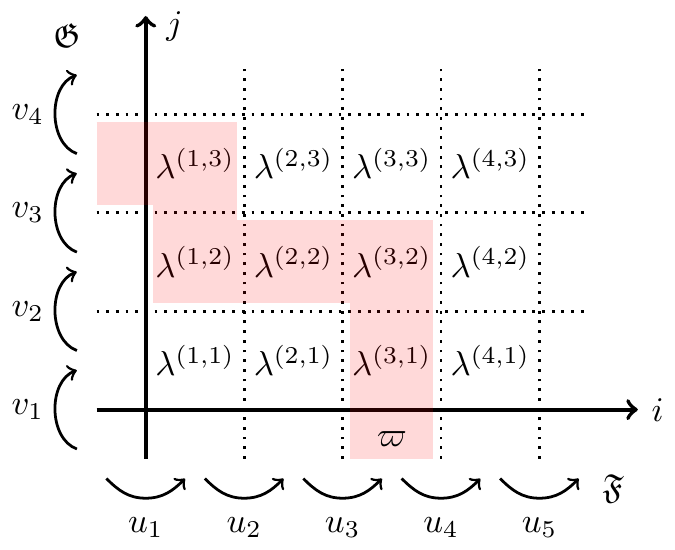}
	\caption{A down-right path (highlighted) in a random field,
	and edge decorations.}
	\label{fig:field}
\end{figure}

In this paper, particularly interesting instances of 
$\mathfrak{F}/\mathfrak{G}$ processes will be 
those arising when considering paths $\varpi$ 
of the form $(0,T)\to (N,T) \to (N,0)$.
Taking the marginal distribution of 
$\lambda^{(N,1)},\ldots, \lambda^{(N,T)}$, we 
arrive at the following definition:
\begin{definition}
	\label{def:F_G_process}
		The \emph{ascending} $\mathfrak{F} / \mathfrak{G}$ \emph{process}
		is the probability measure on the set of signatures
    \begin{equation*}
			\lambda^1 \subseteq \lambda^2 \subseteq \cdots \subseteq \lambda^N,
    \end{equation*}
		assigning to each such sequence the probability weight
    \begin{equation*}
			\frac{1}{\prod_{i=1}^{N}\prod_{j=1}^{T}\Pi(u_i;v_j)}\, \mathfrak{F}_{\lambda^1}(u_1)\,
			\mathfrak{F}_{\lambda^2 / \lambda^1}(u_2) \cdots
			\mathfrak{F}_{\lambda^N / \lambda^{N-1} }(u_N)\,
			\mathfrak{G}_{\lambda^N}(v_1, \cdots , v_T).
    \end{equation*}
\end{definition}

\subsection{\texorpdfstring{Fields based on spin $q$-Whittaker polynomials}{Fields based on spin q-Whittaker polynomials}} \label{sub:sqW_fields}

Here we specialize skew Cauchy structures to 
two cases involving spin $q$-Whittaker and spin Hall--Littlewood functions.

\begin{definition} \label{def:sqW/sHL_field}
    Let $s \in (-1,0)$ and take 
		parameters $x_i \in [-s,-s^{-1}]$, $v_j\in[0,1)$.
    The \emph{sqW/sHL field} is obtained by
		specializing $\mathfrak{F}_{\lambda/\mu}(x_i)=\mathbb{F}_{\lambda/\mu}(x_i)$
		and $\mathfrak{G}_{\lambda/\mu}(v_j)=\mathsf{F}^*_{\lambda'/\mu'}(v_j)$.

		The corresponding skew Cauchy identity is \Cref{prop:skew_Cauchy_ID_sqW_sHL}.
		One readily verifies that the sHL and sqW functions specialized 
		like this are nonnegative, which leads to probability distributions.
		Joint distributions along down-right paths in this field
		are given by \emph{sqW/sHL processes} which are specializations
		of \eqref{eq:F_G_process}.
\end{definition}

\begin{definition} \label{def:sqW/sqW_field}
    Let $s\in(-1,0)$ and take parameters $x_i,y_j\in[-s,-s^{-1}]$. 
		The \emph{sqW/sqW field} is obtained by 
		specializing $\mathfrak{F}_{\lambda/\mu}(x_i)=\mathbb{F}_{\lambda/\mu}(x_i)$ and
		$\mathfrak{G}_{\lambda/\mu}(y_j)=\mathbb{F}^*_{\lambda/\mu}(y_j)$.

		The corresponding skew Cauchy identity is \Cref{prop:skew_Cauchy_ID_sqW_sqW}.
		The range of parameters here also leads to nonnegative functions
		$\mathbb{F},\mathbb{F}^*$, thus producing probability measures.
		Joint distributions along down-right paths in the sqW/sqW field
		are given by \emph{sqW/sqW processes} which are specializations
		of \eqref{eq:F_G_process}.
\end{definition}

\begin{remark}
	Both types of fields were already defined in 
	\cite{BufetovMucciconiPetrov2018},
	though using slightly different 
	versions of the sHL and sqW functions.
\end{remark}

\subsection{Sampling a field via bijectivization}
As mentioned in 
\Cref{rmk:F_G_field_not_unique},
a random field is not determined
uniquely.
Moreover, its properties (like 
marginal stochastic dynamics)
heavily rely on a particular choice of the field's construction.
This choice can be encoded by certain Markov transition operators.
Let us return to the general formalism of 
skew Cauchy structures. 

Suppose that we have 
Markov transition operators
\begin{equation*}
    \Ufwd_{u,v}(\varkappa \to \nu \mid \lambda, \mu) 
    \qquad
    \text{and}
    \qquad
    \Ubwd_{u,v}(\nu \to \varkappa \mid \lambda, \mu),
\end{equation*}
that satisfy the \emph{reversibility condition}
\begin{equation} \label{eq:reversibility}
    \Ufwd_{u,v}(\varkappa \to \nu \mid \lambda, \mu)
    \Pi(u;v) \mathfrak{F}_{\mu / \varkappa}(u) \mathfrak{G}_{\lambda / \varkappa} (v) 
    =
    \Ubwd_{u,v}(\nu \to \varkappa \mid \lambda, \mu)
    \mathfrak{F}_{\nu / \lambda}(u) \,\mathfrak{G}_{\nu / \mu} (v).
\end{equation}
Here $\Ufwd_{u,v} (\varkappa \to \nu \mid \lambda, \mu)$ encodes the
probability of a transition $\varkappa \to \nu$ conditioned on
$\lambda,\mu$, 
whereas $\Ubwd_{u,v} (\nu \to \varkappa \mid \lambda, \mu)$ describes
the probability of the opposite move (specializations
$u,v$ are assumed positive). See Figure
\ref{fig:qudruplet_move}, left, for an illustration. 
Summing \eqref{eq:reversibility}
over both $\nu$ and $\varkappa$ and using the Markov property
of $\Ufwd, \Ubwd$, 
one recovers the skew Cauchy Identity \eqref{eq:sCI}. 
Condition
\eqref{eq:reversibility} determines $\Ubwd$ once $\Ufwd$ if given, and vice versa.

\begin{figure}[htbp]
    \centering
    \begin{minipage}{0.45\textwidth}
        \centering
        \includegraphics[width=0.6\textwidth]{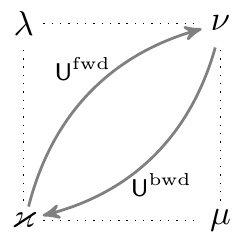}
    \end{minipage}\hfill
    \begin{minipage}{0.45\textwidth}
        \centering
        \includegraphics[width=0.9\textwidth]{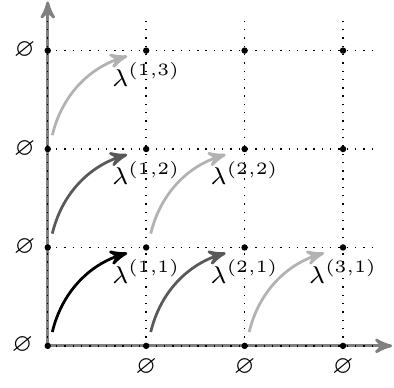}
    \end{minipage}
		\caption{Left: Forward and backward transition operators.
			Right: Construction of a random field using $\Ufwd$,
			where 
			lighter arrows correspond to moves happening later in the update.}
    \label{fig:qudruplet_move}
\end{figure}

If $\Ufwd$ is given, we can construct a random field
$\{\lambda^{(i,j)}\colon i,j \in \mathbb{Z}_{\ge 0} \}$ as 
in Figure \ref{fig:qudruplet_move}, right.
Namely, fix empty boundary conditions. 
Inductively for
$n\ge2$,
assuming we already sampled signatures $\lambda^{(i,j)}$ with
$i+j\le n$, pick signatures $\lambda^{(i',j')}$ for each
$i'+j'=n+1$ at random with probabilities
\begin{equation*}
	\Ufwd_{u_{i'},v_{j'}}
	(\lambda^{(i'-1,j'-1)}\to\lambda^{(i',j')}\mid \lambda^{(i'-1,j')}, \lambda^{(i',j'-1)} ),
\end{equation*}
independently for various pairs $(i',j')$.
We say that the field is \emph{generated} by $\Ufwd$. 
\begin{proposition}
    Assume that $\Ufwd$ is known.
		Then the procedure described right above samples an $\mathfrak{F}/\mathfrak{G}$ field. 
\end{proposition}
\begin{proof}
	One has to show that the distribution of the signatures
	along any down-right path is described by the corresponding 
	$\mathfrak{F}/\mathfrak{G}$ process.
	This is readily verified by induction 
	on adding one box to the area 
	below the down-right path, and using \eqref{eq:reversibility}.
	We omit the details.
\end{proof}

\subsection{Borodin--Ferrari fields}
\label{sub:BF_Fields}

Let us now describe a particular choice 
of the forward transition probabilities
which guarantees the existence of a field for
a skew Cauchy structure.
This construction is based on 
\cite{BorFerr2008DF} and follows an earlier coupling idea of 
\cite{DiaconisFill1990}.
Choose
\begin{equation}\label{eq:Borodin_Ferrari}
    \begin{split}
        &\Ufwd_{u,v}(\varkappa\to \nu \mid \, \lambda,\mu) = \frac{\mathfrak{F}_{ \nu / \lambda }(u) \mathfrak{G}_{\nu/\mu}(v) }{\Pi(u;v) \sum_{\varkappa} \mathfrak{F}_{ \mu/\varkappa }(u) \mathfrak{G}_{ \lambda/\varkappa }(v)  },
        \\
        &
        \Ubwd_{u,v}(\nu \to \varkappa \mid \, \lambda,\mu) = \frac{ \Pi(u;v) \mathfrak{F}_{ \mu/\varkappa }(u) \mathfrak{G}_{ \lambda/\varkappa }(v) }{ \sum_{\nu} \mathfrak{F}_{ \nu / \lambda }(u) \mathfrak{G}_{\nu/\mu}(v)}.
    \end{split}
\end{equation}
In general, although transition probabilities
\eqref{eq:Borodin_Ferrari} are explicit,
in particular examples
their concrete meaning 
may be
far
from transparent. 

A helpful simplification can be made if we
assume that $\mathfrak{G}$ admits expansion 
\begin{equation}
    \mathfrak{G}_{\nu/\mu}(v) = (v-v^*)^{d(\nu/\mu)} (g_{\nu/\mu} + \bigO(v-v^*)),
\end{equation}
for some fixed value $v^*$ independent of $\nu,\mu$,
coefficients $g_{\nu/\mu}$, and a
``nice'' degree
function $d$ 
such that 
$d(\nu/\nu)=0$.
Then one can consider a Poisson-type scaling limit 
of the field \eqref{eq:Borodin_Ferrari} as 
$v_j \to v^*$ for all $j$.
Under this scaling, the discrete vertical axis
becomes continuous, and the field turns into a
Markov dynamics 
$\{
\lambda^{(i,t)}\colon i\in\mathbb{Z}_{\ge 0}, t\in\mathbb{R}_{\ge 0} \}$,
where $t$ is the continuous time variable. 
The dynamics lives on sequences of signatures.

When $\mathfrak{F},\mathfrak{G}$ are Schur functions, such continuous 
processes is the \emph{push--block dynamics} introduced in \cite{BorFerr2008DF}. 

\subsection{Bijectivization of the Yang--Baxter equation} \label{sub:bijectivizations}

In many cases, skew Cauchy Identities
descend directly from the Yang--Baxter equation
(cf. \Cref{sub:sHL_func_def,sub:sqw_dual_def}).
This observation was used
in \cite{BufetovPetrovYB2017}, \cite{BufetovMucciconiPetrov2018} to provide an
explicit construction of random fields for sHL and sqW functions, which we briefly
recall here.
In general, this approach produces 
fields which \emph{differ} from the Borodin--Ferrari ones.
On the other hand, \emph{Yang--Baxter fields} by design possess
Markovian marginals.

For any given identity with positive terms 
\begin{equation} \label{eq:identity_sum}
    \sum_{a\in A} \mathsf{w}(a) = \sum_{b \in B} \mathsf{w}(b),
\end{equation}
we say that two stochastic matrices $\mathsf{p}^{\mathrm{fwd}}(a,b)$ and $\mathsf{p}^{\mathrm{bwd}}(b,a)$ 
(with indices $a\in A$, $b\in B$)
form a (stochastic)
\emph{bijectivization of identity} \eqref{eq:identity_sum}
if they satisfy the reversibility condition
\begin{equation*}
	\mathsf{p}^{\mathrm{fwd}}(a \to b)\,
	\mathsf{w}(a) = \mathsf{p}^{\mathrm{bwd}}(b \to a)\, \mathsf{w}(b)
	\qquad \text{for all $a\in A$, $b\in B$}.
\end{equation*}
A bijectivization always exists since we can take $\mathsf{p}^{\mathrm{fwd}}(a \to b) \propto \mathsf{w}(b)$.
A bijectivization is unique only when $A$ or $B$ has a single element. Another simple case is given when both $A$ and $B$ have only two elements.

\begin{example} \label{ex:bijectivization_2_elements}
    When $A=\{a_1,a_2\}$ and $B=\{b_1,b_2\}$, identity \eqref{eq:identity_sum} becomes
    \begin{equation*}
        \mathsf{w}(a_1) + \mathsf{w}(a_2) = \mathsf{w}(b_1) + \mathsf{w}(b_2).
    \end{equation*}
    In this case all stochastic bijectivizations $\mathbf{p}^{\mathrm{fwd}}$, $\mathbf{p}^{\mathrm{bwd}}$ are expressed as
    \begin{alignat*}{2}
        &
				\mathbf{p}^{\mathrm{fwd}}(a_1 \to b_1) 
				=
				\gamma, 
        \qquad
				&&
        \mathbf{p}^{\mathrm{fwd}}(a_2 \to b_1) 
				=
				\frac{ \mathsf{w}(a_2) - \mathsf{w}(b_2) + (1-\gamma)\mathsf{w}(a_1) }{\mathsf{w}(a_2)},
        \\
        &
				\mathbf{p}^{\mathrm{fwd}}(a_1 \to b_2) 
				=
				1-\gamma,
        \qquad
				&&
        \mathbf{p}^{\mathrm{fwd}}(a_2 \to b_2) 
				=
				1 - \mathbf{p}^{\mathrm{fwd}}(a_2 \to b_1),
    \end{alignat*}
    for a parameter $\gamma \in [0,1]$.
\end{example}

Let now \eqref{eq:identity_sum} be one of the Yang--Baxter equations 
\eqref{eq:YBE_RWw_star},
\eqref{eq:YBE_RWw_star_wall},
\eqref{eq:YBE_RWW_star},
\eqref{eq:YBE_RWW_star_wall} from \Cref{app:YBE}, 
corresponding to Figure \ref{fig:YBE_W_Wstar}. 
Let us rewrite them in a unified notation as
\begin{equation} \label{eq:YBE_general}
    \sum_{K} \mathsf{wl} (K \mid I,J)
    =
    \sum_{K'} 
    \mathsf{wr}(K'\mid I,J),
\end{equation}
where $I=\{i_1,i_2,i_3\}$, $J=\{j_1,j_2,j_3\}$, $K=\{k_1,k_2,k_3\}$ and $K'=\{k_1',k_2',k_3'\}$,
and weight functions $\mathsf{wl},\mathsf{wr}$ denote the 
terms in the left and right-hand sides of each of
\eqref{eq:YBE_RWw_star}--\eqref{eq:YBE_RWW_star_wall}.
Equations 
\eqref{eq:YBE_RWw_star_wall} and \eqref{eq:YBE_RWW_star_wall}
with the right boundary, by agreement, 
correspond to 
$j_1=\varnothing$.

Denote by $\mathsf{p}^{\mathrm{fwd}}_{I,J}(K\to K')$ and
$\mathsf{p}^{\mathrm{bwd}}_{I,J}(K' \to K)$ a stochastic
bijectivization of \eqref{eq:YBE_general}. Then $\mathsf{p}^{\mathrm{fwd}}_{I,J}$
is the probability of moving the cross from left to right 
(in the local configuration in \Cref{fig:YBE_W_Wstar}),
while transforming the occupation numbers $K$ into $K'$.
The probabilities $\mathsf{p}^{\mathrm{bwd}}_{I,J}(K' \to K)$ 
similarly correspond to moving the cross from right to left.
By the
conservation of paths at each vertex, once $I,J$ are fixed,
the configuration $K$ 
is completely determined specifying only one of the numbers $k_1,k_2$, or $k_3$
(and similarly for $K'$).

\medskip

Bijectivizations of the Yang--Baxter equation are building blocks of
operators $\Ufwd, \Ubwd$. Given $\varkappa,\mu \in
\mathrm{Sign}_{N}$, $\lambda,\nu \in \mathrm{Sign}_{N+1}$
we identify path configurations through two rows
of vertices as in \Cref{fig:local_forward_move}
(in the same way as in \Cref{sub:directed_paths}). 
Vertices
crossed by blue paths are assigned non dual weights $W$ 
\eqref{eq:W_left_boundary}--\eqref{eq:W_corner} whereas those in red
have dual weights $w^*$ or $W^*$. We assume that at the leftmost
column an infinite number of paths flows in the vertical direction.
The transition probability $\Ufwd(\varkappa \to \nu \mid \lambda, \mu)$
is the product of probabilities of sequential local moves
$\mathsf{p}^{\mathrm{fwd}}$
obtained dragging the cross vertex from the leftmost column to the right.
The operator
$\Ubwd$ is constructed using the opposite local moves with probabilities
$\mathsf{p}^{\mathrm{bwd}}$, starting from the $N+1$-th
column. See
\Cref{fig:local_forward_move} for an illustration.

\begin{figure}
    \centering
    \includegraphics[scale=0.95]{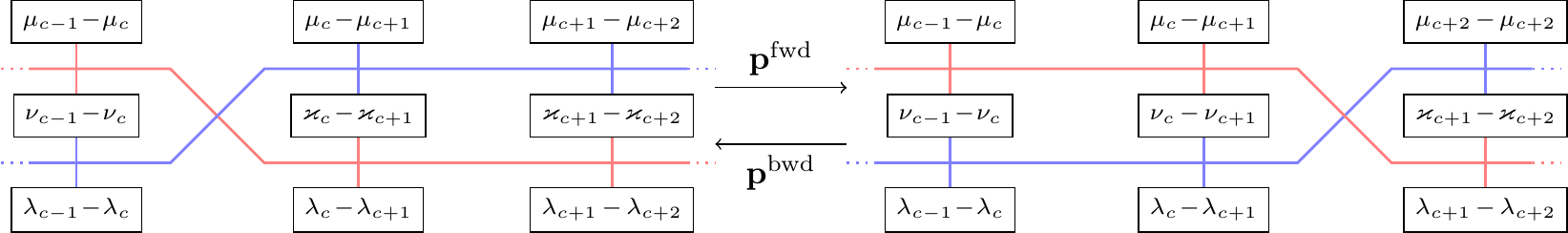}
		\caption{A local random move in a Yang--Baxter field.
			Moving the cross through the column $c$ 
		updates the value of  $\varkappa_c-\varkappa_{c+1}$ to $\nu_c - \nu_{c+1}$.}
    \label{fig:local_forward_move}
\end{figure}

\begin{proposition} \label{prop:bijectivization_samples_sqW_sHL_field}
	Let $\mathsf{p}^{\mathrm{fwd}}$ and $\mathsf{p}^{\mathrm{bwd}}$ be a stochastic bijectivization of Yang--Baxter equations \eqref{eq:YBE_RWw_star}, \eqref{eq:YBE_RWw_star_wall} for the weights $W,w^*$,
	and
  $\Ufwd$ be constructed from sequential local moves $\mathsf{p}^{\mathrm{fwd}}$. 
	Then the random field generated by $\Ufwd$ is a sqW/sHL field.
\end{proposition}
\begin{proof}
    This is analogous to \cite[Section 6.3]{BufetovMucciconiPetrov2018}. See also \cite[Theorem 6.3]{BufetovPetrovYB2017}.
\end{proof}

\begin{proposition}
	Let $\mathsf{p}^{\mathrm{fwd}}$ and $\mathsf{p}^{\mathrm{bwd}}$ be a stochastic bijectivization of Yang--Baxter equations \eqref{eq:YBE_RWW_star}, \eqref{eq:YBE_RWw_star_wall} for the weights $W,W^*$,
	and 
	$\Ufwd$ be constructed from sequential local moves $\mathbf{p}^{\mathrm{fwd}}$. Then the random field generated by $\Ufwd$ is a sqW/sqW field.
\end{proposition}
\begin{proof}
    This is again analogous to \cite[Section 6.3]{BufetovMucciconiPetrov2018}.
\end{proof}

By the very construction, 
we see that for any fixed $c\ge1$, 
the update
$(\varkappa_1,\ldots,\varkappa_c )\to(\nu_1,\ldots,\nu_c )$
is independent
of $\varkappa_i,\mu_i,\lambda_i$ for all $i\ge c+1$.
Therefore, we have:
\begin{proposition} \label{prop:marginal_markov}
	Let $\{\lambda^{(i,j)}\colon i,j\in\mathbb{Z}_{j\in\mathbb\ge 0}\}$ be a Yang--Baxter random field
	as above.
	For any $c\in \mathbb{Z}_{\ge 1}$, the
	marginal process $\{(\lambda_1^{(i,j)}\ge\cdots \ge \lambda_c^{(i,j)})
	:i,j\in\mathbb{Z}_{\ge 0}\}$ is a Markov process.
\end{proposition}
\begin{proof}
	This is \cite[Proposition 6.2]{BufetovPetrovYB2017}.
\end{proof}

In the simplest case $c=1$, transition probabilities of the 
one-dimensional marginal field can be written down explicitly:

\begin{proposition} \label{prop:general_marginal_YBfield}
    Let $\{\lambda^{(i,j)}\colon i,j\in\mathbb{Z}_\ge 0\}$ be a 
		random field generated by $\Ufwd$ constructed from bijectivization
		of the Yang--Baxter equation. Let $\{ \lambda^{(i,j)}_1 \colon i,j
		\in \mathbb{Z}_{\ge 0} \}$ be the first row marginal process. Then
		for all $i,j\ge1$ we have
    \begin{equation} \label{eq:lambda_1_around_a_vertex}
        \mathrm{Prob}\{ \lambda^{(i,j)}_1=n \mid \lambda^{(i,j-1)}_1=m,\lambda^{(i-1,j)}_1=\ell,\lambda^{(i-1,j-1)}_1=k \} = \mathbf{L}_{u_i,v_i}(m-k,\ell-k;n-\ell,n-m),
    \end{equation}
		for all $n,m,k,n \ge 0$, where $\mathbf{L}$ is the stochastic vertex weight
    \begin{equation} \label{eq:general_stochastic_L}
    \mathbf{L}_{u,v} (j_2,j_1;k_1',k_2') = \frac{ \mathsf{wr}_{\{0,0,\infty\},\{j_1,j_2,\infty\}} (\{k_1',k_2',\infty\})}{ \sum_{k_1,k_2} \mathsf{wl}_{\{0,0,\infty\},\{j_1,j_2,\infty\}} (\{k_1,k_2,\infty\})}.
\end{equation}
\end{proposition}
Note that interlacing
		implies that $k\le m\le n$, $k\le \ell\le n$,
		so the arguments of $\mathbf{L}_{u_i,v_i}$ in \eqref{eq:lambda_1_around_a_vertex} are all
		nonnegative.
\begin{proof}[Proof of \Cref{prop:general_marginal_YBfield}]
		This is proven in \cite[Section 6.4]{BufetovMucciconiPetrov2018}
		and we briefly reproduce the argument here. 
		The
		update $\lambda_1^{(i-1,j-1)} \to \lambda_1^{(i,j)}$, once
		$\lambda_1^{(i-1,j)},\lambda_1^{(i,j-1)}$ 
		are fixed, is determined only by a single random move at the
		leftmost column of vertices.
		By construction, 
		the vertical direction at the leftmost column
		has infinitely many paths.
		The corresponding Yang--Baxter equation is
    \begin{equation} \label{eq:YBE_general_leftmost}
    \sum_{k_1,k_2} \mathsf{wl} (\{k_1,k_2,\infty\} \mid \{0,0,\infty\},\{j_1,j_2,\infty\})
    =
    \sum_{k_1',k_2'} 
    \mathsf{wr}(\{k_1',k_2',\infty\}\mid\{0,0,\infty\},\{j_1,j_2,\infty\}).
\end{equation}
		This implies that taking
    \begin{equation*}
			\mathbf{p}^{\mathrm{fwd}}_{\{0,0,\infty\},\{j_1,j_2,\infty \}}
			\left( \{k_1,k_2,\infty\} \to \{k_1',k_2',\infty\} \right) =
			\mathbf{L}_{u,v}(j_2,j_1;k_2',k_1')
    \end{equation*}
		indeed produces a bijectivization.\footnote{This bijectivization
		is in fact unique for our choices of weights
		(this follows similarly to \cite{BufetovMucciconiPetrov2018}). However,
		we do not need this fact.}
		Here $u,v$ denote generic spectral parameters of weights appearing in the Yang--Baxter equation.
		Recall that occupation numbers are related to signatures as
		\begin{equation*}
			\begin{array}{lcl}
				j_1=\lambda_1^{(i-1,j)}-\lambda_1^{(i-1,j-1)},&\qquad 
				&j_2=\lambda_1^{(i,j-1)}-\lambda_1^{(i-1,j-1)},
				\\
				\rule{0pt}{15pt}
        k_1'=\lambda_1^{(i,j)}-\lambda_1^{(i,j-1)},
        &\qquad 
        &k_2'=\lambda_1^{(i,j)}-\lambda_1^{(i-1,j)}.
				
			\end{array}
    \end{equation*}
		This completes the proof.
\end{proof}

The fact that the sqW functions are parametrized by signatures with
specified number of rows also allows to access the
random dynamics of \emph{last rows} of a field by writing down
explicit bijectivizations. 
In particular, the evolution of 
$\{ \lambda_i^{(i,j)} \colon i,j \ge 0 \}$ is related to the Yang--Baxter
equations \eqref{eq:YBE_RWw_star_wall}, \eqref{eq:YBE_RWW_star_wall}
corresponding to configurations depicted in Figure
\ref{fig:YBE_W_Wstar} (b). 

\begin{remark}
		The construction of a random field using stochastic
		bijectivizations \emph{does not guarantee} that the evolution of
		last rows is autonomous. 
		This contrasts with
		the fact that the first rows form autonomous Markov marginal
		processes by the very construction of Yang--Baxter fields
		(\Cref{prop:marginal_markov}).
		In
		\Cref{thm:sqW_sHL_last_row,thm:sqW_sqW_last_row}
		below we show that the marginals $\{
		\lambda_i^{(i,j)} \colon i,j \ge 0 \}$ 
		of sqW/sHL and sqW/sqW fields, respectively, 
		are
		in fact autonomous for a 
		particular
		bijectivization we construct.
\end{remark}

\section{Schur case: Robinson--Schensted--Knuth from Yang--Baxter}
\label{sec:rsk_from_YB}

In this section, as a simpler illustration, we consider the 
degeneration of the vertex weights and the Yang--Baxter
equation obtained by setting $q=s=0$,
and show how this produces (via bijectivization) 
the classical
Robinson--Schensted--Knuth (RSK) row insertion 
algorithm \cite{Knuth1970}, \cite{fulton1997young}, \cite{Stanley1999}.
We would obtain a ``local''
description of the RSK insertion 
in terms of ``toggle'' operations.
We refer to, e.g., \cite{Pak2001hook},
\cite{kirillov1995groups}, \cite{fomin1995schensted} and also
to the recent notes \cite{hopkins2014rsk} for this description.

We consider the $q=s=0$ degeneration
of the Yang--Baxter equations 
(\Cref{prop:YBE_sqW_sqW,prop:YBE_sqW_sqW_corner})
proving the sqW/sqW skew Cauchy identity
(\Cref{prop:skew_Cauchy_ID_sqW_sqW}).
Note that for $q=s=0$, the spin $q$-Whittaker polynomials become the 
Schur polynomials.
The Yang--Baxter equations we need are illustrated
in \Cref{fig:YBE_W_Wstar}.
In fact, in the Schur degeneration the 
corner Yang--Baxter equation 
\Cref{prop:YBE_sqW_sqW_corner}
illustrated in 
\Cref{fig:YBE_W_Wstar}\,{\rm{(b)}}
is the same as the usual one, 
and so we only need the equation from 
\Cref{prop:YBE_sqW_sqW}.

The weights entering the Yang--Baxter equation degenerate as follows:
\begin{equation}
	\begin{split}
		W^{\bulk}_{x,0}(i_1,j_1;i_2, j_2) 
		&=
		\mathbf{1}_{i_1 + j_1 = i_2 + j_2} \, \mathbf{1}_{i_1 \geq j_2}\,  x^{j_2},
		\\
		W^{*,\bulk}_{y,0}(i_1,j_1;i_2,j_2)
		&=
		\mathbf{1}_{i_1 + j_2 = i_2 + j_1} \, \mathbf{1}_{i_2 \geq j_2}\,  y^{j_2},
		\\
		\mathbb{R}_{x,y,s}(i_1,j_1;i_2,j_2)
		&=
		\mathbf{1}_{i_2 + j_1 = i_1 + j_2 }\,
		(xy)^{\min(j_1,j_2)}.
	\end{split}
	\label{eq:from_YBE_to_RSK}
\end{equation}
Equation \eqref{eq:YBE_RWW_star} thus reads for all
fixed $i_1,i_2,i_3,j_1,j_2,j_3\in \mathbb{Z}_{\ge0}$:
\begin{multline}
	\label{eq:Schur_YBE}
	\sum_{k_1,k_2,k_3\ge0}
	\mathbf{1}_{\textnormal{arrow preservation}}
	\,
	\mathbf{1}_{k_3\ge \max(j_1,j_2)}
	\,
	x^{j_2}y^{j_1}(xy)^{\min(k_1,k_2)}
	\\=
	\sum_{k_1',k_2',k_3'\ge0}
	\mathbf{1}_{\textnormal{arrow preservation}}
	\,
	\mathbf{1}_{i_3\ge k_2'}
	\mathbf{1}_{j_3\ge k_1'}
	(xy)^{\min(j_1,j_2)}
	x^{k_2'}y^{k_1'},
\end{multline}
where by ``arrow preservation''
we mean the intersection of 
all the conditions of the form $a_1+b_1=a_2+b_2$ in 
\eqref{eq:from_YBE_to_RSK}
which the indices $k_1,k_2,k_3$ and $k_1',k_2',k_3'$ must satisfy.
In particular, by arrow preservation we have
\begin{equation*}
	k_2=i_2+k_1-i_1,\qquad k_3=i_3+j_1-k_1,
	\qquad 
	k_2'=j_2+k_1'-j_1,
	\qquad 
	k_3'=i_1+j_3-k_1',
\end{equation*}
and thus the summands in the left- and right-hand sides of
\eqref{eq:Schur_YBE} 
are indexed only by $k_1$ or $k_1'$, respectively.
Equation \eqref{eq:Schur_YBE} admits a 
bijective proof (or, in terms of \Cref{sec:dynamics_on_arrays}, 
a bijectivization which is deterministic):
\begin{lemma}
	\label{lemma:Schur_YBE}
	Setting
	\begin{equation}
		\label{eq:Schur_YBE_bijection}
		k_1'=j_1-\min(j_1,j_2)+\min(k_1,i_2+k_1-i_1)
	\end{equation}
	produces a 
	bijection
	between the terms in both sides of 
	\eqref{eq:Schur_YBE}.
\end{lemma}
\begin{proof}
	By \eqref{eq:Schur_YBE_bijection}
	we see that 
	the powers of $y$ in the corresponding terms match.
	The powers of $x$ match, too:
	\begin{equation*}
		j_2+\min(k_1,i_2+k_1-i_1)
		=
		\min(j_1,j_2)+
		j_2+k_1'-j_1,
	\end{equation*}
	where $k_1'$ is given by \eqref{eq:Schur_YBE_bijection}.
	It remains to check that if $k_1$ is such that the 
	product of indicators in the left-hand
	side of \eqref{eq:Schur_YBE} is nonzero, then the same
	holds for $k_1'$ in the right-hand side. 
	This check is straightforward.
\end{proof}

Let us now interpret a sequence of bijective Yang--Baxter
transformations as a row RSK insertion.
Fix signatures $\lambda,\mu$ and use \Cref{lemma:Schur_YBE}
to construct a bijection between the sets
\begin{equation}
	\label{eq:Schur_YBE_RSK_sets_for_bijection}
	\{\varkappa\colon \varkappa\prec\lambda,\ \varkappa\prec \mu \}\times \mathbb{Z}_{\ge0}
	\longleftrightarrow
	\left\{ \nu\colon \nu \succ\lambda,\ \nu\succ\mu \right\}.
\end{equation}
This bijection would correspond to a local move in the 
``Fomin growth diagram'' \cite{fomin1995schensted}
interpretation of the RSK.

Interpret $\mu,\varkappa,\lambda$ as a path configuration
as in 
\Cref{fig:local_forward_move}.
The numbers of paths through vertical edges 
are then equal to 
$\mu_c-\mu_{c+1}$, 
$\varkappa_c-\varkappa_{c+1}$,
and $\lambda_c-\lambda_{c+1}$.
One can also check that the horizontal edges carry 
$\mu_c-\varkappa_c$ and $\lambda_c-\varkappa_c$ paths.

Attach a cross at the leftmost boundary of the path configuration. This 
cross is not uniquely determined since at the leftmost
boundary (with $i_3=j_3=\infty$, $i_1=i_2=0$) the Yang--Baxter
equation
\eqref{eq:Schur_YBE}
takes the form
\begin{equation*}
	x^{j_2}y^{j_1}\sum_{k=0}^{\infty}
	(xy)^{k}
	=
	(xy)^{\min (j_1,j_2)}x^{j_2-j_1}
	\sum_{k_1'\ge0}
	\mathbf{1}_{\textnormal{arrow preservation}}\,
	x^{k_1'}y^{k_1'}.
\end{equation*}
Arrow preservation in the right-hand side here 
means that there exists an arrow configuration 
with the given $k_1'$.

Selecting an arbitrary $k$ in the left-hand side 
is equivalent to selecting an element
of $\mathbb{Z}_{\ge0}$
in the left-hand side of \eqref{eq:Schur_YBE_RSK_sets_for_bijection}.
Then we set based on the Yang--Baxter equation (see \Cref{lemma:Schur_YBE}):
\begin{equation*}
	k_1'=j_1-\min(j_1,j_2)+k.
\end{equation*}
In terms of signatures, $k_1'$
corresponds to the new number of paths
on the horizontal edge, and so the above equation means that
\begin{equation*}
	\nu_1-\mu_1=\lambda_1-\varkappa_1-\min(\lambda_1-\varkappa_1,\mu_1-\varkappa_1)+k,
\end{equation*}
that is,
\begin{equation}
	\label{eq:RSK_toggle_1}
	\nu_1=k+\max(\lambda_1,\mu_1).
\end{equation}

After dealing with the leftmost boundary, we move the cross
one by one to the right, updating each 
$\varkappa_c-\varkappa_{c+1}$ to $\nu_c-\nu_{c+1}$, where $c\ge2$.
At each step the signatures correspond to the path numbers as
\begin{align*}
	j_1&=\lambda_{c+1}-\varkappa_{c+1},\qquad 
	j_2=\mu_{c+1}-\varkappa_{c+1},\qquad 
	k_1=\lambda_c-\varkappa_c
	,
	\qquad 
	k_2=i_2+k_1-i_1=\mu_c-\varkappa_c
	,
	\\
	k_1'&=\nu_{c+1}-\mu_{c+1}
	,\qquad 
	k_2'=j_2+k_1'-j_1=\nu_{c+1}-\lambda_{c+1}.
\end{align*}
The local bijection of \Cref{lemma:Schur_YBE}
then leads to 
\begin{align*}
	k_1'
	&=
	j_1-\min(j_1,j_2)+\min(k_1,k_2)
	\\&=
	\lambda_{c+1}-\varkappa_{c+1}-\min(\lambda_{c+1}-\varkappa_{c+1},\mu_{c+1}-\varkappa_{c+1})+
	\min(\lambda_c-\varkappa_c,\mu_c-\varkappa_c)
	\\&=\lambda_{c+1}-\min(\lambda_{c+1},\mu_{c+1})-\varkappa_c+\min(\lambda_c,\mu_c),
\end{align*}
which leads to 
\begin{equation}
	\label{eq:RSK_toggle_2}
	\nu_{c+1}=\max\left( \lambda_{c+1},\mu_{c+1} \right)+
	\min(\lambda_c,\mu_c)-\varkappa_c.
\end{equation}

Formulas \eqref{eq:RSK_toggle_1}--\eqref{eq:RSK_toggle_2}
for $\nu$ in terms of $\varkappa$ 
provide the local RSK bijection between the two sets
\eqref{eq:Schur_YBE_RSK_sets_for_bijection}.
Moreover, these formulas 
have the ``toggle'' form, e.g., see \cite{hopkins2014rsk}.

Therefore, we see that in the Schur degeneration 
the Yang--Baxter equation of \Cref{prop:YBE_sqW_sqW}
produces a bijection, and this bijection coincides
with the ``toggle'' bijection
in the local description of the 
classical Robinson--Schensted--Knuth row insertion.

\section{\texorpdfstring{Marginals of spin $q$-Whittaker fields}{Marginals of spin q-Whittaker fields}} \label{sec:marginals_sqW_processes}

In this section we study two random fields of signatures defined in
\Cref{sub:sqW_fields} based on sqW functions.
We identify their Markov marginals 
corresponding to the first and last coordinates $\lambda^{(i,j)}_1$ and
$\lambda^{(i,j)}_i$. These are matched with stochastic vertex models or
particle dynamics introduced in 
\cite{Povolotsky2013}, \cite{CorwinPetrov2015}, \cite{CMP_qHahn_Push}. These results
extend the characterization of marginals of the
$q$-Whittaker processes given in \cite{MatveevPetrov2014} by 
adding the spin parameter $s$ into the picture.
The matchings are summarized in the table in \Cref{fig:matchings_table}.
\begin{figure}[htbp]
	\centering
	\begin{tabular}{c|l|l}
		& first row $\lambda^{(i,j)}_{1}$ & last row $\lambda^{(i,j)}_i\phantom{\underbrace{asd}}$
		\\
		\hline
		sqW/sHL field
		&
		\parbox{.35\textwidth}
		{[\ref{sub:sqw_shl_first}]
		Stochastic higher spin six vertex model
		\cite{CorwinPetrov2015}, 
		\cite{BorodinPetrov2016inhom}
		}
		&
		\parbox{.4\textwidth}
		{ {\ }\\[0pt]
		[\ref{sub:sqw_shl_last}]
		Stochastic higher spin six vertex model
		\cite{CorwinPetrov2015}, 
		\cite{BorodinPetrov2016inhom}
		}
		\\
		\hline
		sqW/sqW field
		&
		\parbox{.35\textwidth}{
			[\ref{sub:sqw_sqw_first}]
			$_4\phi_3$ vertex model
			and $q$-Hahn PushTASEP \cite{CMP_qHahn_Push},
			\cite{BufetovMucciconiPetrov2018}
			}
		&
		\parbox{.4\textwidth}{
			{\ }\\[0pt]
			[\ref{sub:sqw_sqw_last}, \ref{sub:sqw_sqw_last_continuous}]
			$q$-Hahn TASEP / Boson particle systems
		\cite{Povolotsky2013}, \cite{Corwin2014qmunu}
	}
	\end{tabular}
	\caption{A summary of matchings of \Cref{sec:marginals_sqW_processes},
	with numbers of relevant subsections.}
	\label{fig:matchings_table}
\end{figure}

\subsection{Stochastic vertex models}

We work with two typologies of stochastic
vertex models: 
\emph{up-right} or \emph{up-left}.
These are probability measures on 
directed path ensembles (of the corresponding direction)
in
the integer quadrant,
constructed from families of \emph{stochastic vertex weights}
$L_{(i,j)}$. By ``stochastic'' we mean that the weights
must satisfy the sum to one
condition
\begin{equation} \label{eq:sum_to_one}
	\sum_{\alpha_2,\beta_2\ge 0} L_{(i,j)}(\alpha_1,\beta_1;\alpha_2,\beta_2) = 1
\end{equation}
for all $\alpha_1,\beta_1$,
where $\alpha_1,\alpha_2,\beta_1,\beta_2\in \mathbb{Z}_{\ge0}$
are the occupation numbers of edges at a vertex $(i,j)$.

\medskip

For the first type of stochastic vertex models,
equip the lattice with 
\emph{up-right} vertex weights $L_{(i,j)}^{\mathrm{ur}}$ subject to the 
arrow preservation condition
\begin{equation*}
    L_{(i,j)}^{\mathrm{ur}}(\alpha_1,\beta_1;\alpha_2,\beta_2)=0 
		\qquad \text{if} \qquad\alpha_1+\beta_1 \neq \alpha_2+\beta_2.
\end{equation*}

\begin{definition}[Up-right stochastic vertex model]
	\label{def:ur_sVM}
		The \emph{up-right stochastic vertex model} with weights
		$L_{(i,j)}^{\mathrm{ur}}$ and boundary conditions
		$B^{\mathrm{h}}=\{b_1^{\mathrm{h}},b_2^{\mathrm{h}},\dots \}$,
		$B^{\mathrm{v}}=\{b_1^{\mathrm{v}},b_2^{\mathrm{v}},\dots \}$,
		with $b_i^{\mathrm{h}} ,b_j^{\mathrm{v}} \ge 0$, is the unique
		probability measure on the set of up-right directed paths on
		$\mathbb{Z}_{\ge 1}\times \mathbb{Z}_{\ge 0}$, such that:
    \begin{itemize}
        \item each vertex $(1,j)$ emanates $b_j^{\mathrm{v}}$ paths initially directed to the right;
        \item each vertex $(i,0)$ emanates $b_i^{\mathrm{h}}$ paths initially directed upwards;
				\item the probability of observing a configuration
					$(\alpha_1,\beta_1;\alpha_2,\beta_2)$ at vertex $(i,j)$,
					conditioned on the configuration at all vertices $(i',j')$
					with $i'+j'<i+j$, is given by
					$L^{\mathrm{ur}}_{(i,j)}(\alpha_1,\beta_1;\alpha_2,\beta_2)$.
					Moreover, this event is independent of 
					choosing arrow configurations
					at other vertices 
					$\dots,(i-1,j+1),(i+1,j-1),\dots$
					on the same diagonal.
    \end{itemize}
\end{definition}

Up-right directed lattice path configurations can be encoded by
the \emph{height function}:
\begin{equation} \label{eq:ur_height}
    \mathcal{H}^{\mathrm{ur}}(i,j) = \# \{ \text{occupations at horizontal edges} \} - \# \{ \text{occupations at vertical edges} \},
\end{equation}
where occupations are counted along the path
$(\frac{1}{2},\frac{1}{2}) 
\to
(i+\frac{1}{2},\frac{1}{2}) 
\to
(i+\frac{1}{2} , j+\frac{1}{2})$ (equivalently, along any up-right directed path from
$(\frac{1}{2},\frac{1}{2})$ to $(i+\frac{1}{2},j+\frac{1}{2})$).
See \Cref{fig:stoch_vertex_models}, right, for an illustration of the vertex model 
and the corresponding height function.

\begin{figure}[htpp]
	\centering
	\includegraphics[height=.3\textwidth]{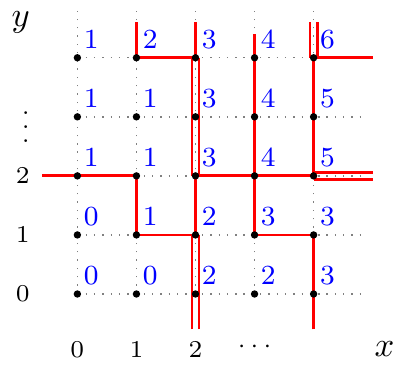}
	\qquad \qquad 
	\includegraphics[height=.3\textwidth]{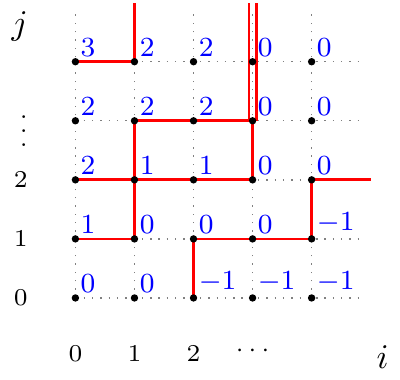}
	\caption{Realizations of the up-left and the up-right stochastic
	vertex models (left and right panels, respectively).}
	\label{fig:stoch_vertex_models}
\end{figure}

\begin{remark}[Up-right model and TASEPs] \label{rem:connection_TASEP}
    Path configurations can be interpreted as trajectories of particles performing totally asymmetric random walks, with time running in the upward direction. In particular, one can define a process 
    \begin{equation*}
			\{\mathsf{X}(t) = (\mathsf{x}_1(t) > \mathsf{x}_2(t) > \cdots) \}_{t\in \mathbb{Z}_{\ge0}}
    \end{equation*}
		by setting $\mathsf{x}_n(t) := \mathcal{H}^{\mathrm{ur}}(n,t) -n$.
		Then $\mathsf{X}$ is a discrete time totally asymmetric simple exclusion
		process, in which the 
		random jump 
		$\mathsf{x}_n(t-1)\to\mathsf{x}_n(t)$
		of the 
		$n$-th particle at time $t$ is governed by 
		\begin{equation*}
			L^{\mathrm{ur}}_{(n,t)}
			\bigl(
				\mathsf{x}_{n-1}(t-1)-\mathsf{x}_n(t-1)-1,
				\mathsf{x}_{n-1}(t)-\mathsf{x}_{n-1}(t-1);
				\mathsf{x}_{n-1}(t)-\mathsf{x}_n(t)-1,
				\mathsf{x}_n(t)-\mathsf{x}_n(t-1)
			\bigr).
		\end{equation*}
\end{remark}

Let us now turn to up-left path ensembles. 
The
\emph{up-left} weights $L_{(i,j)}^{\mathrm{ul}}$ satisfy the following arrow preservation property:
\begin{equation*}
	L_{(i,j)}^{\mathrm{ul}}(\alpha_1,\beta_1;\alpha_2,\beta_2)=0 \qquad \text{if} \qquad\alpha_1+\beta_2 \neq \beta_1+\alpha_2.
\end{equation*}

\begin{definition}[Up-left stochastic vertex model]
	\label{def:ul_sVM}
    The \emph{up-left stochastic vertex model} with weights $L_{(i,j)}^{\mathrm{ul}}$ and boundary conditions $B^{\mathrm{h}}=\{b_1^{\mathrm{h}},b_2^{\mathrm{h}},\dots \}$, $B^{\mathrm{v}}=\{b_1^{\mathrm{v}},b_2^{\mathrm{v}},\dots \}$, with $b_i^{\mathrm{h}},b_j^{\mathrm{v}}\ge 0$, is the unique probability measure on the set of up-left directed path on $\mathbb{Z}_{\ge 1}\times\mathbb{Z}_{\ge 0}$, such that:
    \begin{itemize}
        \item each vertex $(1,j)$ has $b_j^{\mathrm{v}}$ paths entering from its left;
        \item each vertex $(i,0)$ emanates $b_i^{\mathrm{h}}$ paths initially directed upwards;
        \item the probability of observing a configuration $(\alpha_1,\beta_1;\alpha_2,\beta_2)$ at a
					vertex $(i,j)$, conditioned on the path configuration 
					at vertices $(i',j')$ with $i'+j'<i+j$, 
					is given by $L^{\mathrm{ul}}_{(i,j)}(\alpha_1,\beta_1;\alpha_2,\beta_2)$.
					Moreover, this event is independent of 
					choosing arrow configurations
					at other vertices 
					$\dots,(i-1,j+1),(i+1,j-1),\dots$
					on the same diagonal.
    \end{itemize}
\end{definition}

Up-left directed lattice path configurations can be encoded by the 
\emph{height function}:
\begin{equation} \label{eq:ul_height}
    \mathcal{H}^{\mathrm{ul}}(i,j) = \# \{ \text{occupations at horizontal edges} \} + \# \{ \text{occupations at vertical edges} \},
\end{equation}
where occupations are counted along the 
path $(\frac{1}{2},\frac{1}{2}) \to (i + \frac{1}{2},\frac{1}{2}) \to (i+\frac{1}{2} , j+\frac{1}{2})$
(equivalently, along any up-right directed path from
$(\frac{1}{2},\frac{1}{2})$ to $(i+\frac{1}{2},j+\frac{1}{2})$).
Notice the difference in sign with the definition of $\mathcal{H}^{\mathrm{ur}}$ \eqref{eq:ur_height}.
See \Cref{fig:stoch_vertex_models}, left, for an illustration of the up-left vertex model 
and the corresponding height function.

\begin{remark}[Up-left model and PushTASEPs] \label{rem:connection_pushTASEP}
	Define a process 
	\begin{equation*}
		\{\mathsf{Y}(t) = (\mathsf{y}_1(t) > \mathsf{y}_2(t) > \cdots) \}_{t\in\mathbb{Z}_{\ge0}}
	\end{equation*}
	by 
	setting $\mathsf{y}_n(t) = - \mathcal{H}^{\mathrm{ul}}(n,t) - n$.
	Then $\mathsf{Y}$ is a discrete time totally asymmetric simple exclusion
	process under which particles jump to the 
	left, and a 
	\emph{pushing mechanism} is present.
	The random jump 
	$\mathsf{y}_n(t-1)\to\mathsf{y}_n(t)$
	of the $n$-th particle at time $t$ 
	is governed by 
	\begin{equation*}
		L^{\mathrm{ul}}_{(n,t)}
		\bigl(
			\mathsf{y}_{n-1}(t-1)-\mathsf{y}_n(t-1)-1
			,
			\mathsf{y}_{n-1}(t-1)-\mathsf{y}_{n-1}(t)
			;
			\mathsf{y}_{n-1}(t)-\mathsf{y}_n(t)-1
			,
			\mathsf{y}_{n}(t-1)-\mathsf{y}_{n}(t)
		\bigr).
	\end{equation*}
\end{remark}

In the rest of this section we 
establish the matching results
outlined in \Cref{fig:matchings_table}.

\subsection{Last row in sqW/sHL field}
\label{sub:sqw_shl_last}

We start by defining the stochastic higher spin six vertex
model:
\begin{definition}[\cite{CorwinPetrov2015}, \cite{BorodinPetrov2016inhom}] \label{def:ur_HS6VS}
	Specialize the up-right stochastic vertex model
	of 
	\Cref{def:ur_sVM} by taking 
	$L_{(i,j)}^{\mathrm{ur}} =
	\mathsf{L}_{x_i,v_j}^{\mathrm{ur}}$,
	where the latter are given in 
	\Cref{fig:up_right_HS6VM}.
	We refer to this model as the 
	\emph{up-right stochastic higher spin six vertex model}.
	We consider the \emph{step-stationary}
	boundary conditions:
	\begin{equation}
		\label{eq:ur_Bernoulli_bc}
		b_j^{\mathrm{v}} \sim \mathrm{Ber}\left(\frac{x_1 v_j}{1+x_1 v_j}\right) 
		\qquad 
		\text{and} \qquad b_i^{\mathrm{h}}=0,
\end{equation}
where $\mathrm{Ber}(\cdot)$ are independent Bernoulli random variables
with the probability of success given in the 
parentheses.\footnote{A slightly broader class of boundary conditions than the
step-stationary ones, where also $b_i^{\mathrm{h}}$ are allowed to be
positive numbers, can be considered using the fusion argument introduced
in \cite{Amol2016Stationary}; see also
\cite{imamura2019stationary}, \cite{BufetovMucciconiPetrov2018}.}
\end{definition}

\begin{figure}
    \centering
    \includegraphics{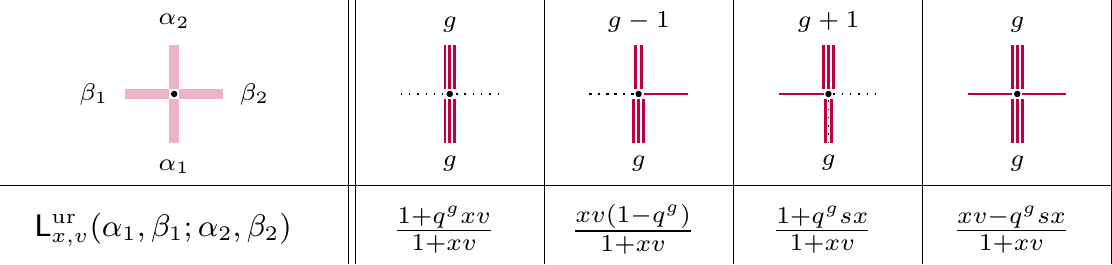}
    \caption{The stochastic vertex weights $\mathsf{L}^{\mathrm{ur}}_{x,v}$ for the up-right stochastic higher spin six vertex model.}
    \label{fig:up_right_HS6VM}
\end{figure}

\begin{remark} \label{rem:matching_to_BP}
	The model in \Cref{def:ur_HS6VS} is equivalent to that of \cite{BorodinPetrov2016inhom}
	(the latter with step boundary conditions
	$b_j^{\mathrm{v}}=1$, $b_i^{\mathrm{h}}=0$), 
	under specializations $\xi_1 \mathsf{s}_1 \to
	x_1$, $\mathsf{s}^2_1\to 0$, $\mathsf{s}_\alpha \xi_\alpha \to
	x_\alpha$, $\mathsf{s}^2_\alpha \to -s x_\alpha$ and $u_\beta \to -
	v_\beta$.
\end{remark}

\begin{theorem}[sqW/sHL last row]
	\label{thm:sqW_sHL_last_row}
	The last row marginal $\{\lambda^{(i,j)}_i \}_{i \ge 1,j \ge 0}$ 
	of the sqW/sHL field
	has the same distribution as 
	the height function 
	$\{\mathcal{H}_{\mathrm{HS}}^{\mathrm{ur}} (i,j)\}_{i \ge 1,j \ge 0}$
	of the
	up-right higher spin six vertex model with step-stationary boundary conditions.
\end{theorem}
\begin{proof}
    We use \Cref{prop:bijectivization_samples_sqW_sHL_field}. 
		During the update 
    \begin{equation*} 
        \lambda^{(n-1,t-1)} \to \lambda^{(n,t)}, \qquad \text{for fixed} \qquad \lambda^{(n-1,t)}, \lambda^{(n,t-1)},
    \end{equation*}
		weighted by the stochastic matrix $\Ufwd$, the law of the rightmost
		local move is given by a stochastic bijectivization of the
		Yang--Baxter equation \eqref{eq:YBE_RWw_star_wall}.
		A straightforward computation
		shows that one such bijectivization is given by the choice
    \begin{equation} \label{eq:choice_bijectivization}
        \mathbf{p}^{\mathrm{fwd}}_{\{i_1,i_2,i_3\},\{\varnothing,j_2,j_3\}}(\{k_1,k_2,k_3\}\to\{k_1',k_2',k_3'\}) = \mathsf{L}_{x,v}^{\mathrm{ur}}(j_3-i_2+i_1-k_1,k_1;j_3-i_2+i_1-k_1',k_1').
    \end{equation}
    This can be readily verified using the parametrization 
		from
		\Cref{ex:bijectivization_2_elements}.
		For simpler notation, let us denote 
		\begin{equation*}
			\varkappa=\lambda^{(n-1,t-1)},\qquad 
			\lambda=\lambda^{(n-1,t)},\qquad 
			\mu=\lambda^{(n,t-1)},\qquad 
			\nu=\lambda^{(n,t)}.
		\end{equation*}
		In terms of elements of these signatures,
		the integers $i_1,i_2,j_3,k_1,k_1'$ are
		interpreted as
		\begin{align*}
			i_1&=\nu_{n-1}-\mu_{n-1},
			\qquad 
			i_2 = \nu_{n-1}-\lambda_{n-1}, 
			\qquad 
			j_3= \mu_{n-1}-\mu_{n},
			\\
			k_1 &= \lambda_{n-1}-\varkappa_{n-1}, 
			\qquad
			k_1' = \nu_n-\mu_n. 
		\end{align*}
		Remarkably, transition weight
		\eqref{eq:choice_bijectivization} only depends on 
		$j_3-i_2+i_1=\lambda_{n-1}-\mu_n$ and on $k_1,k_1'$, 
		but not on other edge occupation
		numbers.
		Observe that these quantities
		involve only the last components of the signatures 
		$\mu,\varkappa,\lambda,\nu$.
		Therefore, the law of the last component $\nu_n=\lambda_n^{(n,t)}$ is fully
		determined by the last components
		$\lambda_{n-1}=\lambda_{n-1}^{(n-1,t)}$,
		$\mu_n=\lambda_n^{(n,t-1)}$,
		and 
		$\varkappa_{n-1}=\lambda_{n-1}^{(n-1,t-1)}$.
		This implies that the last row marginal $\{\lambda^{(i,j)}_i \}_{i
		\ge 1,j \ge 0}$ is an autonomous Markov process.
		Moreover, 
		this autonomous process
		has the same multitime joint distribution
		as the height function of the up-right higher spin six
		vertex model because
		$\mathsf{L}^{\mathrm{ur}}$ appears in
		\eqref{eq:choice_bijectivization}.
		This completes the proof.
\end{proof}

In
\cite{CorwinPetrov2015}, \cite{BorodinPetrov2016inhom}, 
joint $q$-moments
of the 
up-right
stochastic higher spin six vertex model 
were expressed in terms of nested contour integrals. 
These moments completely determine the joint distribution of 
the model's height function 
$\mathcal{H}_{\mathrm{HS}}^{\mathrm{ur}}(\cdot, j)$
along any given horizontal line (because $q\in (0,1)$ and the random
variables in question are nonnegative).
Let us reproduce the $q$-moment formula:

\begin{proposition}[\cite{BorodinPetrov2016inhom}] \label{prop:BP_moments_HS}
	Consider the up-right stochastic higher spin six vertex model with
	step-stationary boundary conditions and assume $v_\alpha \neq q
	v_\beta$. 
	For any $i_1\ge \ldots\ge i_\ell\ge1$ we have
\begin{equation} \label{eq:HS_ur_moments}
    \begin{split}
        \mathbb{E} \prod_{k = 1}^{\ell} q^{\mathcal{H}_{\mathrm{HS}}^{\mathrm{ur}}(i_k,j)}
        =
        q^{\binom{\ell}{2}}
        &
        \oint_{\gamma[-\overline{\mathbf{v}} | 1]} \frac{d z_1}{2 \pi \mathrm{i}} \cdots \oint_{\gamma[-\overline{\mathbf{v}} \mid \ell]} \frac{d z_\ell}{2 \pi \mathrm{i}}
        \prod_{1\le A < B \le \ell} \frac{z_A - z_B}{z_A - q z_B}
        \\
        &
        \hspace{1pt}
        \times
        \prod_{k=1}^\ell \left( \frac{1}{ z_k ( 1 + s z_k ) } \prod_{\alpha=1}^{i_k} \frac{x_\alpha ( 1 + s  z_k ) }{x_\alpha - z_k} \prod_{\alpha=1}^{j} \frac{ 1 + q v_\alpha z_k }{ 1 + v_\alpha z_k } \right).
    \end{split}
\end{equation}
Here, integration contours are $\gamma[-\overline{\mathbf{v}} | k] = \gamma[-\overline{\mathbf{v}}] \cup r^{k-1}c_0$, where $\gamma[-\overline{\mathbf{v}}]$ encircles $-1/v_1,\dots,-1/v_j$ and no other singularity, $c_0$ is a small circle around $0$, and $r>q^{-1}$. All curves are positively oriented, and $r^{k-1}c_0$ never intersects $\gamma[-\overline{\mathbf{v}}]$ for $k=1,\dots,\ell$.
\end{proposition}

\begin{proof}
	This follows from
	Theorem 9.8  
	in 
	\cite{BorodinPetrov2016inhom}
	by identifying the parameters
	as in \Cref{rem:matching_to_BP}
	and noting that 
	$\mathcal{H}^{\mathrm{ur}}_{\mathrm{HS}}(i,j)$ is the same as the height
	function $\mathfrak{h}(i)$ at the $j$-th horizontal slice. 
	Note also that
	\cite[Corollary 10.3]{BorodinPetrov2016inhom} is essentially the 
	same as our $q$-moments \eqref{eq:HS_ur_moments},
	but with contours dragged through infinity, and 
	identification of $\mathsf{s}_1^2$ with $x_1$.
	The latter follows by 
	comparing \eqref{eq:ur_Bernoulli_bc} with \cite[Remark 6.14]{BorodinPetrov2016inhom}.
\end{proof}

Eigenrelations for sqW polynomials given in Theorem \ref{thm:eigenrelation_sqW} can be employed to provide an alternative proof of the moment formula \eqref{eq:HS_ur_moments}.

\begin{proof}[Alternative proof of Proposition \ref{prop:BP_moments_HS}]
	We express $q$-moments of last rows of the sqW/sHL process using 
	the $q$-difference operators 
	$\mathfrak{D}_1$ 
	\eqref{eq:D_1} at several levels, 
	following the argument in 
	\cite[Proposition 4.4]{BCGS2013}.
        
	Denote by
	$\mathfrak{D}_1^{(i)}$ the operator $\mathfrak{D}_1$ 
	acting on $i$ variables $x_1,\dots,x_i$.
	Then for any $\ell$ and any sequence $1 \le i_1 \le \cdots \le
	i_\ell$, we have
    \begin{equation} \label{eq:sHL_sqW_moments_operators}
			\mathbb{E} \prod_{k = 1}^{\ell} q^{\lambda^{(i_k,j)}_{i_k}} = \frac{\mathfrak{D}_1^{(i_1)} \cdots \mathfrak{D}_1^{(i_\ell)} \Pi( x_1,\dots,x_N ; v_1,\dots,v_j ) }{\Pi( x_1,\dots,x_N ; v_1,\dots,v_j )},
    \end{equation}
		where $N\ge i_\ell$ is arbitrary, and 
		\begin{equation*}
			\Pi(x_1,\ldots,x_N;v_1,\ldots,v_j)=
			\prod_{r=1}^j \left(\frac{1}{1-s v_r}\right)^{N-1} 
			\prod_{i=1}^N
			\prod_{r=1}^j (1 + v_r x_i)
		\end{equation*}
		is the partition
		function in the right-hand side of 
		the sqW/sHL Cauchy identity \eqref{eq:Cauchy_Id_sqW_sHL}.
		Equality \eqref{eq:sHL_sqW_moments_operators} is a straightforward
		consequence of the Cauchy identities
		\eqref{eq:skew_Cauchy_ID_sqW_sHL}, 
		\eqref{eq:Cauchy_Id_sqW_sHL},
		eigenrelation
		\eqref{eq:eigenrelation_sqW}, and the branching rules for the sqW
		functions. 

		Let us now express the right-hand side of
		\eqref{eq:sHL_sqW_moments_operators} in terms of nested contour
		integrals. For
    \begin{equation*}
        h(z)=\prod_{r=1}^j (1+ v_r z),
    \end{equation*}
    we have
    \begin{equation*}
        \mathrm{r.h.s.}\, \eqref{eq:sHL_sqW_moments_operators} =  \frac{\mathfrak{D}_1^{(i_1)} \cdots \mathfrak{D}_1^{(i_\ell)} h(x_1) \cdots h(x_N) }{ h(x_1) \cdots h(x_N) }.
    \end{equation*}
    Moreover, for any meromorphic function $\widetilde{h}$ 
		we have 
    \begin{equation*}
			\mathfrak{D}_1^{(n)} \left( \widetilde{h}(x_1) \cdots \widetilde{h}(x_n) \right) =
			\frac{1}{2\pi \mathrm{i}}\oint_{\gamma_{\widetilde{h}}} \prod_{\alpha=1}^n \left( \widetilde{h}(x_\alpha)\, \frac{x_\alpha (1+s z)}{x_\alpha - z}  \right) 
        \frac{\widetilde{h}(qz)}{\widetilde{h}(z)}
        \frac{dz}{z(1+sz)},
    \end{equation*}
    where the curve $\gamma_{\widetilde{h}}$ 
		encircles 0 and all poles of $\widetilde{h}(qz) / \widetilde{h}(z) $.
		The latter poles may include infinity, too.
		In other words, the integral over 
		$\gamma_{\widetilde h}$ is equal to 
		the sum of minus residues of the integrand at $x_1,\ldots,x_n $.

		By iterating this integral representation, 
		we can evaluate \eqref{eq:sHL_sqW_moments_operators} 
		and match the resulting expression with the $q$-moment formula \eqref{eq:HS_ur_moments}. 
		The equivalence of processes $\lambda^{(i,j)}_i$ and $\mathcal{H}_{\mathrm{HS}}^{\mathrm{ur}}(i,j)$ stated in Theorem \ref{thm:sqW_sHL_last_row} allows us to complete the proof.
\end{proof}

\subsection{First row in sqW/sHL field}
\label{sub:sqw_shl_first}

Let us define an \emph{up-left version of the 
stochastic higher spin six vertex model}.
Take an up-left model 
in the sense of \Cref{def:ul_sVM},
with the weights 
$L^{\mathrm{ul}}_{(i,j)}=\mathsf{L}_{x_i,v_j}^{\mathrm{ul}}$,
given in \Cref{fig:ul_HS6VM}.
We take this model with the same step-stationary boundary conditions
\eqref{eq:ur_Bernoulli_bc}.
In fact, this model is essentially
the same as the one from \Cref{def:ur_HS6VS}:

\begin{remark}
	When at most one path occupies each horizontal edge
	(as in our case),
	swapping the horizontal
	occupation numbers $0 \leftrightarrow 1$
	is a bijection between up-left and up-right
	models.
	Their height
	functions are related as 
	$\mathcal{H}_{\mathrm{HS}}^{\mathrm{ur}} (i,j) = j -
	\mathcal{H}_{\mathrm{HS}}^{\mathrm{ul}} (i,j)$.
	Moreover, the weights $\mathsf{L}_{x,v}^{\mathrm{ul}}$ become
	the weights $\mathsf{L}_{x,v}^{\mathrm{ur}}$ from \Cref{fig:up_right_HS6VM}
	after this swapping of horizontal occupations, 
	and the inversion of the parameters $(x,v)\mapsto (x^{-1},v^{-1})$.
\end{remark}

However, it is convenient to 
work with the up-right and the up-left models separately,
as in the sqW/sqW case they are genuinely different.

\begin{figure}
    \centering
    \includegraphics{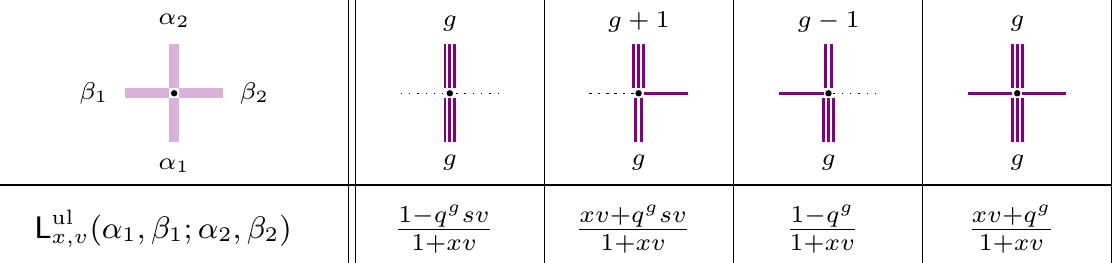}
    \caption{
    The stochastic vertex weights $\mathsf{L}^{ \mathrm{ul} }_{x,v}$ for the up-left stochastic higher spin six vertex model.
    }
    \label{fig:ul_HS6VM}
\end{figure}

\begin{theorem}[sqW/sHL first row]
	\label{thm:sqW_sHL_first_row}
	The first row marginal $\{\lambda^{(i,j)}_1 \}_{i\ge1, j\ge0}$ 
	of the sqW/sHL field 
	has the same distribution as the 
	height function 
	$\{
	\mathcal{H}_{\mathrm{HS}}^{\mathrm{ul}} (i,j) \}_{i\ge1, j\ge0}$
	of
	the up-left stochastic higher spin six vertex model with
	step-stationary boundary conditions.
\end{theorem}

\begin{proof}
	We use Yang--Baxter fields similarly to the approach
	taken in \cite[Section 7.3]{BufetovMucciconiPetrov2018}.
	Let us
	specialize the general notation of \Cref{prop:general_marginal_YBfield}.
	We need to 
	match the stochastic vertex weight 
	$\mathbf{L}$ of \eqref{eq:general_stochastic_L} with $\mathsf{L}^{\mathrm{ul}}$,
	and verify boundary conditions. 

	The random move $\lambda_1^{(i-1,j-1)} \to \lambda_1^{(i,j)}$,
	conditioned on $\lambda_1^{(i,j-1)},\lambda_1^{(i-1,j)}$
	is determined by the bijectivization of the Yang--Baxter equation
	\eqref{eq:YBE_RWw_star} for $i>1,j \ge 1$ and by the bijectivization
	of \eqref{eq:YBE_RWw_star_wall}, if $i=1$. We start with the first
	case, where in \eqref{eq:general_stochastic_L} we get (after canceling common factors)
    \begin{equation*}
        \begin{split}
            &
						\mathsf{wl}_{ \{0,0,\infty\},\{j_1,j_2,\infty \}}(\{k_1,k_2,\infty\}) 
						= 
						\mathcal{R}_{x,v,s}(0,0;k_1,k_2) \, \frac{v^{j_1}}{1-sv}\, x^{j_2} \,
						\frac{(-s/x;q)_{j_2}}{(q;q)_{j_2}},
            \\
            &
						\mathsf{wr}_{ \{0,0,\infty\},\{j_1,j_2,\infty \}}(\{k_1',k_2',\infty\})
						=
						\mathcal{R}_{x,v,s}(k_2',k_1';j_2,j_1) \frac{v^{k_1'}}{1-sv} \,
						x^{k_2'}\, \frac{(-s/x;q)_{k_2'}}{(q;q)_{k_2'}}.
        \end{split}
    \end{equation*}
		One readily sees that then \eqref{eq:general_stochastic_L} gives the 
		stochastic weight $\mathsf{L}^{\mathrm{ul}}_{x,v}$.

		For the boundary signature $\lambda^{(1,j)}_1$ case, configuration
		weights $\mathsf{wl},\mathsf{wr}$ become (after canceling common factors)
    \begin{equation*}
        \begin{split}
            &
						\mathsf{wl}_{ \{0,0,\infty\},\{\varnothing,j_2,\infty \}}(\{k_1,k_2,\infty\}) = \mathcal{R}_{x,v,s}(0,0;k_1,k_2) \, x^{j_2},
            \\
            &
						\mathsf{wr}_{ \{0,0,\infty\},\{\varnothing,j_2,\infty \}}(\{k_1',k_2',\infty\})= \frac{v^{k_1'}}{1-sv}\, x^{k_2'},
        \end{split}
    \end{equation*}
		which leads to the step-stationary boundary conditions \eqref{eq:ur_Bernoulli_bc}
		since $\lambda^{(i,0)}=\varnothing$ for all $i$.
\end{proof}

\subsection{Last row in sqW/sqW field}
\label{sub:sqw_sqw_last}

Define the up-right stochastic weight by
\begin{equation} \label{eq:q_Hahn_vertex_weight}
    \mathbb{L}_{x,y}^{\mathrm{ur}}(\alpha_1,\beta_1;\alpha_2,\beta_2)
    \coloneqq 
    \mathbf{1}_{\alpha_1 + \beta_1 = \alpha_2 + \beta_2 }\,
		\varphi_{q,x y, -s x} ( \beta_2 \mid \alpha_1 ),
\end{equation}
where $\varphi$ is the $q$-beta-binomial
distribution
\eqref{eq:phi_def}--\eqref{eq:phi_def_infty}.

\begin{definition}[\cite{Povolotsky2013}] \label{def:q_Hahn_vertex_model}
	The $q$-\emph{Hahn vertex model} is the up-right stochastic vertex
	model, in the sense of Definition \ref{def:ur_sVM}, with weights
	$L_{i,j}^{\mathrm{ur}} = \mathbb{L}_{x_i,y_j}^{\mathrm{ur}}$. We
	consider \emph{step-stationary} boundary conditions:
	\begin{equation}
		\label{eq:q_Hahn_step_stationary}
		b_j^{\mathrm{v}} \sim \varphi_{q,x_1 y_j,-s x_1}(\bullet\mid
		\infty) \qquad \text{and} \qquad b_i^{\mathrm{h}}=0,
	\end{equation}
	where the random variables for 
	$b_j^{\mathrm{v}}$ are independent.
	Denote the corresponding
	height function by
	$\mathcal{H}^{\mathrm{ur}}_{q\textnormal{-Hahn}}$.
\end{definition}

\begin{remark}
		The model of \Cref{def:q_Hahn_vertex_model} is
		equivalent to that of \cite[Section
		6.6.2]{BorodinPetrov2016inhom}, where parameters have been
		specialized as $\mathsf{s}^2_\alpha \to -s x_\alpha$ and
		$q^{J_\alpha}\to -y_\alpha /s$.
\end{remark}

\begin{theorem}[sqW/sqW last row] \label{thm:sqW_sqW_last_row}
	The last row marginal 
	$\{\lambda^{(i,j)}_i\}_{i \ge 1,j \ge 0}$ 
	of the sqW/sqW field 
	has the same distribution as the 
	height
	function 
	$\{\mathcal{H}_{q\textnormal{-Hahn}}^{\mathrm{ur}} (i,j)\}_{i \ge 1,j \ge 0}$
	of the up-right $q$-Hahn vertex model.
\end{theorem}

\begin{proof}
		This follows from \Cref{thm:sqW_sHL_last_row}
		which established an analogous result 
		matching 
		the last row of the sqW/sHL field
		and the height function of the up-right higher spin six vertex
		model. 
		By fusion, 
		the dual sHL 
		functions turn into the dual sqW functions 
		(cf. \cite{BorodinWheelerSpinq}).
		Therefore, the 
		sqW/sHL field under fusion turns into
		the sqW/sqW field. 

		On the other hand, the same fusion
		procedure turns the
		up-right higher spin six vertex model into the up-right $q$-Hahn
		vertex model\footnote{For a practical explanation of fusion in the
		context of $\mathfrak{sl}_2$ stochastic vertex models see
		\cite{BorodinWheelerSpinq} and \cite{CorwinPetrov2015}, 
		\cite{BorodinPetrov2016inhom} 
		and
		references therein.}. 
		This completes the proof.
\end{proof}

In \cite[Corollary 10.4]{BorodinPetrov2016inhom} the multi-point
$q$-moments of the up-right $q$-Hahn vertex model were expressed in
terms of nested contour integrals:

\begin{proposition}[\cite{BorodinPetrov2016inhom}, Corollary 10.4] \label{prop:BP_moments_qHahn}
	Assume $\min_\alpha |s x_\alpha|>q \max_\alpha |s x_\alpha|$. 
	For any $i_1\ge \ldots\ge i_\ell\ge1$ we have
    \begin{equation} \label{eq:q_moments_q_Hahn}
        \begin{split}
            \mathbb{E} \prod_{k = 1}^{\ell} q^{\mathcal{H}^{\mathrm{ur}}_{q\textnormal{-Hahn}}(i_k,j)}
            =
           (-1)^\ell q^{\binom{\ell}{2}}
            &
            \oint_{\gamma_1^+[- s \mathbf{x}]} \frac{d w_1}{2 \pi \mathrm{i}} \cdots \oint_{\gamma_\ell^+[- s \mathbf{x} ]} \frac{d w_\ell}{2 \pi \mathrm{i}}
            \prod_{1\le A < B \le \ell} \frac{w_A - w_B}{w_A - q w_B}
            \\
            &
            \hspace{1pt}
            \times
            \prod_{k=1}^\ell \left( \frac{1}{ w_k ( 1- w_k ) } \prod_{\alpha=1}^{i_k} \frac{ 1 -  w_k }{1 + w_k/(s x_\alpha)} \prod_{\alpha=1}^{j} \frac{ 1 +  w_k y_\alpha/s }{ 1 - w_k } \right).
        \end{split}
    \end{equation}
		Integration contours encircle $-sx_1,-sx_2,\dots$, and leave out
		0,1 and are $q$-nested in the sense that $q \gamma_{k+1}^+[- s
		\mathbf{x}]$ is inside $\gamma_k^+[- s \mathbf{x}]$ for all
		$k=1,\dots,\ell-1$. 
\end{proposition}


Proposition \ref{prop:BP_moments_qHahn} was obtained in
\cite{BorodinPetrov2016inhom} 
as a corollary (under fusion)
of the multi-point
$q$-moment formula \eqref{eq:HS_ur_moments} for the up-right
higher spin six vertex model. 
Both of these $q$-moment formulas have 
several different proofs: 
via duality 
\cite{CorwinPetrov2015},
manipulations with symmetric functions using Bethe Ansatz
\cite{BorodinPetrov2016inhom}, 
or distributional matchings and difference operators
\cite{OrrPetrov2016}.
Eigenrelations for the sqW
polynomials provide yet another independent proof:

\begin{proof}[Alternative proof of Proposition \ref{prop:BP_moments_qHahn}]
	Similarly to the alternative proof of
	\Cref{prop:BP_moments_HS} given in \Cref{sub:sqw_shl_last}, we will
	use eigenrelations of the sqW polynomials to compute
	$q$-moments. To express $q$-moments of the
	sqW/sqW field, we use formula
	\eqref{eq:sHL_sqW_moments_operators}, after replacing the function
	$\Pi$ with the right-hand side of \eqref{eq:Cauchy_Id_sqW_sqW}. 
	The action of the difference operator $\mathfrak{D}_1$
	\eqref{eq:D_1} (in $n$ variables)
	on a meromorphic function $\widetilde{h}$ can be written as
	\begin{equation*}
		\mathfrak{D}_1 \left( \widetilde{h}(x_1) \cdots
		\widetilde{h}(x_n) \right) 
		= 
		-\frac{1}{2\pi \mathrm{i}}
		\oint_{x_1,\dots,x_n}
		\prod_{\alpha=1}^n \left( \widetilde{h}(x_\alpha)\, \frac{x_\alpha
		(1+s z)}{x_\alpha - z}  \right)
		\frac{\widetilde{h}(qz)}{\widetilde{h}(z)} \frac{dz}{z(1+sz)},
	\end{equation*}
	where the integration contour contains $x_1,\dots,x_n$, but doesn't
	contain 0 or any pole of $\widetilde{h}(qz)/\widetilde{h}(z)$. 
	Using this formula 
	repeatedly, we can match the 
	$q$-moments of the 
	marginal $\lambda^{(i,j)}_i$ 
	to expression \eqref{eq:q_moments_q_Hahn}. The equivalence of processes between last row of the sqW/sqW field and height function of the $q$-Hahn vertex model stated in Theorem \ref{thm:sqW_sqW_last_row} yields the proof.
\end{proof}

\subsection{First row in sqW/sqW field}
\label{sub:sqw_sqw_first}

For our fourth and final vertex model,
define the up-left stochastic weight by
\begin{equation}\label{eq:4phi3_weights}
    \begin{split}
	        &
	        \mathbb{L}_{x,y}^{\mathrm{ul}}(\alpha_1,\beta_1;\alpha_2,\beta_2)
	        \coloneqq
	        \mathbf{1}_{\alpha_1 + \beta_2 = \alpha_2 + \beta_1 }\,  
	        \frac{ y^{\alpha_2} s^{\alpha_1} x^{\alpha_2-\alpha_1} \, q^{ \beta_1 \beta_2 +\frac{1}{2}\alpha_1(\alpha_1 -1) }
	        \, (-s/x;q)_{\alpha_2} (-s/y;q)_{\beta_2} }
	        { (-s/x;q)_{\alpha_1} (-s/y;q)_{\beta_1} 
	        (q;q)_{\beta_2} (-q/(sy);q)_{\beta_2 -\alpha_2} }
	        \\
	        &
	        \hspace{80pt}
	        \times \frac{(s^2 q^{\alpha_1+\beta_2};q)_\infty 
	        (x  y; q)_\infty}{(-s y;q)_\infty (-s x ; q)_\infty} \,{}_4 \overline{ \phi}_3
	        \left(\begin{minipage}{5.2cm}
	        \center{$q^{-\beta_1}; q^{-\beta_2}, -s x,  -q/(s y)$}
        	\\
	        \center{$-s/y,q^{1+\alpha_1-\beta_1}, -x q^{1-\beta_2-\alpha_1}/s$}
			\end{minipage} \Big\vert\, q,q\right),
        \end{split}
    \end{equation}
		where ${}_4 \overline{ \phi}_3$
		is the regularized
		$q$-hypergeometric function
		\eqref{eq:phi_regularized}.

\begin{remark}
    An expression equivalent to \eqref{eq:4phi3_weights} for the stochastic weight $\mathbb{L}_{x,y}^{\mathrm{ul}}$ is given by
    \begin{equation} \label{eq:4phi3_vertex_model_alternative}
        \mathbb{L}_{x,y}^{\mathrm{ul}}(g,\ell;g+L-\ell,L)
        =
        \sum_{k=0}^{\min(\ell,L)}\varphi_{q^{-1},q^g,-syq^{g-1}}(k \mid \ell)
        \,
        \psi_{q,-q^k s/ y, -q^gs/x,s^2q^{g+k} }(L-k),
    \end{equation}
		where we used the $q$-beta-binomial and the $q$-hypergeometric distributions
		\eqref{eq:phi_def},
		\eqref{eq:q_hypergeom_distr}. This can be proved through simple
		manipulations of the $q$-Pochhammer
		terms. From 
		\eqref{eq:4phi3_vertex_model_alternative} it is immediate to see
		that $\mathbb{L}_{x,y}^{\mathrm{ul}}$ possesses the sum to one
		property \eqref{eq:sum_to_one}. The positivity of the weights (under certain 
		restrictions on the parameters)
		follows from \Cref{prop:positivity_R}.
\end{remark}

\begin{definition} \label{def:4phi3_vertex_model}
	The ${}_4\phi_3$ \emph{vertex model} is the up-left stochastic
	vertex model, in the sense of \Cref{def:ul_sVM}, with
	weights $L_{i,j}^{\mathrm{ur}} =
	\mathbb{L}_{x_i,y_j}^{\mathrm{ul}}$. We consider
	the same
	\emph{step-stationary} boundary conditions 
	as in \eqref{eq:q_Hahn_step_stationary}.
	The height function of this model is denoted by 
	$\mathcal{H}_\phi^{\mathrm{ul}}$.
\end{definition}

\begin{theorem}[sqW/sqW first row] \label{thm:sqW_sqW_first_row}
	Let $s\in (-\sqrt{q},0)$. The first row marginal $\{
	\lambda^{(i,j)}_1 \}_{i,j\in \mathbb{Z}_{\ge0}}$ 
	of the sqW/sqW
	field 
	has the same distribution
	as the height function
	$\{ \mathcal{H}_{\phi}^{\mathrm{ul}}(i,j)\}_{i,j\in \mathbb{Z}_{\ge0}}$
	of the 
	${}_4\phi_3$ stochastic
	vertex model.
\end{theorem}

\begin{proof}
	The proof of this matching is similar to that of 
	\Cref{thm:sqW_sHL_first_row},
	and follows from 
	\Cref{prop:general_marginal_YBfield}. 
	Namely, we 
	specialize formula
	\eqref{eq:general_stochastic_L} using the Yang--Baxter equations
	\eqref{eq:YBE_RWW_star}, \eqref{eq:YBE_RWW_star_wall}. 
	For updates
	of ``bulk'' transition $\lambda_1^{(i-1,j-1)} \to \lambda_1^{(i,j)}$,
	for $i>1,j\ge1$, 
	conditioned on $\lambda_1^{(i,j-1)},\lambda_1^{(i-1,j)}$, the
	stochastic weight \eqref{eq:general_stochastic_L} 
	uses
	\begin{equation*}
			\begin{split}
					&
					\mathsf{wl}_{ \{0,0,\infty\},\{j_1,j_2,\infty \}}(\{k_1,k_2,\infty\}) 
					=
					\mathbb{R}_{x,y,s}(0,0;k_1,k_2) \, y^{j_1} \frac{(-s/y;q)_{j_1}}{(q;q)_{j_1}}\,
					x^{j_2} \frac{(-s/x;q)_{j_2}}{(q;q)_{j_2}},
					\\
					&
					\mathsf{wr}_{ \{0,0,\infty\},\{j_1,j_2,\infty \}}(\{k_1',k_2',\infty\})
					=
					\mathbb{R}_{x,y,s}(k_2',k_1';j_2,j_1) \,
					y^{k_1'} \frac{(-s/y;q)_{k_1'}}{(q;q)_{k_1'}}
					\,x^{k_2'}
					\frac{(-s/x;q)_{k_2'}}{(q;q)_{k_2'}}.
			\end{split}
	\end{equation*}
	Using the expression of the R-matrix $\mathbb{R}_{x,y,s}$ and
	summation identity \eqref{eq:sum_R_matrix} one can match
	$\mathbf{L}_{x,y}$ with $\mathbb{L}_{x,y}$. At the boundary
	$\lambda^{(1,j)}_1$, we use a
	stochastic bijectivization of \eqref{eq:YBE_RWW_star_wall} and therefore in this case we have
	\begin{equation*}
		\begin{split}
			&
			\mathsf{wl}_{ \{0,0,\infty\},\{\varnothing,j_2,\infty
			\}}(\{k_1,k_2,\infty\}) = \mathbb{R}_{x,y,s}(0,0;k_1,k_2) \,
			x^{j_2},
			\\
			&
			\mathsf{wr}_{ \{0,0,\infty\},\{\varnothing,j_2,\infty
			\}}(\{k_1',k_2',\infty\})= y^{k_1'}\frac{ (-s/y;q)_{k_1'} }{
			(q;q)_{k_1'} } \frac{(-sy;q)_\infty}{(s^2;q)_\infty}
			\,x^{k_2'},
		\end{split}
	\end{equation*}
	that yields boundary conditions \eqref{eq:q_Hahn_step_stationary}
	after using again summation identity \eqref{eq:sum_R_matrix}.
\end{proof}

\subsection{Push--block dynamics for sqW/sqW process}
\label{sub:sqw_sqw_last_continuous}

Let us now present another, more explicit matching 
of last rows of the sqW/sqW field
in a ``Plancherel'' (or ``Poisson-type'') 
continuous time limit.
Here the dynamics of the last rows is matched
to the corresponding continuous time limit of the 
$q$-Hahn TASEP. 
This construction is very similar to how the continuous time 
$q$-TASEP emerges from $q$-Whittaker processes 
in \cite{BorodinCorwin2011Macdonald}.

\medskip

Consider the Borodin--Ferrari forward transition map
(cf. \Cref{sub:BF_Fields})
\begin{equation} 
	\label{eq:BF_sqW_sqW}
    \Ufwd_{x,y}(\varkappa\to \nu \mid \lambda,\mu) = \frac{\mathbb{F}_{ \nu / \lambda }(x) \mathbb{F}^*_{\nu/\mu}(y) }{\Pi(x;y) \sum_{\varkappa} \mathbb{F}_{ \mu/\varkappa }(x) \mathbb{F}^*_{ \lambda/\varkappa }(y)  },
\end{equation}
where $\Pi(x;y)=\frac{(-sx;q)_\infty (-sy;q)_\infty}{(xy;q)_\infty (s^2;q)_\infty}$. 
In the limit as
$y=-s +
\varepsilon (1-q)$, $\varepsilon\to0$,
the dual sqW function at a single variable becomes
(we use the notation $[r]_q=(1-q^r)/(1-q)$)
\begin{equation*}
	\mathbb{F}^*_{\lambda/\mu}(-s+\varepsilon (1-q)) 
	= 
	\begin{cases}
		1+\bigO(\varepsilon),
		&\lambda=\mu
		;
		\\
		\varepsilon\,
		\dfrac{(-s)^{r-1}}{[r]_q}
		\dfrac{(q^{\mu_{i-1}-\lambda_i+1};q)_r}{(q^{\mu_{i-1}-\lambda_i} s^2
		;q)_r} +\bigO(\varepsilon^2),
		& \lambda=\mu+r \mathbf{e}_i \ \textnormal{for some $i,r>0$}.
	\end{cases}
\end{equation*}
see \eqref{eq:dual_sqw_single_variable_expression}.
Take $y_j=-s + \varepsilon (1-q)$ for all $j$ and
rescale $M=\floor{t/\varepsilon}$, $t\in \mathbb{R}_{\ge0}$, in the sqW/sqW field.
Thus, we get a
continuous time dynamics on interlacing arrays
$\lambda^1(t) \prec \lambda^2 (t) \prec \cdots$, where at time $t$,
each 
$\lambda^k_i$ jumps to  $\lambda^k_i+r$, $r\ge1$, according to an exponential clock with rate
(see \eqref{eq:BF_sqW_sqW})
\begin{equation} \label{eq:sqW_sqW_push_block_rate}
		\mathrm{rate}(\lambda^k \to \lambda^k + r  \mathbf{e}_{i} \mid
		\lambda^{k-1} ) = x_k^{r}\, \frac{(-s)^{r-1}}{[r]_q}
		\frac{(-q^{\lambda^k_i - \lambda^{k-1}_i} s/x_k \, , \,
		q^{\lambda^k_i - \lambda^{k}_{i+1}+1} \, , \,
		q^{\lambda^{k-1}_{i-1} - \lambda^k_i + 1 -r }\, ; \, q )_r}{ (
		q^{\lambda^{k}_{i} - \lambda^{k-1}_i + 1 } \, , \, q^{\lambda^k_i -
		\lambda^{k}_{i+1}} s^2 \, , \, -q^{\lambda^{k-1}_{i-1} - \lambda^{k}_i
		-r} s x_k \, ; \, q )_r }.
\end{equation}
When an update occurs at level $j$ bringing $\lambda^j \to \widetilde{\lambda}^j = \lambda^j+r\mathbf{e}_i$, the
signature $\lambda^{j+1}$ is instantaneously updated to $\widetilde{\lambda}^{j+1}$ in the following way:
\begin{itemize}
    \item if $\widetilde{\lambda}^j_i \le \lambda^{j+1}_i$, then $\widetilde{\lambda}^{j+1}=\lambda^{j+1}$
    \item if $\widetilde{\lambda}^j_i > \lambda^{j+1}_i$, 
			then assume $\widetilde{\lambda}^{j}_i - \lambda^{j+1}_i = m$ 
			and set $\widetilde{\lambda}^{j+1}=\lambda^{j+1}+(m+\ell)\mathbf{e}_i$ with probability
    \begin{equation*}
        \mathrm{prob}(\lambda^{j+1} \to \widetilde{\lambda}^{j+1}\mid\lambda^{j} \to \widetilde{\lambda}^{j})
        =
        \lim_{\varepsilon\to 0}
				\frac{\mathbb{F}_{\widetilde{\lambda}^{j+1}/\widetilde{\lambda}^j }(x) \mathbb{F}^*_{\widetilde{\lambda}^{j+1} / \lambda^{j+1} } (y) }{ \sum_{ \eta = \lambda^{j+1} + (m+\ell')\mathbf{e}_i } \mathbb{F}_{ \eta /\widetilde{\lambda}^j }(x) \mathbb{F}^*_{ \eta / \lambda^{j+1} } (y) }\,\bigg|_{y=-s+\varepsilon(1-q)}.
    \end{equation*}
    for any $\ell\ge0$ (for $\ell$ large enough this probability vanishes).
		See \Cref{fig:push_block} for an illustration.
\end{itemize}

\begin{figure}[htbp]
    \centering
    \subfloat[]{{\includegraphics[width=6.5cm]{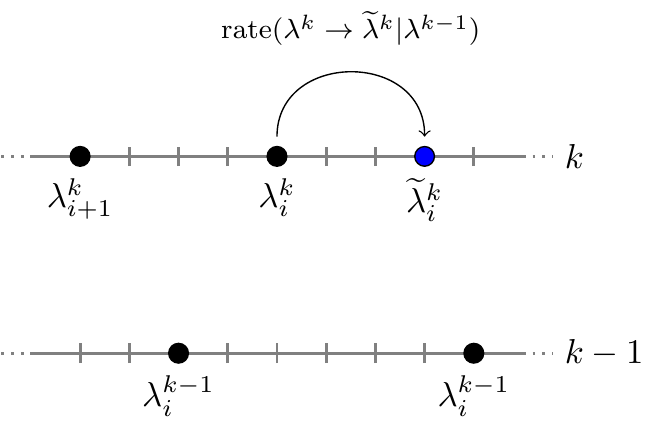} }}%
    \hspace{60pt}
    \subfloat[]{{\includegraphics[width=6.5cm]{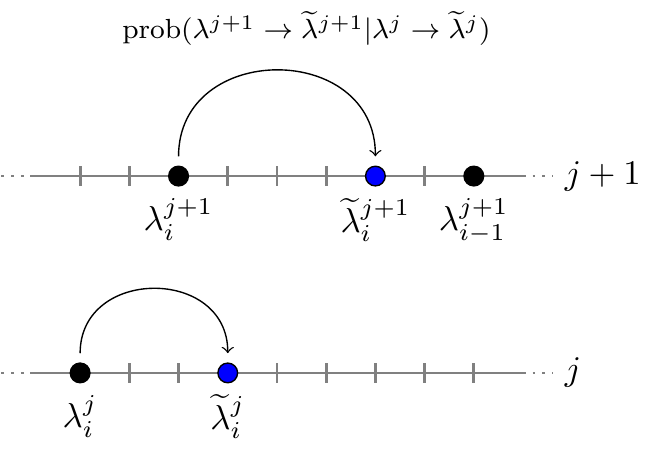} }}%
		\caption{Push--block mechanism in the half-continuous sqW/sqW
		field. Each $\lambda_i^k$ jumps to
		$\widetilde{\lambda}^k_i=\lambda_i^k+r$ at rate
		\eqref{eq:sqW_sqW_push_block_rate}, which only
		depends on $\lambda^{k-1}$; see left panel. When a
		jump happens at level $k$ and breaks interlacing, 
		it
		triggers an instantaneous push at levels above to
		re-establish interlacing; see right panel.}%
    \label{fig:push_block}%
\end{figure}

When $s=0$ and $q\in(0,1)$ in our dynamics, we recover
the continuous time $q$-Whittaker 2d-growth model introduced in
\cite[Definition 3.3.3]{BorodinCorwin2011Macdonald}.
Further setting $q=0$ brings the original
Borodin--Ferrari's push--block process 
corresponding to Schur measures 
\cite{BorFerr2008DF}.
Note that in our case, in contrast with the Schur and $q$-Whittaker situations, jumps are long range.

Restricting attention to the last rows (leftmost diagonal) of the array 
and setting
$i=k$ in 
\eqref{eq:sqW_sqW_push_block_rate}, we see that
the rate 
only depends
on $\lambda^k_k$ and $\lambda^{k-1}_{k-1}$.
Moreover, the pushing mechanism does not affect the leftmost diagonal
of the 
array.
Thus,
the marginal
evolution of the particles in the leftmost diagonal 
is an autonomous Markov process. 
Its jump rates are
\begin{equation*}
    \mathrm{rate}(\lambda^k_k \to \lambda^k_k+r \mid \lambda^{k-1}_{k-1}) 
		=
		x_k^r \,\frac{(-s)^{r-1}}{[r]_q} \frac{(q^{\lambda^{k-1}_{k-1} -
		\lambda^k_k +1-r } ;q )_r}{ (-q^{\lambda^{k-1}_{k-1} -
		\lambda^{k}_k - r} s x_k ; q )_r }.
\end{equation*}
These rates correspond to
an
inhomogeneous version of the continuous time $q$-Hahn TASEP studied in \cite{barraquand2015q},
which is also a continuous time degeneration of the $q$-Hahn TASEP of \cite{Corwin2014qmunu}.
Thus, we see that the continuous time push--block dynamics in the sqW case 
agrees with the last row marginal evolution.

\section{\texorpdfstring{Spin Whittaker Functions from $q\to1$ limit}
{Spin Whittaker Functions from q -> 1 limit}}
\label{sec:sW_functions}

In this section we introduce new 
one-parameter deformations of the $\mathfrak{gl}_n$ Whittaker functions
\cite{Jacquet1967}, \cite{Kostant_Whittaker}. 
These deformations arise from our version of
spin $q$-Whittaker polynomials
in a scaling limit as $q\to1$.
The deformation parameter is denoted by $S>0$.

\subsection{Whittaker functions}
\label{sub:usual_Whittaker}

Before proceeding with deformations of Whittaker functions, 
let us recall the usual $\mathfrak{gl}_{N}$ Whittaker functions.
These functions play a central role in representation theory and 
integrable systems
\cite{Kostant1977Whitt}, \cite{Etingof1999}, \cite{givental1996stationary}
as well as are related to several models of random polymers 
\cite{Oconnell2009_Toda}, \cite{COSZ2011}, \cite{OSZ2012}, 
\cite{BorodinCorwin2011Macdonald}.

The $\mathfrak{gl}_N$ 
Whittaker functions
$\psi_{\lambda_1,\dots,\lambda_N} (\underline{u}_N)$
are indexed\footnote{To match 
the historical notation for Whittaker functions, 
here and in the discussion of the spin Whittaker functions
we place the ``variables'' into the subscript of a Whittaker function,
and the ``index'' in the parentheses.} by $N$-tuples 
$\underline{u}_N=(u_{N,1},\ldots,u_{N,N} )\in \mathbb{R}^N$,
depend on 
$\underline{\lambda}=(\lambda_1,\ldots,\lambda_N )\in \mathbb{C}^N$,
and may be defined through the recursion (following 
from the Givental integral representation \cite{givental1996stationary}, cf.
\cite{Gerasimov_et_al_OnAGaussGivental}): 
\begin{equation}
	\label{eq:usual_Whittaker}
	\psi_{\lambda_1,\dots,\lambda_N} (\underline{u}_N)
	=
	\int_{\mathbb{R}^{N-1}}
	\psi_{\lambda_1,\dots,\lambda_{N-1}} (\underline{u}_{N-1}) \,
	Q^{N \to N-1}_{\lambda_N}(\underline{u}_N,\underline{u}_{N-1})
	\prod_{k=1}^{N-1}  du_{N-1,k},
\end{equation}
where
\begin{equation}
	\label{eq:Baxter_Q}
		Q^{N \to N-1}_{\lambda}(\underline{u}_N,\underline{u}_{N-1}) 
		= e^{\mathrm{i} \lambda \left( \sum_{i=1}^{N} u_{N,i} - \sum_{i=1}^{N-1} u_{N-1,i} \right) }
		\prod_{i=1}^{N-1} 
		\exp\left\{ 
			-
			e^{ u_{N-1,i}-u_{N,i}} 
			- e^{u_{N,i+1}-u_{N-1,i}   } \right\}
\end{equation}
is known as the
Baxter $Q$-operator.
The function $\underline{\lambda}\mapsto
\psi_{\underline{\lambda}}(\underline{u}_N)$
is an entire function of $\underline{\lambda}\in\mathbb{C}^N$ 
for all $\underline{u}_N\in \mathbb{R}^N$.
For $N=1$, we have $\psi_\lambda(u)=e^{\mathrm{i}\lambda u}$.
For $N = 2$, the Whittaker functions can be expressed 
through the (single-variable) 
Bessel $K$ function
$K_v(z)=\frac{1}{2}\int_{-\infty}^{\infty}e^{xv}\exp\left( -\frac{z}{2}(e^x+e^{-x}) \right)dx$.

For the Whittaker functions, 
$Q^{N\to N-1}_{\lambda_{N}}(\underline{u}_N,\underline{u}_{N-1})$ 
plays the role of a
branching function like the single-variable sqW function
$\mathbb{F}_{\nu/\mu}(x)$
\eqref{eq:inc_sqW} (here $x$ plays the same role as $\lambda_N$, and 
$\nu,\mu$ correspond to $\underline{u}_N,\underline{u}_{N-1}$).
Note that the Whittaker functions are not indexed by 
ordered sequences of numbers $\underline{u}_N$. 
Rather, in the Baxter $Q$-operator,
the interlacing condition among arrays
$\underline{u}_{N-1},\underline{u}_{N}$ is replaced by the 
``mild interlacing''.
Namely, 
$Q^{N \to N-1}$ \eqref{eq:Baxter_Q}
decays doubly exponentially whenever
$u_{N,i+1}>u_{N-1,i}$ or $u_{N-1,i}>u_{N,i}$.

The Whittaker functions satisfy the following analogue
of the Cauchy identity due to Bump and Stade
\cite{Bump1989},
\cite{Stade2002},
\cite{gerasimov2008baxter}:
\begin{equation}
	\label{eq:Bump_Stade_cauchy}
	\int_{\mathbb{R}^N}
	e^{-e^{-u_{N,N}}}
	\overline{\psi_{\underline{\lambda}_N}(\underline{u}_N)}
	\psi_{\underline{\nu}_N}(\underline{u}_N)
	\prod_{j=1}^N
	d u_{N,j}
	=\prod_{j,k=1}^N
	\Gamma(\mathrm{i}\nu_j-\mathrm{i}\lambda_k).
\end{equation}
See also 
\cite[(1.2)]{COSZ2011}, \cite[Section 4.2.1]{BorodinCorwin2011Macdonald}
for a generalization when 
one of the Whittaker functions is replaced by a certain integral
coming from the limit of the torus product representation
of Macdonald polynomials:
\begin{equation}
	\label{eq:dual_usual_Whittaker_function}
	\theta_{\underline{Y}}(\underline{u}_N)
	\coloneqq
	\int_{\mathbb{R}^N}\overline{\psi_{\underline{\nu}}(\underline u_N)}
	\prod_{i=1}^{T}\prod_{k=1}^N
	\Gamma(Y_i- \mathrm{i}\nu_k)\,
	\cdot
	\frac{1}{(2\pi)^N N!}
	\prod_{1\le A\ne B\le N}
	\frac{1}{\Gamma(\mathrm{i}\nu_A-\mathrm{i}\nu_B)}\,
	d\underline{\nu},
\end{equation}
where $\underline{Y}=(Y_1,\ldots,Y_T )\in \mathbb{R}^T$.
We refer to $\theta_{\underline{Y}}(\underline{u}_N)$
as the \emph{dual Whittaker function}.

Similar integral representations for dual spin Hall--Littlewood
functions are found in 
\cite[Proposition 7.3]{Borodin2014vertex}, 
\cite[Section 7.3]{BorodinPetrov2016inhom}.

The Whittaker functions are eigenfunctions
of the $\mathfrak{gl}_N$ quantum Toda Hamiltonian
$\mathscr{H}_2^{\mathrm{Toda}}$, 
see formula 
\eqref{eq:intro_Whit_eigenrelations} in the Introduction.

\medskip

\noindent\textbf{Convention on multiplicative notation}. The papers \cite{COSZ2011}, \cite{OSZ2012} 
use \emph{multiplicative parameters}
$U_{N,i}=e^{u_{N,i}}\in \mathbb{R}_{>0}$ instead of the additive ones.
In multiplicative notation,
the integration in \eqref{eq:usual_Whittaker} and \eqref{eq:Bump_Stade_cauchy}
is over the product measures of the form $\prod \frac{dU_{m,i}}{U_{m,i}}$.
It is convenient for us to 
adopt multiplicative notation throughout most of the 
discussion of the spin Whittaker functions.
We will often denote multiplicative variables and parameters by capital letters.

\subsection{Signatures in continuous space}

In contrast with the usual Whittaker functions
indexed by unordered $N$-tuples of reals,
the spin Whittaker functions will be indexed by
\emph{nondecreasing}
sequences of real numbers. 
Introduce the \emph{Weyl chamber} of
$\mathbb{R}_{\ge 1}^N$ by
\begin{equation}
	\label{eq:Whit_Weyl_chamber}
	\mathpzc{W}_N \coloneqq \{ \underline{L}_N = (L_{N,i})_{1
	\le i \le N} \in \mathbb{R}_{\ge 1}^N \colon L_{N,N}\le L_{N,N-1}\le \ldots\le L_{N,1} \}.
\end{equation}
By 
$\mathring{\mathpzc{W}}_N$
denote the interior of the Weyl chamber with strict inequalities in 
\eqref{eq:Whit_Weyl_chamber}.

Given two sequences $\underline{L}_{N-1} \in \mathpzc{W}_{N-1}$ and $\underline{L}_{N} \in \mathpzc{W}_{N}$, we say that they \emph{interlace} if
\begin{equation} \label{eq:interlacing_continuous}
	L_{N,i+1} \le L_{N-1,i} \le L_{N,i}, \qquad \text{for } 1\le i \le N-1.
\end{equation}
As in discrete setting, we denote interlacing by $\underline{L}_{N-1} \prec
\underline{L}_{N}$. 
The interlacing relation is naturally extended to
sequences of the same length by dropping the last inequality in
\eqref{eq:interlacing}. 

We endow the Weyl chamber $\mathpzc{W}_N$ with
the measure $\frac{d \underline{L}_N}{\underline{L}_N} = \prod_{k=1}^N
\frac{d L_{N,k}}{L_{N,k}}$. In most
cases
we do not explicitly indicate the 
integration domain $\mathpzc{W}_N$
when the measure $\frac{d \underline{L}_N}{\underline{L}_N}$ is used.

Define the \emph{continuous Gelfand-Tsetlin cone} as
    \begin{equation}
			\label{eq:cont_GT_cone}
			\mathpzc{GT}_{\hspace{-0.5ex}N} \coloneqq \{ \doubleunderline{L}_N = (L_{k,i})_{1 \le i \le k \le N} \in \mathbb{R}_{\ge 1}^{N(N+1)/2} \colon L_{k+1,i+1} \le L_{k,i} \le L_{k+1,i} \},
    \end{equation}
which is
the set of 
interlacing sequences
$\underline{L}_1 \prec \cdots \prec \underline{L}_N$. The set
$\mathpzc{GT}_{\hspace{-0.5ex}N}$ is endowed with the measure $\frac{d
\doubleunderline{L}_N}{\doubleunderline{L}_N} = \prod_{1\le i \le j
\le N} \frac{d L_{j,i}}{L_{j,i}}$.

\subsection{Spin Whittaker functions}

We begin with a branching function
from which we can recursively build 
spin Whittaker functions.
The branching function is an analogue of the skew polynomial
evaluated at a single variable.

Fix a deformation parameter $S>0$
throughout the section.
Let us denote
\begin{equation} \label{eq:A}
	\mathcal{A}_{S,X}(u,v,z) 
	\coloneqq
	\frac{1}{\mathrm{B}(S+X, S-X)} 
	\left( 1-\frac{v}{z} \right)^{S-X-1} \left( 1-\frac{u}{v} \right)^{S+X-1}
	\left( 1-\frac{ u }{ z } \right)^{1-2S},
\end{equation}
where $1\le u < v < z$ are real, and $|X|<S$. 
Here $\mathrm{B}(\cdot,\cdot)$ is the beta function
\eqref{eq:Beta}.

\begin{definition}
	Let $|X|<S$ and $k\ge1$. 
	The \emph{spin Whittaker branching functions} are given by
	\begin{equation*}
			\mathfrak{f}_X (\underline{L}_k;\underline{L}_{k+1})
			\coloneqq
			\mathbf{1}_{\underline{L}_k \prec \underline{L}_{k+1}}
			\left(
			\frac{ L_{k+1,k+1} \cdots L_{k+1,1} }{ L_{k,k} \cdots L_{k,1} }
		\right)^{-X} \prod_{i=1}^k \mathcal{A}_{S,X} (L_{k+1,i+1}, L_{k,i}, L_{k+1,i}).
	\end{equation*}
\end{definition}

We now introduce the main object of the present section.

\begin{definition}[Spin Whittaker functions]
	\label{def:sW}
    For $N\ge 1$, consider parameters $X_1, \dots, X_N$ and $S$ such that $|X_i|<S$ for all $i$. 
		The \emph{spin Whittaker functions} $\mathfrak{f}_{X_1, \dots ,X_N}(\underline{L}_N)$,
		$\underline{L}_N\in \mathpzc{W}_N$,
		are defined recursively by
    \begin{equation} \label{eq:sW_N=1}
        \mathfrak{f}_{X_1} (L_{1,1}) \coloneqq L_{1,1}^{-X_1}
    \end{equation}
		for $N=1$,
		and via the branching rule
    \begin{equation} \label{eq:sW}
        \mathfrak{f}_{X_1,\dots,X_N}(\underline{L}_N) 
        \coloneqq
        \int_{  \underline{L}_{N-1} \prec \underline{L}_N } \mathfrak{f}_{X_1,\dots,X_{N-1}}(\underline{L}_{N-1})
        \,
        \mathfrak{f}_{X_N}(\underline{L}_{N-1} ; \underline{L}_N)  \frac{d \underline{L}_{N-1}}{\underline{L}_{N-1}}
    \end{equation}
		for $N\ge2$.
\end{definition}

\begin{example}[Two-variable spin Whittaker function]
	\label{ex:two_var_sW}
	Let us 
	compute the integral \eqref{eq:sW} for $N=2$.
	Denote $\underline{X}_2=(X,Y)$, 
	$\underline{L}_2=(u,u+\alpha)$, where $u\ge 1$, $\alpha>0$. Then
	\begin{align*}
		&
		\mathfrak{f}_{X,Y}(u,u+\alpha)
		\\&
		\hspace{5pt}=
		\frac{(u(u+\alpha))^{-Y}}{\mathrm{B}(S+Y,S-Y)}
		\left( 1-\frac{u}{u+\alpha} \right)^{1-2S}
		\int_{u}^{u+\alpha}
		v^{Y-X-1}
		\left( 1-\frac{v}{u+\alpha} \right)^{S-Y-1}
		\left( 1-\frac{u}{v} \right)^{S+Y-1}
		dv
		\\&
		\hspace{5pt}=
		\frac{u^{-Y}(u+\alpha)^{S}}{\mathrm{B}(S+Y,S-Y)}
		\int_{0}^{1}
		(u+t \alpha)^{-X-S}
		\left( 1-t \right)^{S-Y-1}
		t^{S+Y-1}
		dt,
	\end{align*}
	where we changed the variable as $v=u+\alpha t$, $t\in [0,1]$.
	The integral can now be evaluated using 
	Euler's representation of the Gauss hypergeometric function
	$_2F_1$ \eqref{eq:2F1}. Let us also rename back $z=u+\alpha$. We have 
	\begin{equation}
		\label{eq:sW_two_var}
		\mathfrak{f}_{X,Y}(u,z)
		=
		(z/u)^S 
		u^{-X-Y}
		{}_2F_1 
    \left(\begin{minipage}{2.15cm}
		\center{$S+X \, , \, S+Y$}
		\\
		\center{$2S$}
		\end{minipage} \Big\vert\, 1-\frac{z}{u}\right).
	\end{equation}
	When $|1-z/u| \ge 1$, the hypergeometric function 
	in
	\eqref{eq:sW_two_var}
	should be understood in the sense of analytic continuation.

	We remark that most of the properties of the spin Whittaker functions 
	given below in this section 
	can be directly derived for $N=2$ from known properties of the Gauss 
	hypergeometric function ${}_2F_1$.

\end{example}

\begin{proposition} \label{prop:sW_are_well_posed}
	For $\underline{X}_N=(X_1,\ldots,X_N )$ with $|X_i|<S$, the 
	spin Whittaker function $\mathfrak{f}_{\underline{X}_N}(\underline{L}_N)$ is well-defined and 
	continuous in $\underline{L}_N \in \mathpzc{W}_N$.
\end{proposition}
In particular, we can first define 
$\mathfrak{f}_{\underline{X}_N}(\underline{L}_N)$ for $\underline{L}_N\in \mathring{\mathpzc{W}}_N$,
and then extend to the whole Weyl chamber by continuity.
(Note that $\mathcal{A}_{S,X}(u,v,z)$ \eqref{eq:A} might have a singularity at $u=z$.)
The proof of \Cref{prop:sW_are_well_posed} is based on the next two lemmas.

\begin{lemma} \label{lemma:delta_function}
	Let $\ell_1>0$ and let $f(\cdot )$ be a left continuous function on $\mathbb{R}_{\ge1}$. 
	Then, we have
	\begin{equation} \label{eq:delta_function}
			\lim_{\ell_3 \to \ell_1^-} \int_{\ell_3}^{\ell_1} \frac{d\ell_2}{\ell_2} \mathcal{A}_{S,X} (\ell_3, \ell_2 ,\ell_1) f(\ell_2)
			= f(\ell_1).
	\end{equation}
\end{lemma}
\begin{proof}
To compute the limit set $\ell_3=\ell_1 - \delta$ for a small positive $\delta$. After a change of variable $\ell_2=\ell_1 - \delta (1-\ell_2')$, the integral in \eqref{eq:delta_function} becomes 
\begin{equation*}
    \frac{1}{\mathrm{B}(S+X,S-X)} \int_0^1 d\ell_2' 
		\left( \frac{\ell_1}{\ell_1- \delta (1-\ell_2') } \right)^{S+X}
		(1-\ell_2')^{S-X-1} \ell_2'^{S+X-1} f(\ell_1 - \delta (1-\ell_2')).
\end{equation*}
	Using the left continuity of $f$, we see that the integrand converges 
	to $(1-\ell_2')^{S-X-1} \ell_2'^{S+X-1}f(\ell_1)$
	as $\delta\to0$.
	The limiting integrand integrates to 
	$\mathrm{B}(S+X,S-X)$,
	and so by the Dominated Convergence Theorem
	the lemma follows.
\end{proof}

\begin{lemma} \label{lemma:sW_branching_general}
	Let $f:\mathpzc{W}_{N-1}\to \mathbb{C}$ be left continuous in each of 
	$L_{N-1,i}$.
	Define $F:\mathring{\mathpzc{W}}_N \to \mathbb{C}$ as
	\begin{equation}\label{eq:sW_branching_general}
		F(\underline{L}_N) 
		=
		\int f(\underline{L}_{N-1})
		\,
		\mathfrak{f}_{X}(\underline{L}_{N-1} ; \underline{L}_N) \frac{d \underline{L}_{N-1}}{\underline{L}_{N-1}}.
	\end{equation}
	Then $F$ is continuous and can be extended by continuity to $\mathpzc{W}_N$. 
\end{lemma}

\begin{proof}
For $\underline{L}_N \in \mathring{\mathpzc{W}}_N$, 
the singularities of the integrand in \eqref{eq:sW_branching_general}
come only from the branching function
$\mathfrak{f}_{X}(\underline{L}_{N-1},\underline{L}_N)$ and they are
of the form
\begin{equation*}
    \left( 1 - \frac{L_{N-1,i}}{L_{N,i}} \right)^{S-X-1},
    \qquad
    \text{or}
    \qquad
    \left( 1 - \frac{L_{N,i+1}}{L_{N-1,i}} \right)^{S+X-1}
\end{equation*}
for some $i$. 
Because $|X|<S$ these singularities are summable.
Therefore $F$, is
continuous inside the interior $\mathring{\mathpzc{W}}_N$ of the Weyl chamber. 

To prove that $F$ can be
extended by continuity to $\mathpzc{W}_N$ we
first define, from small positive increments
$\delta_1,\dots,\delta_{N-1}$, the quantities 
$d_i =
\delta_i + \cdots + \delta_{N-1}$ for each $i=1,\dots,N-1$.
We aim to compute the limit
\begin{equation*}
    \lim_{\delta_1,\dots, \delta_{N-1} \to 0} F(L_{N,N}, L_{N,N-1} + d_{N-1}, \dots, L_{N,1}+d_1),
\end{equation*}
when some of the $L_{N,i}$'s 
are equal to each other. 
Before the limit, this function is equal to
\begin{equation*}
    \begin{split}
        &
        \int_{L_{N,N}}^{L_{N,N-1}+\delta_{N-1}} \mathcal{A}_{S,X} (L_{N,N},L_{N-1,N-1},L_{N,N-1}+\delta_{N-1}) \frac{d L_{N-1,N-1} }{L_{N-1,N-1}}
        \\
        &
        \cdots
        \int_{L_{N,2}+d_2 }^{L_{N,1}+d_2+\delta_1} \mathcal{A}_{S,X} (L_{N,2}+d_2,L_{N-1,1},L_{N,1}+d_2+\delta_1) \frac{d L_{N-1,1} }{L_{N-1,1}} f(\underline{L}_{N-1}) \left( \frac{\prod_{i=1}^{N-1} L_{N-1,i}}{\prod_{i=1}^N (L_{N,i}+d_i)} \right)^X.
    \end{split}
\end{equation*}
For any $i$ such that $L_{N,i}=L_{N,i+1}$, make the change of
variables $L_{N-1,i}=L_{N,i}+d_{i+1} - \delta_i(1-\ell_{N-1,i})$.
As in the proof of 
\Cref{lemma:delta_function},
this removes all the corresponding singularities.
Therefore, the limit as
$\delta_1,\dots,\delta_{N-1}\to 0$ exists, is finite, and can
be computed using \eqref{eq:delta_function}.
\end{proof}

\begin{proof}[Proof of Proposition \ref{prop:sW_are_well_posed}]
For $N=1$ the spin Whittaker function \eqref{eq:sW_N=1} is clearly
continuous. Therefore, by 
\Cref{lemma:sW_branching_general},
$\mathfrak{f}_{X_1,X_2}(\underline{L}_2)$ is well defined and 
continuous on $\mathpzc{W}_2$. Proceeding by induction on $N$,
we get the result of 
\Cref{prop:sW_are_well_posed}.
\end{proof}

The next corollary gives a Givental type
representation of the spin Whittaker functions, obtained by writing down
explicitly the recursive definition \eqref{eq:sW}.
\begin{corollary}
	\label{cor:sW_Givental}
	We have
	\begin{equation} \label{eq:sW_Givental}
		\mathfrak{f}_{X_1,\dots,X_N}(\underline{L}_N)
		=
		\int
		\prod_{1 \le k \le N} \frac{ \prod_{i=1}^{k-1}L_{k-1,i}^{X_k} }{ \prod_{i=1}^{k}L_{k,i}^{X_k} }
		\prod_{1\le i \le k \le N-1}
		\mathcal{A}_{S,X_{k+1}}(L_{k+1,i+1},L_{k,i},L_{k+1,i}) 
		\frac{d \doubleunderline{L}_{N-1}}{\doubleunderline{L}_{N-1}}.
	\end{equation}
\end{corollary}
\begin{proof}
Because the sequence of integrations as in \eqref{eq:sW} 
leading to $\mathfrak{f}_{\underline{X}_N}(\underline{L}_N)$
is (absolutely) convergent, 
so is the integration over the Gelfand-Tsetlin array 
$\mathpzc{GT}_{\hspace{-0.5ex}N-1}$. 
The two integration procedures give the same result by the Fubini--Tonelli theorem.
\end{proof}

\subsection{Dual Spin Whittaker functions} 
\label{sub:dual_sw}

In this section we define a 
dual family of functions.
Given interlacing sequences
$\underline{\widetilde{L}}_{k}\prec\underline{L}_{k}$ of the same length $k$, introduce the
\emph{dual spin Whittaker branching functions}
    \begin{equation} \label{eq:dual_branch_sW_1}
			\begin{split}
        \mathfrak{g}_Y(\underline{\widetilde{L}}_k ; \underline{L}_k)
        &\coloneqq
				\mathbf{1}_{\underline{\widetilde L}_{k} \prec \underline{L}_k}\,
        \frac{1}{\Gamma(S-Y)}
        \left(
        \frac{ \widetilde{L}_{k,k} \cdots \widetilde{L}_{k,1} }{ L_{k,k} \cdots L_{k,1} }
        \right)^Y
        \left( 1-\frac{\widetilde{L}_{k,1}}{L_{k,1}} \right)^{S-Y-1} 
        \\
				&
				\hspace{180pt}
				\times\prod_{i=2}^{k} 
        \mathcal{A}_{S,-Y}(\widetilde{L}_{k,i},L_{k,i},\widetilde{L}_{k,i-1})
        .
			\end{split}
    \end{equation}
For pairs of interlacing
sequences $\underline{L}_{k-1} \prec \underline{L}_k$, $k\ge1$, of different
lengths, set
\begin{equation*}
	\mathfrak{g}_Y(\underline{L}_{k-1};\underline{L}_k)
	\coloneqq
	\mathfrak{g}_Y((1,\underline{L}_{k-1});\underline{L}_k).
\end{equation*}
\begin{remark}
	One can also write $\mathfrak{g}_Y$ as
			\begin{equation} \label{eq:dual_branch_sW_2}
					\mathfrak{g}_Y(\underline{\widetilde{L}}_k ; \underline{L}_k) 
					= 
					\frac{L_{k,1}^{-Y}}{\Gamma(S-Y)}
					\left( 1 - \frac{\widetilde{L}_{k,1}}{L_{k,1}} \right)^{S-Y-1} \mathfrak{f}_{-Y}(\underline{\ell}_{k-1}; \underline{\widetilde{L}}_k),
			\end{equation}
	where $\underline{L}_k = (\underline{\ell}_{k-1}, L_{k,1})$. 
\end{remark}
\begin{definition}
	Let $N \le M$ and consider parameters $Y_1,\dots, Y_M$ 
	such that $|Y_i|<S$ for all $i$. 
	The \emph{dual spin Whittaker functions} are defined recursively by
    \begin{equation}
			\label{eq:g_branching_sW}
        \mathfrak{g}_{Y_1,\dots,Y_M}(\underline{L}_N)
        =
        \begin{dcases}
        \int
        \mathfrak{g}_{Y_1,\dots,Y_{M-1}}(\underline{\widetilde{L}}_N)
        \mathfrak{g}_{Y_M}(\underline{\widetilde{L}}_N ; \underline{L}_N)  \frac{d \underline{\widetilde{L}}_{N}}{\underline{\widetilde{L}}_{N}} \qquad & \text{if } N<M,
        \\
        \int
        \mathfrak{g}_{Y_1, \dots, Y_{N-1}}(\underline{\widetilde{L}}_{N-1})
        \mathfrak{g}_{Y_N}(\underline{\widetilde L}_{N-1} ; \underline{L}_N)  \frac{d \underline{\widetilde L}_{N-1}}{\underline{\widetilde L}_{N-1}} \qquad & \text{if } N=M.
        \end{dcases}
    \end{equation}
		In particular, 
		for $M=N=1$ we have
		\begin{equation*}
			\mathfrak{g}_Y(L)=
			\mathfrak{g}_Y(1;L)
			=\frac{L^{-Y}(1-L^{-1})^{S-Y-1}}{\Gamma(S-Y)}.
		\end{equation*}
\end{definition}

The next two propositions explain that 
$\mathfrak{g}_{Y_1,\dots,Y_M}$ are well-defined as elements of the ``dual'' space of
compactly supported continuous functions on the Weyl chamber $\mathpzc{W}_{N}$.
    
\begin{proposition} \label{prop:integral_dual_branch_sW}
Let $f(\underline{L}_N)$ be a compactly supported continuous function
on $\mathpzc{W}_N$. Then the function
\begin{equation} \label{eq:integral_dual_branch_sW}
	\underline{\widetilde{L}}_N\mapsto
    \int \mathfrak{g}_Y (\underline{\widetilde{L}}_N;\underline{L}_N) f(\underline{L}_N) \frac{d \underline{L}_N}{\underline{L}_N},
\end{equation}
is also compactly supported and continuous.
\end{proposition}

\begin{proof}
We evaluate the integral \eqref{eq:integral_dual_branch_sW} using expression \eqref{eq:dual_branch_sW_2} for $\mathfrak{g}_Y$ as
\begin{equation*}
    \int 
		\frac{dL_{N,1}}{L_{N,1}^{1+Y}} \frac{1}{\Gamma(S-Y)} \left( 1 - \frac{\widetilde{L}_{N,1}}{L_{N,1}} \right)^{S-Y-1} 
    \int 
		f(\underline{\ell}_{N-1},L_{N,1})\,
		\frac{d \underline{\ell}_{N-1} }{ \underline{\ell}_{N-1} }
		\,
    \mathfrak{f}_{-Y}( \underline{\ell}_{N-1} ; \underline{\widetilde{L}}_k ).
\end{equation*}
By \Cref{lemma:sW_branching_general},
the integral in the variables $\underline{\ell}_{N-1}$ defines a family
of continuous bounded functions in $\underline{\widetilde{L}}_N$,
depending on
$L_{N,1}$. The (improper)
integral in $L_{N,1}$ is
convergent both at $\widetilde{L}_{N,1}$ and $\infty$ (the latter because $f$
vanishes for $L_{N,1}$ large enough). This proves the claim.
\end{proof}

\begin{proposition}
Let $f(\underline{L}_N)$ be a compactly supported continuous function. Then the integral 
\begin{equation*}
    \int \mathfrak{g}_{Y_1,\dots,Y_M} (\underline{L}_N) f(\underline{L}_N) \frac{d \underline{L}_N}{\underline{L}_N}
\end{equation*}
is absolutely convergent.
\end{proposition}
\begin{proof}
This follows from \Cref{prop:integral_dual_branch_sW} 
applied recursively after expanding $\mathfrak{g}_{Y_1,\dots,Y_M}$ 
using the branching rules \eqref{eq:g_branching_sW}.
\end{proof}

\subsection{\texorpdfstring{Convergence of the sqW functions as $q\to1$}{Convergence of the sqW functions as q -> 1}}
\label{sub:conv_sqW}

Here and in the following
subsection we establish that the 
spin Whittaker functions $\mathfrak{f}_{\underline{X}}(\underline{L}_N)$
and 
$\mathfrak{g}_{\underline{Y}}(\underline{L}_N)$
are scaling limits, as $q\to 1$, 
of the spin $q$-Whittaker functions $\mathbb{F}_{\lambda}(x_1,\ldots, x_N)$
and $\mathbb{F}^*_\mu(y_1,\ldots,y_k )$, respectively.
Recall that they also depend on 
two parameters, $q\in(0,1)$ and $s\in(-1,0)$.

First, this subsection we deal with the non-dual functions.
Let us fix a scaling of all parameters.

\begin{definition}[Scaling] \label{def:scaling}
	We consider the following renormalization of parameters:
	\begin{equation} \label{eq:scaling}
		x_i=q^{X_i},
		\qquad 
		s=-q^{S},
		\qquad
		\lambda^i_j= \floor{\log_q (1/L_{i,j})}.
	\end{equation}
	We will assume throughout that
	\begin{equation*}
		S>0, \qquad |X_i|<S, \qquad \text{and} \qquad 1\le 
		L_{i+1,j+1} \le L_{i,j} \le L_{i+1,j}
	\end{equation*}
	for all $i,j$.
	Therefore, the pre-limit quantities in 
	\eqref{eq:scaling} satisfy $s\in(0,1)$,
	$x_i\in(-s,-s^{-1})$, and
	$0\le \lambda^{i+1}_{j+1}\le \lambda^i_j\le \lambda^{i+1}_{j}$.
\end{definition}

For any triple 
of real numbers
$1 \le \ell_3 \le \ell_2 \le \ell_1$, set $n_i\coloneqq\floor{\log_q(1/\ell_i)}$ 
(so $0\le n_3\le n_2\le n_1$).
\begin{lemma}
	\label{lemma:discrete_sums_Delta_q}
	With the above notation,
	for any function $f:\mathbb{Z} \to \mathbb{R}$ we have
	\begin{equation}
		\label{eq:discrete_sums_Delta_q}
		\sum_{n_2=n_3}^{n_1}  f(n_2) 
		= 
		\int_{\ell_3}^{\ell_1} \frac{1}{\Delta_q(\ell_3,\ell_2,\ell_1)} f(\floor{\log_q (1 /\ell_2)})
		\,
		\frac{d \ell_2}{\ell_2},
	\end{equation}
	where 
	\begin{equation} \label{eq:Delta_q}
			\Delta_q (\ell_3,\ell_2,\ell_1)
			\coloneqq
			\int_{\max(\ell_3,q^{-n_2})}^{\min(\ell_1,q^{-n_2-1})} \frac{d \ell_2'}{\ell_2'}
			=
			\begin{cases}
					- \log q \qquad & \text{if } n_3<n_2<n_1;
					\\
					\log(q^{n_1} \ell_1) & \text{if } n_3<n_2=n_1;
					\\
					- \log (q^{n_3+1} \ell_3)
					\qquad & \text{if } n_3=n_2<n_1;
					\\
					\log(\ell_1/\ell_3) \qquad & \text{if } n_3=n_2=n_1.
			\end{cases}
	\end{equation}
	When $\ell_3=\ell_1$, the integral in \eqref{eq:discrete_sums_Delta_q}
	is understood in the limiting sense.
\end{lemma}
\begin{proof}
	This follows by observing that $\Delta_q$ is the measure of intervals where the function
	$\ell_2 \mapsto \floor{\log_q(1/\ell_2)}$ is constant, and simultaneously
	$\ell_2$ lies in the interval $[\ell_3, \ell_1]$. 
\end{proof}

The rescaled spin $q$-Whittaker functions are defined recursively as
\begin{align*}
		\mathfrak{f}_{X_N}^{(q)}(\underline{L}_{N-1} ; \underline{L}_N) 
		&= 
		\prod_{k=1}^{N-1}\frac{1}{\Delta_q (L_{N,k+1}, L_{N-1,k} , L_{N,k})} 
		\,
		\mathbb{F}_{\lambda^N / \lambda^{N-1}}(x_N)\bigg \vert_{\mathrm{scaling \,}\eqref{eq:scaling}};
    \\
		\mathfrak{f}_{X_1}^{(q)}(L_{1,1}) 
		&=
		x_1^{\lambda_1^1}\bigg \vert_{\mathrm{scaling \,}\eqref{eq:scaling}}
		=
		q^{X_1\floor{\log_q (1/L_{1,1})}} ;
		\\
    \mathfrak{f}_{X_1,\dots, X_N}^{(q)} (\underline{L}_N) &=
    \int
		\mathfrak{f}_{X_1,\dots, X_{N-1}}^{(q)} (\underline{L}_{N-1}) \,
		\mathfrak{f}_{X_N}^{(q)}(\underline{L}_{N-1} ; \underline{L}_N) 
		\frac{d \underline{L}_{N-1}}{\underline{L}_{N-1}},
\end{align*}

The next theorem is the 
main result of this subsection:
\begin{theorem}\label{thm:sqW_to_sW}
    We have
    \begin{equation} \label{eq:sqW_to_sW}
        \lim_{q \to 1} 
        \mathfrak{f}_{X_1,\dots, X_N}^{(q)}
        = 
        \mathfrak{f}_{X_1,\dots, X_N} ,
    \end{equation}
    uniformly on any compact subset of $\mathpzc{W}_N$.
\end{theorem} 

Pointwise convergence
in
\eqref{eq:sqW_to_sW} is a
consequence of a simpler result stated 
in Lemma 2.2 of \cite{CorwinBarraquand2015Beta} (reproduced as \Cref{lemma:BC_lemma_uniform}
in \Cref{app:proof_prop}):
\begin{equation} \label{eq:BC_limit}
    \lim_{q\to 1}
    \frac{(\ell q^A ;q)_\infty }{ (\ell q^B ;q)_\infty
    }
    =
    (1-\ell)^{B-A},
\end{equation}
for any $\ell \in (0,1)$ and $A,B>0$. 

By \eqref{eq:BC_limit} and 
through a repeated use of the identity
\begin{equation} \label{eq:ratio_q_pochhammer}
    \frac{(q^a;q)_n}{(q^b;q)_n} = \mathbf{1}_{n=0} + \mathbf{1}_{n \ge 1} \frac{\Gamma_q(b)}{\Gamma_q(a)}(1-q)^{b-a} \frac{( q^{b+n} ;q)_\infty }{ ( q^{a+n} ;q)_\infty},
\end{equation}
where $\Gamma_q$ is the $q$-Gamma
function \eqref{eq:q_Gamma_q_Beta},
one readily gets the 
pointwise convergence of the branching
function $\mathfrak{f}^{(q)}_{X}( \underline{L}_{N-1}; \underline{L}_N
)$ to $\mathfrak{f}_{X}( \underline{L}_{N-1}; \underline{L}_N )$.
Nevertheless, for the finer uniform convergence result of
\Cref{thm:sqW_to_sW}, a slightly more accurate analysis of
ratios of $q$-Pochhammer symbols appearing in the sqW functions is required. We
postpone this technical discussion to 
\Cref{app:proof_prop}.
Let us summarize the main technical result proven in 
\Cref{app:proof_prop}:

\begin{proposition} \label{prop:skew_SW_uniform_conv}
    Let $f(\underline{L}_{N-1})$ be a continuous function on $\mathpzc{W}_{N-1}$. Then
		for any $\underline{L}_N\in \mathpzc{W}_N$ we have
    \begin{equation} \label{eq:skew_SW_uniform_conv}
        \lim_{q \to 1}
				\int f(\underline{L}_{N-1}) \,\mathfrak{f}^{(q)}_X (\underline{L}_{N-1}; \underline{L}_N ) \frac{d \underline{L}_{N-1}}{\underline{L}_{N-1}}
        =
				\int f(\underline{L}_{N-1}) \,\mathfrak{f}_X (\underline{L}_{N-1}; \underline{L}_N ) \frac{d \underline{L}_{N-1}}{\underline{L}_{N-1}},
    \end{equation}
    and the convergence is uniform on compact subsets of $\mathpzc{W}_{N}$.
\end{proposition}

The continuous
function $f$ in \Cref{prop:skew_SW_uniform_conv} can also be replaced by a uniformly converging sequence:
\begin{corollary} \label{cor:skew_SW_unif_conv}
		Let $f^{(q)}(\underline{L}_{N-1})$ be a sequence uniformly
		convergent as $q\to1$ on compact subsets of
		$\mathpzc{W}_{N-1}$ to a continuous function $f(\underline{L}_{N-1})$. 
		Then
    \begin{equation*}
			\lim_{q \to 1}
			\int f^{(q)}(\underline{L}_{N-1}) \,
			\mathfrak{f}^{(q)}_X
			(\underline{L}_{N-1}; \underline{L}_N ) \,
			\frac{d \underline{L}_{N-1}}{\underline{L}_{N-1}}
			=
			\int f(\underline{L}_{N-1})\,
			\mathfrak{f}_X (\underline{L}_{N-1};
			\underline{L}_N ) \,\frac{d \underline{L}_{N-1}}{\underline{L}_{N-1}}
    \end{equation*}
    and the convergence is uniform on compact subsets of $\mathpzc{W}_{N}$.
\end{corollary}
\begin{proof}
	This follows from 
	\Cref{prop:skew_SW_uniform_conv} and 
	the fact that for fixed $\underline{L}_N\in \mathpzc{W}_N$,
	the functions $\underline{L}_{N-1}\mapsto 
	\mathfrak{f}^{(q)}_X(\underline{L}_{N-1}; \underline{L}_N )$
	and
	$\underline{L}_{N-1}\mapsto 
	\mathfrak{f}_X(\underline{L}_{N-1}; \underline{L}_N )$
	are compactly supported on $\mathpzc{W}_{N-1}$.
\end{proof}

\begin{proof}[Proof of Theorem \ref{thm:sqW_to_sW}]
    For $N=1$ we have
    \begin{equation*}
			\mathfrak{f}_{X_1}^{(q)}(L) = q^{X_1\floor{\log_q (1/L)}} \xrightarrow[q \to 1]{} L^{-X_1} = \mathfrak{f}_{X_1}(L),
    \end{equation*}
    uniformly with respect to $L\ge1$ varying in any compact domain. 
		\Cref{cor:skew_SW_unif_conv} then implies  
		\Cref{thm:sqW_to_sW} by induction on $N$.
\end{proof}

\subsection{\texorpdfstring{Convergence of the dual sqW functions as $q\to1$}{Convergence of the dual sqW functions as q -> 1}}
\label{sub:conv_dual_sqW}

We now establish the convergence of functions
$\mathbb{F}^*$ to the dual spin Whittaker functions $\mathfrak{g}$.
The scaling of
parameters we adopt is that of Definition \ref{def:scaling}. For
consistency with the previous sections, dual functions will depend
on $y$ variables for which the scaling is
\begin{equation} \label{eq:scaling_y_variables}
	y_i=q^{Y_i}, \qquad \qquad |Y_i|<S.
\end{equation}
    
For two interlacing arrays $\underline{\widetilde{L}}_k \prec \underline{L}_k$ define the
rescaled dual spin Whittaker branching
functions
\begin{equation}\label{eq:dual_branch_sW_q_def}
		\mathfrak{g}_{Y_k}^{(q)}(\underline{\widetilde{L}}_{k};\underline{L}_k)
		=
		(1-q)^{S-Y_k}
		\left(
		\prod_{j=1}^{k} \frac{1}{\Delta_q( \widetilde{L}_{k,j}, L_{k,j} , \widetilde{L}_{k,j-1} )} \right)
		\mathbb{F}^*_{\lambda^k / \widetilde{\lambda}^{k}} (y_k)\bigg \vert_{\mathrm{scaling \,}\eqref{eq:scaling},\eqref{eq:scaling_y_variables}},
\end{equation}
where, by agreement, $\widetilde{L}_{k,0}=\infty$,
and $\Delta_q$ is given by \eqref{eq:Delta_q}.
In particular, the rescaled one-variable function is
(assuming $L>1$ and $q$ close enough to $1$)
\begin{equation*}
	\mathfrak{g}_Y^{(q)}(L)
	=
	\mathfrak{g}_Y^{(q)}(1;L)
	=(1-q)^{S-Y}\frac{1}{(-\log q)}\frac{(q^{S-Y};q)_{\floor{\log_q(1/L)}}}{(q;q)_{\floor{\log_q(1/L)}}}
	\,q^{Y\floor{\log_q(1/L)}}.
\end{equation*}
For interlacing arrays of different lengths $\underline{L}_{k-1}\prec
\underline{L}_k$, we set
$\mathfrak{g}^{(q)}_{Y}(\underline{L}_{k-1};\underline{L}_k) =
\mathfrak{g}^{(q)}_{Y}((1,\underline{L}_{k-1});\underline{L}_k)$,
as before.
Define the rescaled dual spin $q$-Whittaker functions recursively as 
\begin{equation*}
	\mathfrak{g}_{Y_1,\dots,Y_M}^{(q)}(\underline{L}_N)
	=
	\begin{dcases}
	\int
	\mathfrak{g}_{Y_1,\dots,Y_{M-1}}^{(q)}(\underline{\widetilde{L}}_N)
	\mathfrak{g}_{Y_M}^{(q)}(\underline{\widetilde{L}}_N ; \underline{L}_N)  \frac{d \underline{\widetilde{L}}_{N}}{\underline{\widetilde{L}}_{N}} \qquad & \text{if } N<M,
	\\
	\int
	\mathfrak{g}_{Y_1, \dots, Y_{N-1}}^{(q)}(\underline{\widetilde{L}}_{N-1})
	\mathfrak{g}_{Y_N}^{(q)}(\underline{\widetilde{L}}_{N-1} ; \underline{L}_N)  \frac{d \underline{\widetilde{L}}_{N-1}}{\underline{\widetilde{L}}_{N-1}} \qquad & \text{if } N=M.
		\end{dcases}
\end{equation*}

The next result establishes a weak convergence of rescaled branching functions $\mathfrak{g}^{(q)}$.
\begin{theorem} 
	Let $f(\underline{L}_N)$ be a compactly supported continuous function on $\mathpzc{W}_N$. Then
\begin{equation} \label{eq:weak_unif_conv_sW}
		\lim_{q \to 1} 
		\int \mathfrak{g}^{(q)}_Y (\underline{\widetilde{L}}_N ; \underline{L}_N) f (\underline{L}_N)\,
		\frac{d \underline{L}_N}{\underline{L}_N}
		=
		\int \mathfrak{g}_Y (\underline{\widetilde{L}}_N ; \underline{L}_N) f (\underline{L}_N)\,
		\frac{d \underline{L}_N}{\underline{L}_N},
\end{equation}
and the convergence is uniform with respect to $\underline{\widetilde{L}}_N$.
\end{theorem}
\begin{proof}
We start by rewriting the branching 
function $\mathfrak{g}^{(q)}_Y (\underline{\widetilde{L}}_N ; \underline{L}_N)$ as
(this follows from straightforward algebraic manipulations with \eqref{eq:dual_branch_sW_q_def})
\begin{equation*}
		q^{Y \lambda^k_1} \frac{(q^{S-Y};q)_{\lambda^k_1 - \widetilde{\lambda}^k_1 }}
		{ (q;q)_{\lambda^k_1 - \widetilde{\lambda}^k_1 } }
		\frac{(1-q)^{S-Y}}{\Delta_q(\widetilde{L}_{k,1} , L_{k,1}, \infty )}\,
		\mathfrak{f}_{-Y}^{(q)}(\underline{\ell}_{k-1}; \underline{\widetilde{L}}_k ).
\end{equation*}
The integral in the left-hand side of
\eqref{eq:weak_unif_conv_sW} becomes
\begin{equation} \label{eq:weak_unif_conv_2}
		\int \frac{dL_{k,1}}{L_{k,1}^{1}}
		\left(
		q^{Y \lambda^k_1} \frac{(q^{S-Y};q)_{\lambda^k_1 - \widetilde{\lambda}^k_1 }}{ (q;q)_{\lambda^k_1 - \widetilde{\lambda}^k_1 } }
		\frac{(1-q)^{S-Y}}{\Delta_q(\widetilde{L}_{k,1} , L_{k,1}, \infty )}
		\right)
		\int \frac{d \underline{\ell}_{k-1} }{ \underline{\ell}_{k-1} }\,
		\mathfrak{f}^{(q)}_{-Y}( \underline{\ell}_{k-1} ; \underline{\widetilde{L}}_k ) f(\underline{\ell}_{k-1},L_{k,1}).
\end{equation}
The inner integral involving the function $\mathfrak{f}^{(q)}_{-Y}$ is uniformly (with respect to $\underline{\widetilde{L}}_k$) convergent to 
\begin{equation*}
\int \frac{d \underline{\ell}_{k-1} }{ \underline{\ell}_{k-1} }
\,\mathfrak{f}_{-Y}( \underline{\ell}_{k-1} ; \underline{\widetilde{L}}_k ) f(\underline{\ell}_{k-1},L_{k,1})
\end{equation*}
by virtue of Proposition \ref{prop:skew_SW_uniform_conv}. On the other hand, the term inside the parentheses in \eqref{eq:weak_unif_conv_2} is uniformly convergent to 
\begin{equation*}
	\frac{L_{k,1}^{-Y}}{\Gamma(S-Y)} \left( 1 - \widetilde{L}_{k,1}/L_{k,1} \right)^{S-Y-1},
\end{equation*}
when $L_{k,1}$ is kept away from $\widetilde{L}_{k,1}$.
Moreover, 
the term inside the parentheses
is absolutely bounded by 
$\mathrm{const} \times ( 1 -
\widetilde{L}_{k,1} /L_{k,1} )^{S-Y-1}$ when $L_{k,1}$ approaches
$\widetilde{L}_{k,1}$,
thanks to 
\Cref{lemma:bound_ratio_pochhammer}. 
Since the resulting term after the $q\to1$ limit coincides with the expression
\eqref{eq:dual_branch_sW_2} for the dual branching function
$\mathfrak{g}_Y (\underline{\widetilde{L}}_N ; \underline{L}_N)$,
we are done.
\end{proof}
    
Similarly to \Cref{cor:skew_SW_unif_conv},
we can let the test function $f$ depend on $q$:
    
\begin{corollary} \label{cor:weak_unif_conv_sW}
	Let $f^{(q)}(\underline{L}_N)$ converge,
	as $q\to1$,
	to a compactly supported continuous function $f(\underline{L}_N)$,
	uniformly on $\mathpzc{W}_N$.
Then
\begin{equation*}
		\lim_{q \to 1} 
		\int \mathfrak{g}^{(q)}_Y (\underline{\widetilde{L}}_N ; \underline{L}_N) f^{(q)} (\underline{L}_N) \frac{d \underline{L}_N}{\underline{L}_N}
		=
		\int \mathfrak{g}_Y (\underline{\widetilde{L}}_N ; \underline{L}_N) f (\underline{L}_N) \frac{d \underline{L}_N}{\underline{L}_N}
\end{equation*}
and the convergence is uniform with respect to $\underline{\widetilde{L}}_N$.
\end{corollary}
    
\subsection{Properties of the spin Whittaker functions}
\label{sub:sW_properties}

In this subsection we describe the properties
of the spin Whittaker functions
which follow in the $q\to1$ limit
from the corresponding properties
of the spin $q$-Whittaker functions. 

\begin{proposition}[Symmetry and shifting]
	\label{prop:symmetry_and_shifting}
	The spin Whittaker function 
	$\mathfrak{f}_{X_1,\ldots,X_N }(\underline{L}_N)$
	is symmetric in the $X_i$'s for all $\underline{L}_N\in \mathpzc{W}_N$.
	They also satisfy the shifting property:
	\begin{equation*}
		\mathfrak{f}_{X_1,\ldots,X_N }(a \underline{L}_N)=
		a^{-X_1-\ldots -X_N}
		\mathfrak{f}_{X_1,\ldots,X_N }(\underline{L}_N),\qquad 
		a>1.
	\end{equation*}
\end{proposition}
\begin{proof}
	The symmetry follows from the 
	corresponding symmetry of the sqW  polynomial
	$\mathbb{F}_\lambda(x_1,\ldots,x_N )$,
	which ultimately is a consequence of the Yang--Baxter equation.
	The shifting 
	property can either be deduced from
	\Cref{prop:sqW_shifting},
	or obtained in a similar way by noting that the 
	branching spin Whittaker functions themselves satisfy
	$\mathfrak{f}_X(a\underline{L}_{k};a\underline{L}_k)=a^{-X}\mathfrak{f}_X(\underline{L}_{k};\underline{L}_k)$.
\end{proof}

We now turn to Cauchy type identities for the spin Whittaker functions.

\begin{theorem}[Skew Cauchy type identity] 
	\label{thm:skew_Cauchy_sW}
	Assume $|X|,|Y|<S$ and $X+Y>0$.
	Then, for any $\underline{L}_{N-1},\underline{\widetilde{L}}_N$ we have
	\begin{multline} \label{eq:skew_Cauchy_sW}
		\int \mathfrak{f}_X (\underline{L}_{N-1};\underline{L}_N)
		\mathfrak{g}_Y (\underline{\widetilde{L}}_N;\underline{L}_N)
		\frac{d \underline{L}_N}{\underline{L}_N}
		\\=
		\frac{\Gamma(X+Y) \Gamma(2 S)}{\Gamma(S+X) \Gamma(S+Y)} 
		\int \mathfrak{f}_X (\underline{\widetilde{L}}_{N-1} ; \underline{\widetilde{L}}_N)
		\mathfrak{g}_Y(\underline{\widetilde{L}}_{N-1} ; \underline{L}_{N-1}) 
		\,
		\frac{d \underline{\widetilde{L}}_{N-1} }{\underline{\widetilde{L}}_{N-1}}
	\end{multline}
	and, when $N=1$ we have
	\begin{equation} \label{eq:skew_Cauchy_sW_N=1}
			\int \mathfrak{f}_X (L_{1,1}) \mathfrak{g}_Y (\widetilde{L}_{1,1};L_{1,1}) \frac{d L_{1,1}}{L_{1,1}}
			=
			\frac{\Gamma(X+Y)}{\Gamma(S+X)} \mathfrak{f}_X(\widetilde{L}_{1,1}). 
	\end{equation}
\end{theorem}
    
\begin{proof}
We first observe that \eqref{eq:skew_Cauchy_sW_N=1} is equivalent to
the integral representation of $\mathrm{B}(S-Y,X+Y)$. 

In order to prove the general case
\eqref{eq:skew_Cauchy_sW} we use
\Cref{cor:skew_SW_unif_conv,cor:weak_unif_conv_sW}. 
Take a compactly supported continuous test function
$\phi(\underline{L}_{N-1})$, and set
\begin{equation*}
		\Phi_{\mathfrak{f}}(\underline{L}_N)
		\coloneqq
		\int \phi(\underline{L}_{N-1}) \mathfrak{f}_X(\underline{L}_{N-1};\underline{L}_N) \frac{d \underline{L}_{N-1}}{\underline{L}_{N-1}},
		\quad
		\Phi_{\mathfrak{g}}(\underline{\widetilde{L}}_{N-1})
		\coloneqq
		\int
		\mathfrak{g}_Y(\underline{\widetilde{L}}_{N-1};\underline{L}_{N-1})
		\phi(\underline{L}_{N-1})  \frac{d \underline{L}_{N-1}}{\underline{L}_{N-1}}.
\end{equation*}
Analogously define $\Phi_{\mathfrak{f}}^{(q)}$ and
$\Phi_{\mathfrak{g}}^{(q)}$ by substituting respectively
$\mathfrak{f}_X$ and $\mathfrak{g}_Y$ with $\mathfrak{f}^{(q)}_X$ and
$\mathfrak{g}^{(q)}_Y$ in the above formulas.
It follows from the skew Cauchy Identity for sqW functions
(\Cref{prop:skew_Cauchy_ID_sqW_sqW}) that
\begin{equation} \label{eq:weak_Cauchy_id}
		\int \Phi_{\mathfrak{f}}^{(q)}(\underline{L}_N) \mathfrak{g}^{(q)}_Y (\underline{\widetilde{L}}_N;\underline{L}_N)
		\frac{d \underline{L}_N}{\underline{L}_N}
		=
		\frac{\Gamma_q(X+Y) \Gamma_q(2 S)}{\Gamma_q(S+X) \Gamma_q(S+Y)} 
		\int \mathfrak{f}^{(q)}_X (\underline{\widetilde{L}}_{N-1} ; \underline{\widetilde{L}}_N)
		\Phi_{\mathfrak{g}}^{(q)}(\underline{\widetilde{L}}_{N-1})
		\frac{d \underline{\widetilde{L}}_{N-1} }{\underline{\widetilde{L}}_{N-1}}.
\end{equation}
By \Cref{cor:weak_unif_conv_sW} we have
$\Phi_{\mathfrak{g}}^{(q)} \to \Phi_{\mathfrak{g}}$ uniformly, and
further $\Phi_{\mathfrak{g}}$ is compactly supported and continuous by
\Cref{prop:integral_dual_branch_sW}. 
This implies, by
\Cref{cor:skew_SW_unif_conv}, that the right-hand side of
\eqref{eq:weak_Cauchy_id} converges to
\begin{equation*}
		\frac{\Gamma(X+Y) \Gamma(2 S)}{\Gamma(S+X) \Gamma(S+Y)} 
		\int \mathfrak{f}_X (\underline{\widetilde{L}}_{N-1} ; \underline{\widetilde{L}}_N)
		\Phi_{\mathfrak{g}}(\underline{\widetilde{L}}_{N-1})
		\frac{d \underline{\widetilde{L}}_{N-1} }{\underline{\widetilde{L}}_{N-1}}.
\end{equation*}
The integral in the left-hand side of \eqref{eq:weak_Cauchy_id} is
absolutely convergent when $X+Y>0$. Since
$\Phi_{\mathfrak{f}}^{(q)} \to \Phi_{\mathfrak{f}}$ uniformly by
\Cref{prop:skew_SW_uniform_conv}, 
\Cref{cor:weak_unif_conv_sW} implies that the left-hand side of \eqref{eq:weak_Cauchy_id} 
converges to
\begin{equation*}
		\int \Phi_{\mathfrak{f}}(\underline{L}_N) \mathfrak{g}_Y (\underline{\widetilde{L}}_N;\underline{L}_N)
		\frac{d \underline{L}_N}{\underline{L}_N}.
\end{equation*}
Since the function $\phi$ was arbitrary, equality \eqref{eq:skew_Cauchy_sW} follows.
\end{proof}
    
\begin{corollary}[Full Cauchy type identity]
	\label{cor:full_Cauchy}
	Let $N \le M$ and $|X_i|,|Y_j|<S$,
	$X_i+Y_j>0$ for all $i,j$. We have
		\begin{equation} \label{eq:Cauchy_id_sW}
				\int
				\mathfrak{f}_{X_1,\dots,X_N}(\underline{L}_N)\,
				\mathfrak{g}_{Y_1,\dots,Y_M}(\underline{L}_N)\,
				\frac{d \underline{L}_N}{\underline{L}_N}
				=
				\prod_{j=1}^M \frac{\Gamma(X_1+Y_j)}{\Gamma(S+X_1)} 
				\Bigg(\prod_{i=2}^N \frac{\Gamma(X_i+Y_j) \Gamma(2S)}{\Gamma(S+X_i) \Gamma(S+Y_j)}\Bigg)
				.
		\end{equation}
\end{corollary}
\begin{proof}
	Immediately follows from
	\Cref{thm:skew_Cauchy_sW} and 
	the branching rules for the functions $\mathfrak{f},\mathfrak{g}$.
\end{proof}
    
We also have an identity involving a single spin Whittaker function:
\begin{proposition}
		Let $|X_i|<S$. Then we have
		\begin{equation*}
				\int_{L_{N,N}=1} \mathfrak{f}_{X_1,\dots,X_N} ( \underline{L}_N ) \prod_{j=1}^{N-1} \left( 1- \frac{ L_{N,j+1} }{ L_{N,j} } \right)^{2S-1} \frac{ d L_{N,j} }{ L_{N,j}^{1+S} } = \frac{ \Gamma(S+X_1) \cdots \Gamma(S+X_N) }{ \Gamma( SN + X_1 +\cdots +X_N) }.
		\end{equation*}
\end{proposition}
\begin{proof}
		This is a scaling limit of \Cref{prop:quasi_Cauchy_identity}.
\end{proof}
    
We now consider the scaling limits of eigenrelations 
for the sqW functions
stated in 
\Cref{thm:eigenrelation_sqW,thm:eigenrelation_sqW_conj}. 
This produces two operators acting in the $X_i$ variables
which are diagonal in the spin Whittaker functions.
For the next definition we use the
shift operator
\begin{equation} \label{eq:shift}
		\mathcal{T}_{X} f(X) \coloneqq f(X + 1).
\end{equation}
\begin{definition}
	\label{def:sW_eigenoperators}
		For any $N \ge 1$ set
		\begin{equation*}
				\mathscr{D}_1 \coloneqq \sum_{i=1}^N \prod_{\substack{j=1\\j\neq i}}^N \frac{X_i + S}{X_i - X_j} 
				\,
				\mathcal{T}_{X_i}
				,
				\qquad
				\overline{\mathscr{D}}_1 \coloneqq \sum_{i=1}^N \prod_{\substack{j=1\\j\neq i}}^N \frac{X_i - S}{X_i - X_j} 
				\,
				\mathcal{T}_{X_i}^{-1}.
		\end{equation*}
\end{definition}

The next proposition represents a partial generalization of eigenrelations satisfied by Whittaker functions
(e.g., see \cite{Kharchev_2001}). 
\begin{proposition}[Eigenrelations for spin Whittaker functions]
	\label{prop:eigenrelation_sW}
		We have
		\begin{gather*}
			\mathscr{D}_1 \mathfrak{f}_{X_1,\dots,X_N} (\underline{L}_N) = L_{N,N}^{-1}\, \mathfrak{f}_{X_1,\dots,X_N} (\underline{L}_N),
				\\
				\overline{\mathscr{D}}_1 \mathfrak{f}_{X_1,\dots,X_N} (\underline{L}_N) = L_{N,1}\, \mathfrak{f}_{X_1,\dots,X_N} (\underline{L}_N).
		\end{gather*}
\end{proposition}
\begin{proof}
	We easily see that operators
	$\mathscr{D}_1,\overline{\mathscr{D}}_1$ are limiting forms of
	$\mathfrak{D}_1,\overline{\mathfrak{D}}_1$ 
	(\Cref{def:sqW_operators})
	under the scaling
	\eqref{eq:scaling}. At the same time we have $q^{\lambda_N} \to
	L_{N,N}^{-1}$ and $q^{-\lambda_1} \to L_{N,1}$ under the same
	scaling. Therefore, \eqref{eq:eigenrelation_sqW},
	\eqref{eq:eigenrelation_sqW_lambda_1} and convergence
	\eqref{eq:sqW_to_sW} imply the claimed eigenrelation.
\end{proof}

\subsection{Formal reduction to the usual Whittaker functions} \label{sub:sW_to_W}
    
Just like the sqW polynomials reduce to the $q$-Whittaker
polynomials setting $s=0$, it should be possible to prove that, under the
correct scaling, our spin Whittaker functions converge 
to the Whittaker functions. 
An evidence for
this is suggested by the following computation. 

Set
\begin{equation} \label{eq:scaling:sW_to_W}
L_{k,i} = S^{k + 1 - 2 i} e^{u_{k,i}}, \qquad  X_k= -\mathrm{i}\lambda_k,
\end{equation}
then, in the limit $S\to \infty$ we have
\begin{equation} \label{eq:limit_sW_W}
\renewcommand{\arraystretch}{2}
		\begin{tabular}{rcl}
		$\displaystyle \left( \frac{ L_{k,k} \cdots L_{k,1} }{ L_{k+1,k+1} \cdots L_{k+1,1} } \right)^{X_{k+1}}$ 
		& 
		$\xrightarrow{\hspace*{1.5cm}}$
		& 
		$\displaystyle \exp \left\{ \mathrm{i} \lambda_k \left( \sum_{i=1}^{k+1} u_{k+1,i} -\sum_{i=1}^k u_{k,i} \right) \right\};$
		\\[1.5ex]
		$\displaystyle \left( 1-\frac{L_{k,i}}{L_{k+1,i}} \right)^{S-X_{k+1}-1}$ 
		&
		$\xrightarrow{\hspace*{1.5cm}}$ 
		&
		$\displaystyle \exp \left\{ -e^{u_{k,i} - u_{k+1,i}}  \right\};$ 
		\\[1.5ex]
		$\displaystyle \left( 1-\frac{ L_{k+1,i+1} }{ L_{k,i} } \right)^{S+X_{k+1}-1}$ 
		&
		$\xrightarrow{\hspace*{1.5cm}}$ 
		&
		$\displaystyle \exp \left\{ -e^{u_{k+1,i+1} - u_{k,i}}  \right\};$
		\\[1.5ex]
		$\displaystyle \left( 1-\frac{ L_{k+1,i+1} }{ L_{k+1,i} } \right)^{1-2S}$
		&
		$\xrightarrow{\hspace*{1.5cm}}$
		&    
		1;
		\\[1.5ex]
		$\displaystyle 
		4^S S^{\frac12}\,
		\mathrm{B}(S+X,S-X)$
		&
		$\xrightarrow{\hspace*{1.5cm}}$
		&    
		$\displaystyle  2\sqrt{ \pi }$.
\end{tabular}
\end{equation}
All the limits in \eqref{eq:limit_sW_W} are straightforward
(note that the last one requires the Stirling approximation).
Thus, the
branching function $\mathfrak{f}$, rescaled by a factor depending
solely on $S$, converges locally uniformly to 
the Baxter $Q$-operator
$Q^{N \to N-1}$ 
\eqref{eq:Baxter_Q} for the usual Whittaker functions:
\begin{equation*}
	\left( \frac{ 4 \pi }{S 16^{S}} \right)^{\frac{N-1}{2}}
	\mathfrak{f}_{X_N}(\underline{L}_{N-1}; \underline{L}_N) 
	\xrightarrow[S\to \infty]{\mathrm{scaling \,}\eqref{eq:scaling:sW_to_W}}
		Q_{\lambda_N}^{N \to N-1}(\underline{u}_N,\underline{u}_{N-1}). 
\end{equation*}
These computations
suggest that the same type of convergence 
should hold for the full functions.
Namely, under \eqref{eq:scaling:sW_to_W} and as 
$S\to+\infty$,
the spin Whittaker functions $\mathfrak{f}_{\underline{X}_N}(\underline{L}_N)$ 
rescaled by $(4S^{-1}\pi/16^{S})^{\frac{N(N-1)}{4}}$ 
should converge 
to the usual
Whittaker functions $\psi_{\underline{\lambda}_N}(\underline{u}_N)$. 
A proof of this convergence
would require a finer analysis to justify the exchange of the $S\to+\infty$
limit and integration,
and goes beyond the scope of this
paper.

\section{Spin Whittaker Processes and beta polymers} \label{sec:sW_processes}

In this section define \emph{spin Whittaker processes}, and establish
their connection with two beta polymer type models 
introduced in \cite{CorwinBarraquand2015Beta} and
\cite{CMP_qHahn_Push}, respectively.

\subsection{Spin Whittaker processes}

The definition of spin Whittaker processes is 
a straightforward analogy of the discrete level
$\mathfrak{F}/\mathfrak{G}$ processes (\Cref{def:F_G_process}).
The key role is played by the Cauchy type identities (established for spin Whittaker functions
in \Cref{sub:sW_properties}).

\begin{definition} \label{def:sW_processes}
		Set $\mathbf{X}=(X_1,\dots,X_N)$ and $\mathbf{Y} =
		(Y_1,\dots,Y_T)$, with $|X_i|,|Y_j|<S$ and $X_i+Y_j>0$ for all
		$i,j$. The (ascending) \emph{spin Whittaker process} is the probability
		measure on interlacing sequences
		$\doubleunderline{L}_N(T) = (L_{k,i}(T))_{1\le i \le k \le N}$
		(that is, on the Gelfand-Tsetlin cone $\mathpzc{GT}_N$ \eqref{eq:cont_GT_cone})
		with the following density with respect 
		to the measure 
		$\frac{d
		\doubleunderline{L}_N}{\doubleunderline{L}_N} = \prod_{1\le i \le j
		\le N} \frac{d L_{j,i}}{L_{j,i}}$:
		\begin{equation}
			\label{eq:ascending_sW}
				\mathfrak{P}_{\mathbf{X};\mathbf{Y}}(\doubleunderline{L}_N) =
				\frac{ \mathfrak{f}_{X_1}(\underline{L}_1)
				\mathfrak{f}_{X_2}(\underline{L}_1;\underline{L}_2) \cdots
				\mathfrak{f}_{X_N}(\underline{L}_{N-1};\underline{L}_N)
			\mathfrak{g}_{\mathbf{Y}}(\underline{L}_N)}{\Pi(\mathbf{X};\mathbf{Y})}.
		\end{equation}
		The
		normalizing constant in
		\eqref{eq:ascending_sW}
		follows from the Cauchy identity
		of \Cref{cor:full_Cauchy}: 
		\begin{equation*}
				\Pi(\mathbf{X};\mathbf{Y})
				=
				\prod_{j=1}^T 
				\frac{\Gamma (X_1+Y_j)}{\Gamma (S+X_1)} 
				\Biggl(
					\prod_{i=2}^N \frac{\Gamma (X_i+Y_j) \Gamma (2S)}
					{\Gamma (S+X_i) \Gamma (S+Y_j)}
				\Biggr).
		\end{equation*}
\end{definition}

For the next result we denote the ascending sqW/sqW
process, subject to the rescaling \eqref{eq:scaling}, \eqref{eq:scaling_y_variables},
by
\begin{equation*}
		\mathfrak{P}^{(q)}_{\mathbf{X};\mathbf{Y}}(\doubleunderline{L}_N) = \frac{ \mathfrak{f}^{(q)}_{X_1}(\underline{L}_1) \mathfrak{f}^{(q)}_{X_2}(\underline{L}_1;\underline{L}_2) \cdots \mathfrak{f}^{(q)}_{X_N}(\underline{L}_{N-1};\underline{L}_N) \mathfrak{g}^{(q)}_{\mathbf{Y}}(\underline{L}_N) }{\Pi^{(q)}(\mathbf{X};\mathbf{Y})},
\end{equation*}
where the normalization constant is
(cf. \eqref{eq:weak_Cauchy_id})
\begin{equation*}
		\Pi^{(q)}(\mathbf{X};\mathbf{Y})
		=
		\prod_{j=1}^T \frac{\Gamma_q(X_1+Y_j)}{\Gamma_q(S+X_1)} 
		\Biggl(
			\prod_{i=2}^N \frac{\Gamma_q(X_i+Y_j) \Gamma_q(2S)}{\Gamma_q(S+X_i) \Gamma_q(S+Y_j)}
		\Biggr).
\end{equation*}

\begin{theorem} \label{thm:sqW_process_to_sW_process}
	Under the scaling \eqref{eq:scaling}, \eqref{eq:scaling_y_variables}, 
	the ascending sqW/sqW process converges weakly to the spin Whittaker process
	\begin{equation}
			\mathfrak{P}_{\mathbf{X};\mathbf{Y}}^{(q)} \xrightarrow[q \to 1]{} \mathfrak{P}_{\mathbf{X};\mathbf{Y}}.
	\end{equation}
\end{theorem}
\begin{proof}
For any continuous bounded test function $\phi(\doubleunderline{L}_N)$ 
on $\mathpzc{GT}_N$
we have
\begin{equation}
		\mathbb{E}^{(q)}(\phi)
		=
		\frac{1}{\Pi^{(q)}(\mathbf{X};\mathbf{Y})}
		\int \frac{d \underline{L}_N}{\underline{L}_N} \mathfrak{g}^{(q)}_{\mathbf{Y}}(\underline{L}_N)
		\int
		\frac{ d \doubleunderline{L}_{N-1} }{ \doubleunderline{L}_{N-1} }
		\mathfrak{f}^{(q)}_{X_1}(\underline{L}_1) \mathfrak{f}^{(q)}_{X_2}(\underline{L}_1;\underline{L}_2) \cdots \mathfrak{f}^{(q)}_{X_N}(\underline{L}_{N-1};\underline{L}_N)
		\phi(\doubleunderline{L}_N).
\end{equation}
The integral is absolutely convergent and as a consequence of 
\Cref{cor:skew_SW_unif_conv,cor:weak_unif_conv_sW} 
it converges to the average $\mathbb{E}(\phi)$
with respect to the spin Whittaker process.
\end{proof}

\begin{remark} \label{rem:sW_process_to_W_process}
	Developing the argument sketched in \Cref{sub:sW_to_W} 
	it should be possible to 
	show that the spin Whittaker process
	of \Cref{def:sW_processes}
	converges 
	to the $\alpha$-Whittaker process from
	\cite{COSZ2011}, \cite{BorodinCorwin2011Macdonald}. 
	In this case the correct way to rescale the random
	variables $L_{k,i}(T)$ is
	\begin{equation} \label{eq:scaling_sW_process_to_sW_process}
			L_{k,i}(T) = S^{T+k+1-2i} e^{u_{k,i}(T)}.
	\end{equation}
	In the limit $S\to \infty$ the
	process $\{u_{k,i}(T)\colon 1 \le i \le k \le N\}$ should be then described by the density
	\begin{equation} \label{eq:W_process}
			\frac{ Q^{1 \to 0}_{\mathrm{i} X_1}(\underline{u}_1,0) Q^{2 \to 1}_{\mathrm{i} X_2}(\underline{u}_2,\underline{u}_1) \cdots Q^{N \to N-1}_{\mathrm{i} X_N}(\underline{u}_{N},\underline{u}_{N-1}) \theta_{\mathbf{Y}}(\underline{u}_N) }{ \prod_{i=1}^N \prod_{j=1}^T \Gamma(X_i + Y_j) },
	\end{equation}
	where $Q^{k+1 \to k}$
	and 
	$\theta_{\mathbf{Y}}$
	are given in 
	\eqref{eq:Baxter_Q} and \eqref{eq:dual_usual_Whittaker_function}, respectively.
\end{remark}

\subsection{Strict-weak beta polymer model}

We will now recall the strict-weak beta polymer
introduced in \cite{CorwinBarraquand2015Beta}.

\begin{definition}
	\label{def:strict_weak_beta}
	Let $B_{i,j} \sim \mathpzc{B}(X_i+Y_j, S-Y_j)$ be a family of
	independent beta random variables. The
	\emph{strict-weak beta polymer model} partition function $Z(i,j)$, $i\ge1$, $j\ge0$, is
	the random function satisfying the recurrence
	\begin{equation*}
		\begin{dcases}
			Z(i,j) = Z(i,j-1) B_{i,j} + Z(i-1,j-1) (1-B_{i,j}) \qquad &\text{for } 1 < i \le j;
			\\
			Z(1,j)= Z(1,j-1)B_{1,j} \qquad &\text{for } j>0;
			\\
			Z(i,0)= 1 & \text{for $i>0$}.
		\end{dcases}
	\end{equation*}
\end{definition}
Note that all the partition functions $Z(i,j)$
belong to $(0,1]$.
In particular, the probability distribution 
of the strict-weak beta polymer
is completely determined by the joint moments.

\begin{proposition} \label{prop:q_Hahn_to_Beta_polymer}
Recall the $q$-Hahn vertex model height function
$\mathcal{H}_{q\textnormal{-Hahn}}^{\mathrm{ur}}$
(\Cref{sub:sqw_sqw_last}). 
Define
$Z^{(q)}(i,j)=q^{\mathcal{H}_{q\textnormal{-Hahn}}^{\mathrm{ur}}
(i,j)}$. Then, under the scaling \eqref{eq:scaling}, $Z^{(q)}$
converges weakly to the strict-weak beta polymer partition function:
\begin{equation*}
		Z^{(q)}(i,j) \xrightarrow[q \to 1]{} Z(i,j). 
\end{equation*}
\end{proposition}
\begin{proof}
	This result is equivalent to Proposition 2.1 of
	\cite{CorwinBarraquand2015Beta} in the homogeneous case $X_i=X$,
	$Y_j=Y$ for all $i,j$. One can easily check that the proof given
	there also works in our inhomogeneous setting.
\end{proof}
    
\begin{theorem}
	\label{thm:strict_weak_beta_polymer}
	The marginal process $\{L_{k,k}(T)^{-1}\colon k=1,\dots N,\,T \ge N\}$ of the spin
	Whittaker process $\mathfrak{P}_{\mathbf{X},\mathbf{Y}}$ is equivalent 
	in distribution to
	the strict-weak beta polymer partition functions model
	$\{Z(k,T)\colon k=1,\dots, N,\,T \ge N\}$. 
\end{theorem}

Note that since $L_{k,k}(T)\in \mathbb{R}_{\ge1}$, 
we have $L_{k,k}(T)^{-1}\in(0,1]$, which agrees with the range
of the beta polymer partition functions.
    
\begin{proof}[Proof of \Cref{thm:strict_weak_beta_polymer}]
	This is a direct consequence of 
	\Cref{thm:sqW_sqW_last_row,thm:sqW_process_to_sW_process}
	and \Cref{prop:q_Hahn_to_Beta_polymer}. 
	Indeed, the last row marginal of the sqW/sqW
	process is matched to the 
	$q$-Hahn vertex model height function,
	and in the
	limit $q\to 1$ this implies that the last coordinate marginal of the
	spin Whittaker process is matched to the beta polymer $Z(i,j)$.
\end{proof}

Let us make two remarks on this result.

\begin{remark}
	A weaker version of \Cref{thm:strict_weak_beta_polymer} that
	matches $L_{k_i,k_i}(T)^{-1}$ and $Z(k_i,T)$ for each single time $T$
	can alternatively be proved using moment formulas. Namely, the
	eigenoperators of \Cref{def:sW_eigenoperators}
	may be used to extract multiple integral formulas for the 
	joint moments
	of $L_{k_i,k_i}(T)^{-1}$
	under the spin Whittaker processes.
	These formulas can then be matched to the
	ones for the joint moments of the 
	beta polymer.
	The latter in the homogeneous case are obtained in 
	\cite[Proposition 3.4]{CorwinBarraquand2015Beta},
	and their inhomogeneous generalization
	is rather straightforward, cf. \cite[Proposition 6.1]{petrov2019qhahn}.
\end{remark}
	
\begin{remark}
	\label{rmk:strict_weak_beta_to_gamma}
	It was noticed in \cite[Remark 1.5]{CorwinBarraquand2015Beta} that
	under the scaling $Z(i,T)=S^{T-i+1} z(i,T)$ the process $z(i,T)$
	converges, when $S \to \infty$, to the strict-weak gamma polymer
	model introduced by Sepp{\"a}l{\"a}inen in an unpublished note and
	studied in
	\cite{OConnellOrtmann2014}, \cite{CorwinSeppalainenShen2014}. This
	scaling of polymers corresponds to the scaling
	\eqref{eq:scaling_sW_process_to_sW_process} for the full spin
	Whittaker process.
	As $S$ goes to infinity,
	\Cref{thm:strict_weak_beta_polymer}
	turns into the 
	matching between strict weak gamma polymer model and
	$\alpha$-Whittaker process that was originally discovered in
	\cite{OConnellOrtmann2014}.
	This observation is another piece of evidence supporting 
	the formal $S\to+\infty$
	scaling described in \Cref{sub:sW_to_W}.
\end{remark}

\subsection{Another beta polymer type model}

Let us now recall the beta polymer type model which was
introduced in \cite{CMP_qHahn_Push}. We employ notation from 
\Cref{app:sw_special}.
    
\begin{definition} \label{def:polymer_like}
The random function $\widetilde{Z}(i,j)$, for $i,j\in \mathbb{Z}_{\ge 0}$ is defined by the recurrence
\begin{equation} \label{eq:recurrence_polymer_like}
		\widetilde{Z}(i,j)=
		\begin{dcases}
				1 \qquad & \text{for } j=0,
				\\
				\widetilde{Z}(1,j-1) \widetilde{B}_{1,j} \qquad & \text{for } i=1,
				\\
				W_{i,j}^{>} \widetilde{Z}(i,j-1) + (1-W_{i,j}^{>}) \widetilde{Z}(i-1,j)
				\qquad & \text{if } \widetilde{Z}(i,j-1)>\widetilde{Z}(i-1,j),
				\\
				(1-W_{i,j}^{<}) \widetilde{Z}(i,j-1) + W_{i,j}^{<} \widetilde{Z}(i-1,j)
				\qquad & \text{if } \widetilde{Z}(i,j-1) < \widetilde{Z}(i-1,j),
		\end{dcases}
\end{equation}
where $\widetilde{B}_{1,j} \sim \mathpzc{B}^{-1}(X_1+Y_j,S-Y_j)$ 
are independent inverse beta random variables, and
\begin{equation} \label{eq:random_var_W}
		\begin{split}
				&
				W_{i,j}^> \sim \mathpzc{NBB}^{-1}\left( 2S -1, \frac{\widetilde{Z}(i-1,j)-\widetilde{Z}(i-1,j-1)}{\widetilde{Z}(i,j-1)-\widetilde{Z}(i-1,j-1)}, X_i + Y_j, S-Y_j \right),
				\\
				&
				W_{i,j}^< \sim \mathpzc{NBB}^{-1}\left( 2S -1, \frac{\widetilde{Z}(i,j-1)-\widetilde{Z}(i-1,j-1)}{\widetilde{Z}(i-1,j)-\widetilde{Z}(i-1,j-1)}, X_i + Y_j, S - X_i \right).
		\end{split}
\end{equation}
For $i>1$, $j>0$, we have 
$ \widetilde{Z}(i,j-1) \ne \widetilde{Z}(i-1,j)$
with probability one.
\end{definition}

\begin{proposition}
	Recall the ${}_4\phi_3$ vertex model height function
	$\mathcal{H}^{\mathrm{ul}}_{\phi}(i,j)$
	(\Cref{sub:sqw_sqw_first}). Define
	$\widetilde{Z}^{(q)}(i,j)=q^{-\mathcal{H}^{\mathrm{ul}}_{\phi}(i,j)}$.
	Then, under the scaling \eqref{eq:scaling}, $Z^{(q)}$ converges
	weakly to the process $\widetilde{Z}$:
	\begin{equation*}
		\widetilde{Z}^{(q)}(i,j) \xrightarrow[q \to 1]{} \widetilde{Z}(i,j).
	\end{equation*}
\end{proposition}
\begin{proof}
	In the homogeneous case $X_i=0$, $Y_j=Y$ for all $i,j$, this was
	proven in \cite{CMP_qHahn_Push}. The same argument also essentially applies to the
	inhomogeneous case, and we will not repeat the computations here.
	The non-trivial part of the proof is to understand how the $X_i$
	parameters appear in the definition of $W_{i,j}^<,W_{i,j}^>$. The
	interested reader can check the validity of our statements
	starting from \eqref{eq:4phi3_vertex_model_alternative} and
	reproducing the computations of Section 4.3 of
	\cite{CMP_qHahn_Push}.
\end{proof}
    
\begin{theorem} \label{thm:sW_process_first_row}
	The marginal process $\{L_{k,1}(T)\colon k=1\dots N,\, T \ge N \}$ of the spin
	Whittaker process $\mathfrak{P}_{\mathbf{X},\mathbf{Y}}$ is
	equivalent in distribution
	to the process $\{ \widetilde{Z}(k,T)\colon k=1,\dots N,\,T \ge N\}$.
\end{theorem}
\begin{proof}
	This is established similarly to 
	\Cref{thm:strict_weak_beta_polymer}
	by combining the matching of \Cref{thm:sqW_sqW_first_row}
	with the $q\to1$ scaling limits. 
\end{proof}
    
\subsection{Reduction to log-gamma polymer}

Here we show that the model $\widetilde Z(i,j)$
of \Cref{def:polymer_like}
reduces, as $S\to+\infty$, to the well-known 
log-gamma polymer model introduced in \cite{Seppalainen2012}.
This proof is more involved than the rather straightforward 
observation for the strict-weak beta polymer (\Cref{rmk:strict_weak_beta_to_gamma}).

\begin{definition}[Log-gamma polymer]
	Let $\{g_{i,j} \colon i,j \in \mathbb{Z}_{\ge 1} \}$ be a sequence
	of independent inverse gamma random variables, $g_{i,j} \sim
	\mathrm{Gamma}^{-1}(X_i+Y_j)$ with density \eqref{eq:inverse_gamma_density}.
	The random function $\widetilde{z}$
	defined by the recurrence
	\begin{equation} \label{eq:log_Gamma_polymer}
		\widetilde{z}(i,j)
		=
		\begin{cases}
		g_{i,j} \left( \widetilde{z}(i-1,j) + \widetilde{z}(i,j-1) \right) 
		\qquad &\text{if $i,j \ge 1$ and $i+j\ge3$};
		\\
		g_{1,1}
		\qquad &\text{if } i=j=1;
		\\
		0
		\qquad &\text{else},
		\end{cases}
	\end{equation}
	is the \emph{point-to-point log-gamma polymer partition function}.
\end{definition}

The log-gamma polymer model was introduced 
(with a proof of its exact solvability)
by Sepp{{\"a}}l{{\"a}}inen
in \cite{Seppalainen2012}. 
One can view
$\widetilde z(i,j)$ as a partition function of 
up-right directed paths from $(1,1)$ to $(i,j)$, 
where the weight of each path equals the product of the 
quantities $g_{i',j'}$ along the path.
In \cite{COSZ2011} the log-gamma polymer
model was given a powerful combinatorial interpretation using
Kirillov's geometric RSK (Robinson--Schensted--Knuth) algorithm.
This showed the distributional matching 
of the log-gamma polymer with a
marginal of the Whittaker process~\eqref{eq:W_process}.

The next statement shows that the log-gamma polymer model 
can be obtained in a $S\to+\infty$ scaling limit 
from the beta polymer like model of
\Cref{def:polymer_like}. Modulo
\Cref{rem:sW_process_to_W_process},
this together with 
\Cref{thm:sW_process_first_row} 
produces an alternative derivation of the results of 
\cite{COSZ2011}.

\begin{proposition}
	\label{prop:log_gamma_reduction}
	Consider the scaling $\widetilde{Z}(i,j)=S^{j+i-1}
	\widetilde{\mathpzc{z}}^{(S)}(i,j)$ of the process from
	\Cref{def:polymer_like}. Then the rescaled process
	$\widetilde{\mathpzc{z}}^{(S)}$ converges weakly to
	the log-gamma polymer:
	\begin{equation*}
		\widetilde{\mathpzc{z}}^{(S)}(i,j) \xrightarrow[S \to \infty]{}
		\widetilde{z}(i,j).
	\end{equation*}
\end{proposition}

\begin{proof}
	We argue by induction. 
	When $i=j=1$, then $\widetilde{\mathpzc{z}}^{(S)}(1,1) 
	= S^{-1}
	\mathpzc{B}^{-1}(X_1+Y_1,S-Y_1)$.
	In the large $S$ limit this
	converges to $\mathrm{Gamma}^{-1}(X_1+Y_1)$, which is precisely
	$\widetilde{z}(1,1)$.

	Fix $i,j$ and assume that for all $i',j'$ such that
	$i'+j'<i+j$ the convergence $\widetilde{\mathpzc{z}}^{(S)}(i',j')
	\to \widetilde{\mathpzc{z}}(i',j')$ holds. 
	Let us compute the
	densities of random variables $S^{-1}W_{i,j}^>$ and $S^{-1}
	W_{i,j}^<$, that are rescalings of \eqref{eq:random_var_W}, in the
	large $S$ limit. 
	We show the computations only for
	$W_{i,j}^>$ since the other case is very similar. 
	The density of 
	$S^{-1} W_{i,j}^>$ (depending on the variable $x\in(0,1)$) is, from
	\eqref{eq:random_var_W} and \eqref{eq:NBB},
	equal to 
	\begin{multline} \label{eq:rescaled_NBB}
		\frac{ \left( \frac{1}{x} \right)^{X_i+Y_j+1} \left( 1 - \frac{1}{Sx} \right)^{S-Y_j-1} }{\Gamma(X_i + Y_j)}
		\frac{\Gamma(S+X_i)}{S^{X_i+Y_j} \Gamma(S-Y_j)} 
		\\\times
		(1-p_S)^{2S-1}
		{}_2F_1 
		\left(\begin{minipage}{2.8cm}
				\center{$2S -1,S+X_i$}
		\\
		\center{$S-Y_j$}
		\end{minipage} \Big\vert\,
		p_S\left(1- \frac{1}{Sx} \right)\right),
	\end{multline}
	where
	\begin{equation*}
		p_S=
		\frac{\widetilde{Z}(i-1,j)-\widetilde{Z}(i-1,j-1)}
		{\widetilde{Z}(i,j-1)-\widetilde{Z}(i-1,j-1)}
		=
		\frac{
			S\widetilde{\mathpzc z}^{(S)}(i-1,j)
			-
			\widetilde{\mathpzc z}^{(S)}(i-1,j-1)
		}
		{
			S\widetilde{\mathpzc z}^{(S)}(i,j-1)
			-
			\widetilde{\mathpzc z}^{(S)}(i-1,j-1)
		}
		\sim 
		p\coloneqq
		\frac{\widetilde{z}(i-1,j)}{\widetilde{z}(i,j-1)}
	\end{equation*}
	is smaller than $1$.

	The limit of the first few factors in
	\eqref{eq:rescaled_NBB} is straightforward:
	\begin{equation*}
			\frac{ \left( \frac{1}{x} \right)^{X_i+Y_j+1} 
			\left( 1 - \frac{1}{Sx} \right)^{S-Y_j-1} }
			{\Gamma(X_i + Y_j)}
			\frac{\Gamma(S+X_i)}{S^{X_i+Y_j} \Gamma(S-Y_j)} 
			\xrightarrow[ \hspace{5pt} S\to \infty \hspace{5pt} ]{} 
			\frac{ \left( \frac{1}{x} \right)^{X_i+Y_j+1} e^{-\frac{1}{x}} }{\Gamma(X_i + Y_j)}.
	\end{equation*}
	To compute the limit of the Gaussian
	hypergeometric function, we use the Euler transformation 
	\begin{equation*}
		{}_2F_1 
		\left(\begin{minipage}{1cm}
		\center{$a\,,\,b$}
		\\
		\center{$c$}
		\end{minipage} \Big\vert\,
		z\right) = (1-z)^{c-a-b} {}_2F_1 
		\left(\begin{minipage}{2.1cm}
		\center{$c-a\,,\,c-b$}
		\\
		\center{$c$}
		\end{minipage} \Big\vert\,
		z\right),
	\end{equation*}
	so that the remaining terms in \eqref{eq:rescaled_NBB} become
	\begin{equation*}
		\left( \frac{1-p_S}{1-p_S+\frac{p_S}{Sx}} \right)^{2S-1} \left(1-p_S+\frac{p_S}{Sx} \right)^{-X_i-Y_j} 
		{}_2F_1 
		\left(\begin{minipage}{4cm}
		\center{$-S-Y_j+1,-X_i-Y_j$}
		\\
		\center{$S-Y_j$}
		\end{minipage} \Big\vert\,
		p_S\left(1- \frac{1}{Sx} \right)\right).
	\end{equation*}
	We have
	\begin{equation*}
		\left( \frac{1-p_S}{1-p_S+\frac{p_S}{Sx}} \right)^{2S-1} \left(1-p_S+\frac{p_S}{Sx} \right)^{-X_i-Y_j} 
			\xrightarrow[ \hspace{5pt} S\to \infty \hspace{5pt} ]{}
			e^{-\frac{2p}{(1-p)x}} (1-p)^{-X_i-Y_j},
	\end{equation*}
	whereas
	\begin{equation*}
		\begin{split}
			{}_2F_1 \left(\begin{minipage}{4cm}
			\center{$-S-Y_j+1,-X_i-Y_j$}
			\\
			\center{$S-Y_j$}
			\end{minipage} \Big\vert\,
			p_S\left(1- \frac{1}{Sx} \right)\right)
			\xrightarrow[ \hspace{5pt} S\to \infty \hspace{5pt} ]{}
			{}_1F_0 \left(\begin{minipage}{1.8cm}
			\center{$-X_i-Y_j$}
			\\
			\center{$-$}
			\end{minipage} \Big\vert\,
			-p \right)=(1+p)^{X_i+Y_j}.
		\end{split}
	\end{equation*}

	Our computations imply that 
	\begin{equation*}
		S^{-1} W_{i,j}^> \xrightarrow[ \hspace{5pt} S\to \infty
		\hspace{5pt} ]{} \frac{\widetilde{\mathpzc{z}}(i-1,j) +
		\widetilde{\mathpzc{z} }(i,j-1) }{ \widetilde{\mathpzc{z} }(i,j-1)
		- \widetilde{\mathpzc{z} }(i-1,j) } \,
		\mathrm{Gamma}^{-1}(X_i+Y_j).
	\end{equation*}
	Essentially repeating the computations for 
	$S^{-1} W_{i,j}^<$, we obtain
	\begin{equation*}
		S^{-1} W_{i,j}^< \xrightarrow[ \hspace{5pt} S\to \infty
		\hspace{5pt} ]{} \frac{\widetilde{\mathpzc{z}}(i-1,j) +
		\widetilde{\mathpzc{z} }(i,j-1) }{ \widetilde{\mathpzc{z} }(i-1,j)
		- \widetilde{\mathpzc{z} }(i,j-1) } \,
		\mathrm{Gamma}^{-1}(X_i+Y_j).
	\end{equation*}
	Thus, we see that in the scaling limit as $S\to+\infty$,
	the beta polymer like model 
	recurrence relation \eqref{eq:recurrence_polymer_like}
	becomes \eqref{eq:log_Gamma_polymer}, the 
	recurrence for the log-gamma polymer partition functions.
	This completes the proof.
\end{proof}
    
\section{Deformed quantum Toda Hamiltonian} 
\label{sec:Toda}
    
In this section we consider the scaling limit of the Pieri rule
\eqref{eq:second_Pieri_rule}
which states that the spin $q$-Whittaker polynomials
$\mathbb{F}_\lambda(x_1,\ldots,x_N )$
are eigenfunctions of an operator
acting on the \emph{label} $\lambda$.
This scaling limit leads to an eigenoperator
for the spin Whittaker functions. This operator acts
as a second order differential operator
in the 
(additive versions of the) variables
$\underline{L}_N$. 
We call is the \emph{$S$-deformed quantum Toda Hamiltonian}.
Our scaling of the Pieri rules are
inspired by \cite{GerasimovLebedevOblezin2011} where the Pieri rule for the
$q$-Whittaker polynomials was understood as a discretization of the (undeformed)
quantum Toda Hamiltonian.
    
\subsection{Refined Pieri operators}

We start by refining the Pieri operator
$
(\mathfrak{H}^{\mathrm{sHL}} f)(\mu) = \sum_\lambda f(\lambda) \,\mathsf{F}^{*}_{\lambda' / \mu'} (v)
$
introduced in 
\eqref{eq:pieri_op}, by considering its expansion in powers of $v$. 
Recall identity 
\eqref{eq:second_Pieri_rule} which states that
\begin{equation*}
	(\mathfrak{H}^{\mathrm{sHL}}f)(\lambda)=
	\left(\left(\frac{1}{1-sv} \right)^{N-1} \prod_{i=i}^N (1+x_i v)\right)
	f(\lambda),\qquad 
	f(\lambda)=\mathbb{F}_{\lambda}(x_1,\ldots,x_N ).
\end{equation*}
Defining $\mathfrak{H}_k$ by
\begin{equation} \label{eq:pieri_op_expansion}
		(1-vs)^{N-1} \mathfrak{H}^{\mathrm{sHL}} = \sum_{k=0}^N  v^k \mathfrak{H}_k,
\end{equation}
we see that
\begin{equation} \label{eq:refined_pieri_rules}
		\mathfrak{H}_k \mathbb{F}_\lambda = e_k(x_1,\dots , x_N) \mathbb{F}_\lambda.
\end{equation}

The action of the $\mathfrak{H}_k$'s
on functions $f(\lambda)$
can in principle
be recovered using the vertex weights $w^*$ in
\Cref{fig:table_w_star} (without the denominator $1-vs$)
which compose the sHL functions 
$(1-vs)^{N-1}\mathsf{F}^{*}_{\lambda' / \mu'}$.
In the simplest
cases $k=0$ or $N$ one can verify that
\begin{equation*}
	\mathfrak{H}_0 = \mathrm{Id}, 
	\qquad 
	\mathfrak{H}_N= \mathcal{T}_{\mu_1} \cdots \mathcal{T}_{\mu_N},
\end{equation*}
where $\mathcal{T}$ is the shift operator
\begin{equation*}
	(\mathcal{T}_{\mu_i}f)(\mu)
	=
	\begin{cases}
		f(\mu+\mathrm{e}_i),&\text{if $\mu_i<\mu_{i-1}$ or $i=1$};
		\\
		0,&\text{otherwise}.
	\end{cases}
\end{equation*}
Indeed, $\mathfrak{H}_0$ requires no vertical arrows to change from 
$\mu$ to $\lambda$, and $\mathfrak{H}_N$
corresponds to adding a full horizontal path
starting at zero and ending at $N$, so that 
the arrow configuration corresponding to $\lambda$ 
is obtained from the one for $\mu$ by adding one vertical arrow at 
location $N$.

When $1\le k \le N-1$, explicit formulas for $\mathfrak{H}_k$ look significantly more involved.
We need only the one for $k=1$, and will not discuss the other operators 
$\mathfrak{H}_k$.\footnote{Appropriate scaling limits of 
	the higher operators $\mathfrak{H}_k$ could potentially lead
	to higher order differential operators 
	commuting with the deformed quantum Toda Hamiltonian $\mathscr{H}_2$
	introduced below in this section.
	We leave this investigation to a future work.}
In the next statement, by agreement, we set 
$\mu_0=+\infty$, $\mu_{N+1}=-\infty$.

%
%
%

\begin{proposition}
		We have
		\begin{equation*}
			\mathfrak{H}_1 = h_{0,0} \, \mathrm{Id} + \sum_{0\le k < \ell \le N} 
			h_{k,\ell} \mathcal{T}_{\mu_{k+1}} \cdots \mathcal{T}_{\mu_\ell},
		\end{equation*}
		with
		\begin{align*}
				h_{0,0} &= -s \sum_{j=1}^{N-1} q^{\mu_j - \mu_{j+1}},
				\\
				h_{k,\ell} &= (1-q^{\mu_k - \mu_{k+1}}) (-s)^{\ell-k-1} q^{\mu_{k+1} - \mu_\ell} (1-s^2 q^{\mu_\ell - \mu_{\ell+1}} ).
		\end{align*}
\end{proposition}

\begin{proof}
		Express the action of $\mathfrak{H}_1$ as 
		\begin{equation*}
				\mathfrak{H}_1 f(\mu) = \sum_{\mu\prec \lambda} H(\mu; \lambda) f(\lambda),
		\end{equation*}
		where the term $H(\mu; \lambda)$ corresponds to the
		weight of a row of vertices having a configuration $\lambda$ 
		at the top and $\mu$ at the bottom.
		Recall that we are using down-right directed paths as in 
		\Cref{fig:table_w_star}, and all the 
		individual vertex weights are multiplied by $(1-vs)$.
		
		Observe that each vertex 
		\tikz[baseline=-0.1cm]{\draw[dotted]
		(-0.3,0)--(0,0); \draw[red] (0,0)--(0.3,0); \draw[red]
		(0.03,0.3)--(0.03,-0.3); \draw[red] (-0.03,0.3)--(-0.03,-0.3);
		\draw[fill=black] (0,0) circle(0.04cm);} 
		somewhere in the bulk
		comes with the weight
		$v(1-q^{\mu_i-\mu_{i+1}})$.
		Therefore, 
		terms proportional to $v$ can only come
		from configurations with no horizontal arrows
		\tikz[baseline=-0.1cm]{\foreach \n in {0,1,2,3,4} {\draw[dotted]
		(0.6*\n-0.3,0)--(0.6*\n,0); \draw[dotted] (0.6*\n,0)--(0.6*\n+0.3,0);
		\draw[red] (0.6*\n+0.03,+0.3)--(0.6*\n+0.03,-0.3); \draw[red]
		(0.6*\n-0.03,0.3)--(0.6*\n-0.03,-0.3); \draw[fill=black] (0.6*\n,0)
		circle(0.04cm);}}
		or with a single sequence of horizontal arrows
		\tikz[baseline=-0.1cm]{
		\draw[red] (0.6,0)--(4*0.6,0);
		\foreach \n in {0,2,3} {
		\draw[red] 
		(0.6*\n+0.03,+0.3)--(0.6*\n+0.03,-0.3); 
		\draw[red] (0.6*\n-0.03,0.3)--(0.6*\n-0.03,-0.3); 
		\draw[fill=black] (0.6*\n,0) circle(0.04cm);}
		\draw[red] 
		(0.6*1+0.03,+0.3)--(0.6*1+0.03,-0); 
		\draw[red] (0.6*1-0.03,0.3)--(0.6*1-0.03,-0.3); 
		\draw[fill=black] (0.6*1,0) circle(0.04cm);
		\draw[red] 
		(0.6*4+0.04,+0.3)--(0.6*4+0.04,-0.3); 
		\draw[red] (0.6*4-0.04,0.3)--(0.6*4-0.04,-0.3); 
		\draw[red] (0.6*4,0)--(0.6*4,-0.3);
		\draw[fill=black] (0.6*4,0) circle(0.04cm);
		\draw[dotted] (-0.3,0)--(0.6,0);
		\draw[dotted] (2.4,0)--(2.7,0);
		}. 
		The first case, corresponding to $\lambda=\mu$, provides us with
		the term $h_{0,0}$. For the latter case, let $k$ be the column of
		the leftmost vertex emanating an horizontal arrow, and let $\ell$
		be the column where the single horizontal path stops. Such
		configuration corresponds to a partition
		$\lambda=\mu+\mathbf{e}_{k+1}+\cdots+\mathbf{e}_{\ell}$.
		Isolating the coefficient of $v$ in the expansion of product of
		vertex weights, we recover $h_{k,\ell}$.
\end{proof}

\subsection{Scaling of the Pieri operators}

Introduce the differential operators acting in the variables 
$u_1,\ldots,u_N$:
\begin{align}
	\label{eq:Toda_1}
	\mathscr{H}_1 &
	\coloneqq \sum_{i=1}^N \partial_{u_i};
	\\
	\label{eq:Toda_2_main_operator}
	\mathscr{H}_2 &
	\coloneqq 
	-\frac{1}{2} \sum_{i=1}^N \partial^2_{u_i} 
	+
	\sum_{1\le i<j \le N} 
	S^{-2(j-i)} e^{u_j - u_i} 
	( 
		S - \partial_{u_i} 
	) 
	( 
		S+ \partial_{u_j} 
	).
\end{align}
In the second operator, 
the product is understood in the usual way as
\begin{equation*}
	(S-\partial_{u_i})(S+\partial_{u_j})
	=
	S^2
	\,\mathrm{Id}+S(\partial_{u_j}-\partial_{u_i})-
	\partial_{u_i}\partial_{u_j}
	.
\end{equation*}

For the next result we define the rescaling 
\begin{equation} \label{eq:scaling_toda}
	q=e^{-\varepsilon}, \qquad q^{\lambda_i}=\frac{e^{-u_i}}{S^{N+1-2i}}, \qquad s=-e^{-\varepsilon S}.
\end{equation}
\begin{proposition} \label{prop:Pieri_to_Toda}
	Under the scaling \eqref{eq:scaling_toda}, we have
	\begin{align*}
			\mathscr{H}_1 &= \lim_{\varepsilon \to 0} \frac{1}{\varepsilon} \left( \mathfrak{H}_1 - N \right),
			\\
			\mathscr{H}_2 &= - \lim_{\varepsilon \to 0} \frac{1}{\varepsilon^2} \left( \mathfrak{H}_1 - N + \frac{1}{2} \mathfrak{H}_N^2 - 2 \, \mathfrak{H}_N + \frac{3}{2} \right).
	\end{align*}
\end{proposition}
We remark that the combinations
of the
refined Pieri 
operators leading to $\mathscr{H}_1$ and $\mathscr{H}_2$
are the same as in the $q$-Whittaker case
\cite[Proposition 2.1]{GerasimovLebedevOblezin2011}, and correspond
to the scaling of eigenvalues in the proof of 
\Cref{thm:deformed_Toda_eigen} below.
The scaling \eqref{eq:scaling_toda} of the variables, however, 
is different.
\begin{proof}[Proof of \Cref{prop:Pieri_to_Toda}]
	First, expand the shift operator as $\varepsilon\to0$.
	From \eqref{eq:scaling_toda} we see that the
	increment by $1$ in $\mu_i$ corresponds to 
	the increment by $\log(q^{-1})=\varepsilon$
	in the scaled variable $u_i$. Therefore, 
	\begin{equation*}
		\mathcal{T}_{\mu_k}=\mathrm{Id}+\varepsilon \partial_{u_k} +\frac{\varepsilon^2}{2} \partial_{u_k}^2 +
		\smallO(\varepsilon^2),
	\end{equation*}
		and hence
		\begin{equation*}
				\mathcal{T}_{\mu_{k+1}}\cdots \mathcal{T}_{\mu_\ell}=
				\mathrm{Id} + \varepsilon \sum_{\alpha=k+1}^\ell
				\partial_{u_\alpha} + \frac{\varepsilon^2}{2} \sum_{k+1\le
				\alpha , \beta \le \ell} \partial_{u_\alpha,u_\beta} +
				\smallO(\varepsilon^2).
		\end{equation*}
		This implies that
		\begin{equation} \label{eq:expansion_H_N}
			\frac{1}{2}\mathfrak{H}_N^2 - 2 \mathfrak{H}_N + \frac{3}{2} =
			- \varepsilon \sum_{i=1}^N \partial_{u_i}  +
			\smallO(\varepsilon^2).
		\end{equation}
		
		Next we address the scaling of $\mathfrak{H}_1$. Set
		\begin{equation*}
				a_{k,\ell} = \begin{dcases}
				 S^{-2(\ell-k)}e^{u_{\ell} - u_k} \qquad &\text{if } 1\le k \le \ell \le N,
				\\
				0 \qquad & \text{else}.
				\end{dcases}
		\end{equation*}
		Expanding the coefficients $h_{k,\ell}$, we have
		\begin{equation*}
				h_{0,0}=\left(1-\varepsilon S +\frac{\varepsilon^2}{2}S^2\right)\sum_{j=1}^{N-1} a_{j,j+1} + \smallO(\varepsilon^2),
		\end{equation*}
		and (recall that we assume $k<\ell$)
		\begin{equation*}
				\begin{split}
						h_{k,\ell} 
						=
						&
						(a_{k+1,\ell}-a_{k,\ell}-a_{k+1,\ell+1}+a_{k,\ell+1})
						\\
						&
						+\varepsilon S \bigg\{ -( \ell-k-1 ) (a_{k+1,\ell}-a_{k,\ell}) + ( \ell-k+1 )  (a_{k+1,\ell+1} - a_{k,\ell+1}) \bigg\}
						\\
						&
						+ \frac{\varepsilon^2}{2} S^2 \bigg\{ (\ell-k-1)^2(a_{k+1,\ell}-a_{k,\ell})
						- (\ell-k+1)^2  (a_{k+1,\ell+1}-a_{k,\ell+1})
						\bigg\}+ \smallO(\varepsilon^2).
				\end{split}
		\end{equation*}

		Together with the action of the shifts 
		$\mathcal{T}$, we see that $h_{k,\ell} \mathcal{T}_{k+1}\cdots\mathcal{T}_{\ell}$ expands as
		\begin{align*}
					&(a_{k+1,\ell}-a_{k,\ell}-a_{k+1,\ell+1}+a_{k,\ell+1})\,\mathrm{Id}
						\\
						&
						+\varepsilon \bigg\{ -S( \ell-k-1 ) (a_{k+1,\ell}-a_{k,\ell})
							\,\mathrm{Id}
							+ S ( \ell-k+1 )  (a_{k+1,\ell+1} - a_{k,\ell+1})
							\,\mathrm{Id}
						\\
						&
						\qquad\qquad
						+(a_{k+1,\ell}-a_{k,\ell}-a_{k+1,\ell+1}+a_{k,\ell+1}) \sum_{\alpha=k+1}^\ell \partial_{u_\alpha} \bigg\}
						\\
						&
						+\frac{\varepsilon^2}{2}
						\bigg\{
							S^2 (\ell-k-1)^2(a_{k+1,\ell}-a_{k,\ell})\,\mathrm{Id}
							- S^2 (\ell-k+1)^2  (a_{k+1,\ell+1}-a_{k,\ell+1})\,\mathrm{Id}
						\\
						&
						\qquad \qquad
						+(a_{k+1,\ell}-a_{k,\ell}-a_{k+1,\ell+1}+a_{k,\ell+1}) \sum_{k+1\le \alpha , \beta \le \ell} \partial^2_{u_\alpha, u_\beta}
						\\
						&
						\qquad
						+2 S \big(
						-( \ell-k-1 ) (a_{k+1,\ell}-a_{k,\ell}) + ( \ell-k+1 )  (a_{k+1,\ell+1} - a_{k,\ell+1})
						\big) \sum_{\alpha=k+1}^\ell \partial_{u_\alpha}
						\bigg\}+ \smallO(\varepsilon^2).
		\end{align*}
		To evaluate the summation $\sum_{0\le k < \ell \le N} h_{k,\ell}
		\mathcal{T}_{\mu_{k+1}} \cdots \mathcal{T}_{\mu_\ell},$ we use the
		identities in 
		\Cref{prop:summation_identities}. We obtain
		\begin{equation*}
				\begin{split}
						&
						h_{0,0}+\sum_{0\le k < \ell \le N} h_{k,\ell} \mathcal{T}_{\mu_{k+1}} \cdots \mathcal{T}_{\mu_\ell}
						\\
						&
						\qquad
						=
						N + \varepsilon \sum_{i=1}^N \partial_{u_i} 
						+ \varepsilon^2 \left\{ \frac{1}{2}\sum_{i=1}^N \partial_{u_i}^2 - \sum_{1\le i < j \le N} S^{2(i-j)} e^{u_j - u_i} \left(S-\partial_{u_i}\right)\left(S+\partial_{u_j}\right) \right\} +\smallO(\varepsilon^2).
				\end{split}
		\end{equation*}
		Together with \eqref{eq:expansion_H_N} this yields the proof.
\end{proof}

For the next result we employ the 
spin Whittaker functions in the \emph{additive parameters}
$u_i$, where the multiplicative parameters $\underline{L}_N=(L_{N,N},\ldots,L_{N,1} )$ 
are expressed through the $u_i$'s as 
\begin{equation}
	\label{eq:sW_additive_variables}
	L_{N,i}=S^{N+1-2i}e^{u_i}.
\end{equation}
Denote 
the spin Whittaker function 
$\mathfrak{f}_{\underline{X}}(\underline{L}_N)$
in the additive parameters 
by $\mathfrak{f}^{\,add}_{\underline{X}}(u_1,\ldots,u_N )$. 
Here $\underline{X}=(X_1,\ldots,X_N )$ are such that $|X_i|<S$ for all $i$.

\begin{theorem}
	\label{thm:deformed_Toda_eigen}
		The spin Whittaker functions
		$\mathfrak{f}^{\,add}_{\underline{X}}(u_1,\ldots,u_N )$ 
		in the additive variables 
		\eqref{eq:sW_additive_variables}
		are eigenfunctions of
		the differential operators 
		$\mathscr{H}_1$ \eqref{eq:Toda_1}
		and $\mathscr{H}_2$ \eqref{eq:Toda_2_main_operator}. 
		In
		particular, we have
		\begin{align*}
				\mathscr{H}_1 \mathfrak{f}^{\,add}_{\underline{X}} (u_1,\dots,u_N) &= -\left( X_1 + \cdots + X_N \right) \mathfrak{f}^{\,add}_{\underline{X}} (u_1,\dots,u_N),
				\\
				\mathscr{H}_2 \mathfrak{f}^{\,add}_{\underline{X}} (u_1,\dots,u_N) &= -\frac{1}{2}\left( X_1^2 + \cdots + X_N^2 \right) \mathfrak{f}^{\,add}_{\underline{X}} (u_1,\dots,u_N).
		\end{align*}
\end{theorem}
\begin{proof}
	This result is a combination of 
	the refined Pieri rules
	\eqref{eq:refined_pieri_rules} 
	viewed as eigenrelations
	for the spin $q$-Whittaker functions,
	and the convergence of the 
	functions 
	(\Cref{thm:sqW_to_sW})
	and the operators
	(\Cref{prop:Pieri_to_Toda}).
	More precisely, 
	under the scaling $x_i=e^{-\varepsilon X_i}$ 
	for the eigenvalues $e_k(x_1,\ldots,x_N )$ 
	we have
	\begin{equation*}
		\frac{1}{\varepsilon}\big(e_1(x_1,\dots,x_N)-N \big) \xrightarrow[\varepsilon \to 0]{} -X_1 -\cdots -X_N,
	\end{equation*}
	and
	\begin{equation*}
			\frac{1}{\varepsilon^2}\left( e_1(x_1,\dots,x_N) - N +
			\frac{1}{2} e_N(x_1,\dots,x_N)^2 - 2 e_N(x_1,\dots,x_N)
		+\frac{3}{2}  \right)\xrightarrow[\varepsilon \to 0]{}
		\frac{1}{2}( X_1^2 +\cdots X_N^2).
	\end{equation*}
	This leads to the desired results.
\end{proof}    

\subsection{Reduction to the quantum Toda Hamiltonian}

It is natural to call the second order differential 
operator $\mathscr{H}_2$ \eqref{eq:Toda_2_main_operator}
the \emph{deformed quantum Toda Hamiltonian}.
Namely, it is diagonal 
in the spin Whittaker functions
which (formally) reduce, as $S\to+\infty$, 
to the classical $\mathfrak{gl}_n$ Whittaker functions
(\Cref{sub:sW_to_W}).
Further justification to this
name comes from the fact that
the operator 
$\mathscr{H}_2$ itself degenerates as $S\to+\infty$ to the
usual quantum Toda Hamiltonian
\begin{equation}
	\label{eq:quantum_usual_Toda}
	\mathscr{H}^{\mathrm{Toda}}_2\coloneqq
	-\frac{1}{2}  \sum_{i=1}^{N} \partial_{u_i}^2 + \sum_{i=1}^{N-1} e^{ u_{i+1} - u_i }.
\end{equation}
\begin{proposition}
	\label{prop:S_Toda_to_Toda}
	As $S\to+\infty$, the operator $\mathscr{H}_2$ \eqref{eq:Toda_2_main_operator}
	converges to the quantum Toda Hamiltonian
	$\mathscr{H}^{\mathrm{Toda}}_2$
	\eqref{eq:quantum_usual_Toda}.
\end{proposition}
\begin{proof}
	The factors $S^{-2(j-i)}$, $1\le i\le j\le N$, 
	in the sum in \eqref{eq:Toda_2_main_operator}
	decay at least as fast as $S^{-2}$ as $S\to+\infty$.
	Therefore, the only surviving contribution 
	in the limit $S\to+\infty$
	comes from the terms with $j=i+1$, for which we have
	\begin{equation*}
		S^{-2} e^{u_{i+1} - u_i} 
		( 
			S - \partial_{u_i} 
		) 
		( 
			S+ \partial_{u_{i+1}} 
		)
		\to
		e^{u_{i+1}-u_i},\qquad S\to+\infty.
	\end{equation*}
	This completes the proof.
\end{proof}

\section{Desired properties and conjectures}
\label{sec:concluding_remarks}

This paper developed the 
spin $q$-Whittaker polynomials and spin Whittaker functions,
and established many of their 
properties
which are 
one-parameter generalizations of the
corresponding facts about the $q$-Whittaker 
polynomials and $\mathfrak{gl}_n$ Whittaker functions.
In this final section we briefly discuss further desired properties and
conjectures corresponding to our deformed situation.

\subsection{Orthogonality and spectral theory for spin $q$-Whittaker polynomials}

The $q$-Whittaker polynomials satisfy orthogonality
relations coming from (the $t=0$ degeneration of) the Macdonald
torus scalar product \cite[Ch. VI.9]{Macdonald1995}.
This relation states that the $s=0$ versions of
$\mathbb{F}_\lambda(z_1,\ldots,z_N )$
and $\mathbb{F}_\mu(1/z_1,\ldots,1/z_N )$
are orthogonal to each other when $\mu\ne \lambda$
with respect to a certain weight
on the $N$-dimensional torus 
$\mathbb{T}^N=\{|z_i|=1,\, i=1,\ldots,N \}$.
\begin{remark}
	Under the generalization with a spin parameter,
	the spin Hall--Littlewood polynomials also
	satisfy a version of the torus orthogonality
	(called spatial orthogonality in 
	\cite[Corollary 3.10]{BCPS2014_arXiv_v4},
	see also
	\cite{Borodin2014vertex},
	\cite[Proposition 8.6]{BufetovMucciconiPetrov2018}),
	as well as another biorthogonality
	involving the summation 
	of over $\lambda$ instead of integration over $z$.
	Here we discuss only the former conjectural orthogonality
	of the spin $q$-Whittaker polynomials.
\end{remark}

Define 
\begin{equation*}
		m_{q,s}^N(z_1,\dots , z_N) \coloneqq 
		\frac{1}{N!} \prod_{1\le i \ne j \le N} 
		\frac{(s^2,z_i / z_j ; q)_\infty}{(-s z_i, -s / z_i ; q)_\infty} 
		\prod_{i=1}^N \frac{1}{ 2 \pi \mathrm{i} z_i},
		\qquad (z_1,\ldots,z_N )\in \mathbb{T}^N.
\end{equation*}
When $s=0$, $m_{q,s}^{N}$
reduces to the 
orthogonality measure of the $q$-Whittaker polynomials
on $\mathbb{T}^N$, which is a $t=0$ degeneration of the Macdonald's torus 
scalar product $\Delta(z_1,\ldots,z_N;q,t )$, cf. \cite[VI.(9.2)]{Macdonald1995}.
\begin{lemma}
	\label{lemma:self_adjoint}
	Both eigenoperators $\mathfrak{D}_1$, $\overline{\mathfrak{D}}_1$
	\eqref{eq:D_1}, \eqref{eq:D_1_bar}
	for the spin $q$-Whittaker polynomials
	are self-adjoint with respect to the scalar product
	\begin{equation*}
		\langle f,g \rangle_{q,s} \coloneqq
		\int_{\mathbb{T}^N}
		f(z_1,\ldots,z_N )
		\,
		\overline{g(z_1,\ldots,z_N )}
		\,m_{q,s}^{N} (z_1,\ldots,z_N )\,dz_1\ldots dz_N ,
	\end{equation*}
	where $f,g$ are Laurent polynomials with 
	coefficients in $\mathbb{R}(q,s)$.
\end{lemma}
\begin{proof}
	A direct verification.
\end{proof}

\begin{conjecture}\label{conj:sW_orthogonality}
	We have for all signatures $\lambda,\mu$:
	\begin{equation} \label{eq:conj_orthogonality_sqW}
			\int_{\mathbb{T}^N}
			\mathbb{F}_{\lambda} (z_1, \dots, z_N)\,
			\mathbb{F}_{\mu}(1/z_1, \dots, 1/z_N)\,
			m_{q,s}^N(z_1,\dots , z_N)\,
			d z_1 \cdots d z_N = c_\lambda 
			\mathbf{1}_{\lambda=\mu},
	\end{equation}
	where
	\begin{equation}
		\label{eq:c_lambda_s_q}
			c_\lambda = \prod_{i=1}^{N-1} \frac{ (s^2;q)_\infty }{(q ; q)_\infty} \frac{ (q ; q)_{\lambda_i - \lambda_{i+1}} }{ (s^2 ; q)_{\lambda_i - \lambda_{i+1}} }.
	\end{equation}
\end{conjecture}


Note that for $N\le 2$ the statement (up to the 
concrete formula for the norm $c_\lambda$)
follows from \Cref{lemma:self_adjoint}
and the eigenrelations of \Cref{thm:eigenrelation_sqW,thm:eigenrelation_sqW_conj}.
However, for $N\ge 3$ the two operators 
$\mathfrak{D}_1$, $\overline{\mathfrak{D}}_1$
are not sufficient to conclude orthogonality.
\begin{remark}
When $s=0$, the constant $c_\lambda$
\eqref{eq:c_lambda_s_q}
coincides with the $t=0$ degeneration of the torus scalar product
norm of a Macdonald polynomial \cite[Ch. VI.9, Example 1]{Macdonald1995}.
\end{remark}

Let us present one further argument in favor of 
\Cref{conj:sW_orthogonality}.
It was proven in \cite[Proposition 4.10]{imamura2019stationary} that the
probability mass function of a tagged particle in the homogeneous
$q$-Hahn Tasep with parameters $\nu=s^2$ and $\mu=- y s$ is
\begin{equation} \label{eq:q_Hahn_TASEP_IMS}
		\begin{split}
				\mathbb{P}(\mathcal{H}(N,t)=\ell) = \left( \frac{(q;q)_\infty}{(s^2;q)_\infty} \right)^{N-1} \int_{\mathbb{T}^N}
				&
				m^N_{q,s}(z_1,\dots,z_N) \prod_{j=1}^N \left( \frac{\Pi(z_j;y)}{\Pi(-s;y)} \right)^t 
				\\
				&
				\times
				\left( \frac{ (-s)^N }{z_1 \cdots z_N} \right)^\ell 
				\frac{(\frac{(-s)^N}{z_1\cdots z_N};q)_\infty (s^2;q)_\infty^{N-1}}
				{(-sz_1,\dots,-sz_N;q)_\infty}
				dz_1\cdots dz_N,
		\end{split}
\end{equation}
where
$\Pi (z;y)= \dfrac{( - s z;q)_\infty}{(-sy;q)_\infty}$.
The same probability can be expressed as
\begin{equation*}
		\sum_{\substack{\lambda \in \mathrm{Sign}_N \\ \lambda_N=\ell }} 
		\mathbb{F}_\lambda(-s,\dots,-s)\,
		\frac{\mathbb{F}^*_\lambda(y,\dots,y)}{\Pi(-s;y)},
\end{equation*}
Assuming \Cref{conj:sW_orthogonality},
this sum becomes
\begin{equation*}
		\int_{\mathbb{T}^N} m_{q,s}^N(z_1,\dots,z_N)
		\prod_{j=1}^N \left( \frac{\Pi(z_j;y)}{\Pi(-s;y)} \right)^t \left( \frac{ (-s)^N }{z_1 \cdots z_N} \right)^\ell \sum_{\substack{\lambda \in \mathrm{Sign}_N \\ \lambda_N=0 }} \frac{(-s)^{|\lambda|}}{c_\lambda} \,\mathbb{F}_\lambda(1/z_1,\dots,1/z_N)\, dz_1\cdots dz_N.
\end{equation*}
Here we used the Cauchy Identity and the torus scalar product to express
the dual function $\mathbb{F}_\lambda^*$, and the shifting rule of
Proposition \ref{prop:sqW_shifting} to take out the monomial of degree
$\ell$. Evaluating the sum inside the integral as in
\eqref{eq:quasi_Cauchy_identity}, we recover exactly 
\eqref{eq:q_Hahn_TASEP_IMS}.

\subsection{Accessing spin $q$-Whittaker polynomials via free field}

In this paper we did not focus on Fredholm determinantal structures for
marginals of spin $q$-Whittaker processes. These aspects have been
covered quite extensively in literature for specific models in the
last few years \cite{Corwin2014qmunu}, \cite{CorwinPetrov2015},
\cite{imamura2017fluctuations},
\cite{imamura2019stationary},
\cite{BufetovMucciconiPetrov2018}.
Techniques to access such Fredholm
determinantal formulas usually rely on manipulations
with integral representations of $q$-moments (as in
\cite{BorodinCorwin2011Macdonald}, \cite{BorodinCorwinSasamoto2012}).

In the realm of Macdonald processes there exists an alternative approach to expose
the determinantal nature of specific observables.
This is done via
a free
field realization of Macdonald functions and Macdonald operators
\cite{FeiginEtAl_Macdonald}, \cite{koshidaMac}, see also
\cite{BCGS2013} and, e.g., \cite{foda2009hall} for the Hall--Littlewood
case. In the yet simpler case of the Schur processes this
reduces to the infinite wedge representation of Schur functions
\cite{okounkov2001infinite}, \cite{okounkov2003correlation}.
It would be of great interest to
understand to what extent our spin $q$-Whittaker functions, operators,
and processes admit a description in terms of Fock type representations of
a hypothetical $(q,s)$-deformed Heisenberg algebra. It is worth mentioning that an example where symmetric functions coming from solvable vertex models have been incorporated in the language of Fock space representation can be found in \cite{Brubaker_metaplectic}.

\subsection{Sampling and RSK like constructions}

The sqW/sHL and sqW/sqW random
fields of signatures (described in \Cref{sec:dynamics_on_arrays}) 
can be sampled
using the bijectivization
of the corresponding Yang--Baxter equations. 
While these sampling algorithms are well-adapted to 
particle system 
marginals, there could be other randomized procedures
to sample the whole signatures (and resulting in potentially different random fields). 

In particular, there could exist distinguished ``least random''
(i.e., using the least possible
number of random variables)
sampling procedures
resembling the classical Robinson--Schensted--Knuth (RSK)
insertion algorithms.
At $s=0$, such (rather complicated)
RSK-like algorithms were developed in \cite{MatveevPetrov2014}
for sampling $q$-Whittaker processes. 
Further setting $q=0$ recovers the classical RSK algorithm 
related to Schur processes
(which we 
rederive directly
from the Yang--Baxter
equation
in 
\Cref{sec:rsk_from_YB}).
We refer to 
\cite[Section 2.6]{BufetovMucciconiPetrov2018}
for a detailed historical discussion 
of sampling random Young diagrams / signatures
whose probability weights are expressed through various families of symmetric functions.
It would be very interesting to extend RSK-like 
sampling algorithms to the spin $q$-Whittaker level.

In the scaling limit as $q\to1$, the 
RSK-like sampling algorithms of \cite{MatveevPetrov2014}
degenerate into the well-known geometric RSK algorithms introduced and studied in
\cite{Kirillov2000_Tropical}, \cite{NoumiYamada2004}.
The geometric RSK's are naturally 
associated
with Brownian and log-gamma polymer models and $\mathfrak{gl}_n$ Whittaker functions
\cite{Oconnell2009_Toda}, \cite{COSZ2011}.
It would be very interesting to lift geometric RSK's to the 
spin Whittaker processes / beta polymer level developed in the 
present paper. This beta polymer 
version of the RSK could
arise in the corresponding scaling limit of the spin $q$-Whittaker 
RSK.

\subsection{Higher polymer interpretations and random walks}

The strict-weak log-gamma polymer model 
is matched in distribution
to the last row marginals of the Whittaker process (cf. \Cref{rmk:strict_weak_beta_to_gamma}).
Moreover, this connection 
extends (via the geometric RSK)
to the so-called higher polymer partition
functions, i.e., partition functions
of $k$-tuples of noncrossing paths
(in the same log-gamma environment), $k=1,2,\ldots $,
where
$k=1$ corresponds to the original strict-weak log-gamma polymer \cite{NoumiYamada2004}, 
\cite{COSZ2011}.
The higher log-gamma partition functions are matched
with joint distributions of several components of the Whittaker
process.
It is very interesting to find similar higher polymer like
interpretations of joint distributions
of multiple components in the spin Whittaker process introduced in \Cref{sec:sW_processes}.

The strict-weak beta polymer partition function
admits an alternative  
description 
as a certain probability
for the random walk in beta random environment 
\cite{CorwinBarraquand2015Beta}. 
Could multiple components in the spin Whittaker process
be matched to certain probabilities of interacting
random walks in beta random environment?

\subsection{Further properties of spin Whittaker functions}

Whittaker functions have a number of
important properties whose generalization to the 
spin Whittaker level seems potentially very interesting.
This includes connections to representation theory \cite{Kostant1977Whitt},
Mellin-Barnes integral representation \cite{Gerasimov_et_al_OnAGaussGivental}, and 
orthogonality relations
\cite{semenov1994quantization}.
We only make a conjecture about the latter
which is in effect a scaling limit
of \Cref{conj:sW_orthogonality}.

\begin{conjecture}
	For all $\underline{L}_N, \underline{L}_N'\in \mathpzc{W}_N$ we have
	\begin{equation}
		\label{eq:S_Sklyanin_orthogonality}
			\int_{(\mathrm{i}\mathbb{R})^N}
			\mathfrak{f}_{\underline{Z}}(\underline{L}_N) \, \mathfrak{f}_{-
			\underline{Z}}(\underline{L}_N') \,
			\mathfrak{M}^N_S(\underline{Z}) \, dZ_1 \dots dZ_N =
			\prod_{i=1}^{N-1} \left( 1-\frac{L_{N,i+1}}{L_{N,i}}
			\right)^{1-2S} \, \delta_{\underline{L}_N -\underline{L}_N'},    
	\end{equation}
	where $\mathfrak{M}^N_S$ is the $S$-deformation of the Sklyanin measure
	\begin{equation*}
			\mathfrak{M}^N_S (\underline{Z}) = \frac{1}{N! ( 2 \pi
			\mathrm{i})^N } \prod_{1\le i \neq j \le N} \frac{\Gamma(S+Z_i)
			\Gamma(S-Z_i)}{ \Gamma(2S) \Gamma(Z_i - Z_j) },
	\end{equation*}
	and $\delta_{\underline{L}_N -\underline{L}_N'}$ is a delta function.
\end{conjecture}

In support of this conjecture we note that
the eigenoperators 
$\mathscr{D}_1$ and $\overline{\mathscr{D}}_1$
for the spin Whittaker functions (\Cref{def:sW_eigenoperators})
are self-adjoint with respect
to the scalar product defined
by the 
$S$-deformed Sklyanin measure 
$\mathfrak{M}^N_S$ (this can also be checked directly).
This implies the desired statement 
for $N\le 2$, up to the concrete expression for the norm
in the right-hand side of~\eqref{eq:S_Sklyanin_orthogonality}.

\medskip

The theory 
quantum Toda Hamiltonians 
and Whittaker functions extends from $\mathfrak{gl}_n$
to other classical Lie groups 
\cite{Kostant1977Whitt},
\cite{gerasimov2012new}.
It would be interesting to extend our deformation
\eqref{eq:Toda_2_main_operator}
of the $\mathfrak{gl}_n$ quantum Toda Hamiltonian
to other symmetry (Killing-Cartan) types.


\appendix

\section{Special functions and probability distributions}
\label{app:special}

We use the $q$-Pochhammer symbol notation
\eqref{eq:q_Pochhammer}.

\subsection{q-beta binomial distribution}

Recall the definition of
the $q$-deformed beta-binomial distribution
$\varphi_{q,\mu,\nu}$ from \cite{Povolotsky2013}, \cite{Corwin2014qmunu}.

\begin{definition}
	\label{def:phi_distribution}
	For $m\in \mathbb{Z}_{\ge0}$, consider the following distribution on 
	$\left\{ 0,1,\ldots,m  \right\}$:
	\begin{equation}
	\label{eq:phi_def}
		\varphi_{q,\mu,\nu}(j\mid m)=
		\mu^j\,\frac{(\nu/\mu;q)_j(\mu;q)_{m-j}}{(\nu;q)_m}
		\frac{(q;q)_m}{(q;q)_j(q;q)_{m-j}}
		,
		\qquad 
		0\le j\le m.
	\end{equation}
	When $m=+\infty$, extend the definition as
	\begin{equation}
	\label{eq:phi_def_infty}
		\varphi_{q,\mu,\nu}(j\mid \infty)=
		\mu^j\frac{(\nu/\mu;q)_j}{(q;q)_j}\frac{(\mu;q)_\infty}{(\nu;q)_\infty},
		\qquad 
		j\in \mathbb{Z}_{\ge0}.
	\end{equation}
	The distribution depends on $q$ 
	and two other parameters~$\mu,\nu$.
\end{definition}

When $0\le \mu\le 1$ and $\nu\le \mu$, the weights $\varphi_{q,\mu,\nu}(j\mid m)$
are nonnegative.\footnote{These conditions do not
exhaust the full range of $(q,\mu,\nu)$
for which the weights are nonnegative.
See, e.g., \cite[Section 6.6.1]{BorodinPetrov2016inhom}
for additional families of parameters leading to nonnegative weights.}
They also sum to one:
\begin{equation*}
\sum_{j=0}^{m}\varphi_{q,\mu,\nu}(j\mid m)=1,\qquad m\in\left\{ 0,1,\ldots  \right\}
\cup\left\{ +\infty \right\}.
\end{equation*}

\subsection{$q$-hypergeometric function and related quantities}

The unilateral basic hypergeometric series $_{k+1}\phi_{k}$ is defined via
\begin{equation}\label{q_hyp_defn}
	_{k+1}\phi_{k} \left(\begin{array}{ccc} a_{1} & \ldots & a_{k+1} \\ b_{1} & \ldots & b_{k} \end{array}; q, z \right) 
	= \sum_{n=0}^{\infty}
	\frac{(a_{1}, \ldots, a_{k+1}; q)_{n}}{(b_{1}, \ldots, b_{k}, q; q)_{n}} z^{n}.
\end{equation}
If one of $a_{j}$ is $q^{-y}$ for a positive integer $y$, then this series is
terminating. Otherwise we assume  $|q|, |z| < 1$ for the sum to be convergent.
In the terminating case, we also define the regularized version by
\begin{equation}
	\label{eq:phi_regularized}
	\begin{split}
	{}_{k+1} \overline{ \phi}_k
	\left(\begin{array}{cccc} q^{-n}&a_{1} & \ldots & a_{k} \\ &b_{1} & \ldots & b_{k} \end{array}; q, z \right)
	&:=
	(b_1,\ldots,b_k;q )_n\cdot\,
	_{k+1}\phi_{k} \left(\begin{array}{cccc}q^{-n}& a_{1} & \ldots & a_{k+1} \\ b_1&b_{2} & \ldots & b_{k} \end{array}; q, z \right) 
	\\&=
	\sum_{j=0}^{n}
	z^j\,\frac{(q^{-n};q)_j}{(q;q)_j}
	\,
	(a_1,\ldots,a_k ;q)_j
	\,
	(q^jb_1,\ldots,q^jb_k;q) _{n-j}.
	\end{split}
\end{equation}

The $q$-gamma and the $q$-beta functions are
\begin{equation} \label{eq:q_Gamma_q_Beta}
    \Gamma_q (X) = \frac{(q;q)_\infty}{(q^X;q)_\infty} (1-q)^{1-X}, \qquad \mathrm{B}_q(X,Y) = \frac{\Gamma_q (X) \Gamma_q (Y)}{\Gamma_q (X+Y)}, \qquad \text{for }X,Y>0.
\end{equation}

The $q$-hypergeometric distribution is
\begin{equation} \label{eq:q_hypergeom_distr}
	\psi_{q,a,b,c}(n)=\left( \frac{c}{a c} \right)^n \frac{(a,b;q)_n}{(c,q;q)_n} \frac{(c,c /(ab);q)_\infty}{(c/a,c/b;q)_\infty}.
\end{equation}
The fact that the weights \eqref{eq:q_hypergeom_distr} sum to one over $n\in \mathbb{Z}_{\ge0}$
follows from the Heine summation formula 
\cite[(II.8)]{GasperRahman}:
\begin{equation*}
	_{2}\phi_{1} \left(\begin{array}{cc} a & b \\ \multicolumn{2}{c}{c} \end{array}; q, c/(ab) \right) 
	=
	\frac{(c/a,c/b;q)_\infty}{(c,c / (ab);q)_\infty}.
\end{equation*}

\subsection{Spin Whittaker level quantities}
\label{app:sw_special}

It is well-known that $\Gamma_q(X)$ converges to $\Gamma(X)$ as $q \to 1$ uniformly for $X>0$,
where $\Gamma$ is the usual gamma function
$\Gamma(z)=\int_0^\infty e^{-t}t^{z-1}dt$, $z>0$
(e.g., see \cite{andrews1986q}). 
Hence $\mathrm{B}_q(X,Y) \to \mathrm{B}(X,Y)$ uniformly for $X,Y>0$,
where
$\mathrm{B}$ is the beta function
\begin{equation} \label{eq:Beta}
	\mathrm{B}(x,y) = \frac{\Gamma(x)\Gamma(y)}{\Gamma(x+y)} = \int_{0}^1 t^{x-1} (1-t)^{y-1} dt,
	\qquad  x,y >0.
\end{equation}

The \emph{inverse gamma distribution} 
$\Gamma^{-1}(\alpha)$
on $(0,+\infty)$ with a parameter $\alpha>0$
is 
\begin{equation}
	\label{eq:inverse_gamma_density}
	\Gamma^{-1}(\alpha)[x]=
	\frac{x^{-1-\alpha}e^{-1/x}}{\Gamma(\alpha)}.
\end{equation}

The \emph{beta distribution} on $(0,1)$ with (real) parameters $m,n>0$ has density
\begin{equation*}
    \mathpzc{B}(m,n)[x] = \frac{x^{m-1} (1-x)^{n-1} }{\mathrm{B}(n,m)} \qquad \text{for } x \in (0,1).
\end{equation*}
We also recall that a random variable with 
\emph{negative binomial} distribution has probability mass function
\begin{equation*}
    \mathpzc{NB}(r,p)[k] 
    =
    p^k (1-p)^r \binom{k+r-1}{k},
    \qquad
    \text{for } k\in \mathbb{Z}_{\ge 0}
\end{equation*}
and $r>0$, $0 \le p \le 1$.
Sampling $x$ in the interval $(0,1)$ with $\mathpzc{B}(m,n+k)$ law, where $k$ is a $\mathpzc{NB}(r,p)$ independent random variable generates the \emph{negative beta binomial} distribution on $(0,1)$. It has the probability density
\begin{equation} \label{eq:NBB}
    \mathpzc{NBB} (r,p,m,n)[x]
    =
		\frac{(1-p)^r x^{m-1} (1-x)^{n-1} }{\mathrm{B}(n,m)}\,
    {}_2F_1 
    \left(\begin{minipage}{1.55cm}
	\center{$r,n+m$}
	\\
	\center{$n$}
	\end{minipage} \Big\vert\, p(1-x)\right),
\end{equation}
where
we used the Gauss hypergeometric function
\begin{equation}
	\label{eq:2F1}
    {}_2F_1 
    \left(\begin{minipage}{0.85cm}
	\center{$a \, , \, b$}
	\\
	\center{$c$}
	\end{minipage} \Big\vert\, z\right)
	=
	\sum_{k \ge 0} \frac{(a)_k (b)_k}{(c)_k}
	\frac{z^k}{k!},
\end{equation}
and $(r)_k=r (r+1) \cdots (r+k-1)$ is the Pochhammer symbol.
Note that the inverse gamma, beta, and the negative beta binomial are continuous distributions,
while the negative binomial is a discrete distribution.


\section{Yang--Baxter equations}
\label{app:YBE}

In this section we list the Yang--Baxter equations 
used throughout the paper.
We employ the special function notation from \Cref{app:special}.

\subsection{sqW/sqW Yang--Baxter equation}

Let us introduce the cross vertex weight
\begin{equation}
    R_{x,y}(i_1,j_1;i_2,j_2) \coloneqq
    \mathbf{1}_{i_1 + j_1 = i_2 + j_2} \, \mathbf{1}_{i_1 \geq j_2}\,  (y/x)^{j_2}\, 
	\frac{(- s/y;q)_{j_2} (y/ x ; q )_{i_1 - j_2} (q;q)_{i_1} }{(q;q)_{j_2} (q;q)_{i_1 - j_2} (-s/x;q)_{i_1} }.
	\label{eq:R_matrix_WW}
\end{equation}

\begin{figure}[htbp]%
    \centering
    \subfloat[]{{\includegraphics[width=7cm]{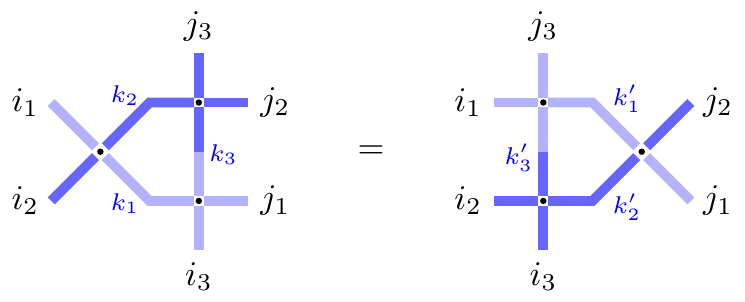} }}%
    \hspace{60pt}
    \subfloat[]{{\includegraphics[width=6cm]{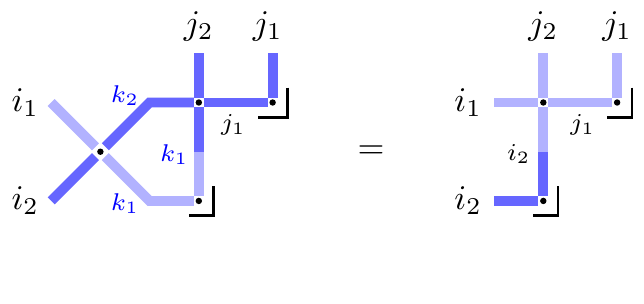}
    }}%
    \caption{Yang--Baxter equations
		\eqref{eq:YBE_RWW}, \eqref{eq:YBE_RWW_wall}
		correspond to local changes in the lattice illustrated by
		(a) and (b), respectively.}%
    \label{fig:YBE_W_W}%
\end{figure}

This cross vertex weight is involved in the following Yang--Baxter equations:
\begin{proposition}
\label{prop:YBE_bulk_W}
For any $i_1,i_2,i_3,j_1,j_2,j_3 \in \mathbb{Z}_{\ge 0}$, we have
\begin{equation} 
	\label{eq:YBE_RWW}
	\begin{split}
    \sum_{k_1,k_2,k_3 \ge 0}
		&R_{x,y}(i_2,i_1;k_2,k_1) W^{\bulk}_{y,s}(i_3,k_1;k_3,j_1)
		W^{\bulk}_{x,s}(k_3,k_2;j_3,j_2) 
		\\&\hspace{50pt}
    =
    \sum_{k_1',k_2',k_3' \ge 0}
    W^{\bulk}_{x,s}(i_3,i_2;k_3',k_2')
		W^{\bulk}_{y,s}(k_3',i_1;j_3,k_1') R_{x,y}(k_2',k_1';j_2,j_1),
	\end{split}
\end{equation}
where 
$W^{\bulk}$ are the bulk weights defined by 
\eqref{eq:Whit_W}.
See \Cref{fig:YBE_W_W}\,{\rm{}(a)} for a graphical interpretation.
\end{proposition}
\begin{proof}
	This is obtained in \cite[Corollary 4.3]{BorodinWheelerSpinq} via fusion 
	from the elementary Yang--Baxter equation for the 
	higher spin $\mathfrak{sl}_2$ vertex model.
	Note that the claim of 
	\cite[Corollary 4.3]{BorodinWheelerSpinq} contains a typo: the spectral parameters
	$x,y$ in the definition of the cross vertex weight should be swapped.
	This is corrected here by defining $R_{x,y}$ 
	in 
	\eqref{eq:R_matrix_WW}
	with parameters already swapped.
\end{proof}

\begin{proposition}
	\label{prop:YBE_corner_W}
For any $i_1,i_2,j_1,j_2 \in \mathbb{Z}_{\ge 0}$, we have
\begin{equation} \label{eq:YBE_RWW_wall}
	\begin{split}
		&
    \sum_{k_1,k_2 \ge 0}
    R_{x,y}(i_2,i_1;k_2,k_1) W_{y,s}^{\rightcorner}(k_1)
		W^{\bulk}_{x,s}(k_1,k_2;j_2,j_1) W_{x,s}^{\rightcorner}(j_1)
		\\&\hspace{200pt}
    =
    W_{x,s}^{\rightcorner}(i_2)
		W^{\bulk}_{y,s}(i_2,i_1;j_2,j_1) W_{y,s}^{\rightcorner}(j_1),
	\end{split}
\end{equation}
where $W^{\rightcorner}$ are the right corner weight
defined by \eqref{eq:W_corner}.
See \Cref{fig:YBE_W_W}\,{\rm{}(b)} for an illustration.
\end{proposition}
\begin{proof}
Expanding both right and left-hand side of \eqref{eq:YBE_RWW_wall} and simplifying common factors we end up with the identity
\begin{equation*}
    \sum_{k=j_1}^{i_2} (y/x)^{k-j_1} \frac{(y/x;q)_{i_2-k}}{(q;q)_{i_2-k}} \frac{(-s x;q)_{k-j_1}}{(q;q)_{k-j_1}}=\frac{(-s y ; q)_{i_2-j_1}}{(q;q)_{i_2-j_1}},
\end{equation*}
which follows from the $q$-Chu--Vandermonde identity
(e.g., \cite[(II.6)]{GasperRahman}).
\end{proof}

\begin{figure}[htbp]
    \centering
		\includegraphics[width=\textwidth]{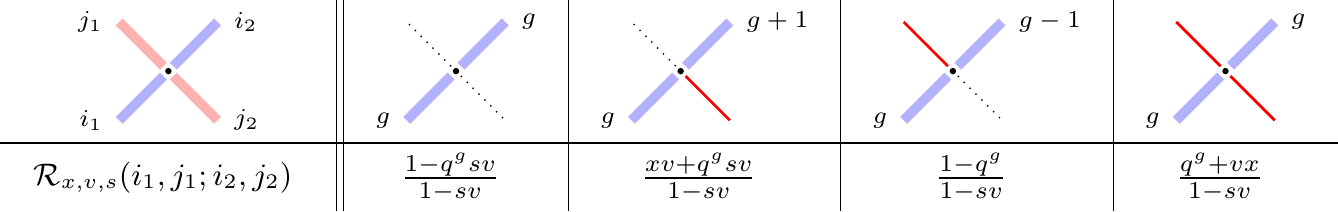}
    \caption{The cross vertex weights involved in the Yang--Baxter equations
		for the sHL and sqW vertex weights.
		Note that these weights vanish unless 
		$i_1+j_2=j_1+i_2$.}
    \label{fig:table_R_cal}
\end{figure}

\subsection{Yang--Baxter equations with dual weights}

Our additional Yang--Baxter equations
involve the dual sHL weights 
$w^*_{v,s}$ which are given in 
\Cref{fig:table_w_star} in the text
and the dual sqW weights 
\eqref{eq:Wstar_leftwall}--\eqref{eq:Wstar_bulk}.
We use the 
cross vertex weights
$\mathcal{R}_{x,v,s}$ given in \Cref{fig:table_R_cal}
and the 
following cross vertex weights:
\begin{equation}\label{eq:Whittaker_cross_weight}
    \begin{split}
        \mathbb{R}_{x,y,s}(i_1,j_1;i_2,j_2)
		& \coloneqq
		\mathbf{1}_{i_2 + j_1 = i_1 + j_2 }\,
		\frac{ q^{ i_2 i_1 +\frac{1}{2}j_2(j_2 -1) }(s x)^{j_2} (q;q)_{j_1}  }
		{ (s^2;q)_{j_1 + i_2} (q;q)_{j_2} (q;q)_{i_2} (-q/(s x );q)_{i_1 -j_1} } 
		\\
		& \hspace{100pt} \times 
		{}_4 \overline{ \phi}_3
		\left(\begin{minipage}{5.2cm}
		\center{$q^{-i_2}; q^{-i_1}, -s y,  -q/(s x)$}
		\\
		\center{$-s/ x,q^{1+j_2-i_2}, -y q^{1-i_1-j_2}/s$}
		\end{minipage} \Big\vert\, q,q\right).
	\end{split}
\end{equation}
Here
${}_4 \overline{ \phi}_3$
is the regularized
$q$-hypergeometric function
\eqref{eq:phi_regularized}.
We remark that one of the first 
${}_4 \overline{ \phi}_3$
type formulas for vertex weights of the fused
six vertex model appeared in \cite{Mangazeev2014}.
See also \cite{CorwinPetrov2015}, \cite{BorodinPetrov2016inhom}
for a probabilistic explanation of the fusion procedure
which goes back to \cite{KulishReshSkl1981yang}.

Next we list Yang--Baxter equations involving 
a usual and a dual family of vertex weights. There are
two instances of these Yang--Baxter equations,
one with sqW/sHL weights, and another with sqW/sqW weights.
Moreover, each of these has two different forms, in the bulk and 
at the boundary. In total there are four statements.
The bulk statements are available
from 
\cite{BufetovMucciconiPetrov2018}
(and also can be found in \cite{borodin_wheeler2018coloured}),
and the statements on the boundary need to be proven.

\begin{figure}[htbp]
    \centering
    \subfloat[]{{\includegraphics[width=7cm]{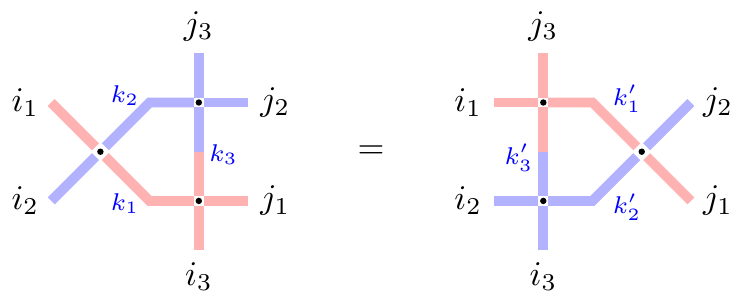} }}%
    \hspace{60pt}
    \subfloat[]{{\includegraphics[width=6cm]{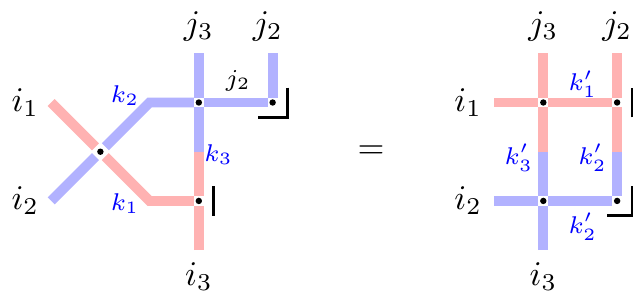} }}%
    \caption{Graphical representation of the Yang--Baxter equation for dual weights.}%
    \label{fig:YBE_W_Wstar}%
\end{figure}

\begin{proposition}
	\label{prop:sHL_sqW_YBE_bulk}
For any $i_1,j_1 \in \{0,1\}$ and $i_2,i_3,j_2,j_3 \in \mathbb{Z}_{\ge 0}$ we have
\begin{equation} \label{eq:YBE_RWw_star}
	\begin{split}
		&
    \sum_{k_1,k_2,k_3 \ge 0}
    \mathcal{R}_{x,v,s}(i_2,i_1;k_2,k_1)\, 
		w^{*,\bulk}_{v,s}(i_3,k_1;k_3,j_1)\,
		W^{\bulk}_{x,s}(k_3,k_2;j_3,j_2) 
		\\
		&\hspace{80pt}
    =
    \sum_{k_1',k_2',k_3' \ge 0}
		W^{\bulk}_{x,s}(i_3,i_2;k_3',k_2')\,
		w^{*,\bulk}_{v,s}(k_3',i_1;j_3,k_1')\, 
		\mathcal{R}_{x,v,s}(k_2',k_1';j_2,j_1).
	\end{split}
\end{equation}
See \Cref{fig:YBE_W_Wstar}\,{\rm{(a)}} for a graphical interpretation.
\end{proposition}
\begin{proof}
	This is \cite[(A.11)]{BufetovMucciconiPetrov2018}.
\end{proof}

\begin{proposition}
	\label{prop:sHL_sqW_YBE_border}
	For any $i_1\in\left\{ 0,1 \right\}$ and $i_2,i_3,j_2,j_3 \in \mathbb{Z}_{\ge 0}$, we have
\begin{equation} \label{eq:YBE_RWw_star_wall}
	\begin{split}
    &
		\sum_{k_1,k_2,k_3 \ge 0}
    \mathcal{R}_{x,v,s}(i_2,i_1;k_2,k_1)
		\,
		W^{*,\rightwall}(i_3,k_1;k_3)\,
		W^{\bulk}_{x,s}(k_3,k_2;j_3,j_2)\,
		W_{x,s}^{\rightcorner}(j_2)
		\\&\hspace{50pt}=
    \sum_{k_1',k_2',k_3' \ge 0}
		W^{\bulk}_{x,s}(i_3,i_2;k_3',k_2')\,
		W_{x,s}^{\rightcorner}(k_2')\,
    w^*_{v,s}(k_3',i_1;j_3,k_1')
		\,
		W^{*,\rightwall}(k_2',k_1';j_2).
	\end{split}
\end{equation}
See \Cref{fig:YBE_W_Wstar}\,{\rm{(b)}} for a graphical interpretation.
\end{proposition}
\begin{proof}
Consider separately the cases $i_1=0$ and $i_1=1$. 
Start with $i_1=0$. 
We see that \eqref{eq:YBE_RWw_star_wall} is nontrivial only when 
$i_2+i_3 = j_2+j_3$ 
(say $j_2=i_2+i_3-j_3$) and $j_3 \ge i_2$.
Under these conditions we have
\begin{equation*}
    \begin{split}
    &
    \mathcal{R}_{x,v,s}(i_2,0;i_2,0) W^{\bulk}_{x,s}(i_3,i_2;j_3,i_2+i_3-j_3) W_{x,s}^{\rightcorner}(i_2+i_3-j_3) \\
    &
    \hspace{25pt}
    +
    \mathcal{R}_{x,v,s}(i_2,0;i_2+1,1) W^{\bulk}_{x,s}(i_3-1,i_2+1;j_3,i_2+i_3-j_3) W_{x,s}^{\rightcorner}(i_2+i_3-j_3)
    \\
    &
    =
    w^{*}_{v,s}(j_3,0;j_3,0) W^{\bulk}_{x,s}(i_3,i_2;j_3,i_2+i_3-j_3) W_{x,s}^{\rightcorner}(i_2+i_3-j_3) \\
    &
    \hspace{25pt}
    +
    w^{*}_{v,s}(j_3-1,0;j_3,1) W^{\bulk}_{x,s}(i_3,i_2;j_3-1,i_2+i_3-j_3+1) W_{x,s}^{\rightcorner}(i_2+i_3-j_3+1).
    \end{split}
\end{equation*}
After the required simplifications, the previous relation reduces to 
\begin{multline*}
    (1-q^{i_2}s v)(1+s x q^{j_3-i_2-1}) + (xv+s v q^{i_2})(1 - q^{j_3-i_2}) 
    \\
    = 
		(1-sv q^{j_3})(1+s x q^{j_3-i_2-1}) + xv (1- q^{j_3-i_2})(1 - s^2 q^{j_3-1}),
\end{multline*}
that can be checked directly.

When $i_1=1$, as in the $i_1=0$ case, \eqref{eq:YBE_RWw_star_wall} is an equality between sums of at most two terms that after simplification reduces to
\begin{multline*}
    (1-q^{i_2})(1+s x q^{j_3-i_2}) + (q^{i_2} + xv)(1 - q^{j_3-i_2+1}) 
    \\
		= (1-q^{j_3+1})(1+s x q^{j_3-i_2}) + x (1- q^{j_3-i_2+1})(v - s q^{j_3}),
\end{multline*}
which is again checked directly.
\end{proof}

\begin{proposition}
	\label{prop:YBE_sqW_sqW}
For any $i_1,j_1,i_2,i_3,j_2,j_3 \in \mathbb{Z}_{\ge 0}$ we have
\begin{equation} \label{eq:YBE_RWW_star}
	\begin{split}
		&
    \sum_{k_1,k_2,k_3 \ge 0}
		\mathbb{R}_{x,y,s}(i_2,i_1;k_2,k_1)\, W^{*,\bulk}_{y,s}(i_3,k_1;k_3,j_1)\,
		W^{\bulk}_{x,s}(k_3,k_2;j_3,j_2) 
    \\
		&\hspace{50pt}=
    \sum_{k_1',k_2',k_3' \ge 0}
		W^{\bulk}_{x,s}(i_3,i_2;k_3',k_2')\,
		W^{*,\bulk}_{y,s}(k_3',i_1;j_3,k_1')\, \mathbb{R}_{x,y,s}(k_2',k_1';j_2,j_1).
	\end{split}
\end{equation}
See \Cref{fig:YBE_W_Wstar}\,{\rm{(a)}} for a graphical interpretation.
\end{proposition}
\begin{proof}
	This is \cite[(A.13)]{BufetovMucciconiPetrov2018}.
\end{proof}

\begin{proposition}
	\label{prop:YBE_sqW_sqW_corner}
For any $i_1,i_2,i_3,j_1,j_2,j_3 \in \mathbb{Z}_{\ge 0}$, we have
\begin{equation} \label{eq:YBE_RWW_star_wall}
	\begin{split}
		&
    \sum_{k_1,k_2,k_3 \ge 0}
    \mathbb{R}_{x,y,s}(i_2,i_1;k_2,k_1) W^{*,\rightwall}(i_3,k_1;k_3)
    W^{\bulk}_{x,s}(k_3,k_2;j_3,j_2) W_{x,s}^{\rightcorner}(j_2)\\
		&\hspace{30pt}=
	\frac{(-sy;q)_\infty}{(s^2;q)_\infty}
    \sum_{k_1',k_2',k_3' \ge 0}
    W^{\bulk}_{x,s}(i_3,i_2;k_3',k_2')
    W_{x,s}^{\rightcorner}(k_2')
		W^{*,\bulk}_{y,s}(k_3',i_1;j_3,k_1') W^{*,\rightwall}(k_2',k_1';j_2).
	\end{split}
\end{equation}
See \Cref{fig:YBE_W_Wstar}\,{\rm{(b)}} for a graphical interpretation.
\end{proposition}
\begin{proof}
	This follows from the analogous relation
	\eqref{eq:YBE_RWw_star_wall}. In fact both the R-matrix $\mathbb{R}$
	and the vertex weight $W^{*,\bulk}$ can be constructed respectively
	from $\mathcal{R}$ and $w^{*,\bulk}$ via fusion with respect to the 
	spectral parameter $v$ (see
	\cite{BorodinWheelerSpinq}, \cite{BufetovMucciconiPetrov2018} for details). The coefficient $\frac{(-sy;q)_\infty}{(s^2;q)_\infty}$ arises from fusion of $w^{*,\bulk}$ and does not simplify since in the left hand side of \eqref{eq:YBE_RWw_star_wall} the bulk weight $w^{*,\bulk}$ is missing. One
	can check that the fusion procedure preserves identity
	\eqref{eq:YBE_RWw_star_wall} and hence \eqref{eq:YBE_RWW_star_wall}
	holds.
\end{proof}

We also need a useful summation identity that was stated
in a slightly a more general form in 
Proposition A.5 of \cite{BufetovMucciconiPetrov2018}:
\begin{proposition}
For $|xy|<1$ and under the usual conditions
$0<q<1$ and $-1<s<0$, we have
\begin{equation} \label{eq:sum_R_matrix}
	\sum_{i,j \ge 0} \mathbb{R}_{x,y,s}(0,0;i,j)
	=
	\sum_{j=0}^{\infty}
	(xy)^{j}\,
	\frac{(-s/x;q)_j(-s/y;q)_j}{(s^2;q)_j(q;q)_j}
	=
	\frac{(-s x;q)_{\infty} (-s y ;q)_\infty}
	{(s^2;q)_{\infty} (x y ;q)_\infty}.
\end{equation}
\end{proposition}

The R-matrix $\mathbb{R}_{x,y,s}$ is positive if we assume its parameters to be in a specific range:
\begin{proposition}[{\cite[Proposition A.8]{BufetovMucciconiPetrov2018},
	see also \cite[Proposition 3.1]{CMP_qHahn_Push}}] \label{prop:positivity_R}
    Let us take the parameters
		$q\in (0,1)$, $s\in (-\sqrt{q},0)$ and $x,y \in [-s,-s^{-1}]$.
		Then $\mathbb{R}_{x,y,s}(i_1,j_1;i_2,j_2) \ge 0$ for all $i_1,i_2,j_1,j_2 \ge 0$.
\end{proposition}

\section{Proof of Proposition \ref{prop:skew_SW_uniform_conv}} \label{app:proof_prop}

    \begin{lemma}[\cite{CorwinBarraquand2015Beta}, Lemma 2.2] \label{lemma:BC_lemma_uniform}
    Let $A,B>0$. Then 
    \begin{equation}
			\label{eq:lemma_2_2_qPochhammer}
        \lim_{q \to 1}
        \frac{(\ell q^{A};q)_\infty}{(\ell q^{B};q)_\infty} = (1-\ell)^{B-A},
    \end{equation}
    uniformly in $\ell$ belonging to any compact subset of $(0,1)$.
    \end{lemma}
		Note that the uniformity in $\ell$
		in \eqref{eq:lemma_2_2_qPochhammer}
		is not claimed in 
		\cite{CorwinBarraquand2015Beta}
		but easily follows from the uniformity of 
		all Taylor expansions involved in the proof in the cited paper
		(which we do not reproduce).
    
    \begin{lemma} \label{lemma:bound_ratio_pochhammer}
        Let $A,B>0$. Then, for all $n \in \mathbb{Z}_{\ge 1}$ and all $q \in (\frac{1}{2},1)$, we have
        \begin{equation} \label{eq:bound_ratio_pochhammer}
            \frac{(q^{A+n};q)_\infty}{(q^{B+n};q)_\infty} \le c \, (1-q^n)^{B-A},
        \end{equation}
        where $c$ is a constant independent of $q,n$.
    \end{lemma}
    \begin{proof}
    Set $q=e^{-\varepsilon}$. The result of the Lemma is restated, taking the logarithm of both sides of \eqref{eq:bound_ratio_pochhammer}, as
    \begin{equation} \label{eq:bound_log_pochhammer}
        \sum_{k \ge 0} \log \frac{(1 - e^{- \varepsilon( A+n+k )})}{(1 - e^{- \varepsilon (B+n+k)})} - (B-A) \log (1-e^{- \varepsilon n}) \le c',
    \end{equation}
		for all $\varepsilon \in (0,-\log 2)$ and a constant $c'$
		independent of $\varepsilon, n$. Using Lagrange mean value theorem, we
		can rewrite the generic term of the infinite sum as
    \begin{equation*}
        \log \frac{(1 - e^{-\varepsilon(A+n+k)})}{(1 - e^{- \varepsilon (B+n+k) })} = (A-B)\, \frac{\varepsilon}{e^{\varepsilon (\widetilde{t}_k + n + k)} -1},
    \end{equation*}
    where numbers $\widetilde{t}_k$ belong to the interval $(\min(A,B), \max(A,B))$. We show that for any positive bounded sequence $\{t_k\}_k \subset (0,M)$, with $M$ fixed, the quantity 
    \begin{equation} \label{eq:_remainder_ratio_pochhammer}
        \sum_{k\ge 0} \frac{\varepsilon}{e^{ \varepsilon (t_k + n +k )} -1} + \log ( 1 - e^{-\varepsilon n})
    \end{equation}
		is absolutely bounded uniformly in $\varepsilon$ and $n$ and this would prove \eqref{eq:bound_log_pochhammer} and hence \eqref{eq:bound_ratio_pochhammer}. To evaluate the infinite sum over $k$ we fix a positive constant $\delta$ and distinguish two cases.

\smallskip
\noindent\textbf{Case 1, $k\ge \delta/\varepsilon$}. 
We use the estimate
        \begin{equation*}
            \frac{\varepsilon}{e^{\varepsilon (t_k + n + k)} -1} \le e^{-\varepsilon k} \frac{\varepsilon}{1- e^{-\delta} },
        \end{equation*}
        that implies, summing over $k$,
        \begin{equation} \label{eq:bound_k_large}
            \sum_{ k \ge \delta / \varepsilon } \frac{\varepsilon}{e^{\varepsilon (t_k + n + k)} -1} \le \frac{1}{ e^{\delta} -1 } \frac{\varepsilon}{1-e^{-\varepsilon}} \le \frac{2 \log 2}{ e^{\delta} -1 }.
        \end{equation}

\smallskip
\noindent\textbf{Case 2, $k < \delta/\varepsilon$}. In this case we use again Lagrange
				mean value theorem to express the denominator of the generic term of the summation of \eqref{eq:_remainder_ratio_pochhammer} as
        \begin{equation*}
            \frac{\varepsilon}{e^{\varepsilon (t_k + n + k)} -1} = \frac{e^{-\varepsilon \xi_{k,n} }}{t_k + n + k }, \qquad \text{for some } \xi_{k,n} \in (0,t_k+n+k).
        \end{equation*}
        This implies the bounds
        \begin{equation} \label{eq:lower_upper_bound}
            \frac{e^{-\varepsilon (M+n+k)}}{M+n+k} \le \frac{\varepsilon}{e^{\varepsilon (t_k + n + k)} -1} \le \frac{1}{n+k}.
        \end{equation}
        We focus first on the lower bound given by the first inequality in \eqref{eq:lower_upper_bound}. Summing over $k$ we find
        \begin{equation*}
            \sum_{k=0}^{\delta/\varepsilon}
            \frac{e^{-\varepsilon (M+n+k)}}{M+n+k}
            \ge
            \int_{0}^{\delta/\varepsilon}
            \frac{e^{-\varepsilon (M+n+k)}}{M+n+k} dk 
            = \int_{\varepsilon(M+n)}^{\delta+\varepsilon(M+n)} \frac{e^{-k'}}{k'} dk' \ge
            \int_{\varepsilon(M+n)}^{\delta+\varepsilon(M+n)} \left(\frac{1}{k'} -1 \right) dk',
        \end{equation*}
        which gives
        \begin{equation} \label{eq:bound_k_small_1}
            \sum_{k=0}^{\delta/\varepsilon} \frac{\varepsilon}{e^{\varepsilon (t_k + n + k)} -1} \ge \log\left( 1+ \frac{\delta}{\varepsilon(M+n)} \right) - \delta.
        \end{equation}
        We turn now our attention to the second inequality in \eqref{eq:lower_upper_bound} and, since
        \begin{equation*}
            \sum_{k=0}^{\delta/\varepsilon} \frac{1}{n+k} \le \int_n^{\delta/\varepsilon+n} \frac{dk}{k-1/2},
        \end{equation*}
        we obtain
        \begin{equation} \label{eq:bound_k_small_2}
            \sum_{k=0}^{\delta/\varepsilon} \frac{\varepsilon}{e^{\varepsilon (t_k + n + k)} -1} \le \log \left( 1+ \frac{\delta}{\varepsilon (n - 1/2)} \right).
        \end{equation}

				\smallskip
    Combining results obtained from the analysis of cases $k\ge \delta/\varepsilon$ and $k <  \delta/\varepsilon$ in \eqref{eq:bound_k_large}, \eqref{eq:bound_k_small_1} \eqref{eq:bound_k_small_2} 
    we can finally write
    \begin{equation*}
        \log\left( (1 + \frac{\delta}{\varepsilon(M+n)}) (1-e^{-\varepsilon n}) \right) + \bigO(\delta) \le \eqref{eq:_remainder_ratio_pochhammer} \le \log\left( (1 + \frac{\delta}{\varepsilon(n-1/2)}) (1-e^{-\varepsilon n}) \right) + \bigO(\delta).
    \end{equation*}
    This concludes our proof since the arguments of the logarithms in the left and right-hand side are bounded functions for $\varepsilon \in (0, \log 2)$ and $n\ge 1 $.
    \end{proof}
    
    For the next lemma we define the quantity
    \begin{equation*}
        \mathcal{A}_{S,X}^{(q)}(\ell_3,\ell_2,\ell_1) = \frac{1}{\Delta_q (\ell_3,\ell_2,\ell_1) } \frac{(q^{S-X};q)_{n_1-n_2}}{(q;q)_{n_1-n_2}}
        \frac{(q^{S+X};q)_{n_2-n_3}}{(q;q)_{n_2-n_3}}
        \frac{(q;q)_{n_1-n_3}}{(q^{2S};q)_{n_1-n_3}},
    \end{equation*}
		where we assumed $1\le \ell_3 \le \ell_2 \le \ell_1$ and $n_i=\floor{\log_q (1/\ell_i)}$. Here  
		$\Delta_q$ is defined in \eqref{eq:Delta_q}. 
		We think of $\mathcal{A}_{S,X}^{(q)}$ as a $q$-deformation of
		$\mathcal{A}_{S,X}$ \eqref{eq:A}.
    \begin{lemma} \label{lemma:A_integral_unif_conv}
        For any continuous function $f(\ell_2)$ we have
        \begin{equation} \label{eq:A_integral_unif_conv}
            \lim_{q \to 1}
            \int_{\ell_3}^{\ell_1} f(\ell_2) \mathcal{A}_{S,X}^{(q)}(\ell_3,\ell_2,\ell_1) \frac{d \ell_2}{\ell_2}
            =
            \int_{\ell_3}^{\ell_1} f(\ell_2) \mathcal{A}_{S,X}(\ell_3,\ell_2,\ell_1) \frac{d \ell_2}{\ell_2},
        \end{equation}
        uniformly for any $\ell_3 \le \ell_1$ bounded away from $\infty$.
    \end{lemma}
    \begin{proof}
    Fix small positive $\delta$.
		We will prove our claim distinguishing two cases, based on the
		distance between $\ell_3$ and $\ell_1$.
    
		\smallskip
		\noindent\textbf{Case 1, $\ell_1 - \ell_3 > \delta$}. The integral in the left-hand side of \eqref{eq:A_integral_unif_conv} can be decomposed as 
        \begin{equation*}
            \int_{\ell_3}^{\ell_1}=
            \int_{\ell_3 + \delta /2}^{\ell_1-\delta/2}
            + 
            \int_{\ell_3}^{\ell_3 + \delta /2}
            +
            \int_{\ell_1-\delta/2}^{\ell_1}.
        \end{equation*}
        When $\ell_3+\delta/2 \le \ell_2 \le \ell_1 - \delta/2$ , by virtue of Lemma \ref{lemma:BC_lemma_uniform}, we have
        \begin{equation*}
            \int_{\ell_3 + \delta/2}^{\ell_1 - \delta/2} f(\ell_2) \mathcal{A}^{(q)}_{S,X}(\ell_3, \ell_2, \ell_1) \frac{d \ell_2}{\ell_2} \xrightarrow[q \to 1]{}
            \int_{\ell_3 + \delta/2}^{\ell_1 - \delta/2} f(\ell_2) \mathcal{A}_{S,X}(\ell_3, \ell_2, \ell_1) \frac{d \ell_2}{\ell_2},
        \end{equation*}
        uniformly. 
        
        On the other hand, when $\ell_3 \le \ell_2 < \ell_3 + \delta/2$ we use estimates
        \begin{equation*}
            \mathcal{A}^{(q)}_{S,X}(\ell_3, \ell_2, \ell_1)
            \le
            \begin{dcases}
            C \mathcal{A}_{S,X}(\ell_3, \ell_2, \ell_1) \qquad & \text{if } n_3 < n_2,
            \\
            \frac{C}{\Delta_q (\ell_3,\ell_2,\ell_1)} \left( \frac{1-q}{1-\ell_3/\ell_1} \right)^{S+X}
            \qquad & \text{if } n_3 = n_2.
            \end{dcases},
        \end{equation*}
				valid for some constant $C$ independent of $q$ of $\ell_2$ and
				that can be deduced using \Cref{lemma:bound_ratio_pochhammer}
				and identity \eqref{eq:ratio_q_pochhammer}. This implies that
        \begin{equation*}
            \int_{\ell_3}^{\ell_3 + \delta/2} f(\ell_2) \mathcal{A}^{(q)}_{S,X}(\ell_3, \ell_2, \ell_1) \frac{d \ell_2}{\ell_2}
            =
            \bigO(\delta) + \bigO\left( \frac{1-q}{\delta} \right)^{S+X}.
        \end{equation*}
        In an analogous fashion one can also show that
        \begin{equation*}
            \int^{\ell_1}_{\ell_1 - \delta/2} f(\ell_2) \mathcal{A}^{(q)}_{S,X}(\ell_3, \ell_2, \ell_1) \frac{d \ell_2}{\ell_2}
            =
            \bigO(\delta) + \bigO\left( \frac{1-q}{\delta} \right)^{S-X}.
        \end{equation*}
        This concludes the proof of \eqref{eq:A_integral_unif_conv} when $\ell_1-\ell_3>\delta$.

\smallskip
\noindent\textbf{Case 2, $\ell_1 - \ell_3 \le \delta$}.
				Assuming $\delta$ is very small, for any $\ell_2 \in[\ell_3, \ell_1]$, we can write, by continuity, $f(\ell_2)= f(\ell_1) + \smallO (1)$, where $\smallO (1)$ tends to 0 as $\delta \to 0$. Thus, we have
        \begin{equation*}
            \begin{split}
                \int_{\ell_3}^{\ell_1} f(\ell_2) \mathcal{A}^{(q)}_{S,X}(\ell_3, \ell_2, \ell_1) \frac{d \ell_2}{\ell_2}
                &
                =
                \sum_{n_2 = n_3}^{n_1} \left( f(\ell_1) + \smallO (1) \right) q^{-(S+X)(n_1-n_2)} \varphi_{q,q^{S+X},q^{2S}}(n_1-n_2| n_1-n_3) 
                \\
                &
                = f(\ell_1) + \smallO (1),
            \end{split}
        \end{equation*}
        and by Lemma \ref{eq:delta_function} this concludes the analysis of the case $\ell_1 - \ell_3 \le \delta$.
    
				\smallskip

				Since all the estimates we provided are controlled as functions of $\delta$,
				the convergence in \eqref{eq:A_integral_unif_conv} 
				is uniform provided that $\ell_1$ stays bounded.
    \end{proof}
    \begin{proof}[Proof of Proposition \ref{prop:skew_SW_uniform_conv}] The integral in the left-hand side of \eqref{eq:skew_SW_uniform_conv} is equal to
    \begin{equation*}
    \begin{split}
        &
        \int_{L_{N,N}}^{L_{N,N-1}} \frac{dL_{N-1,N-1}}{L_{N-1,N-1}} \mathcal{A}^{(q)}_{S,X} (L_{N,N}, L_{N-1,N-1}, L_{N,N-1})\cdots
        \\
        &
        \cdots
        \int_{L_{N,2}}^{L_{N,1}} \frac{dL_{N-1,1}}{L_{N-1,1}} \mathcal{A}^{(q)}_{S,X} (L_{N,2}, L_{N-1,1}, L_{N,1}) \left( \frac{ L_{N-1,N-1} \cdots L_{N-1,1} }{L_{N,N} \cdots L_{N,1} } \right)^X
        f(\underline{L}_{N-1})
    \end{split}
    \end{equation*}
    and we can 
		take the $q\to1$ limit in
		each of the $N-1$ integrals using \Cref{lemma:A_integral_unif_conv}. 
		This establishes the convergence 
		to the right-hand side of 
		\eqref{eq:skew_SW_uniform_conv} as $q \to 1$,
		uniformly on any compact subset of~$\mathpzc{W}_N$. 
    \end{proof}
    
\section{Triangular Sums}
\label{app:triangular_sums}

Here we write down a number of identities of summations of certain symbols $a_{k,\ell}, b_\alpha$ used in
the proof of 
\Cref{prop:Pieri_to_Toda}.
Fix a positive integer $N$, and 
assume that the symbols $b_\alpha$, $\alpha=1,\ldots,N $ commute with each other.
Let $a_{k,\ell}$ be
\begin{equation*}
	a_{k,\ell}
	=
	\begin{dcases}
		0 \qquad & \text{if $0=k$, or $\ell=N+1$};
		\\
		1 \qquad & \text{if } 0\le k = \ell \le N;
		\\
		\in \mathbb{R} \qquad & \text{else.} 
	\end{dcases}
\end{equation*}
\begin{proposition} \label{prop:summation_identities}
For any $N\ge 1$, the following identities hold
\begin{align*}
	&\sum_{0\le k < \ell \le N} a_{k+1,\ell} - a_{k,\ell} -a_{k+1,\ell+1} + a_{k,\ell+1} 
	= 
	N -\sum_{j=1}^{N-1} a_{j,j+1};
	\\
	&\sum_{0\le k < \ell \le N} (\ell-k+1)(a_{k+1,\ell+1} - a_{k,\ell+1}) - (\ell - k -1)(a_{k+1,\ell} - a_{k,\ell} ) 
	= 
	\sum_{j=1}^{N-1} a_{j,j+1};
	\\
	&\sum_{0\le k < \ell \le N} (a_{k+1,\ell} - a_{k,\ell} 
	- a_{k+1,\ell+1} + a_{k,\ell+1} )\sum_{\alpha=k+1}^\ell b_{\alpha} 
	= 
	\sum_{\alpha=1}^{N}  b_{\alpha};
	\\
	&\sum_{0\le k < \ell \le N} (\ell - k -1)^2(a_{k+1,\ell} - a_{k,\ell} ) -(\ell - k +1)^2 ( a_{k+1,\ell+1} - a_{k,\ell+1}) 
	\\&\hspace{180pt}
	= -\sum_{j=1}^{N-1} a_{j,j+1} - 2 \sum_{1\le k < \ell \le N} a_{k,\ell};
	\\
	&\sum_{0\le k < \ell \le N} (a_{k+1,\ell} - a_{k,\ell} - a_{k+1,\ell+1} + a_{k,\ell+1} )\sum_{k+1 \le \alpha, \beta \le l } b_{\alpha} b_{\beta}
	=
	\sum_{\alpha=1}^N b_\alpha^2 + 2 \sum_{1\le k < \ell \le N} a_{k,\ell} b_{k}b_{\ell};
	\\
	&\sum_{0\le k < \ell \le N} \Big( (\ell-k+1)(a_{k+1,\ell+1} - a_{k,\ell+1}) - (\ell - k -1)(a_{k+1,\ell} - a_{k,\ell} ) \Big)\sum_{\alpha=k+1}^\ell b_{\alpha} 
	\\&\hspace{180pt}= 
	\sum_{1\le k<\ell \le N}^{N} a_{k,\ell} (b_{k}-b_{\ell}).
\end{align*}
\end{proposition}
\begin{proof}
	All these identities
	are elementary and can be 
	proven by induction in a straightforward way.
\end{proof}

\bibliographystyle{alpha}
\bibliography{bib}

\end{document}